\documentclass[a4paper]{amsart}
\pdfoutput=1 
\usepackage{verbatim}
\usepackage{amssymb, mathtools, amsthm}
\usepackage{pinlabel}
\usepackage[normalem]{ulem}
\usepackage[shortlabels]{enumitem}
\usepackage{stackengine}
\usepackage{tikz}
\usepackage{multirow}
\usepackage{mathabx}
\usetikzlibrary{calc,positioning}
\usepackage{tikz-3dplot}
\usetikzlibrary{3d}
\usepackage[only,llbracket,rrbracket]{stmaryrd}
\usepackage[labelfont=up]{subcaption}
\usepackage{graphicx}
\usepackage{longtable}
\usepackage{pinlabel}
\usepackage{hyperref}
\usepackage[normalem]{ulem}
\usepackage{microtype}

\hypersetup{
    pdftitle={RAAGedy right-angled Coxeter groups},    
    pdfauthor={Christopher H. Cashen, Pallavi Dani, Alexandra
      Edletzberger, Annette Karrer},     
    pdfkeywords={Right-angled Coxeter group, right-angled Artin group, RACG,
  RAAG, quasiisometry}, 
    colorlinks=true,       
    linkcolor=black,          
    citecolor=black,        
    filecolor=black,      
    urlcolor=black           
}

\NewCommandCopy{\proofqedsymbol}{\qedsymbol}
\newcommand{\exampleqedsymbol}{$\diamond$}

\AtBeginEnvironment{proof}{\renewcommand{\qedsymbol}{\proofqedsymbol}}

\theoremstyle{definition}

\theoremstyle{plain}
\newtheorem{theorem}{Theorem}[section]
\newtheorem{lemma}{Lemma}[section]
\newtheorem{proposition}{Proposition}[section]
\newtheorem{corollary}{Corollary}[section]

\newtheorem{question}{Question}[section]

\theoremstyle{remark}

\newtheorem*{remark}{Remark}

\theoremstyle{definition}
\newtheorem{definition}{Definition}[section]
\newtheorem{example}{Example}[section]
\newtheorem{numberedremark}{Remark}[section]
\AtBeginEnvironment{example}{%
  \pushQED{\qed}\renewcommand{\qedsymbol}{\exampleqedsymbol}%
  }
\AtEndEnvironment{example}{\popQED\endexample}

\def\makeautorefname#1#2{\expandafter\def\csname#1autorefname\endcsname{#2}}
\let\fullref\autoref

\makeautorefname{theorem}{Theorem} 
\makeautorefname{lemma}{Lemma} 
\makeautorefname{proposition}{Proposition} 
\makeautorefname{corollary}{Corollary}
\makeautorefname{numberedremark}{Remark} 
\makeautorefname{definition}{Definition}
\makeautorefname{example}{Example}
\makeautorefname{section}{Section}
\makeautorefname{subsection}{Section}
\makeautorefname{subsubsection}{Section}
\makeautorefname{claim}{Claim}
\makeautorefname{question}{Question}

\makeatletter 
\let\c@lemma=\c@theorem 
\makeatother
\makeatletter 
\let\c@proposition=\c@theorem 
\makeatother
\makeatletter 
\let\c@corollary=\c@theorem 
\makeatother
\makeatletter 
\let\c@definition=\c@theorem 
\makeatother
\makeatletter 
\let\c@example=\c@theorem 
\makeatother
\makeatletter 
\let\c@question=\c@theorem 
\makeatother
\makeatletter 
\let\c@numberedremark=\c@theorem 
\makeatother
\makeatletter
\@addtoreset{claim}{proposition}
\makeatother
\makeatletter
\@addtoreset{claim}{lemma}
\makeatother

\mathtoolsset{centercolon} 

\newcommand{\bdry}{\partial} 
\newcommand{\act}{\curvearrowright} 
\newcommand{\from}{\colon\thinspace} 

\makeatletter
\DeclareRobustCommand{\mod}{%
  \mathbin{\mathpalette\scaletoplus@{\mkern1mu\%\mkern1mu}}%
}
\newcommand{\scaletoplus@}[2]{%
  \vcenter{%
    \hbox{\sbox\z@{$\m@th#1+$}\resizebox{\wd\z@}{!}{$\m@th#1#2$}}%
  }%
}
\makeatother

\DeclareMathOperator{\diam}{diam}
\DeclareMathOperator{\lk}{lk}
\DeclareMathOperator{\st}{st}
\DeclareMathOperator{\supp}{supp}
\newcommand{\X}{\mathcal{X}}
\newcommand{\Davis}{\Sigma}
\newcommand{\Salvetti}{\Sigma}
\newcommand{\compliant}{\mathcal{C}}
\newcommand{\diag}{\boxslash}
\newcommand{\nbhd}{\mathcal{N}}
\newcommand{\join}{*}
\newcommand{\edge}{\mathbin{\tikz[baseline=-\the\dimexpr\fontdimen22\textfont2\relax
    ]{\filldraw (0,0) circle (1pt) (.2,0) circle (1pt);\draw
      (0,0)--(.2,0);}}}
\newcommand{\longdash}{\mathbin{\tikz[baseline=-\the\dimexpr\fontdimen22\textfont2\relax ]{\draw (0,0)--(1,0);}}}
\DeclareMathOperator{\double}{\mathfrak{D}}
\renewcommand{\setminus}{-}

\newcommand{\bridge}{\mathbin{\tikz[baseline=-.7ex] 
  \draw[line width=0.2mm] (0,.1) -- (0.025,0.05) -- (0.225,0.05) --
  (0.25,0.1)  (0,-.1) -- (0.025,-0.05) -- (0.225,-0.05) -- (0.25,-0.1);}}

\DeclareMathOperator{\ric}{Ric}
\newcommand{\Act}{.}
\newcommand{\mprg}{\Pi}
\newcommand{\wall}{\mathcal{H}}
\newcommand{\cint}{\stackrel{c}{\cap}}
\newcommand{\ceq}{\stackrel{c}{=}}
\newcommand{\quadrat}{Q}

\newcommand{\ten}{\textsc{x}}
\newcommand{\elf}{\textsc{e}}

\begin{document}

\title{RAAGedy right-angled Coxeter groups}
\author[Cashen]{Christopher H.\ Cashen}
\address{Faculty of Mathematics\\University of
  Vienna\\Oskar-Morgenstern-Platz 1\\1090 Vienna, Austria\\
\href{https://orcid.org/0000-0002-6340-469X}{\includegraphics[scale=.75]{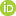}
  0000-0002-6340-469X}}
\email{christopher.cashen@univie.ac.at}
\author[Dani]{Pallavi Dani}
\address{Department of Mathematics, Louisiana State University, Baton
  Rouge, LA 70803–4918, USA\\
 \href{https://orcid.org/0000-0001-7653-2618}{\includegraphics[scale=.75]{ORCID-iD_icon-16x16} 0000-0001-7653-2618}}
\email{pdani@math.lsu.edu}
\author[Edletzberger]{Alexandra Edletzberger}
\address{Faculty of Mathematics\\University of
  Vienna\\Oskar-Morgenstern-Platz 1\\1090 Vienna, Austria\\
\href{https://orcid.org/0000-0002-6584-5149}{\includegraphics[scale=.75]{ORCID-iD_icon-16x16}
  0000-0002-6584-5149}}
\email{alexandra.edletzberger@gmail.com}
\author[Karrer]{Annette Karrer}
\address{Department of Mathematics, The Ohio State University\\231 W. 18th Ave.\\
Columbus, OH 43210, USA\\
\href{https://orcid.org/0000-0003-3761-5861}{\includegraphics[scale=.75]{ORCID-iD_icon-16x16} 0000-0003-3761-5861}}
\email{karrer.14@osu.edu}

\date{\today}


\begin{abstract}
We give criteria for deciding whether or not a triangle-free
simple graph is the presentation graph of a right-angled Coxeter group
that is quasiisometric to some right-angled Artin group, and, if so,
producing a presentation graph for such a right-angled Artin group.

We introduce two new graph modification operations, cloning and
unfolding, to go along with an existing operation called link
doubling. These operations change the presentation graph but not the
quasiisometry type of the resulting group.
We give criteria on the graph that imply it can be transformed by these
operations into a graph that is recognizable as presenting a
right-angled Coxeter group commensurable to a right-angled Artin group. 

In the converse direction we derive coarse geometric obstructions to
being quasiisometric to a right-angled Artin group, first by specializing
existing results from the literature to this setting, then by
developing new approaches using configurations of maximal product
regions.
In all cases we give sufficient graphical conditions that imply these
geometric obstructions. 

We implemented our criteria on a computer and applied them to an
enumeration of small graphs. 
Our methods completely answer the motiving question when
the graph has at most 10 vertices.

\end{abstract}

\maketitle

\tableofcontents

\section{Introduction}\label{sec:intro}

We would like to understand  the large scale geometry of right-angled
Coxeter groups (RACGs).
Previous work on this problem has focused on understanding a related
collection of quasiisometry invariants (divergence, thickness, and
hypergraph index) \cite{DanTho15,BehFalHag18,Lev18, Lev19, Lev22},
understanding certain hyperbolic and relatively hyperbolic cases
\cite{CriPao08, BehHagSis17cox,DanTho17, DanStaTho18, HruStaTra20,
  BouXie20}, or certain other hyperbolic-like features
\cite{MR3966609, GraKarLaz21, FioKar22, Karrer23}, and cases with nontrivial JSJ
decompositions \cite{DanTho17, NguTra19, Edl24}. See also the survey \cite{DaniRACGsurvey}.

The hyperbolic and relatively hyperbolic cases have exponential
divergence, while linear divergence corresponds to being a product. 
This leaves the quadratic divergence case as the simplest
`interesting' case orthogonal to hyperbolicity, in the sense of having
polynomial divergence.
This class of RACGs is relatively unexplored and wide open for
investigation.
Furthermore, this class of groups contains, up to
commensurability, all 1-ended right-angled Artin groups (RAAGs)~\cite{DavJan00}, and 
there is extensive work on understanding the large scale geometry of
RAAGs \cite{BehCha12,MR2727658,BehNeu12,
  MR2421136, Hua16ii, MR3692971,Hua18, MR3761106,Mar20}. So, it is natural to ask when a given RACG with quadratic divergence is quasiisometric to some RAAG, in which case we say it is
\emph{RAAGedy}.
By extension, we say that a simple graph $\Gamma$ is \emph{RAAGedy} if it is
the presentation graph of a  right-angled Coxeter group $W_\Gamma$ that is RAAGedy.

\begin{question}\label{classification_question}
  Which RACGs are  RAAGedy? 
\end{question}
Surprisingly little is known about this problem.
There is a graph property \emph{CFS} characterizing quadratic
divergence, which is therefore necessary for a graph to
be RAAGedy.
This and other existing results are reviewed in \fullref{sec:previous_work}.
We develop multiple criteria to address \fullref{classification_question}.
These are summarized in \fullref{sec:summary_main_results}.
We take a `hands dirty' approach: large scale
geometric features are interpreted in terms of the presentation graph
of the Coxeter group, so that our criteria are effectively verifiable.
In fact, we have computerized\footnote{Code available at: \href{https://github.com/cashenchris/RACG}{https://github.com/cashenchris/RACG}} all of the constructions, enumerated small
triangle-free CFS graphs, and then checked all of them for
RAAGediness.
The results are shown in \fullref{fig:venn}, where the
region labels are the list items from \fullref{sec:summary_main_results}. 

\begin{figure}[h!]
  \centering\tiny
  \labellist
  \pinlabel {planar} [c] at 0 112
  \pinlabel {non-planar} [c] at 0 100
  \pinlabel {RAAGedy} [c] at 80 135
  \pinlabel {not RAAGedy} [c] at 200 135
  \pinlabel {DL} [c] at 110 90
  \pinlabel {DJ} [c] at 43 90
  \pinlabel {CND} [c] at 70 30
  \pinlabel {ZZ} [c] at 191 76
  \pinlabel {CC} [c] at 240 25
  \pinlabel {NS} [c] at 223 23
  \pinlabel {NRR} [r] at 246 91
  \pinlabel {L} [c] at 258 70
  \pinlabel {$\Xi$R} [c] at 15 15
  \pinlabel {$\Xi$NR} [l] at 250 15
  \pinlabel 67 [c] at 100 110 
  \pinlabel 124 [c] at 75 110 
  \pinlabel 7 [c] at 56.5 110 
  \pinlabel 156 [c] at 115 80 
  \pinlabel 3 [c] at 40 83 
  \pinlabel 1175 [c] at 70 70 
  \pinlabel 1 [c] at 28.5 69 
  \pinlabel 232 [c] at 55 30  
  \pinlabel 42 [c] at 25 15 
  \pinlabel 56 [r] at 215 110
  \pinlabel 2 [c] at 250 15  
  \pinlabel 341 [c] at 250 33   
  \pinlabel 53 [c] at 259 65    
  \pinlabel 299 [c] at 250 52  
  \pinlabel 7 [l] at 242 90      
  \pinlabel 87 [c] at 226 44  
  \pinlabel 46 [c] at 243 75   
  \pinlabel 140 [c] at 233 63 
  \pinlabel 113 [c] at 210 23 
  \pinlabel 29 [c] at 217 33  
  \pinlabel 655 [c] at 190 65 
  \pinlabel 74 [c] at 210 53   
  \pinlabel 192 [c] at 219 72 
  \pinlabel 29 [c] at 208 39  
  \pinlabel 8? [c] at 135 15 
  \endlabellist
  \includegraphics{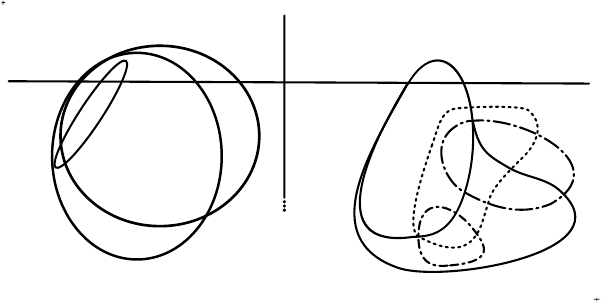}
  \caption{The 3938 isomorphism types of triangle-free CFS graphs with
    at most 11 vertices. There are exactly 8 graphs, all with 11 vertices, for which we do not know if they are
  RAAGedy/non-RAAGedy. They are listed in \fullref{sec:furtherexamples}.}
  \label{fig:venn}
\end{figure}
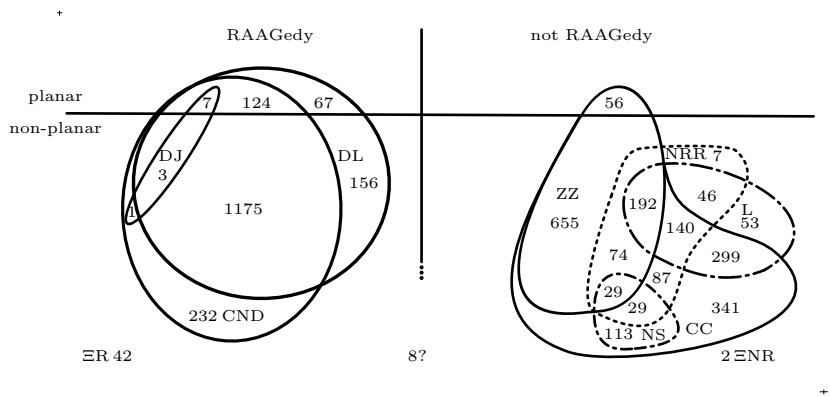

Preexisting work covers only the `planar' part of the diagram, 
the regions labelled `DL' and `DJ', and a handful of isolated examples, so the figure shows that our new
constructions vastly improve the state of knowledge
about \fullref{classification_question}, and almost completely settle
it for small triangle-free graphs.
It also illustrates the inherent complexity of the problem---
particularly on the non-RAAGedy side we see several overlapping
but independent reasons that a graph can fail to be RAAGedy.

\medskip

Practical criteria for answering
\fullref{classification_question} are an important outcome of this
work, but the techniques developed to establish these criteria also
support our broader aim of understanding the connections
between the combinatorial properties of presentation graphs and the
large scale geometry of the resulting right-angled Coxeter groups.
By systematically identifying and exploiting these connections, we
establish a foundation for exploring more general questions about the
quasiisometric classification for RACGs.

\subsection{Background}\label{sec:previous_work}
We summarize the current state of knowledge of
\fullref{classification_question}:

There is a standard reduction of the quasiisometry problem to the case
of 1--ended factors of the Grushko--Stallings--Dunwoody decomposition of
a finitely presented group.
In both the RAAG and RACG cases these are special subgroups, so they are again RAAGs or RACGs,
respectively.
Thus, for RAAGs we restrict to the case that $\Gamma$ is connected
with more than one vertex, and for RACGs we restrict to the case that
$\Gamma$ is incomplete without separating cliques.

A RAAG $A_\Delta$ has linear divergence when it is a product, which happens
when $\Delta$ is a join, and it has quadratic
divergence whenever $\Delta$ is not a join and $A_\Delta$ is 1--ended \cite[Corollary~4.8]{BehCha12}.
One-ended RACGs have a richer divergence spectrum
\cite{DanTho15,Lev19,Lev22}, and the divergence can be calculated from
the structure of $\Gamma$. In particular, linear
divergence corresponds to $W_\Gamma$ being a product of infinite subgroups, which
occurs when $\Gamma$ is a thick join, and quadratic
divergence occurs when $\Gamma$ is not a join and has property
CFS (component of full support/constructed from squares)
\cite{DanTho15, BehFalHag18, Lev18}.
Thus, for a 1--ended, irreducible RACG $W_\Gamma$, the CFS property
for $\Gamma$ is necessary for it to be RAAGedy.

Behrstock \cite{MR3966609} mentions the `Folk Question' of whether
quadratic divergence for RACGs implies commensurability to a RAAG and gives a
counterexample, that is not even RAAGedy, by constructing a RACG with
quadratic divergence containing a stable subgroup (see
\fullref{fig:Behrstock} and \fullref{sec:morse}).

Nguyen and Tran \cite{NguTra19} show that Behrstock's obstruction is
not the only kind. They give a complete characterization of RAAGedy
graphs that are triangle-free, have no separating cliques, and are planar.
Their proof uses planarity in an essential way to port the problem to
the realm of 3--manifolds.

In the RAAGedy direction there are constructions of Davis and
Januszkiewicz \cite{DavJan00} and Dani and Levcovitz \cite{DanLev24}
that give sufficient conditions on $\Gamma$ for $W_\Gamma$ to be
commensurable to a RAAG.

\subsection{Summary of our main
  results}\label{sec:summary_main_results}
\paragraph{\it RAAGedy criteria:}
There are two existing criteria for showing $W_\Gamma$ is
commensurable to a RAAG:
\begin{enumerate}
\item[(DJ)] The graph is a `double' as in \fullref{def:graphdoubles},
  so $W_\Gamma$ is commensurable to a RAAG by a result of Davis and Januszkiewicz.
  \item[(DL)] The graph satisfies the conditions of Dani and Levcovitz, so
    $W_\Gamma$ has a finite-index RAAG subgroup. 
  \end{enumerate}
There are also some graph operations that change a graph in such a way that 
the resulting RACG is a finite-index subgroup of the one we started
with.
These include \emph{link doubling}
  (\fullref{sec:double}), as well as some others
  (\fullref{sec:other_fi}).
We call a graph a \emph{near double} if it can be turned into a
double via a
sequence of link doubling operations. 
It turns out that there are concise criteria
(\fullref{recognizeneardouble}) for recognizing when a graph $\Gamma$ is a near
double in terms of its \emph{twin graph} $\Gemini(\Gamma)$
(\fullref{def:twingraph}), and when a graph is a near double it is
always possible to change it into a double using only one or two link
doubling steps.

We develop two more graph modification operations, 
\emph{cloning} (\fullref{sec:clone}) and \emph{unfolding} (\fullref{sec:unfolding})
that allow us to change the graph $\Gamma$ without changing the
quasiisometry type of $W_\Gamma$.
Cloning, in particular, leads to a nice generalization of
\fullref{recognizeneardouble} as \fullref{einzelkind}.
We call graphs satisfying the hypotheses of \fullref{einzelkind}
\emph{coarse near doubles}. These include doubles and near doubles. 
\begin{enumerate}
\item[(CND)] If $\Gamma$ is a coarse near double, then $\Gamma$ is
  RAAGedy.
  A particularly simple case is that if every vertex of $\Gamma$ has a
  twin then $\Gamma$ is RAAGedy. See \fullref{einzelkind}.
\end{enumerate}
Finally, any graph that can be changed into one of our known RAAGedy
types using a sequence of graph modification operations is also RAAGedy:
\begin{enumerate}
\item[($\Xi$R)] If $\Gamma$ can be changed into a coarse near double
  or a graph satisfying Dani and Levcovitz's conditions by a sequence
  of link doubles, clonings, and unfoldings, then $\Gamma$ is
  RAAGedy. See \fullref{sec:qiclasses}.
\end{enumerate}

\paragraph{\it Non-RAAGedy crietria:}
We also work in the opposite direction, detecting geometric obstructions
to a graph being RAAGedy. 

In \fullref{sec:minsquare} and \fullref{sec:morse} we describe such
obstructions by exploring the Morse subgroups of 2--dimensional RACGs. The presence of 1--ended,
infinite-index Morse subgroups is a known obstruction to being
RAAGedy.
Having a stable cycle in $\Gamma$ is a sufficient condition
for the existence of such a subgroup.
We show by example that stable cycles can appear after iterated link
doubling, and recount a condition for the Morse boundary of $W_\Gamma$
to be totally disconnected, which rules out such behavior. 

The results of  \fullref{sec:JSJ}  and \fullref{mprg} can be seen as
two incarnations of a general strategy that has been used often in 
quasiisometric rigidity arguments (see
the survey \cite[Chapter~25]{DruKap18}): find a geometrically
distinguished feature and make a combinatorial object encoding
interactions between subspaces so distinguished.
One such geometrically distinguished feature is that of being a
coarsely separating quasiline.
Classes of parallel quasilines are grouped together
into `cylinders', and JSJ theory organizes these cylinders into a
tree, the JSJ tree of cylinders. 
The second version of distinguished feature is top dimensional
quasiflats. These are grouped together into maximal product regions and
there is a maximal
product region graph (MPRG) encoding their intersections.
In both situations quasiisometries of the base groups induce isomorphisms
of the corresponding combinatorial objects. 

In \fullref{sec:JSJ}  we compare JSJ decompositions of RAAGs and RACGs
and derive several quasiisometry obstructions:
\begin{enumerate}
  \item[(NRR)] If the JSJ graph of cylinders of $W_\Gamma$ contains a
    rigid vertex whose group is not quasiisometric to a RAAG, then
    $\Gamma$ is not RAAGedy, see \fullref{sec:JSJ}. 
\item[(ZZ)] If the JSJ graph of cylinders of $W_\Gamma$ contains a virtually
  $\mathbb{Z}^2$ edge incident to a not virtually $\mathbb{Z}^2$ rigid
  vertex then $\Gamma$ is not RAAGedy, by \fullref{ZZobstruction}.
  \item[(CC')] If the JSJ graph of cylinders of $W_\Gamma$ contains a
    collection of cylinders forming a cycle of cuts, as in
    \fullref{no_cycle_of_cuts}, then $\Gamma$ is not RAAGedy. 
\end{enumerate}

 In \fullref{sec:mprg_connectivity} we make precise the fact that the MPRG of a RAAG is a (connected) quasitree with a 
1--bottleneck property. 
Then we give criteria on the presentation graph of a RACG that show
this is not always the case for RACGs.

\begin{enumerate}
  \item[(NS)] $\Gamma$ is not RAAGedy if it is not strongly CFS,
    \fullref{raagedyimpliestrongcfs}, because strongly CFS is
    equivalent to connectivity of the MPRG.
  \item[(L)]
We define a `ladder' in the MPRG as a subgraph that is a coarse axis for
the action of an infinite order element of $W_\Gamma$ on its MPRG that is too wide to be compatible with the 1--bottleneck property.
  We give sufficient conditions for the existence of a ladder in
  \fullref{thm:ladders}, which therefore prevent $\Gamma$ from being
  RAAGedy. 
  A novel point here is that we are really discovering an invariant that has a fine
      dependence on the isometry type of the MPRG, not just its
      quasiisometry type.
      That is, an MPRG containing a ladder might still be a quasitree,
      it is just not a quasitree in precisely the same way that the
      MPRG of a RAAG is. 
    \end{enumerate}
   In \fullref{sec:compliant_cycles} we construct a bespoke
   obstruction that is tailored to our specific problem comparing
   RACGS to RAAGs.
   To do this, we leverage a slight difference between the behavior of
      closest point projection to standard subcomplexes in RAAGs and
      RACGs:  In RAAGs there is a dichotomy, the coarse intersection of two
      standard subcomplexes has diameter either infinte or 0.
      In RACGs it can be finite and nonzero, so many small diameter
      projections can add up to a large total.

   It is a fact that maximal product regions are standard subcomplexes
   in RAAGs and RACGs, so quasiisometries between them coarsely
   preserve this particular family of standard subcomplexes.
   We bootstrap from this fact to inductively define  (\fullref{def:constructible_compliant}) a class of
   \emph{compliant subcomplexes} such that quasiisometries send them
   close to other compliant subcomplexes.
      
    \begin{enumerate}
    \item[(CC)] 
      We deduce an obstruction to being RAAGedy if $\Davis_\Gamma$
      contains a cycle $X_0,\dots, X_{n-1}$ of compliant subcomplexes
      such that consecutive $X_i$ come close to one another and 
      for each $i\neq 0$ the diameter of the  projection of  $X_i$ to $X_0$ is
      finite but the diameter of the union of all the projections is
      large.
      Such a cycle can exist for RAAGedy $\Gamma$ only if $\Gamma$
      satisfies some very restrictive conditions, as described in 
      \fullref{nocycles}, so if we additionally rule these out then
      $\Gamma$ is not RAAGedy.
    \end{enumerate}
    It turns out that condition (CC) generalizes conditions (CC') and
    (ZZ).
    
    Finally, we reconsider the graph modification operations:
      \begin{enumerate}
   \item[($\Xi$NR)] If $\Gamma$ can be changed into one of the above
     non-RAAGedy types  by a sequence
  of link doubles, clonings, and unfoldings, then $\Gamma$ is not
  RAAGedy. 
\end{enumerate}

The practical criteria defining conditions (CC) and, especially, (L) are
designed to take some link doubling into account, and what we actually
computed for the purpose of \fullref{fig:venn} is whether there is a sequence
of at most 3 link doubles after which the presentation graph has the
desired property.
Even allowing this, we see that (CC), (L), and (NRR) appear to be independent.

  \subsection{Comparison to Huang-Kleiner}
  Huang and Kleiner \cite{MR3761106} have results on groups
quasiisometric to RAAGs, but their paper is very different.
They start with a fixed RAAG $A_\Delta$, with the additional hypothesis that
it has finite outer automorphism group, and classify finitely generated groups
quasiisometric to $A_\Delta$.
Our goal is start with a RACG $W_\Gamma$ and decide whether or not it
can be quasiisometric to any RAAG whatsoever, and, if so,  produce $\Delta$
such that $W_\Gamma$ and $A_\Delta$ are quasiisometric.
Furthermore, many of the groups we consider do not have the rigid geometry
associated with RAAGs with finite outer automorphism group.
The extra flexibility makes it harder to find obstructions to being
quasiisometric to some RAAG.

\subsection*{Acknowledgements}
This paper contains material from the PhD thesis of A.E.\ \cite{Edl24thesis}, submitted in partial fulfilment of the requirements for the degree of
Doktorin der Naturwissenschaft (Dr.\ rer.\ nat.) at the University of
Vienna.

This research was funded in whole or in part by the Austrian Science
Fund (FWF) \href{https://doi.org/10.55776/P34214}{10.55776/P34214} and
\href{https://doi.org/10.55776/PAT7799924}{10.55776/PAT7799924}.

This material is based upon work supported by the National Science Foundation under Award No.~2407104 and Award No.~2407438.

Part of this work was conducted during the Erwin Schr\"odinger Institute
Research in Teams program ``Rigidity in Coxeter groups'', July 2023, C.C.\
and P.D.\ members.

We also thank Jingyin Huang and Sangrok Oh.

\section{Preliminaries}\label{sec:prelim}
In this section we fix terminology and notation, and recall some
existing technology that can be used in quasiisometry results, which
we will apply in subsequent sections in the special case of
2--dimensional RACGs and RAAGs.
We also adapt some ideas that are present in the literature to forms
that will be useful to us later on. 
This makes the paper more self-contained, and serves as a warm-up.
We do not regard this material as new.

\subsection{Graphs}\label{prelim:graphs}
A \emph{simple graph}, also known as a simplicial graph, is a
1--dimensional simplicial complex, or, equivalently, it is a
collection of vertices and edges such that every edge connects
distinct vertices and two edges have at most one
vertex in common.
{\bf All graphs in this paper are simple.}
On the rare occasion that we want to entertain the possibility of
self-loops or multiple edges between a pair of vertices we will call
that structure a \emph{multigraph}.
A \emph{full} or \emph{induced} subgraph $\Gamma'$ of a graph $\Gamma$ is one that
has an edge between two vertices if and only if $\Gamma$ does.
A \emph{spanning} subgraph is one that contains every vertex of
$\Gamma$.
The graph is \emph{complete} if it contains an edge between every pair
of distinct vertices, and \emph{incomplete} otherwise. 
A (possibly empty) set of vertices form a \emph{clique} if they induce a complete
subgraph, and an \emph{anticlique} if there is no edge between any
pair of vertices in the set.
The \emph{join} of nonempty $\Gamma$ and $\Gamma'$, written $\Gamma \join \Gamma'$, is the graph made from the
disjoint union of $\Gamma$ and $\Gamma'$ by adding an edge between
every vertex of $\Gamma$ and every vertex of $\Gamma'$.
A graph is a \emph{cone} if it can be written as the join of a subgraph
with a singleton. 
The singleton is then called a \emph{cone vertex}.
We can also \emph{cone-off a subgraph} $\Gamma'$ of $\Gamma$ by adding
to $\Gamma$ a new vertex $c$ connected to every vertex of
$\Gamma'$.
This is written as $\Gamma \join_{\Gamma'} c$.

A graph is a \emph{suspension} if it can be written as a join of a subgraph with a 2 point anticlique.
The anticlique is called the \emph{pole} or the \emph{suspension points} for this
particular realization of the graph as a suspension.

If $\Gamma'$ is a subgraph of $\Gamma$, an \emph{$n$--chord} is a path
of length $n$ in $\Gamma$ whose endpoints are vertices of $\Gamma'$
whose path distance in $\Gamma'$ is strictly greater than $n$.
Saying a subgraph has no 1--chords is equivalent to it being an
induced subgraph. 

If two graphs $\Gamma$ and $\Gamma'$ are isomorphic, we write $\Gamma \cong \Gamma'$.

When treating a connected graph as a metric object, we always use the path
metric induced by metrizing edges as unit intervals.

\subsubsection{Twins and satellites}\label{sec:twins}
A vertex $v\neq w$ is a \emph{satellite} of $w$ if
$\lk(v)\subset\lk(w)$.
In the context of RAAGs there is an existing notion of a relation
between vertices called the \emph{domination relation (of Servatius)},
where $w$ dominates $v$ if $\lk(v)\subset\st(w)$.
In a triangle-free graph $w$ dominates $v$ if $v$ is a satellite of
$w$ or if $v$ is a leaf at $w$.

Vertices $v\neq w$ are \emph{twins} if $\lk(v)=\lk(w)$.\footnote{In
  the literature the definition given here is sometimes called `open
  twins' or `false twins' and there is another definition in which $v$ and $w$ are
twins (`closed twins' or `true twins')  if $\{v\}\cup\lk(v)=\{w\}\cup\lk(w)$, which requires twins to be
adjacent.
In our setting of connected, incomplete, triangle-free graphs, if $v$
and $w$ are adjacent then
$\lk(v)\setminus\{w\}\neq\lk(w)\setminus\{v\}$, so only the definition
with $v$ and $w$ nonadjacent is meaningful.}

A \emph{module} $M$ in a graph $\Gamma$ is a set of vertices such that
for all $v\in \Gamma\setminus M$, either every vertex in $M$ is
adjacent to $v$ or no vertex in $M$ is adjacent to $v$.
A module is \emph{non-trivial} if it is a proper subset containing
more than one vertex.

There is a general fact that any partition of $\Gamma$ into
modules gives a quotient graph with one vertex for each module.
For two modules $M_0$ and $M_1$ in the partition, either
$M_0*M_1\subset \Gamma$ or there are no edges of $\Gamma$ between
$M_0$ and $M_1$.
The quotient graph contains an edge between $M_0$ and $M_1$ when
$M_0*M_1\subset\Gamma$.

Given a vertex $v$ of $\Gamma$, the set consisting of $v$ and all its twins forms a module; call this a \emph{twin module}.
The set of twin modules partitions $\Gamma$ into disjoint anticliques.

\begin{definition}\label{def:twingraph}
 The \emph{twin graph} $\Gemini(\Gamma)$ of $\Gamma$ is the quotient
 graph coming
from the decomposition of $\Gamma$ into twin modules.
\end{definition}

\subsubsection{RACGs and RAAGs}
Given a finite graph $\Gamma$ the \emph{right-angled Artin
  group with presentation graph $\Gamma$} is the group $A_\Gamma$
whose generators are the vertices of $\Gamma$, and whose defining
relations are commutation relations corresponding to pairs of vertices that are connected by
an edge in $\Gamma$. The \emph{right-angled Coxeter group with
  presentation graph $\Gamma$} is the group $W_\Gamma$ that includes
the same generators and relations as $A_\Gamma$ and additional
defining relations saying that each generator has order 2.
In both cases, an induced subgraph $\Gamma'$ of $\Gamma$ determines a
\emph{special subgroup} $A_{\Gamma'}<A_\Gamma$ or $W_{\Gamma'}<W_\Gamma$. 
Tits' solution to the word problem for Coxeter groups
\cite[Sec.~3.4]{davisbook} implies that for $w\in
W_{\Gamma'}<W_\Gamma$, every minimal length word representing $w$ in
$W_\Gamma$ uses generators only from $\Gamma'$.
The same is true in RAAGs \cite{MR1285550}.

We will call a RAAG or RACG \emph{irreducible} if the given
presentation graph is not a join. This is equivalent to saying that
the group does not split as a direct product of nontrivial subgroups \cite{439976,MR2333366}.

Both families of groups admit cocompact actions on CAT(0) cube
complexes, which have the additional property that the 1--skeleton is
the Cayley graph corresponding to the generating set defined by the
presentation graph.
For a RAAG defined by $\Delta$, this cube complex is the universal
cover of the Salvetti complex, denoted $\Salvetti(A_\Delta)$ or
$\Salvetti_\Delta$.
For a RACG defined by $\Gamma$, the cube complex is the Davis complex,
denoted $\Davis(W_\Gamma)$ or $\Davis_\Gamma$.
It will be clear from context whether $\Sigma_\Gamma$ refers to a
Salvetti complex or a Davis complex. We will usually use $\Gamma$ for
the presentation graph of a RACG and $\Delta$ for the presentation
graph of a RAAG.
A consequence of word problem result above is that for
$\Gamma'<\Gamma$ there is a convex subcomplex
$\Davis_{\Gamma'}\subset\Davis_\Gamma$, and, again, the analogue is
also true in RAAGs.

For both RAAGs and RACGs, there is a
bijection between join subgraphs of the presentation graph and special
subgroups that split as direct products. 
We will be interested in the case where both factors of the
direct product are infinite groups.
For RAAGs this is always true, but for RACGs we need an additional
condition that both factors of the join subgraph are incomplete.
In the context of RACGs we use the term \emph{thick join} to mean
$A\join B\subset \Gamma$ with $A$ and
$B$ incomplete.

For further background on RAAGs and RACGs, see \cite{Cha07, davisbook,DaniRACGsurvey,Dav24}.

\subsubsection{Doubles}\label{sec:double}
\begin{definition}\label{def:graphdoubles}
  Let $\Gamma$ be a graph.
  \begin{itemize}
  \item The \emph{double} $\double(\Gamma)$ of $\Gamma$ is the 
graph with vertex set  $\Gamma\times\{0,1\}$ such that for every 
edge $v\edge w$ in $\Gamma$ there are four edges 
$(v,\epsilon)\edge(w,\delta)$ for $\epsilon,\delta\in\{0,1\}$.
\item The \emph{link double of $\Gamma$ at a vertex $v$} is the 
  graph:\[ \double^\circ_v(\Gamma):=\left(\Gamma\times\{0,1\}/\{(w,0)\sim
    (w,1)\mid w\in\lk(v)\}\right)\setminus \left(\{v\}\times\{0,1\}\right)\]
  \item The \emph{star double of $\Gamma$ at a vertex $v$} is the 
  graph: \[\double^*_v(\Gamma):=\Gamma\times\{0,1\}/\{(w,0)\sim(w,1)\mid
    w\in\st(v)\}\]
    \end{itemize}
  \end{definition}

These doubles have the following, well known, algebraic significance:
  \begin{lemma}\label{vertexdouble}
        Let $\Gamma$ and $\Delta$ be graphs. 
  \begin{itemize}
  \item There is an exact sequence:
    \[1\to W_{\double^\circ_v(\Gamma)}\to W_\Gamma\to\mathbb{Z}_2\to
      1\]
    The first map sends $(w,0)\mapsto w$ and
    $(w,1)\mapsto v^{-1}wv$ for all $w\in\Gamma\setminus\{v\}$. The second map sends $v$
    to 1 and all other generators to 0.
  \item There is an exact sequence:
    \[1\to A_{\double^*_v(\Delta)}\to A_\Delta\to\mathbb{Z}_2\to 1\]
    The maps are as in the previous case, with the addition
    that $(v,0)\mapsto v^2$.
  \end{itemize}
\end{lemma}

\begin{theorem}[{Davis and Januszkiewicz \cite{DavJan00}}]\label{thm:davisjanuszkiewicz}
  $A_\Gamma$ is commensurable to $W_{\double(\Gamma)}$.
\end{theorem}

\subsection{More finite-index subgroup constructions}
\fullref{vertexdouble} gave us one way to pass to a finite-index RACG
subgroup of a RACG, and \fullref{thm:davisjanuszkiewicz} gave us one
way to find a commensurable RAAG.
In this section we give an additional construction of each of these types. 
\subsubsection{Some other finite-index RACG subgroups}\label{sec:other_fi}
In the RAAG case of \fullref{vertexdouble} the `2' is not special; one
can take the kernel of the map $A_\Delta\to\mathbb{Z}_n$ obtained by
killing the $n$--th power of $v$ and all of the other generators, and express
it as the RAAG on the graph obtained by taking $n$--copies of $\Delta$
and identifying them along the star of $v$.
Obviously, we cannot make such a short exact sequence with a RACG, but
actually we can do a similar graph construction to get an index-$n$
subgroup if $\Gamma$ contains a pair of twins, as follows:

\begin{proposition}
Fix $n>2$ and a pair of twins $v$ and $w$ in $\Gamma$.
Let $\{a_0,\dots,a_{\ell-1}\}$ be the vertices of $\lk(v)=\lk(w)$.
Consider $$\Gamma':=\Gamma\times\{0,\dots, n-1\}/\{(x,i)\sim(x,0) \mid
x\in \{v,w\}\join\{a_0,\dots,a_{\ell-1}\}, 1\le i <n\}.$$  Then there is an injective homomorphism $\iota: W_{\Gamma'} \to W_{\Gamma}$ such that $\iota(W_{\Gamma'})$ is an index-$n$ subgroup of $W_{\Gamma}$.
\end{proposition}
See \fullref{sec:catzero} for a refresher on walls in cube complexes.
\begin{proof}
Let $\{b_0,\dots, b_{m-1}\}$ be the vertices of $\Gamma$ not in $\{v,w\}\join\{a_0,\dots,a_{\ell-1}\}$. 
Define the homomorphism $\iota$ as follows:  
\begin{align*}
  (v,0)&\mapsto v\\
  (w,0)&\mapsto
         \begin{cases}
           (wv)^{(n-1)/2}w(vw)^{(n-1)/2} &\text{if $n$ is odd}\\
           (wv)^{(n-2)/2}wvw(vw)^{(n-2)/2} &\text{if $n$ is even}
         \end{cases}\\
  (a_i,0)&\mapsto a_i\\
  (b_j,k)&\mapsto  \begin{cases}
           (wv)^{k/2}b_j(vw)^{k/2} &\text{if $k$ is even}\\
           (wv)^{(k-1)/2}wb_jw(vw)^{(k-1)/2} &\text{if $k$ is odd}
         \end{cases}
\end{align*}

To see that the image is really an index-$n$ subgroup, consider the
$vw$--bicolored geodesic $\Davis_{\{v,w\}}$ through the identity vertex
$1$ of $\Davis_\Gamma$.
Let $\wall_v$ be the wall dual to the $1\edge v$ edge, and
similarly for the other generators.
Consider the subsegment $X$ of $\Davis_{\{v,w\}}$ containing vertices $1$,
$w$, $wv$, $wvw$,\dots, $x$, where $x$ is the alternating product of
$w$ and $v$ of length $n-1$, starting with $w$.
This is a finite convex subcomplex of $\Davis_\Gamma$ containing $n$
vertices.
It is a consequence of a theorem of Dyer \cite{MR1076077} and Deodhar
\cite{MR1023969} (in not necessarily right-angled Coxeter groups)  that the reflection
  subgroup of $W_\Gamma$ generated by reflections through walls closest to $X$ that do not cross X is a
  right-angled Coxeter group with those reflections as fundamental
  generators.
 We work through edges incident to $X$ and see that the walls dual to
 such edges correspond to the elements described above, and commuting
 relations between them correspond to the edges described for
 $\Gamma'$. 
 At vertex 1 we have edges corresponding to each of the generators except $w$, since
 the edge $1\edge w$ is contained in $X$.
 The reflections in these walls correspond to the $(y,0)$ for $y\neq w$.
 For each $i$, $a_i$ commutes with both $v$ and $w$, so the single wall
 $\wall_{a_i}$ runs through every $a_i$--colored edge along
 $\Davis_{\{v,w\}}$. Thus, we only need one reflection, given by the
 image of $(a_i,0)$. 
 For each $j$, since $v$ and $w$ are twins and $b_j$ does not commute
 with both of $v$ and $w$, it commutes with neither of them, so every
 $b_j$--colored edge incident to $\Davis_{\{v,w\}}$ is dual to a different
 wall, and we get $n$ different walls for each $j$. 
 The corresponding reflections are $b_j$, $wb_jw$, $wvb_jvw$,\dots,
 which are the images of $(b_j,0)$, $(b_j,1)$, $(b_j,2)$,\dots under the homomorphism $\iota$.
Finally, at vertex $x$ there is one additional edge outside of $X$ that continues along $\Davis_{\{v,w\}}$,
colored either $v$ or $w$ according to whether $n$ is even or odd.
The reflection in the wall dual to this edge is the image of $(w,0)$ under $\iota$.

It remains to argue that $\Gamma'$ is actually the presentation graph
for the subgroup.
It is clear that the edges that are present belong there, as the
corresponding elements of $W_\Gamma$ commute.
In the other direction, the only case of any interest is to show that
the commutator of the images of $(b_i,j)$ and $(b_k,\ell)$ is a
nontrivial element of $W_\Gamma$  when $j\neq
\ell$.
However, such a commutator has incomplete cancellation of the
alternating $v$, $w$ conjugating words if $j\neq \ell$, so it is a
word of the form $z_0 b_i z_1 b_k z_2 b_i z_3 b_k z_4$, where each $z_p$ is
a nontrivial alternating word in $v$ and $w$.
Since neither $b_i$ nor $b_k$ commute with either of $v$ or $w$, and
$v$ and $w$ do not commute with each other, there are no elementary
operations (see \cite[Section~3.4]{davisbook}) that rearrange or
shorten that word, so it represents a nontrivial element of
$W_\Gamma$. 
\end{proof}

Using the theorem of Deodhar and Dyer cited above, one can find other finite-index RACG
subgroups of $W_\Gamma$ by considering various finite convex
subcomplexes of $\Davis_\Gamma$.
The challenge, from the point of view of this paper, would then be to
understand the presentation graph of the subgroup in terms of
operations on $\Gamma$.

We will mention one more case in which we understand the resulting graph: suppose $v$ and $w$ are adjacent in
$\Gamma$, and take $X$ to be the square
$\Davis_{\{v,w\}}\subset\Davis_\Gamma$.
We claim the presentation graph of the corresponding index-4 subgroup
is
$\double^\circ_{(w,0)}\circ\double^\circ_v(\Gamma)\cong\double^\circ_{(v,0)}\circ\double^\circ_w(\Gamma)$.
Explicitly, vertices of those two graphs are of the form
$(\gamma,i,j)$ for $\gamma\in\Gamma$ and $i,j\in\{0,1\}$, and the
isomorphism exchanges the second and third coordinates. 
When $v$ and $w$ are not adjacent it can happen that
the graphs $\double^\circ_{(w,0)}\circ\double^\circ_v(\Gamma)$ and
$\double^\circ_{(v,0)}\circ\double^\circ_w(\Gamma)$ are not
isomorphic.
More generally, in higher dimensional RACGs one can also double over larger
cliques, cf.\ \cite[Section~5]{BouXie20}.

\subsubsection{Visible RAAG subgroups}\label{danilevcovitz}
Let $\Gamma$ be a finite graph, and let $\Gamma^c$ be its
complement: the graph with the same vertex set as $\Gamma$ that has an
edge whenever $\Gamma$ does not.
Let $\Lambda$ be a subgraph of $\Gamma^c$, and let $\Delta$ be its
commuting graph: an edge $a\edge b$ of $\Lambda$ is a vertex $\{a,b\}$
of $\Delta$ and vertices $\{a,b\}$ and $\{c,d\}$ of $\Delta$ span an edge if
they commute, which is to say, $\{a,b\}\join\{c,d\}$ is an induced
square in $\Gamma$.
Obtain a homomorphism $A_\Delta\to W_\Gamma$ by sending a vertex $\{a,b\}$ of
$\Delta$ to the element $ab$ of $W_\Gamma$.
It is easy to see that this homomorphism need
not be injective.
Dani and Levcovitz \cite{DanLev24} give criteria on $\Lambda$ that
imply that the corresponding homomorphism $A_\Delta\to W_\Gamma$ is
injective and has finite-index image.
We refer to a graph $\Lambda$ satisfying their conditions 
as a 
\emph{finite-index Dani-Levcovitz $\Lambda$} (FIDL--$\Lambda$).
The corresponding subgroup $A_\Delta$ is a \emph{visible RAAG
  subgroup}.
It is `visible' (or `visual') in the sense that we see generators of $A_\Delta$ as
(products of the entries of) diagonals of squares of $\Gamma$, and we
see relations of $A_\Delta$ as squares in $\Gamma$. 

The simplest example is to take $\Gamma$ to be a square and
$\Lambda=\Gamma^c$, so that $\Delta$ is a single edge whose vertices
are the two diagonals of $\Gamma$.
Then $A_\Delta\cong\mathbb{Z}^2$ includes into $W_\Gamma\cong
D_\infty\times D_\infty$ as an index 4 subgroup.

{}

\subsection{Coarse geometry}
Let $X$ and $Y$ be metric spaces. 

For $Z\subset X$, let $\nbhd_r(Z):=\{x\in X\mid d(x,Z)<r\}$, and let
$\bar\nbhd_r(Z):=\{x\in X\mid d(x,Z)\leq r\}$.
Two subsets $Z$, $Z'$ of $X$ are \emph{coarsely equivalent} if they
are at finite 
Hausdorff distance, where Hausdorff distance is defined as: \[d_{Haus}(Z,Z'):=\inf\{r\mid Z\subset \bar\nbhd_r(Z'),\,
  Z'\subset\bar\nbhd_r(Z)\}\]
A subset $Z\subset X$ is \emph{coarsely dense} if it is coarsely
equivalent to $X$.

A \emph{coarse map} $\phi\from X\to Y$ is a map
from $X$ to uniformly bounded diameter subsets of $Y$. 
Two coarse maps $\phi,\psi\from X\to Y$ are \emph{coarsely equivalent} if there exists $C$
such that $\diam(\phi(x)\cup\psi(x))\leq C$ for all $x\in X$.
A coarse map $\phi\from X\to Y$ is
  $(L,A)$--\emph{coarse Lipschitz} if $\diam(\phi(x)\cup\phi(x'))\leq
  Ld(x,x')+A$ for all $x,\,x'\in X$.
  It is 
  $(L,A)$--\emph{coarse biLipschitz}, or an $(L,A)$--\emph{quasiisometric
  embedding}, if, in addition,
  $(1/L)d(x,x')-A\leq \diam(\phi(x)\cup\phi(x'))
  $ for all $x,\,x'\in X$.
  It is an \emph{$(L,A)$--quasiisometry} if it is $(L,A)$--biLipschitz with
  $A$--coarsely dense image.

  When we do coarse geometry and $X$ and $Y$ are graphs we take the
  point of view that `the spaces' are discrete metric spaces given by
  the vertex sets of $X$ and $Y$. The edges serve to help visualize
  the discrete metrics.
  Thus, a coarse map $\phi\from X\to Y$ only need be defined on vertices of
  $X$.
  Of course, it is possible to adjust $\phi$ to be an actual map and
  then extend it to interior points of edges of the geometric
  realization of $X$, but such choices are non-canonical, and all reasonable choices yield coarsely equivalent maps. 

The following says that coarsely inverse coarse Lipschitz maps are actually quasiisometries. It can be proved by using the inverse map to establish the quasiisometry lower bounds.
  \begin{lemma}\label{inversecoarselipimpliesqi}
    If $\phi\from X\to Y$ and $\bar\phi\from Y \to X$ are
    coarse Lipschitz and $\bar\phi\circ\phi$ is coarsely equivalent to
    the identity map on $X$ and $\phi\circ\bar\phi$ is coarsely
    equivalent to the identity map on $Y$, then $\phi$ and
    $\bar\phi$ are inverse quasiisometries. 
  \end{lemma}

  A \emph{geodesic} is an isometric embedding of an interval, and a
  \emph{quasigeodesic} is a quasiisometric embedding of an interval.

  Given subsets $Y$ and $Z$ of $X$, if for all sufficiently large $r$
  the sets $\bar\nbhd_r(Y)\cap\bar\nbhd_r(Z)$ are coarsely equivalent
  to one another, then we call that coarse equivalence class the
  \emph{coarse intersection} of $Y$ and $Z$, and denote it $Y\stackrel{c}{\cap}Z$.
  
  \begin{lemma}\label{qi_preserves_coarse_intersection}
    If $\phi\from X\to X'$ is a quasiisometry and $Y,Z<X$ such that
    $Y\stackrel{c}{\cap}Z$ and $\phi(Y)\stackrel{c}{\cap}\phi(Z)$ exist, then $\phi(Y\stackrel{c}{\cap}Z)$ is
    coarsely equivalent to $\phi(Y)\stackrel{c}{\cap}\phi(Z)$.
  \end{lemma}
  \begin{proof}
    Let $\phi$ be an $(L,A)$--quasiisometry.
    For any $r\geq LA$:
    \[\bar\nbhd_{\frac{r}{L}-A}(\phi(Y))\cap\bar\nbhd_{\frac{r}{L}-A}(\phi(Z))\subset \phi(\bar\nbhd_r(Y)\cap\bar\nbhd_r(Z))\subset
      \bar\nbhd_{Lr+A}(\phi(Y))\cap\bar\nbhd_{Lr+A}(\phi(Z))\]
      \end{proof}

      \begin{lemma}[Coarse intersections of cosets exist]
        If $X$ is the Cayley graph of a finitely generated group, $a$
        and $b$ are elements of the group, and $G$ and $H$ are subgroups,
        then $aG\stackrel{c}{\cap} bH$ exists and is represented by $aGa^{-1}\cap bHb^{-1}$.
      \end{lemma}
      \begin{proof}
        This is an immediate corollary of
        \cite[Lemma~2.2]{MosSagWhy11}, which says the intersection of
        subgroups represents their coarse intersection.
        Then observe that $aG$ is coarsely equivalent to $aGa^{-1}$,
        and $bH$ is coarsely equivalent to $bHb^{-1}$.
      \end{proof}

  A tree is connected graph with no cycles.
  It is \emph{bushy} if the set of vertices with more than two unbounded
  complementary components is coarsely dense.
  A graph is a \emph{quasitree} if it is quasiisometric to a tree.

  In a tree, if $\gamma$ is a geodesic segment then every path between
  its endpoints touches every point of $\gamma$.
  There are various ways to coarsen this property to characterize
  quasitrees in terms of `bottleneck constants'
  \cite[Theorem~4.6]{Man05}, 
  \cite[Lemma~2.16]{Man06}.
  The formulation we will use is that $X$ is a quasitree if and only
  if  there exist $A$, $B$, $C$ such
  that for every geodesic $\gamma$ of length at least $A$ there is a vertex
  $v$ within distance $B$ of the midpoint of $\gamma$ such that the
  $C$--ball about $v$ separates the endpoints of the geodesic in $X$.
  We refer to $C$ as `the' bottleneck constant.

The following example shows why some additional constraint on the
location of the separating ball is necessary; it is not sufficient to
say a space is a quasitree if every sufficiently long geodesic has its
endpoints separated by a $C$--ball.
\begin{example}
  Let $\Gamma$ be any connected graph that is not a quasitree--- the
  $\mathbb{Z}^2$ grid, for example. 
  Its double $\double(\Gamma)$ is quasiisometric to $\Gamma$, so it is
  also not a quasitree.
  However, points at distance at least 3 in $\double(\Gamma)$ are
  separated by balls of radius 1: For any $x,y$ with $d(x,y)\geq 3$
  let  $z\neq x$ be the vertex that has the same image as $x$ under the projection of $\double(\Gamma)$ onto $\Gamma$.
  Then $d(x,z)=2$ and $\lk(x)=\lk(z)$, so $x,y\notin \bar \nbhd_1(z)$
  but every path from $x$ to $y$ goes through $\bar \nbhd_1(z)$. 
\end{example}

\subsection{Group decompositions and model spaces}\label{prelim:jsj}
A \emph{graph of groups} consists of a finite connected multigraph, \emph{local groups} associated to each vertex and edge, and
injections of each edge group into each of its two incident vertex
groups.
A graph of groups has an associated \emph{fundamental group} obtained
by amalgamating the vertex groups along the edge groups.
Conversely, a graph of groups is said to give a \emph{decomposition}
of $G$ if $G$ is isomorphic to its fundamental group.

An equivalent formulation comes from letting a group $G$ act
cocompactly and without
edge inversions on a tree $T$. 
One gets a graph of groups decomposition of $G$ by considering the
quotient multigraph of $G\act T$, and setting the local groups to be
stabilizers.
In this setting, a subgroup is said to be \emph{elliptic} if it fixes
a vertex of $T$.
In the other direction, given a graph of groups with fundamental group
$G$ there is an associated
\emph{Bass-Serre tree} $T$ with a $G$ action such that the graph of
groups coming from the action of $G$ on $T$ recovers the original graph of groups.

Associated to a graph of groups decomposition of $G$ with Bass-Serre
tree $T$ we can make a geometric model for $G$ as a \emph{tree of
  spaces} $X$ over $T$:
For each vertex and edge of $T$ choose a metric space quasiisometric
to the stabilizer group. 
The space $X$ is built by taking each edge space $X_e$ crossed with a unit
interval $[0,1]$, and gluing $X_e\times\{0\}$ to the vertex space
$X_v$ of the initial vertex $v=\iota(e)$ of $e$ compatibly with the group
inclusion, and similarly for the terminal vertex $\tau(e)$.
See \cite{CasMar17}
for details. 
The resulting metric space $X$ is quasiisometric to the group $G$.

The \emph{Grushko-Stallings-Dunwoody (GSD) decomposition} of a finitely
presented group $G$ is a graph of groups decomposition of $G$ whose
vertex groups are finite or 1--ended, and whose edge groups are
finite. It turns out that this is well-defined, and the quasiisometry
type of $G$ is determined by the quasiisometry types of the 1--ended
vertex groups in its GSD decomposition \cite{PapWhy02}.
Thus, for quasiisometry questions we restrict to 1--ended groups.

The next higher complexity for splittings of 1--ended, finitely
presented groups are the \emph{Jaco-Shalen-Johannson (JSJ)
  decompositions}.
In this paper we always use ``JSJ decomposition'' to mean ``JSJ
decomposition over 2--ended subgroups''.
This is a graph of groups decomposition of $G$ that, in a sense,
captures all possible splittings over 2--ended subgroups.
The edge groups are all 2--ended.
Unlike with finite edge groups, there can be incompatible splittings
over 2--ended subgroups, and these are combined into \emph{hanging
  vertices} in the graph of groups. The other vertices are either
2--ended or they are \emph{rigid}, in the sense that they do not
split further over 2--ended subgroups \emph{relative to the existing
incident 2--ended edge groups}, meaning a splitting in which all of these
subgroups are elliptic. 

The JSJ decomposition of a group $G$ is not a well-defined graph of
groups; instead, it is a deformation space of graphs of groups (see
\cite{MR3758992}).
However, the existence of a non-trivial JSJ decomposition is a
quasiisometry invariant \cite{Pap05}.
One can make a canonical graph of groups decomposition from (any) JSJ
decomposition of $G$ as follows.
Let $T$ be the Bass-Serre tree of some nontrivial JSJ decomposition of
$G$.
A \emph{cylinder} is a maximal collection of edges of $T$ whose stabilizers
are commensurable in $G$.
One makes a \emph{JSJ tree of cylinders} by collapsing all of the
cylinders to single vertices.
These are known as \emph{cylinder vertices}.
Vertices of $T$ that are contained in a single cylinder are absorbed
into that cylinder vertex; vertices that are not, which are
necessarily rigid or hanging, survive in the JSJ tree of cylinders.
The quotient of the JSJ tree of cylinders by the $G$ action gives a
graph of groups decomposition of $G$ known as the \emph{JSJ graph of
  cylinders} of $G$.
It is a bipartite multigraph where one part consists of cylinder vertices,
and the other part is rigid and hanging vertices that are not
contained in a single cylinder. 
It turns out that the JSJ tree of cylinders and the JSJ graph of cylinders are well-defined,
independent of the starting JSJ decomposition of $G$
\cite{MR3758992} (but the edges are not necessarily 2--ended anymore).
Furthermore, it follows from \cite{Pap05} (see, eg, \cite{CasMar17})
that a quasiisometry $\phi\from G\to G'$ between
two groups induces an isomorphism $\phi_*\from T\to T'$ between their JSJ trees of
cylinders, and if $X$ and $X'$ are trees of spaces of $G$ and $G'$
over $T$ and $T'$, respectively, then the restriction of $\phi$ to each vertex space
$X_v$ is uniformly coarsely equivalent to a quasiisometry $\phi_v\from
X_v\to X'_{\phi_*(v)}$ such that for each edge $e$ of $T$ incident to
$v$, the set $\phi_v(X_e)$ is uniformly coarsely equivalent to 
$X'_{\phi_*(e)}$.
 
In \cite{CasMar17} the construction is carried out in the other
direction: If $X$ and $X'$ are trees of spaces of $G$ and $G'$
over some trees $T$ and $T'$ (not necessarily related to JSJ
decompositions) and if there is an isomorphism $\chi\from T\to T'$
and a collection of uniform quasiisometries $\phi_v\from X_v\to
X'_{\chi(v)}$ that uniformly coarsely agree on common edge spaces, then $X$ and
$X'$ are quasiisometric by a map $\phi$ such that
$\phi|_{X_v}$ coarsely agrees with $\phi_v$ for each vertex $v\in T$.
The collection of quasiisometries is
called a \emph{tree of quasiisometries over $\chi$}. 

\begin{proposition}\label{tree_of_quasiisometries}
  Suppose that $G$ and $G'$ are two groups that are defined by graphs
  of groups on the same underlying multigraph $\Gamma$, which is a tree. 
Suppose for each vertex $v$ of $\Gamma$ there is a quasiisometry
$\psi_v$ from the local group $G_v$ of $G$ to the local group $G'_v$ of $G'$.
If the collection of quasiisometries
$\{\psi_v\mid v\in\Gamma\}$ satisfies the following
conditions, then $G$ and $G'$ are quasiisometric. 
\begin{enumerate}
 \item For each edge $e$ of $\Gamma$, $\psi_{\iota(e)}$ and
   $\psi_{\tau(e)}$ coarsely agree on $G_e$.\label{item:base_consistency}
\item For each edge $e$ of $\Gamma$, $\psi_{\iota(e)}$ induces a
  bijection $\psi_{\iota(e)}^*$ between  $G_{\iota(e)}/G_e$ and 
  $G'_{\iota(e)}/G'_e$ by taking each coset of $G_e$ in $G_{\iota(e)}$
  to within uniformly bounded Hausdorff distance of a unique coset of
  $G'_e$ in $G'_{\iota(e)}$, and vice versa. \label{item:coset_bijection}
\item For each edge $e$ of $\Gamma$ and each coset $gG_e$ of $G_e$ in
  $G_{\iota(e)}$
  there exists $h\in gG_e$ and $h'\in \psi_{\iota(e)}^*(gG_e)$  such that $h'\psi_{\iota(e)}h^{-1}$
  and $\psi_{\iota(e)}$ coarsely agree on $gG_e$, and such that the
  coarseness constants are bounded uniformly over all cosets. \label{item:consistency}
\end{enumerate}
\end{proposition}

\begin{proof}
  We build a
  tree of quasiisometries as described above.
Let $T$ be the Bass-Serre tree corresponding to the given splitting of
$G$, and let $X$ be a tree of spaces for $G$ over $T$.
Since $\Gamma$ is a tree it admits a lift $\widetilde{\Gamma}$ to
$T$, such that if $v\stackrel{e}{\edge}w$ is an edge of $\Gamma$ then
$\tilde{v}\stackrel{\tilde{e}}{\edge}\tilde{w}$ is an edge of
$\widetilde{\Gamma}$.
We use all the same notation with $'$s for the corresponding concepts
with respect to $G'$.

For each vertex $v\in\Gamma$  define $\chi(\tilde{v}):=\tilde{v}'$ and 
define $\phi_{\tilde{v}}\from
X_{\tilde{v}}\to X'_{\tilde{v}'}$ to be $\psi_v\from G_v\to G'_v$,
where we implicitly identify $X_{\tilde{v}}$ with $G_v$ and
$X'_{\tilde{v}'}$ with $G'_v$.
Condition~\eqref{item:base_consistency} implies that for each edge
$e\in \Gamma$, $\phi_{\iota(\tilde{e})}$ and $\phi_{\tau(\tilde{e})}$
coarsely agree on $X_{\tilde{e}}$, which is the intersection of their
domains. 

Now we inductively expand the domain of $\chi$, specifying
local quasiisometries between the corresponding vertex spaces as we go, in such
a way that the maps on adjacent vertex spaces coarsely agree on their
intersection. 
Once this is done, \cite[Corollary~2.16]{CasMar17} says the local
quasiisometries patch together to give a quasiisometry.

As we will go, it will be useful for bookkeeping to define
$\kappa\from T\to G$ and $\kappa'\from T'\to G'$ such that:
\begin{itemize}
\item $\kappa(g\tilde{v})\in gG_v$
  \item $\kappa'(\chi(g\tilde{v}))\in g'G'_v$, where
    $g'\tilde{v}'=\chi(g\tilde{v})$
    \item $\phi_{g\tilde{v}}=\kappa'(\chi(g\tilde{v}))\psi_v\kappa(g\tilde{v})^{-1}$
\end{itemize}
For the base cases of vertices in $\widetilde{\Gamma}$ and
$\widetilde{\Gamma}'$ define $\kappa(\tilde{v})=1$ and
$\kappa'(\tilde{v}')=1$. 
For the initial inductive step, consider an edge $e$ of $\Gamma$ with $\iota(e)=v$.
The edges of $T$ incident to $\tilde{v}$ and covering $e$ are of the form $g_i\tilde{e}$, where $\{g_i\}$ is a set
of coset representatives of $G_v/G_e$.
Similarly, the edges of $T'$ incident to
$\tilde{v}'$ and covering $e$ are translates of $\tilde{e}'$ by coset
representatives of $G'_v/G'_e$.
For each $i$ let $g_i'G'_e=\psi_{\iota(e)}^*(g_iG_e)$ and define
$\chi(g_i\tilde{e}):=g_i'\tilde{e}'$. 
By Condition~\eqref{item:coset_bijection}, this defines a bijection
between edges of $T$ incident to $\tilde{v}$ covering $e$ and edges of
$T'$ incident to $\tilde{v}'=\chi(\tilde{v})$ covering $e$. 
Now choose $h_i$ and $h'_i$ as in Condition~\eqref{item:consistency}
with respect to $g_iG_e$, and
define:
\[\phi_{g_i\tau(\tilde{e})}:=h_i'\psi_{\tau({e})}h_i^{-1}\]
We must check that $\phi_{\tau(g_i\tilde{e})}$ and
$\phi_{\iota(g_i\tilde{e})}$ coarsely agree on $X_{g_i\tilde{e}}=g_iG_{e}$.
For $k\in G_{e}$:
\begin{align*}
  \phi_{\tau(g_i\tilde{e})}(g_ik)&=\phi_{g_i\tau(\tilde{e})}(g_ik)\\
  &=h_i'\psi_{\tau(e)}h_i^{-1}(g_ik)&\\
                                 &=h_i'\psi_{\tau(e)}(h_i^{-1}g_ik)\\
  &\approx h_i'\psi_{\iota(e)}(h_i^{-1}g_ik)&\text{by
                                          Condition~\eqref{item:base_consistency},
                                          since $h_i^{-1}g_ik\in
                                              G_e$}\\
                                 &= h_i'\psi_{\iota(e)}h_i^{-1}(g_ik)\\
   &\approx \psi_{\iota(e)}(g_ik)&\text{by
                                   Condition~\eqref{item:consistency},
                                   since $g_ik\in g_iG_e$}\\
  &=\phi_{\iota(\tilde{e})}(g_ik)=\phi_{\iota(g_i \tilde{e})}(g_ik)
\end{align*}
Define $\kappa(g_i\tau(\tilde{e})):=h_i$ and $\kappa'(\chi(g_i\tau(\tilde{e}))):=h_i'$.

For the general inductive step, suppose $\chi$ is defined on a subtree of $T$
containing $\widetilde{\Gamma}$, and for each vertex $g\tilde{v}$ in
the subtree we have defined a quasiisometry $\phi_{g\tilde{v}}\from
X_{g\tilde{v}}\to X'_{\chi(g\tilde{v})}$ and $\kappa$ and $\kappa'$ as above.
Suppose $g\tilde{v}\notin\widetilde{\Gamma}$ is a leaf of the subtree and  $g\tilde{v}\stackrel{g\tilde{e}}{\edge}g\tilde{w}$ is an
edge of $T$ such that $g\tilde{w}$ is farther from $\widetilde{\Gamma}$ than
$g\tilde{v}$. 
Pull $g\tilde{v}$ back to $\tilde{v}$ by applying
$\kappa(g\tilde{v})^{-1}$, which takes $g\tilde{e}$ to some edge
$\kappa(g\tilde{v})^{-1}g\tilde{e}$ incident to $\tilde{v}$.
Then $\psi_v^*$ identifies this with some edge incident to
$\tilde{v}'$ in the same orbit as $\tilde{e}'$. Push this edge forward
by $\kappa'(\chi(g\tilde{v}))$ to get an outgoing edge of $T'$ at
$\chi(g\tilde{v})$.
Define this edge to be $\chi(g\tilde{e})$, and define
$\chi(g\tilde{w})$ to be its other endpoint.
Similarly, to define $\phi_{g\tilde{w}}$, pull back via
$\kappa(g\tilde{v})^{-1}$, do the construction of the previous
paragraph, and then push forward the result by $\kappa'(\chi(g\tilde{v}))$. 

The key points are that all quasiisometries between vertex spaces are
one of the base quasiisometries $\psi_v$, pre- and post- composed by
multiplication in the groups, which are isometries, so the
quasiisometry constants are bounded by the maxima of the constants for
the base maps. Similarly, the coarse agreement between maps of neighboring
vertex spaces on their common intersection is, up to isometry, the
same as that coarse agreement described in
Condition~\eqref{item:consistency}, which is assumed to be uniform for
each vertex/edge pair in $\Gamma$, of which there are finitely many.
\end{proof}

\subsubsection{JSJ decompositions of RAAGs}\label{jsj}
A graph is \emph{biconnected} if it connected with no cut
vertex.
By this definition, a biconnected graph either has at most two
vertices, or every vertex has valence at least 2. 

It is not hard to see that a RAAG on at least two generators is 1--ended if and only if its
presentation graph is connected.
According to Clay \cite{Cla14} and Margolis \cite[Proposition 3.6]{Mar20}, the JSJ
graph of cylinders of a RAAG $A_\Delta$ can be described ``visually'' in
$\Delta$: cylinders are stars of cut vertices, rigid vertices are
maximal biconnected subgraphs that either contain two cut vertices or
are not contained in any cylinder, and edges between them are
intersections of the corresponding subgraphs. 
From this fact and quasiisometry invariance of the JSJ tree of
cylinders, several quasiisometry invariants appear:

\begin{lemma}\label{RAAGnohanging}
  RAAGs have no hanging vertices in their JSJ decompositions. 
\end{lemma}

\begin{lemma}\label{RAAGflat}
  Let $\Delta$ be a connected graph with more than one vertex. 
  The rigid vertex groups of the JSJ graph of cylinders of $A_\Delta$
  are one-ended special subgroups $A_{\Delta'}$ of $A_\Delta$ that do not split
  further over 2--ended subgroups. 
\end{lemma}
\begin{proof}
  $A_\Delta$ is finitely presented and one-ended, so it has a JSJ
  decomposition.
  Any rigid vertices correspond to maximal biconnected subgraphs $\Delta'$ of
  $\Delta$ that either contain at least two cut vertices or are not
contained in the star of any cut vertex.
  In particular, rigid vertex groups are special subgroups.  
   Since $\Delta$ is connected with more than one vertex, every vertex
   is contained in an edge.
   Edges are biconnected, so no single vertex is a maximal biconnected subgraph.
  Thus, $\Delta'$ is connected with more than one vertex, so
  $A_{\Delta'}$ is one-ended.
  Furthermore, since $\Delta'$ is biconnected it contains no cut
  vertex, so $A_{\Delta'}$ has no two-ended splittings. 
\end{proof}

\subsubsection{JSJ decompositions of RACGs}\label{sec:jsjracg}
Mihalik and Tschantz \cite{MihTsc09} show that in some sense all
decompositions of Coxeter groups are visual.

For RACGs it is not hard to see that splittings over finite subgroups
correspond to separating cliques in the presentation graph.
We will also exclude the well understood cases that the presentation
graph $\Gamma$ is complete ($W_\Gamma$ is finite) and $\Gamma$ is a
cycle ($W_\Gamma$ is virtually a surface group).
Assuming that $\Gamma$ is triangle-free, incomplete, with no
separating clique, and is not a cycle, Edletzberger \cite{Edl24},
extending work of Dani and Thomas \cite{DanTho17} from the hyperbolic
case,  gives
a description of the JSJ graph of cylinders in terms of subgraphs of
the defining graph.
In particular, 2--ended splittings arise in two ways:
\begin{itemize}
\item $\{a,b\}$ is a cut pair of $\Gamma$, meaning that
  $\Gamma\setminus\{a,b\}$ is not connected.
  \item $\{a,b\}$ is not a cut pair, but there is a common neighbor
    $c\in\lk(a)\cap\lk(b)$ such that $\Gamma\setminus\{a,b,c\}$ is not
    connected. In this case $a\edge c\edge b$ is called a \emph{cut 2--path}.
\end{itemize}
We combine the two by saying \emph{$\{a-b\}$ is a cut} to mean that
either $\{a,b\}$ is a cut pair or that there exists $c$ such that
$a\edge c\edge b$ is a cut 2--path.
In the first case $\langle ab\rangle\cong \mathbb{Z}$ is an index-2 subgroup of
$W_{\{a,b\}}$, and in the second case it is an index-4 subgroup of
$W_{\{a,b,c\}}$.
A cut is \emph{crossed} if there is another cut containing vertices in
multiple of its complementary components, and \emph{uncrossed} otherwise.
Crossed cuts group together to make hanging vertices of the JSJ graph
of cylinders.
Cylinder vertices of the JSJ graph of cylinders are commensurators of
uncrossed cuts, which can be described explicitly: if $\{a-b\}$ is an
uncrossed cut of $\Gamma$ then there is a corresponding cylinder
vertex with vertex group $W_{\{a,b\}\join (\lk(a)\cap\lk(b))}$.
Such a group is commensurable to one of $\mathbb{Z}$, $\mathbb{Z}^2$,
or $F_2\times \mathbb{Z}$, according to whether $|\lk(a)\cap\lk(b)|$
is less than 2, equal to 2, or greater than 2, respectively.

Rigid vertices can also be described explicitly
(\cite[Proposiiton~3.8]{Edl24}), they are subsets of size at least 4 of
vertices of $\Gamma$, each with valence at least 3, that cannot be separated by any pair of vertices
or 2--path in $\Gamma$, and that are maximal with respect to
inclusion among sets with these properties. 

\subsubsection{Further splittings}
In this paper we are always using `JSJ decomposition' in the sense of
decomposition over 2--ended subgroups.
JSJ theory also exists for other classes of splittings \cite{DunSag99,
  FujPap06, MR3758992}.
In particular, Groves and Hull \cite{MR3728497} describe the structure of splittings of
RAAGs over Abelian groups.
In the 2--dimensional case, it would therefore be interesting to
consider whether the kind of invariants we develop for 2--ended JSJ
decompositions of RACGs vs RAAGs can be extended to splittings over 
virtually $\mathbb{Z}^2$
subgroups.
We leave this line of inquiry for future work.

\subsection{Morse property and stability}\label{prelim:morsestable}

  A subspace $Z$ of a metric space $X$ is \emph{$\mu$--Morse} if for
  every $L$ and $A$ we have that
  every $(L,A)$--quasigeodesic segment $\gamma$ of $X$ with endpoints
  in $Z$ is contained in the $\mu(L,A)$--neighborhood of $Z$.

  A subspace is \emph{Morse} if there is some $\mu$ for which it is $\mu$--Morse.

  Morseness is a quasigeodesic quasiconvexity condition on $Z$ that describes
how it sits in $X$, but says nothing about the intrinsic geometry of
$Z$.
There is a further property, \emph{stability}, that, when $X$ is a
geodesic metric space, is equivalent to $Z$ being Morse and $Z$ itself
being a hyperbolic space. 

The concept of subgroup stability was introduced by Durham and Taylor
\cite{MR3426695} as a geometric group-theoretic interpretation of
convex-cocompactness for subgroups mapping class groups of surfaces.
Another characterization of such subgroups is that the subgroup is
convex cocompact if and only if its orbit map into the curve graph of
the surface is a quasiisometric embedding \cite{MR2465691,Ham05}.
These results inspired work to characterize stability in other
families of groups, and then to search for a \emph{stability
  recognizing space} to play the role of the curve graph, in the sense
that a subgroup of the given group is stable if and only if its orbit
map into the stability recognizing space is a quasiisometric
embedding.

The theory of Morse and stable subgroups of RAAGs and RACGs is well
developed, as we will briefly recall. We will return to the topic of
stability recognizing spaces in \fullref{sec:hhs}.

\begin{theorem}[{\cite[Theorem~F]{CorHum17},\cite[Corollary~7.4]{RusSprTra23}\footnote{This
    result also says that in a strongly CFS RACG every Morse
    subset is either hyperbolic or coarsely dense.}}]\label{Morsehyperbolic}
  In a 1--ended RAAG, every Morse subset is 
  coarsely dense or quasiisometric to a finite valence tree. 
\end{theorem}

The Morse property is easily identifiable for special subgroups of RACGs.
A subgraph $\Gamma'$ of $\Gamma$ is \emph{square complete} if whenever
$\Gamma'$ contains a diagonal of an induced square it contains
the whole square.
An induced subgraph is \emph{minsquare} if it contains an induced square, is
square complete, and is minimal with respect to inclusion among all
subgraphs that satisfy the first two conditions.

\begin{theorem}[{\cite[Theorem~1.11]{Tra19},\cite[Proposition~4.9]{Gen19}}]\label{morsesquarecomplete}
  A special subgroup of a RACG is Morse if and
  only if its presentation subgraph is square complete. 
\end{theorem}
\begin{corollary}\label{cor:stable}
   A special subgroup of a RACG is stable if and
  only if its presentation subgraph is square complete and contains no
  induced square.
\end{corollary}

\begin{proposition}\label{raagedy_implies_minsquare}
  A 1--ended  RAAGedy RACG  has presentation graph that is a join of a
  clique and a minsquare subgraph. 
\end{proposition}
\begin{proof}
  Removal of the clique factor of the join is passage to a finite-index subgroup. The resulting RACG is still 1--ended and
  quasiisometric to a RAAG, so we may assume the presentation graph
  has no cone vertices.

  A 1--ended RAAG contains $\mathbb{Z}^2$, so it  is not hyperbolic.
  Thus, the presentation graph of the RACG contains an induced
  square.
  Given a square, the smallest square complete subgraph containing it
  is a minsquare subgraph, so such subgraphs exist. 
  Suppose the presentation graph contains a proper subgraph that is
  minsquare. 
  Its complement does not consist of cone vertices, since there are
  none, so the special subgroup defined by the subgraph is of infinite
  index and is Morse and 1--ended.
  Its image under the quasiisometry to the RAAG is a Morse subset that
  is neither coarsely dense nor quasiisometric to a tree, contradicting
  \fullref{Morsehyperbolic}.
  Thus, the only minsquare subgraph is the presentation graph itself. 
\end{proof}

\subsection{Maximal product regions}\label{sec:oh}

This section recapitulates the setup for a theorem of Oh \cite{Oh22},
who builds on work of Haglund and Wise \cite{MR2377497} and
Huang \cite{Hua17}.
\begin{definition}\label{def:weaklyspecial}
 A \emph{weakly special square complex} is a non-positively curved
 square complex $X$ such that no wall of $X$ self-osculates or
 self-intersects. 
\end{definition}
Huang \cite[Section~5.3]{Hua17} considers two families of examples of
compact weakly special cube complexes: the Salvetti complex of a
RAAG and the 
\emph{commutator complex}\footnote{Davis \cite{Dav24} calls
  the complex $P_\Gamma$ the \emph{polyhedral
    product of intervals}.} $P_\Gamma$  of a RACG
$W_\Gamma$.
The commutator complex  is a standard construction \cite{MR1747268,davisbook,Dav24} of
a cube complex whose fundamental group is the commutator subgroup of
$W_\Gamma$, so is a finite-index normal subgroup of $W_\Gamma$, and whose universal cover is the Davis complex
$\Sigma_\Gamma$.

Huang asserts that a compact weakly special cube complex has a finite
cover in which walls are 2--sided.
In particular, the universal cover has 2--sided walls.
Oh makes a standing assumption that we have already passed to such a
finite cover.
The examples that we care about in this paper, the Salvetti complex
and commutator complex, already have 2--sided walls, but to 
 accurately quote Oh, we make an auxiliary
definition:
\begin{definition}\label{def:stronglyweaklyspecial}
  A compact cube complex is \emph{weakly special$^*$} if it is weakly
  special with 2--sided walls.
\end{definition}
\begin{remark}
  Compared to  \emph{special}
  \cite[Def.~3.1]{MR2377497},
  \fullref{def:stronglyweaklyspecial} legalizes inter-osculation.
\end{remark}
\begin{definition}\label{def:standard}
  A \emph{standard graph} is a topologically nontrivial graph without
  leaves.
  A \emph{standard product subcomplex} of a compact weakly special$^*$
  square complex $X$ is the image of a local isometry from the
  product of two standard graphs into $X$.
  A \emph{standard product subcomplex/region} of $\tilde{X}$ is a product
  subcomplex of $\tilde{X}$ that is a lift to $\tilde{X}$ by the
  universal covering map of a standard product subcomplex of $X$.
\end{definition}

The following lemma of Oh allows us to work directly with standard product regions
in $\tilde{X}$ without explicitly demonstrating that they are lifts of
standard product subcomplexes of $X$.

\begin{lemma}\cite[Lemma~2.8]{Oh22} \label{Oh:standard}
Every product subcomplex of $\tilde{X}$ that is a product of infinite trees without leaves is a
standard product region.
\end{lemma}

\begin{definition}
  The \emph{intersection graph}, or  \emph{maximal standard product region graph (MPRG)}, $\mprg(X)$ of a compact, weakly
  special$^*$ square complex $X$ is the graph whose vertices are maximal
  standard product regions in $\tilde X$ such that two vertices are
  connected by an edge if the corresponding product regions intersect in a standard product region. 
\end{definition}

\begin{theorem}[{\cite[Corollary~3.3]{Oh22}}]\label{Oh}
For each $L$ and $A$ there is $C$ such that if  $X$ and $Y$ are compact, weakly special$^*$ square complexes and
 $\phi\from \tilde X\to \tilde Y$ is an $(L,A)$--quasiisometry between
 their universal covers then $\phi$ induces a bijection $\phi_*$ between
 maximal standard product subcomplexes of $\tilde X$ and $\tilde Y$
 such that for each maximal standard product subcomplex $P$ of $\tilde X$, $d_{Haus}(\phi(P),\phi_*(P))\leq C$. It follows that $P$
 and $\phi_*(P)$ are quasiisometric. 
\end{theorem}

\begin{corollary}\label{mpgqiinvariant}
  The maximal standard product region graph of a compact weakly special$^*$
  cube complex can be decorated by adding to each vertex $v$ the
  quasiisometry type of the maximal product region corresponding to
  $v$. 
 A quasiisometry between universal covers of compact, weakly special$^*$
 square complexes induces an isomorphism of their maximal standard product
 regions graphs that respects these decorations. 
\end{corollary}

\begin{definition}\label{def:ric_weakly_special}
  The \emph{reduced intersection graph} $\ric(X)$ of a compact, weakly special$^*$
  square complex $X$ is the graph whose vertices are maximal standard
  product subcomplexes of $X$ such that two vertices are connected by
  an edge if the corresponding product subcomplexes intersect in a
  standard product subcomplex. 
\end{definition}

We will work with $\mprg(X)$ and $\ric(X)$ as graphs, but actually
these graphs are the 1--skeleta of higher dimensional complexes that 
Oh calls the \emph{intersection
  complex} and \emph{reduced intersection complex}, respectively.
Higher dimensional cells are made by filling in a simplex whenever there is a
clique of maximal standard product subcomplexes all containing a common standard
product subcomplex. 

\begin{theorem}[{\cite[Theorem~3.9]{Oh22}}]\label{ric_is_the_fundamental_domain}
  If $X$ is a compact, weakly special$^*$ square complex, $\mprg(X)$ is
  its intersection complex, and $\ric(X)$ is its reduced
  intersection complex, then the $\pi_1(X)$ action on $\tilde X$ by
  deck transformations 
  induces an action on $\mprg(X)$ with
  fundamental domain isomorphic to $\ric(X)$.
  Furthermore, this isomorphism is compatible with the definitions of
  $\mprg(X)$ and $\ric(X)$ in the sense that if $v\in\ric(X)$ and
  $\tilde{v}$ is the corresponding vertex in the fundamental domain of
  $\mprg(X)$ then the maximal standard product region of $\tilde X$
  corresponding to $\tilde v$ is a lift to $\tilde X$ of the maximal
  standard product
  subcomplex corresponding to $v$ in $X$. 
\end{theorem}
\begin{lemma}\label{weakconvexity}
  With notation as in the previous theorem, if
  $\gamma=e_0,\dots,e_{n-1}$ is an edge path in $\mprg(X)$ then there are
  $g_i\in\pi_1(X)$ such that $\gamma'=g_0e_0,\dots,g_{n-1}e_{n-1}$ is a path
  in $\ric(X)$.
  Furthermore, vertices and edges of $\gamma$ contained in $\ric(X)$
  are invariant under this projection. 
\end{lemma}
\begin{proof}
  By definition, an edge $\tilde e$ in $\mprg(X)$ represents a
  standard product region $\tilde P_{\tilde e}$ of $\tilde X$ that is the intersection of two maximal
  standard product regions.
  Standard product regions of $\tilde X$ are, by definition, lifts to
  $\tilde X$ of standard product subcomplexes of $X$, so there is a
  standard product subcomplex $P_e$ of $X$ and corresponding edge $e$
  of $\ric(X)$ such that the quotient by the $\pi_1(X)$--action sends
  $\tilde P_{\tilde e}$ to $P_e$.
  Thus, identifying $\ric(X)$ with the fundamental domain of
  $\pi_1(X)\backslash \mprg(X)$ as in
  \fullref{ric_is_the_fundamental_domain}, for each of the edges $e_i=v_i\edge v_{i+1}$ of $\gamma$ 
  there exists $g_i\in\pi_1(X)$ such that $g_ie_i\in\ric(X)$.
  Then $g_iv_i$ and $g_{i+1}v_i$ are vertices of $\ric(X)$ in the same
  $\pi_1(X)$--orbit, so they are the same vertex, so $\gamma'$ is a path.
\end{proof}

\begin{proposition}[Visibility of reduced intersection graphs for RAAGs]\label{icuricraag}
  Let $\Delta$ be a connected, triangle-free graph. The reduced
  intersection graph $\ric_\Delta$ of the Salvetti complex of $A_\Delta$ is
  isomorphic to the graph obtained from $\Delta$ by taking a vertex
  for each maximal join subgraph and connecting them by an edge if
  the join subgraphs have an edge in common. 
\end{proposition}
\begin{proof}
  Edges of the universal cover $\Salvetti_\Delta$ of the Salvetti
  complex are colored by the corresponding vertex of $\Delta$. 
  Let $\Theta=\Theta_1\times \Theta_2$ be a product region of
  $\Salvetti_\Delta$.
Let $\Delta_i$ be the set of vertices  $v\in \Delta$ such that some edge
of $\Theta_i$ has color $v$.
Since $\Theta$ is a product, every edge in $\Theta_1$ spans a square with every
edge in $\Theta_2$, which means that $\Delta_1\join\Delta_2$ is a join
subgraph of $\Delta$.
Conversely, if $\Delta_1\join\Delta_2$ is a subgraph of $\Delta$, then
$\Sigma_{\Delta_1}\times\Sigma_{\Delta_2}$ is a product region of
$\Sigma_\Delta$, some translate of which contains $\Theta$.
If $\Delta'_1\join\Delta'_2$ is a strictly larger join subgraph
containing $\Delta_1\join\Delta_2$ then
$\Sigma_{\Delta_1'}\times\Sigma_{\Delta_2'}$ is a product region
strictly containing $\Sigma_{\Delta_1}\times\Sigma_{\Delta_2}$, hence,
up to translation, $\Theta$. 
Thus, $\Theta$ is a maximal product region if and only if it is a translate
of $\Sigma_{\Delta_1}\times\Sigma_{\Delta_2}$, such that
$\Delta_1\join\Delta_2$ is a maximal join in $\Delta$.

Edges in the product region graph correspond to having a common
product sub-region, which equates to two maximal join subgraphs of
$\Delta$ having a common join subgraph. Every join subgraph contains
an edge, and edges are joins, so it suffices to consider edges. 
\end{proof}

\begin{proposition}[Visibility of reduced intersection graphs for RACGs]\label{icuricracg}
  Let $\Gamma$ be a  triangle-free graph without separating cliques.
  The reduced intersection graph $\ric_\Gamma$ of the commutator complex $P_\Gamma$ of $\Gamma$ is
  isomorphic to the graph obtained from $\Gamma$ by taking a vertex
  for each maximal thick join subgraph, and connecting them by an edge if
  the join subgraphs have a square in common.

  Furthermore, the action of $W_\Gamma$ on the Davis complex
  $\Sigma_\Gamma=\widetilde{P_\Gamma}$ induces an action of $W_\Gamma$
  on the maximal standard product region graph
  $\mprg(W_\Gamma):=\mprg(P_\Gamma)$ that has a fundamental domain
  isomorphic to $\ric_\Gamma$.
\end{proposition}
\begin{proof}
  The argument for identifying $\ric_\Gamma$ with the graph of thick joins in
  $\Gamma$ is the same as for \fullref{icuricraag}, except that
  product regions are supposed to be infinite in both factors, so we
  add the requirement that the join subgraphs are thick.

  The content of the second part is that
  \fullref{ric_is_the_fundamental_domain} tells us that $\ric_\Gamma$ is the
  fundamental domain for the action of $\pi_1(P_\Gamma)$ on
  $\Sigma_\Gamma$, and $W_\Gamma$ is a supergroup of
  $\pi_1(P_\Gamma)$, so a priori might have had smaller fundamental
  domain, but it does not, because $W_\Gamma\act\Sigma_\Gamma$
  preserves the edge coloring.  
\end{proof}
\begin{definition}\label{def:ric}
  If $\Upsilon$ is the presentation graph of a RAAG/RACG, let
  $\ric_\Upsilon$ be a choice of lift of the reduced intersection
  graph of the Salvetti complex/commutator complex to the maximal
  standard product region graph $\mprg_\Upsilon$, which gives a
  fundamental domain for the group action. 

  For $v\in\ric_\Upsilon$ let $J_v$ be the corresponding maximal
  (thick, in the RACG case) join subgraph of
  $\Upsilon$, as described in \fullref{icuricraag} and \fullref{icuricracg}.
\end{definition}

\subsection{Convex sets and projections in CAT(0) cube complexes}\label{sec:catzero}

Huang calls a subcomplex of the universal cover $\Salvetti_\Delta$ of
a RAAG $A_\Delta$ \emph{standard} if it is a translate of
$\Salvetti_{\Delta'}$ for some $\Delta'\subset \Delta$.
After fixing an identity vertex, thus identifying the 1--skeleton of
$\Salvetti_\Delta$ with the Cayley graph of $A_\Delta$, vertex sets of
standard subcomplexes are precisely cosets of special subgroups of
$A_\Delta$.
This description applies equally well to RACGs, so we can also speak
of standard subcomplexes of Davis complexes of RACGs. 
By \fullref{icuricraag} and \fullref{icuricracg}, this is consistent
with the `standard product subcomplex' terminology of Oh, in the sense
that standard product subcomplexes are examples of standard subcomplexes. 

In \fullref{sec:compliant_cycles} we need some results about coarse
intersections of standard subcomplexes.
These follow from some properties of projections to
convex sets in (finite dimensional) CAT(0) cube complexes.
We will work in the combinatorial metric: the metric on the vertex set
obtained by restricting the path metric on the 1--skeleton.

\begin{lemma}\label{gate}
  Let $X$ be a CAT(0) cube complex with its combinatorial metric, and
  let $Y$ be a convex subcomplex. 
 There is a Lipschitz map
$\pi_Y\from X\to Y$ sending each vertex $x\in X$ to the vertex of $Y$
that is closest to $x$.
The map $\pi_Y$ is a \emph{gate projection}: for all $x\in X$ and $y\in Y$ the vertex $\pi_Y(x)$ is on
a geodesic from $x$ to $y$.
\end{lemma}

In a CAT(0) cube complex, consider the equivalence relation on edges generated by the condition that two edges are equivalent if they are opposite edges of some
square in the complex.
Dual to each equivalence class is a \emph{wall} (also called a \emph{hyperplane}), which can be thought of geometrically as the union of
midcube hyperplanes dual to these edges in the cubes containing them.
A wall $\wall$ separates the cube complex into two complementary
 sets of vertices, each of which spans a convex subcomplex, called \emph{halfspaces} and usually denoted $\wall^+$
and $\wall^-$, such that every combinatorial geodesic from $\wall^-$
to $\wall^+$ contains an edge in the equivalence class defining the
wall. 
The number of walls separating two vertices is equal to the
combinatorial distance between them.
Combinatorial geodesics cross each wall at most
once, and two combinatorial geodesics between the same points cross
the same set of walls.
Two walls \emph{cross} if there is a cube containing an edge dual to
each of them, or, equivalently, if all four possible intersections of
their halfspaces are nonempty. 

The following result is well known to experts. 
Similar results have been proved and reproved several
times in different settings.
In the CAT(0) metric, compare \cite[Lemma~2.3]{MR2421136}, \cite[Lemma~2.10]{Hua17}, also \cite{BehCha12}.
Chatterji, Fern\'os, and Iozzi \cite[Lemma~2.18]{ChaFerIoz16} prove a version of \fullref{bridge} in the
combinatorial metric, but only state it for halfspaces $Y$ and $Z$.
Full proofs of the more general statement appear in unpublished
sources \footnote{We thank Anthony Genevois for the references.}
\cite[Theorem~1.22]{Hag19}, \cite[Proposition~1.5.2]{Gen23}.
A proof can also be deduced from results of
Isbell on parallelism between gate projections in discrete median algebras \cite[cf.\ Proclamation~2.5 and the
Corollary to Theorem~5.12]{MR574784},
together with the well known equivalence between discrete median
algebras and 0-skeleta of CAT(0) cube complexes 
\cite{MR1624115,Rol98,MR1748966}.
\begin{proposition}\label{bridge}[Bridge Lemma]
  Let $X$ be a CAT(0) cube complex with its combinatorial metric and
  let $Y$ and $Z$ be convex subcomplexes.
  The \emph{bridge} $Y\bridge Z$, consisting of the subcomplex
  spanned by vertices that lie on some minimal length geodesic
  between $Y$ and $Z$, is a combinatorially convex subcomplex isomorphic
  to $\pi_Y(Z)\times [y,\pi_Z(y)]$, where $y$ is any vertex in
  $\pi_Y(Z)$ and $[y,\pi_Z(y)]$ is the subcomplex spanned by vertices that lie on a
  combinatorial geodesic from $y$ to $\pi_Z(y)$.

  The walls that meet $Y\bridge Z$  are partitioned into the set that
  traverse the bridge and the set that transect the bridge.
  A wall traverses the bridge if and only if meets both $\pi_Z(Y)$ and
  $\pi_Y(Z)$, which is true if and only if it meets both $Y$ and $Z$.
  A wall transects the bridge if and only if it separates $Y$ and $Z$,
  which is true if and only if its intersection with
  $Y\bridge Z$ separates $\pi_Y(Z)$ from $\pi_Z(Y)$ in $Y\bridge Z$.
  Every wall that traverses the bridge crosses every wall that
  transects the bridge, and vice versa. 
\end{proposition}

\begin{corollary}\label{coarse_intersection_in_cube_cpx}
  For $X$, $Y$, $Z$ as in \fullref{bridge}, 
$\pi_Y(Z)$, $\pi_Z(Y)$, and $Y\bridge Z$ are coarsely equivalent and
  represent $Y\stackrel{c}{\cap}Z$.
\end{corollary}
\begin{proof}
 Coarse equivalence is the observation $\pi_Y(Z)\subset Y\bridge Z\subset
 \bar\nbhd_{d(Y,Z)}(\pi_Y(Z))$.
 
  For $r\geq \lceil d(Y,Z)/2\rceil$, consider $x\in
  \bar\nbhd_r(Y)\cap\bar\nbhd_r(Z)\neq \emptyset$, and let
  $y:=\pi_Y(x)$ and $z:=\pi_Z(x)$.
  By \fullref{gate}, $\pi_Z(y)$ is on a geodesic from $y$ to $z$,
  so $2r\geq d(y,z)\geq d(y,Y\bridge Z)+d(z,Y\bridge Z)$, so at least one of
  them is bounded above by $r$, which implies $d(x,Y\bridge Z)\leq 2r$.
  Thus, $Y\bridge Z\subset\bar\nbhd_{r}(Y)\cap
    \bar\nbhd_{r}(Z)\subset \bar\nbhd_{2r}(Y\bridge Z)$.
\end{proof}

\subsection{Coarse geometry of standard subcomplexes in RAAGs and
  RACGs}\label{sec:coarse_geometry_standard}

The Davis complex of a RACG and the universal cover of the Salvetti
complex of a RAAG are CAT(0) cube complexes in which every  edge is
labelled/colored by a generator of the group.
Furthermore, since all squares correspond to commutation relations in the group,
opposite sides of any square have the same label. 

\begin{lemma}\label{commutator_labelling}
  If $\Sigma$ is the Davis complex of a RACG or the universal cover of the Salvetti
complex of a RAAG, $X_0$ and $X_1$ are convex subcomplexes, $e$
is an edge of $\pi_{X_0}(X_1)$ labelled $a$,
and $\gamma\from[0,L]\to\Sigma$ is a geodesic from
$\iota(e)$
to $\pi_{X_1}(\iota(e))$, then $\Sigma$ contains a subcomplex
$\gamma\times[0,1]$ such that $\gamma(0)\times[0,1]=e$,
$\gamma(L)\times[0,1]$ is an edge of $X_1$, every edge $\gamma(i)\times [0,1]$
has label $a$, and for every $0<i\leq L$, the two edges
$\gamma([i-1,i])\times\{0,1\}$ have the same label $b_i$, which commutes with $a$. 
\end{lemma}
\begin{proof}
  If $X_0$ and $X_1$ intersect there is nothing to prove: $L=0$ and
  there is a single edge $e\in X_0\cap X_1$ labelled $a$ and the set
  of commuting $b_i$'s is empty.
  Otherwise, \fullref{bridge} implies that $\Sigma$ contains a
  subcomplex $\gamma\times [0,1]$ with $\gamma(0)\times[0,1]=e$ and
  $\gamma(L)\times[0,1]\subset X_1$.
  Since squares are labelled by commutators, 
$\gamma(1)\times[0,1]$ has label $a$, and $\gamma([0,1])\times\{0\}$
and $\gamma([0,1])\times\{1\}$ are edges with the same label, which
commutes with $a$. Repeat for each edge of $\gamma$. 
\end{proof}
\begin{proposition}\label{projection_of_standard_is_standard}
  Let $\Upsilon$ be a graph, let $G_\Upsilon$ be either a RACG or RAAG
  with presentation graph $\Upsilon$, and let $\Sigma_\Upsilon$ be its
  Davis complex or the universal cover of its Salvetti complex,
  respectively.
  Consider  $S_0,S_1\subset\Upsilon$ and $g_0,g_1\in G_\Upsilon$.
  Let $w$ be a minimal length representative of 
  $G_{S_0}(g_0^{-1}g_1)G_{S_1}$, and let $T$ be the set of generators that appear in $w$.
  Then either $T=\emptyset$, which occurs when $g_0\Sigma_{S_0}\cap
  g_1\Sigma_{S_1}\neq\emptyset$, or $T$ contains an element not in
  $S_0$ and an element not in $S_1$.
  Furthermore:
  \[\pi_{g_0\Sigma_{S_0}}(g_1\Sigma_{S_1})=g_0\Sigma_{S_0\cap
      S_1\cap\bigcap_{t\in T}\lk(t)}\]
In particular, projection of a standard subcomplex to a standard
subcomplex is a standard subcomplex.
\end{proposition}
\begin{proof}
  The description of $w$ is equivalent to that of a word read on the
  edges of a minimal length geodesic connecting $g_0\Sigma_{S_0}$ to
  $g_1\Sigma_{S_1}$.
\fullref{bridge} implies that the set $T$ does not depend on the
choice of such a geodesic.
Using the group action, take $h\in
\pi_{g_0\Sigma_{S_0}}(g_1\Sigma_{S_1})\subset g_0\Sigma_{S_0}$ so that  $1\in
h^{-1}\pi_{g_0\Sigma_{S_0}}(g_1\Sigma_{S_1})=\pi_{h^{-1}g_0\Sigma_{S_0}}(h^{-1}g_1\Sigma_{S_1})=\pi_{\Sigma_{S_0}}(h^{-1}g_1\Davis_{S_1})$.
Again by \fullref{bridge}, every vertex in $\pi_{\Sigma_{S_0}}(h^{-1}g_1\Sigma_{S_1})$ is
connected to a vertex in $\pi_{h^{-1}g_1\Sigma_{S_1}}(\Sigma_{S_0})$
by a geodesic labelled $w$, so the geodesic starting at 1 labelled $w$
ends at $w\in \pi_{h^{-1}g_1\Sigma_{S_1}}(\Sigma_{S_0})\subset h^{-1}g_1\Sigma_{S_1}$.
Thus, it suffices to consider the case that $g_0=1$ and $g_1=w$.

  The set $T$ is empty, and $w$ is the empty word, precisely when
$\Sigma_{S_0}\cap \Sigma_{S_1}\neq\emptyset$.
In this case $S_0\cap S_1\cap\bigcap_{t\in T}\lk(t)=S_0\cap S_1$, and
$\pi_{\Sigma_{S_0}}(\Sigma_{S_1})=\Sigma_{S_0\cap S_1}$. 

Suppose that $T$ is not empty and $w$ is not the empty word.
Minimality of $w$ implies that it does not start with an element of
$S_0$ or end with an element of $S_1$, so $T$ contains an element not
in $S_0$ and an element not in $S_1$. 

Let $S'$ be the set of labels that occur on edges of
$\pi_{\Davis_{S_0}}(w\Sigma_{S_1})$.
By \fullref{commutator_labelling} every $s\in S'$ commutes with every
letter of $w$, and if $e\in \pi_{\Davis_{S_0}}(w\Sigma_{S_1})$ is an edge
labelled $s$ then the edge $we\in \pi_{w\Sigma_{S_1}}(\Sigma_{S_0})$
is also labelled $s$, so $S'\subset S_0\cap S_1\cap\bigcap_{t\in
  T}\lk(t)$.
Conversely, if $s\in S_0\cap S_1\cap\bigcap_{t\in
  T}\lk(t)$ then $1\edge s$ is an edge of $\Sigma_{S_0}$ that is
parallel via $w$ to the edge $w\edge ws$ in $w\Sigma_{S_1}$, so 
$S'= S_0\cap S_1\cap\bigcap_{t\in
  T}\lk(t)$.
Furthermore, $\Sigma_{S'}\subset \Sigma_{S_0}$ is parallel to
$w\Sigma_{S'}\subset w\Sigma_{S_1}$ via $w$, so $\Sigma_{S'}=
\pi_{\Davis_{S_0}}(w\Sigma_{S_1})$.
\end{proof}
\begin{corollary}\label{parallelgeodesics}
  If $X_0$ and $X_1$ are standard subcomplexes in either the Davis
  complex of a RACG or the universal cover of the Salvetti complex of
  a RAAG, then $X_0$ and $X_1$ have unbounded coarse intersection if
  and only if there is an unbounded standard subcomplex $X_2$ such
  that $X_0$ and $X_1$ contain parallel copies of $X_2$. 
\end{corollary}
\begin{proof}
  \fullref{coarse_intersection_in_cube_cpx} and
  \fullref{projection_of_standard_is_standard} say 
  $\pi_{X_0}(X_1)$
  and $\pi_{X_1}(X_0)$ are parallel standard subcomplexes representing $X_0\stackrel{c}{\cap}X_1$.
\end{proof}
\begin{corollary}\label{raagtrivialprojection}
  Two standard subcomplexes $X_0$ and $X_1$ of the universal cover of the Salvetti complex of
  a RAAG  have bounded coarse intersection if and only if
  the combinatorial closest point projection map from $X_0$ to $X_1$
  is constant. 
\end{corollary}
\begin{proof}
In a RAAG, once a standard subcomplex contains an edge it contains the
entire bi-infinite monochrome geodesic containing that edge, so
$\pi_{X_0}(X_1)$ and  $\pi_{X_1}(X_0)$ are either single vertices or
unbounded.
\end{proof}

\begin{lemma}\label{joinclique}
If $W_\Gamma$ is a RACG and $A,B\subset\Gamma$ then  $\Davis_B\subset \Davis_A\ceq\Davis_B$
if and only if $A=B\join C$, where $C$ is a clique.
\end{lemma}
\begin{proof}
  If $C$ is a clique then $d_{Haus}(\Davis_B,\Davis_{B\join C})=|C|$.
  Conversely, if $\Davis_B\subsetneq \Davis_A\ceq\Davis_B$ and $A\neq
  B\join C$ for a clique $C$, then either there is an $a\in A\setminus B$ not adjacent
  to some $b\in B$, or there are $a$ and $b$ in $A\setminus B$ not
  adjacent to each other. In either case $\Davis_{\{a,b\}}$ is a line.
  Since $\pi_{\Davis_{\{a,b\}}}$ is Lipschitz,
  $d_{Haus}(\Davis_A,\Davis_B)\geq
  d_{Haus}(\pi_{\Davis_{\{a,b\}}}(\Davis_A),
  \pi_{\Davis_{\{a,b\}}}(\Davis_B))$, the latter of which is infinite,
  since $\Davis_{\{a,b\}}\subset\Davis_A$, but
  $\pi_{\Davis_{\{a,b\}}}(\Davis_B)$ is finite.
\end{proof}
\begin{corollary}\label{cor:joinclique}
  If $W_\Gamma$ is a RACG and $S_0,S_1\subset\Gamma$ then
  $\Davis_{S_0}\ceq\Davis_{S_1}$ implies  $d_{Haus}(\Davis_{S_0},\Davis_{S_1})$ is
  bounded above by the size of the largest clique in $\Gamma$.
\end{corollary}
\begin{proof}
  Let $C$ be the largest clique of $\Gamma$.
  Let $S_i'$ be $S_i$ minus its cone vertices. 
  By \fullref{joinclique}, $d_{Haus}(\Davis_{S_i},\Davis_{S_i'})\leq
  |C|$, so $\Davis_{S_0}\ceq\Davis_{S_1}\implies
  \Davis_{S_0'}\ceq\Davis_{S_1'}$.
  But then $\Davis_{S_0'\cap S_1'}\subset \Davis_{S_0'}\ceq
  \Davis_{S_0'\cap S_1'}$, so \fullref{joinclique} says $S_0'\cap
  S_1'=S_0'$.
  Repeat the argument for $S_1'$, and conclude $S:=S_0'=S_1'$.
  Now every vertex of $\Davis_{S_i}$  is distance at
  most $|C|$ from $\Davis_S\subset\Davis_{S_0}\cap \Davis_{S_1}$.
\end{proof}
\begin{corollary}\label{cor:coarse_intersection_of_standard}
    If $W_\Gamma$ is a RACG, $g\in W_\Gamma$, and $S_0,S_1\subset\Gamma$ then
  $\Davis_{S_0}\ceq g\Davis_{S_1}$ implies there exists $S$ and
  cliques $C_0$ and $C_1$ such that $S_0=S\join C_0$, $S_1=S\join C_1$
  and $g$ centralizes $S$. 
\end{corollary}
\begin{proof}
  By \fullref{cor:joinclique}, we may assume $S_0$ and $S_1$ have no
  cone vertices.
By  \fullref{coarse_intersection_in_cube_cpx}, if
  if $T$ is the set of generators appearing in minimal length
  elements of $W_{S_0}gW_{S_1}$ then:
  \[\Davis_{S_0}\cint
  g\Davis_{S_1}\ceq\pi_{\Davis_{S_0}}(g\Davis_{S_1})=\Davis_{S_0\cap
    S_1\cap\bigcap_{t\in T}\lk(t)}\]
So if $\Davis_{S_0}\ceq g\Davis_{S_1}$ then:
\[
\Davis_{S_0\cap
    S_1\cap\bigcap_{t\in T}\lk(t)}\subset\Davis_{S_0}\ceq(\Davis_{S_0}\cint g\Davis_{S_1})\ceq \Davis_{S_0\cap
    S_1\cap\bigcap_{t\in T}\lk(t)}\]
 \fullref{joinclique} says $S_0=(S_0\cap
    S_1\cap\bigcap_{t\in T}\lk(t))\join C$ for a clique $C$, but
  $S_0$ has no cone vertex, so $C=\emptyset$, $S_0\subset S_1$, and
  $S_0\subset\bigcap_{t\in T}\lk(t)$.
The same argument applies to $S_1$, so we conclude that $S:=S_0=S_1$
and $T$ centralizes $S$.  
\end{proof}
\begin{lemma}
  If $A_\Delta$ is a RAAG, $g\in A_\Delta$, and $S_0,S_1\subset\Delta$
  then $\Salvetti_{S_0}\ceq g\Salvetti_{S_1}$ implies $S:=S_0=S_1$ and
  $g$ centralizes $S$.
\end{lemma}
\begin{proof}
  Let $\Salvetti_S:=\pi_{\Salvetti_{S_0}}(g\Salvetti_{S_1})$. 
Since $\Salvetti_{S_0}\ceq g\Salvetti_{S_1}$, $\Salvetti_{S_0}\ceq
\Salvetti_{S_0}\cint g\Salvetti_{S_1}$, which, by
\fullref{coarse_intersection_in_cube_cpx}, is coarsely equivalent to
$\Salvetti_S$.
By \fullref{projection_of_standard_is_standard}, $S=S_0\cap
S_1\cap\bigcap_{t\in T}\lk(t)$, where $T$ is the letters appearing in
a minimal word in $A_{S_0}gA_{S_1}$. 
But in a RAAG no standard subcomplex is coarsely equivalent to one of
its proper standard subcomplexes, since generators have infinite
order, so $S=S_0$, which implies $S_0=S_1$ and $g$ centralizes $S$.
\end{proof}

\section{CFS graphs}\label{sec:firstexamples}
We know from Dani and Thomas \cite{DanTho15} that an incomplete,
triangle-free graph without separating cliques is RAAGedy only if it
is CFS.
In this section we establish some basic results about the structure of
CFS graphs.
\subsection{The diagonal graph}
\begin{definition}
  The \emph{diagonal graph} $\diag(\Gamma)$ of
  $\Gamma$ is the graph with a vertex $\{a,b\}$ if
  $a$ and $b$ are vertices of $\Gamma$ that are the diagonal vertices
  of some induced square. There is an edge $\{a,b\}\edge\{c,d\}$ in
  $\diag(\Gamma)$ when $\{a,b\}\join\{c,d\}$ is an induced square of
  $\Gamma$. 
\end{definition}
The \emph{support} of a vertex $\{a,b\}$ of $\diag(\Gamma)$ is the set
$\{a,b\}\subset\Gamma$. The support of a subset of $\diag(\Gamma)$ is
the union of supports of its vertices.

\begin{definition}\footnote{This is equivalent to other definitions of CFS 
    \cite{DanTho15,BehFalHag18, BehCicFal25} and strongly CFS
    \cite{RusSprTra23}. These sources make definitions in terms of a
    `square graph' of $\Gamma$ that is slightly different than
    $\diag(\Gamma)$ but ultimately contains the same information.}$^,$\footnote{It is possible \cite{BehCicFal25}
  to produce CFS graphs with multiple components of full support.}
  $\Gamma$ is \emph{CFS} if $\diag(\Gamma)$ contains a connected
  component whose support is all non-cone vertices of $\Gamma$.

  $\Gamma$ is \emph{strongly CFS} if it is CFS and $\diag(\Gamma)$ is
  connected. 
\end{definition}

\begin{lemma}
  A triangle in $\diag(\Gamma)$ corresponds to an induced octahedron
  in $\Gamma$.
\end{lemma}
  \begin{corollary}
      A triangle-free graph has a triangle-free diagonal graph. 
    \end{corollary}

    If $A$ is a set, let $\binom{A}{2}$ be the collection of 2--element
    subsets of $A$.
    The following relates joins in $\Gamma$ and joins in
    $\diag(\Gamma)$, and will be used later. 
    \begin{lemma}\label{diagonaljoins}
      Let $\Gamma$ be triangle-free. 
If $A\join B$ is a thick join in $\Gamma$ then $\binom{A}{2} \join \binom{B}{2}$ is
a join in $\diag(\Gamma)$.
If $A\join B$ is a join in $\diag(\Gamma)$ then $\supp(A)\join\supp(B)$ is a thick join in $\Gamma$ and
$\binom{\supp(A)}{2}\join\binom{\supp(B)}{2}$ is a join in $\diag(\Gamma)$. 
  \end{lemma}
  \begin{proof}
    Suppose $A\join B$ is a thick join in $\Gamma$.
    Let $\{a_0,a_1\}$ be any 2--element subset of $A$, and let
    $\{b_0,b_1\}$ be any 2--element subset of $B$.
    Since $A\join B$ is a join, there is a square
    $\{a_0,a_1\}\join\{b_0,b_1\}$, which is induced, since $\Gamma$ is
    triangle-free, so there is an edge $\{a_0,a_1\}\edge\{b_0,b_1\}$ in
    $\diag(\Gamma)$.

    Conversely, suppose $A\join B$ is a join in $\diag(\Gamma)$.
    Let $a\in\supp(A)$ and $b\in \supp(B)$.
    Then there exists $a'$ and $b'$ such that $\{a,a'\}\in A$ and
    $\{b,b'\}\in B$, so $\{a,a'\}\edge\{b,b'\}$ is an edge in
    $A\join B$, hence in 
    $\diag(\Gamma)$, so $\{a,a'\}\join \{b,b'\}$ is an induced square in
    $\Gamma$.
    In particular, $a$ and $b$ are adjacent in $\Gamma$.
    Since this was true for arbitrary elements of $\supp(A)$ and
    $\supp(B)$, $\Gamma$ contains $\supp(A)\join \supp(B)$.
    Moreover, each factor has size at least two, since supports of
    single vertices of $\diag(\Gamma)$ have size two, and since
    $\Gamma$ is triangle-free this implies each factor is incomplete, so this is a
    thick join in $\Gamma$.
    Finally, by the first part of this lemma, $\binom{\supp(A)}{2}\join\binom{\supp(B)}{2}$ is a join in $\diag(\Gamma)$.
    \end{proof}

Next we state two easy lemmas to formalize constructions that have
appeared in examples elsewhere in the literature.
We use these results without further comment. 
\begin{lemma}
   Let $\Gamma'$ be a CFS graph and let $V$ be a subset of the vertices
  of $\Gamma'$.
  Let $\Gamma$ be obtained from $\Gamma'$ by adding some new edges
  between vertices of $V$.
  If no new edge connects the diagonals of a square of $\Gamma'$ then
  $\Gamma$ is CFS.
\end{lemma}

In particular, if $V$ is a set of vertices of $\Gamma'$ that are
pairwise at distance at least 3 from one another, then no squares of $\Gamma'$
are killed, so $\Gamma$ is CFS.
If we space $V$ even wider, the new edges do not interact with squares
of $\Gamma'$ at all:
\begin{lemma}\label{largediameterleavesroomforsubgraphs}
  Let $\Gamma'$ be a CFS graph and let $V$ be a subset of the vertices
  of $\Gamma'$ that are pairwise at distance at least 4 apart.
  Let $\Gamma$ be obtained from $\Gamma'$ by adding some new edges
  between vertices of $V$, and let $\Gamma''$ be the subgraph induced
  by the new edges. 
Then $\Gamma$ is CFS and $\diag(\Gamma)$ is a disjoint union of $\diag(\Gamma')$ and $\diag(\Gamma'')$. 
\end{lemma}

In this way one can build all kinds of exotic examples of CFS graphs
starting from ones with sufficiently large diameter.
This is the content of an example of Behrstock \cite{MR3966609}, which
can be interpreted as taking $\Gamma'$
to be the large diameter CFS graph obtained by doubling a path graph of
length at least 12 and taking $V$ to be 5 vertices at pairwise
distance at least 3 in $\Gamma'$ whose induced graph
is a pentagon, see \fullref{fig:Behrstock}.
Russell, Spriano, and Tran do something similar \cite[Proposition~7.6]{RusSprTra23}.

\begin{figure}[h]
  \centering
  \includegraphics{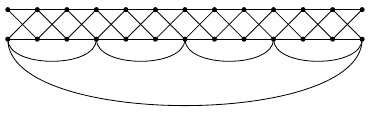}
  \caption{The example of Behrstock of a strongly CFS graph containing
    a stable cycle. }
  \label{fig:Behrstock}
\end{figure}

\subsection{First examples}
\subsubsection{Large diameter contructions}
Here are some simple examples of strongly CFS graphs that
can be constructed to have large diameter.
They serve as basic building blocks for further examples.

\begin{example}\label{snakewormspider}
  Consider the path graph $P_n$ of length $n>0$.
Doubling gives a strongly CFS
  graph of diameter $n$.

  For example $\double(P_7):=$
  \begin{tikzpicture}[scale=.5]
    \tiny
    \coordinate (a0) at (0,0);
    \coordinate (a1) at (0,1);
    \coordinate (b0) at (1,0);
    \coordinate (b1) at (1,1);
    \coordinate (c0) at (2,0);
    \coordinate (c1) at (2,1);
    \coordinate (d0) at (3,0);
    \coordinate (d1) at (3,1);
    \coordinate (e0) at (4,0);
    \coordinate (e1) at (4,1);
     \coordinate (f0) at (5,0);
     \coordinate (f1) at (5,1);
      \coordinate (g0) at (6,0);
      \coordinate (g1) at (6,1);
       \coordinate (h0) at (7,0);
    \coordinate (h1) at (7,1);
    \filldraw (a0) circle (1pt) (a1) circle (1pt) (b0) circle (1pt)
    (b1) circle (1pt) (c0) circle (1pt) (c1) circle (1pt) (d0) circle
    (1pt) (d1) circle (1pt) (e0) circle (1pt) (e1) circle (1pt) (f0)
    circle (1pt) (f1) circle (1pt) (g0) circle (1pt) (g1) circle (1pt)
    (h0) circle (1pt) (h1) circle (1pt);
    \draw (a0)--(b0)--(c0)--(d0)--(e0)--(f0)--(g0)--(h0) (a1)--(b1)--(c1)--(d1)--(e1)--(f1)--(g1)--(h1)
    (a0)--(b1)--(c0)--(d1)--(e0)--(f1)--(g0)--(h1) (a1)--(b0)--(c1)--(d0)--(e1)--(f0)--(g1)--(h0);
  \end{tikzpicture}

  Deleting a vertex from one or both ends does not change the strong CFS
  property:
 the graphs 
    \begin{tikzpicture}[scale=.5]
    \tiny
    \coordinate (a0) at (.5,0.5);
    \coordinate (b0) at (1,0);
    \coordinate (b1) at (1,1);
    \coordinate (c0) at (2,0);
    \coordinate (c1) at (2,1);
    \coordinate (d0) at (3,0);
    \coordinate (d1) at (3,1);
    \coordinate (e0) at (4,0);
    \coordinate (e1) at (4,1);
     \coordinate (f0) at (5,0);
     \coordinate (f1) at (5,1);
      \coordinate (g0) at (6,0);
      \coordinate (g1) at (6,1);
       \coordinate (h0) at (7,0);
    \coordinate (h1) at (7,1);
    \filldraw (a0) circle (1pt)  (b0) circle (1pt)
    (b1) circle (1pt) (c0) circle (1pt) (c1) circle (1pt) (d0) circle
    (1pt) (d1) circle (1pt) (e0) circle (1pt) (e1) circle (1pt) (f0)
    circle (1pt) (f1) circle (1pt) (g0) circle (1pt) (g1) circle (1pt)
    (h0) circle (1pt) (h1) circle (1pt);
    \draw (a0)--(b0)--(c0)--(d0)--(e0)--(f0)--(g0)--(h0) (b1)--(c1)--(d1)--(e1)--(f1)--(g1)--(h1)
    (a0)--(b1)--(c0)--(d1)--(e0)--(f1)--(g0)--(h1) (b0)--(c1)--(d0)--(e1)--(f0)--(g1)--(h0);
  \end{tikzpicture}
 and 
  \begin{tikzpicture}[scale=.5]
    \tiny
    \coordinate (a0) at (.5,0.5);
    \coordinate (b0) at (1,0);
    \coordinate (b1) at (1,1);
    \coordinate (c0) at (2,0);
    \coordinate (c1) at (2,1);
    \coordinate (d0) at (3,0);
    \coordinate (d1) at (3,1);
    \coordinate (e0) at (4,0);
    \coordinate (e1) at (4,1);
     \coordinate (f0) at (5,0);
     \coordinate (f1) at (5,1);
      \coordinate (g0) at (6,0);
      \coordinate (g1) at (6,1);
       \coordinate (h0) at (6.5,.5);
    \filldraw (a0) circle (1pt)  (b0) circle (1pt)
    (b1) circle (1pt) (c0) circle (1pt) (c1) circle (1pt) (d0) circle
    (1pt) (d1) circle (1pt) (e0) circle (1pt) (e1) circle (1pt) (f0)
    circle (1pt) (f1) circle (1pt) (g0) circle (1pt) (g1) circle (1pt)
    (h0) circle (1pt);
    \draw (a0)--(b0)--(c0)--(d0)--(e0)--(f0)--(g0)--(h0) (b1)--(c1)--(d1)--(e1)--(f1)--(g1)
    (a0)--(b1)--(c0)--(d1)--(e0)--(f1)--(g0) (b0)--(c1)--(d0)--(e1)--(f0)--(g1)--(h0);
  \end{tikzpicture}
  are both strongly CFS.
    
Consider a graph $\Gamma$ that is a star with $m>2$ legs
of length $n\geq 2$, so that there are $m$ leaves that are pairwise at
distance $2n$ from one another. Take $\double(\Gamma)$.
This is the $(m,n)$--\emph{spider}.
Optionally, from each foot we can either leave both valence 2 vertices,
and call this a pincer foot, or delete one of them. 
The result is commensurable to a RAAG tree, and has the property that
vertices on different feet have pairwise distance $2n$. 

Here is the $(4,3)$--spider with one pincer foot:
     \begin{tikzpicture}[scale=.4]
    \tiny
    \coordinate (x) at (0,0);
    \coordinate (y) at (0,.75);
    \coordinate (a00) at (3,1);
    \coordinate (a01) at (2.5,1.5);
    \coordinate (a10) at (3.25,-.5);
    \coordinate (a11) at (2.75,0);
    \coordinate (a20) at (2.75,-1.5);
    \coordinate (a21) at (2.5,-1.25);
     \coordinate (b00) at (3.5,3);
    \coordinate (b01) at (3,3.5);
    \coordinate (b10) at (4.25,1);
    \coordinate (b11) at (3.75,1);
    \coordinate (b2) at (3.75,-1);
      \coordinate (c00) at (-3,1);
    \coordinate (c01) at (-2.5,1.5);
    \coordinate (c10) at (-3.25,-.5);
    \coordinate (c11) at (-2.75,0);
    \coordinate (c20) at (-2.75,-1.5);
      \coordinate (d00) at (-3.5,3);
    \coordinate (d01) at (-3,3.5);
    \coordinate (d10) at (-4.25,1);
    \coordinate (d11) at (-3.75,1);
    \coordinate (d2) at (-3.75,-1);
    \filldraw (x) circle (1pt)  (y) circle (1pt)
    (a00) circle (1pt) (a01) circle (1pt) (a10) circle (1pt) (a11) circle
    (1pt) (a20) circle (1pt) (a21) circle (1pt) (b00) circle (1pt) (b01) circle (1pt) (b10) circle (1pt) (b11) circle
    (1pt) (b2) circle (1pt) (c00) circle (1pt) (c01) circle (1pt) (c10) circle (1pt) (c11) circle
    (1pt) (c20) circle (1pt) (d00) circle (1pt) (d01) circle (1pt) (d10) circle (1pt) (d11) circle
    (1pt) (d2) circle (1pt);
    \draw (x)--(a00)--(a10)--(a20) (y)--(a01)--(a11)--(a21)
    (x)--(a01)--(a10)--(a21) (y)--(a00)--(a11)--(a20);
    \draw (x)--(b00)--(b10)--(b2) (y)--(b01)--(b11)--(b2)
    (x)--(b01)--(b10) (y)--(b00)--(b11);
    \draw (x)--(c00)--(c10)--(c20) (y)--(c01)--(c11)
    (x)--(c01)--(c10) (y)--(c00)--(c11)--(c20);
    \draw (x)--(d00)--(d10)--(d2) (y)--(d01)--(d11)--(d2)
    (x)--(d01)--(d10) (y)--(d00)--(d11);
  \end{tikzpicture}
  \end{example}

\begin{proposition}\label{embeddingintostrong}
  Every finite connected graph $\Gamma$ can be isometrically embedded in a
  strongly CFS graph $\Gamma'$.
  Furthermore, if $\Gamma$ is triangle-free then so is $\Gamma'$, and the construction can be arranged so that 
  $v,w\in\Gamma$ are the diagonal vertices of an induced square of
  $\Gamma'$ if and only if they are diagonal vertices of an induced
  square of $\Gamma$.
\end{proposition}
\begin{proof}
  Let $\Gamma$ be a connected finite graph. Let $\Delta$ be a spider
  with legs of length $\max\{3,\lceil \frac{\diam(\Gamma)}{2}\rceil\}$, with one
  single-vertex foot for each vertex of $\Gamma$ that is not contained
  in an induced square of $\Gamma$ and one pincer foot
  for each component of $\diag(\Gamma)$.
  Identify each single vertex foot with its corresponding vertex of
  $\Gamma$.
  For each component of $\diag(\Gamma)$, choose a vertex, which is a
  pair of vertices of $\Gamma$ that are the diagonal of some induced
  square, and identify these two vertices with the two vertices of the
  corresponding pincer foot of the spider. 
  The resulting graph $\Gamma'$ contains $\Gamma$ as an
  isometrically embedded subgraph because the legs of the spider are
  long, so we have not created any shortcuts between vertices of
  $\Gamma$. In particular, all of the squares in $\Gamma$ survive as squares in $\Gamma'$.
  Some vertices of $\Delta$ may get identified, but those that do
  come from different feet, so they have distance at least 6 in $\Delta$; 
  thus, all of the squares of $\Delta$ survive as squares in
  $\Gamma'$, and we did not create any triangles.

 By construction, every vertex of $\Gamma'$ lies in a square, so to establish that $\Gamma'$ is strongly CFS, it is enough to show that $\diag(\Gamma')$ is connected.
Since $\Delta$ is strongly CFS, we conclude that vertices of
$\diag(\Gamma')$ that are the diagonal of a square with edges only
from $\Gamma$ or only from $\Delta$ are contained in a common
connected component of $\diag(\Gamma')$. 
  This leaves us to consider vertices of $\diag(\Gamma')$ coming from squares that use edges from
  both $\Gamma$ and $\Delta$.
  The only vertices of $\Delta$ that attach to $\Gamma$ and
  are at distance less than $6$ from one another inside $\Delta$ are the pincers of a
  foot, so we may assume $a$ and $b$ are the pincers, and take $d_0$ and $d_1$ to be
  their two common neighbors in $\Delta$.
By construction, $\{a,b\}$ is also the diagonal of a square in
$\Gamma$, so $a$ and $b$ have common neighbors $c_0,c_1,\dots$ in
$\Gamma$. 
 For any $i$ and $j$, there is a square $\{a,b\}\join\{c_i,d_j\}$ with
 $\Gamma$--edges $c_i\join\{a,b\}$ and $\Delta$ edges
 $d_j\join\{a,b\}$.
 We have that $\{c_i,d_j\}$ is a vertex of $\diag(\Gamma')$ that does
 not occur in either $\diag(\Delta)$ or $\diag(\Gamma)$, but it is
 adjacent to the vertex $\{a,b\}$, which occurs in both.

  The further claim about squares follows from the construction, as the `mixed'
  squares with edges from $\Gamma$ and $\Delta$ only occur when a
  pincer foot attached to the diagonal of an existing square of $\Gamma$.
\end{proof}

We remark that $\Gamma$ embeds isometrically in $\double(\Gamma)$,
which is strongly CFS, but $\double(\Gamma)$ does not satisfy the further claim about squares. 

\cite[Proposition~7.6]{RusSprTra23} says
any graph $\Gamma$ can be isometrically embedded into a CFS graph
$\Gamma'$ in such a way that $\Gamma$ is square complete in
$\Gamma'$.
If $\Gamma$ contains a square then the $\Gamma'$ produced by that construction will not be strongly CFS, whereas
the $\Gamma'$ produced by 
\fullref{embeddingintostrong} will not contain $\Gamma$ as a square
complete subgraph. 
This tradeoff is unavoidable; it is not possible to embed a subgraph
containing a square into a strongly CFS graph and have it be square
complete, see
\fullref{strongly_CFS_and_no_cone_implies_minsquare}.

\subsubsection{Blow-up graphs}
\begin{definition}\label{def:mixedmultiple}
  Let $(\Delta,\omega)$ be a \emph{weighted graph}, consisting of a
  graph $\Delta$ and a weight function $\omega\from
  \mathrm{Vertices}(\Delta)\to\mathbb{N}$.
  The \emph{blow-up graph} $\Delta^\omega$ is the graph with $\omega(v)$--many vertices
  $(v,0)$,\dots,$(v,\omega(v)-1)$ for each vertex $v\in \Delta$, and
  such that $(v,i)$ and $(w,j)$ are connected by an edge in $\Delta^\omega$
  if and only if $v$ and $w$ are connected by an edge in $\Delta$.
\end{definition}
\begin{remark}
 $\Gamma\cong
  \Gemini(\Gamma)^\omega$, where $\omega(M)$ is the number
  of vertices in the twin module $M$. 
\end{remark}
\begin{numberedremark}\label{twinfree}
Without changing the graph structure of $\Delta^\omega$, 
  we may assume that $\Delta$ is twin-free by replacing
  $\Delta$ with $\Gemini(\Delta)$ and labelling each twin module by
  the sum of the labels of its vertices.
  This only changes the labelling of vertices of $\Delta^\omega$.

Assuming $\Delta$ is twin-free implies, in particular, that every vertex of $\Delta$ has at most one  adjacent leaf.
Furthermore, if $\Delta$ is twin-free then for $\Gamma:=\Delta^\omega$
we have that $\Gemini(\Gamma)\cong\Delta$ and that $\omega$ gives
  the cardinality of each twin module of $\Gamma$.
Thus, if $\omega$ takes any odd values then $\Gamma$ is not a graph
double.
On the other hand, if $\omega$ takes only even values then $\Gamma\cong\double(\Delta^{\omega/2})$.
\end{numberedremark}
\begin{proposition}\label{mixed_multiples_are_strongly_CFS}
 Let $\Gamma:=\Delta^\omega$ be a blow-up graph, with the conditions that $\Delta$ is connected, triangle-free,
 twin-free, has at least two vertices, and 
 $\omega$ takes the value 1 only on leaves of $\Delta$, and if
 $\Delta$ is a single edge then $\omega$ does not take the value 1.  
  Then $\Gamma$ is triangle-free and strongly CFS.
\end{proposition}
\begin{proof}
  Projection to the first coordinate gives a map $\Gamma\to\Delta$
  that sends edges to edges, and the preimage of a vertex is an
  anticlique.
  It follows that, when $\Delta$ is not a single vertex, $\Gamma$ is
  connected and triangle-free if and only if $\Delta$ is.

  If $\Delta$ is a single edge then by hypothesis $\Gamma\cong
  K_{m,n}$ with $m,n\geq 2$.
  It is easy to see that this graph is strongly CFS, so assume
  $\Delta$ is not a single edge.

  For each edge $v\edge w$ in $\Delta$ such that both of $\omega(v)$ and
  $\omega(w)$ are greater than one there are induced squares
  $\{(v,i),(v,j)\}\join\{(w,k),(w,\ell)\}$ in $\Gamma$ for all $i\neq j$ and
  $k\neq \ell$.
  Also, since weight 1 only occurs on leaves, for every embedded segment $v\edge w\edge x$ in
  $\Delta$ there are squares $\{(v,i),(x,j)\}\join\{(w,k),(w,\ell)\}$
  in $\Gamma$ for all $k\neq\ell$.
Finally, if $\{u,v\}\join\{w,x\}$ is an induced square in $\Delta$
then there are induced squares
$\{(u,i),(v,j)\}\join\{(w,k),(x,\ell)\}$ in $\Gamma$.
Conversely, each induced square in $\Gamma$ projects to either a
square or a path of length 1 or 2 in $\Delta$, so 
the squares constructed above account for 
all of the induced squares of $\Gamma$, and this, in turn,  accounts for
all vertices and edges of $\diag(\Gamma)$.
Since $\Delta$ is not a single edge, every vertex lies on some embedded
segment of length 2, so every vertex belongs to some square.
Thus, the support of $\diag(\Gamma)$ is all of $\Gamma$.
We will show that $\diag(\Gamma)$ is connected, which then implies
$\Gamma$ is strongly CFS.

  Since $\omega$ takes the value 1 only on leaves and $\Delta$ is
  twin-free, each vertex is adjacent to at most one vertex with
  weight 1.
  Furthermore, since $\Delta$ is twin-free and not a single edge,
  every vertex is adjacent to a non-leaf, so every vertex is adjacent
  to a vertex with weight greater than 1.

Define $\phi\from\Delta\to\diag(\Gamma)$ by $\phi(v)=\{(v,0),(v,1)\}$ if $\omega(v)>1$.
If $\omega(v)=1$ then the hypotheses on $\Delta$ imply there is a unique
neighbor $w$ of $v$, $\omega(w)>1$, and there exist vertices at distance 2 from $v$.
Choose any $x$ at distance 2 from $v$, and define
$\phi(v)=\{(v,0),(x,0)\}$, which is a diagonal of the induced square
$\{(v,0),(x,0)\}\join\{(w,0),(w,1)\}$ of $\Gamma$. 
Then $\phi$ sends edges of $\Delta$ to edges of $\diag(\Gamma)$ and is
injective on vertices; the only potential source of collisions would
be if $v$ and $x$ are both leaves weighted 1 at distance 2 from each
other such that $\phi(v)=\{(v,0),(x,0)\}=\phi(x)$, but this is ruled
out by the assumption that $\Delta$ is twin-free.
Thus, $\phi(\Delta)$ is connected.

We claim that every vertex of $\diag(\Gamma)$ is in $\phi(\Delta)$
or is adjacent to a vertex of $\phi(\Delta)$.
Since vertices of $\diag(\Gamma)$ come from one of the three types of
induced squares in $\Gamma$ enumerated above, we only need to consider the following three cases. 
If $\omega(v)>1$ then, since $v$ has some neighbor $w$ not weighted 1,
$\{(v,i),(v,j)\}\edge\{(w,0),(w,1)\}=\phi(w)$ in $\diag(\Gamma)$ for all
$i\neq j$. 
Similarly, if $d(v,x)=2$ then they have a common neighbor $w$, which
is not a leaf, so not weighted 1, and
$\{(v,i),(x,j)\}\edge\{(w,0),(w,1)\}=\phi(w)$ in $\diag(\Gamma)$ for
all $i\neq j$.
Finally, if $\{(u,i),(v,j)\}\join\{(w,k),(x,\ell)\}$ is an induced
square in $\Gamma$ coming from an induced square
$\{u,v\}\join\{w,x\}$ of $\Delta$ then none of the vertices are
leaves in $\Delta$, so none are weighted 1, so
$\{(v,i),(x,j)\}\edge\{(w,0),(w,1)\}=\phi(w)$. 
\end{proof}

In \fullref{sec:clone} we upgrade the conclusion of 
\fullref{mixed_multiples_are_strongly_CFS} to `RAAGedy'.

\subsection{Enumeration of small triangle-free CFS graphs}\label{enumeration}
We enumerated triangle-free CFS graphs of small order.
\fullref{cfsenumeration} gives the number
of isomorphism types of graph by number of vertices (V) and edges (E). Blank entries
are 0.

\begin{table}[h!]
  \tiny
\begin{tabular}[h]{c|cccccccccc}
  E\textbackslash V& 4 & 5&6&7&8&9&10&11&12\\
  \hline \cline{2-2}
  4  &1&&&&&&&&\\
  5  &&&&&&&&&\\\cline{3-3}
  6  &&1&&&&&&&\\
  7  &&&&&&&&&\\\cline{4-4}
  8  &&&2&&&&&&\\
  9  &&&1&&&&&&\\
  10&&&&3&&&&&\\\cline{5-5}
  11&&&&1&&&&&\\
  12&&&&1&8&&&&\\
  13&&&&&6&&&&\\\cline{6-6}
  14&&&&&3&19&&&\\
  15&&&&&2&21&&&\\
  16&&&&&1&17&61&&\\
  17&&&&&&7&115&&\\\cline{7-7}
  18&&&&&&4&119&207&\\
  19&&&&&&1&71&616&\\
  20&&&&&&1&37&950&828\\
  21&&&&&&&17&782&3820\\\cline{8-8}
  22&&&&&&&7&461&8722\\
  23&&&&&&&3&212&10863\\
  24&&&&&&&2&103&8492\\
  25&&&&&&&1&42&4856\\
  26&&&&&&&&19&2385\\\cline{9-9}
  27&&&&&&&&7&1082\\
  28&&&&&&&&4&477\\
  29&&&&&&&&1&204\\
  30&&&&&&&&1&89\\
  31&&&&&&&&&38\\\cline{10-10}
  32&&&&&&&&&17\\
  33&&&&&&&&&7\\
  34&&&&&&&&&3\\
  35&&&&&&&&&2\\
  36&&&&&&&&&1\\
\end{tabular}
\caption{Number of isomorphism types of triangle-free CFS graphs by
  numbers of vertices and edges. In each column the horizontal line designates the point
  below  which all graphs are forced to be bipartite. 
  }
\label{cfsenumeration}
\end{table}

Mantel's Theorem says that a triangle-free graph
with $n$ vertices has at most $\lfloor\frac{n^2}{4}\rfloor$ edges,
with equality if and only if the graph is $K_{\frac{n}{2},\frac{n}{2}}$
when $n$ is even or $K_{\frac{n+1}{2},\frac{n-1}{2}}$ when $n$ is
odd.
A complete bipartite graph is strongly CFS when both parts are non-singletons,
so this gives us the unique largest triangle-free CFS graph for
each fixed number of vertices.
Similarly, it is an exercise in extremal graph theory to show that
  a triangle-free graph with $n$ vertices and strictly greater than 
  $\frac{(n-1)^2}{4}+1$ edges is bipartite.
  Further, it can be shown that these conditions also imply the graph
  is strongly CFS.

\subsection{Inductive construction}
\cite[Theorem~II]{BehHagSis17cox} says that the class of thick RACGs
is the class whose defining graphs belong to the smallest class of
graphs containing the square and closed under the following
operations:
\begin{itemize}
\item Cone off a subgraph that is not a clique.
\item Amalgamate two graphs $\Gamma_1$, $\Gamma_2$ in the class over a
  common subgraph $\Gamma'$ that is not a
  clique. 
\item Add additional edges to the previous item between vertices of
  $\Gamma_1-\Gamma'$ and vertices of $\Gamma_2-\Gamma'$.
\end{itemize}

CFS graphs are thick, so they can be built iteratively as above.
However:
\begin{itemize}
\item Not all thick graphs are CFS.
  \item The `amalgamate and add edges' operation is useful for
    constructing examples, but makes inductive
    arguments more complicated. 
  \end{itemize}
  We show that for CFS graphs only the cone-off operation is needed.

We say that a property
$\mathcal{P}$ is \emph{constructible by coning from a square} if for
every graph $\Gamma$ with property $\mathcal{P}$ there exists a
sequence $\Gamma_0\subset\cdots\subset\Gamma_n$ where each $\Gamma_i$
has property $\mathcal{P}$, $\Gamma_0$ is a
square, $\Gamma_n=\Gamma$, and $\Gamma_{i+1}=\Gamma_i\join_{N_i} v_i$
is obtained from $\Gamma_i$ by coning off some subgraph $N_i$.
  
\begin{proposition}\label{inductiveconstruction}
 The property of being an incomplete CFS graph is constructible by
 coning from a square:
An incomplete CFS graph can be built from a square as a sequence of CFS graphs
such that each successive step is a cone-off of an incomplete
subgraph of the previous graph. 
\end{proposition}
\begin{proof}
  Let $\Gamma$ be an incomplete CFS graph.
 Since cone vertices can be added last and not every vertex is a cone
 vertex, it suffices to assume that $\Gamma$ has no cone vertex.
 Let $C$ be a component of $\diag(\Gamma)$ with $\supp(C)=\Gamma$.

  Pick an edge $\Delta_0:=\{a,b\}\edge\{c,d\}$ in $C$.
  By definition, this corresponds to an induced square
  $\Gamma_0:=\{a,b\}\join\{c,d\}$ in $\Gamma$.
  If $\Gamma$ is a square there is nothing more to prove, so assume
  $\Delta_0$ is not all of $C$. 

  We inductively construct a sequence of connected subgraphs
  $\Delta_0\subset\cdots\subset\Delta_i\subset C$ and induced CFS
  subgraphs $\Gamma_0\subset\cdots\subset\Gamma_i\subset \Gamma$, such
  that each $\Gamma_{i}$ is obtained from $\Gamma_{i-1}$ by coning off
  an incomplete subgraph, and 
  with the property that for each $i$ one of the following is true:
  \begin{itemize}
  \item $\Delta_i$ is a full support connected
    subgraph of $\diag(\Gamma_i)$ (so $\Gamma_i$ has no cone vertex).
   \item $\Gamma_i$ has exactly one cone vertex, $v$, and $\Delta_i$
     is a connected subgraph of $\diag(\Gamma_i)$ with support
     $\Gamma_i\setminus v$.
   \end{itemize}
   
Suppose $\Delta_i$ and $\Gamma_i$ have been constructed. 
  Suppose $\Delta_i$ is not all of $C$. 
Extend $\Delta_i$ to the largest connected subgraph
 $\Delta_i'$ of $\diag(\Gamma)$ that contains $\Delta_i$ and whose
 support is contained in $\Gamma_i$.
 Suppose $\Delta_i'$ is not all of $C$. 
 We extend to $\Delta_{i+1}$ and $\Gamma_{i+1}$ according to two 
 (exhaustive)
 cases:
 
 Case A: There is a vertex $\{v,w\}$ adjacent to some $\{e,f\}\in\Delta_i'$ in
 $\diag(\Gamma)$ such that $v\notin \Gamma_i$ and $w\in\Gamma_i$.

Case B: There is no such vertex as in Case A, but there is a vertex $\{v,w\}$ adjacent to some $\{e,f\}\in\Delta_i'$ in
 $\diag(\Gamma)$ such that $v\notin \Gamma_i$ and $w\notin\Gamma_i$.

 In Case A, let $\Delta_{i+1}:=\Delta_i'\cup\{v,w\}$, and let $\Gamma_{i+1}$ be
 the induced subgraph of $\Gamma$ containing $\Gamma_i$ and $v$.
 We have $\Gamma_{i+1}= \Gamma_i\join_{\lk_\Gamma(v)\cap\Gamma_i}v$.
 Since $\{e, f\}* \{v, w\}$ is an induced square in $\Gamma$, the subset $\lk_\Gamma(v)\cap\Gamma_i\subset \Gamma_i$ contains the non-adjacent vertices $e$
 and $f$, but does not contain $w$.  So it is a proper, incomplete
 subgraph of $\Gamma_i$.

 In Case B, supposing $\Gamma_i$ has no cone vertex, we claim that both  $v$ and $w$ are adjacent to every
 vertex of $\Gamma_i$ and define $\Delta_{i+1}:=\Delta_i'$ and
 $\Gamma_{i+1}:=\Gamma_i \join v$.
 By construction, $\Gamma_i$ is not a clique.
 
 To see the claim, suppose, to the contrary, that there is a vertex
 $x$ of $\Gamma_i$ that is not adjacent to $v$.
Since $x$ is not a cone vertex of $\Gamma_i$, there exist vertices of
$\Delta_i'$ with support containing $x$.
We may assume $x$ and $\{x,y\}\in\Delta_i'$ have been chosen such that
that  $\{x,y\}$ is a closest vertex to $\{e,f\}$
 in $\Delta_i'$ among those whose support contains a vertex not
 adjacent to $v$ or $w$. 
 Take a geodesic
 $\{e,f\}=\{e_0,f_0\}$, $\{e_1,f_1\}$,\dots, $\{e_n,f_n\}=\{x,y\}$ in
 $\diag(\Gamma)$.
 The `closest' hypothesis implies that for $j<n$ both of $e_j$ and $f_j$ are adjacent
to both of 
 $v$ and $w$,
so $\{v,w\}\edge \{e_j,f_j\}$ is an edge in
 $\diag(\Gamma)$.
 Thus, by replacing $\{e,f\}$ with $\{e_{n-1},f_{n-1}\}$, we may assume $\{x,y\}$ and $\{e,f\}$ are adjacent in
 $\diag(\Gamma)$.
 But this gives a join $\{v,w,x,y\}\join\{e,f\}$ in $\Gamma$ with $v$ not adjacent to $x$, yielding an induced square
 $\{v,x\}*\{e,f\}$ in $\Gamma$,
 which is a contradiction, because in that case 
 $\{v,x\}$ is available for Case A.

 Notice that after performing a Case B extension, the vertex $\{v,w\}$
 has $v\in\Gamma_{i+1}$ and $w\notin\Gamma_{i+1}$, so the subsequent
 inductive step will be a Case A extension.
 This justifies the supposition that $\Gamma_i$ has no cone vertex
 when Case B is applied. 

We now check that if the induction hypotheses are true up to stage $i$
then they are true at $i+1$.
Since the $\Gamma_i$ are induced subgraphs of $\Gamma$, their diagonal
graphs are subgraphs of $\diag(\Gamma)$.

Suppose $\Gamma_i$ was obtained from
$\Gamma_{i-1}$ by a Case B extension.
Since we never perform consecutive Case B extensions, the support of
$\Delta_{i-1}$ is all of $\Gamma_{i-1}$, so the support of $\Delta_i$
is all of $\Gamma_i$ except the cone vertex added in the last
extension. 

 Suppose $\Gamma_i$ is obtained by a Case A extension and
 $\Gamma_{i-1}$ has no cone vertex.
 Let $\{v,w\}$ and $\{e,f\}$ be as in the description of Case A. 
 Since $\Gamma_{i-1}$ has no cone vertex, $\Delta_{i-1}$ has support
 all of $\Gamma_{i-1}$.
 Since $\Delta_i$ contains $\Delta_{i-1}$ and $\{v,w\}$, its support
 contains all vertices of $\Gamma_i$.

 Finally, suppose $\Gamma_{i-1}$ has a cone vertex.
 This occurs when $\Gamma_{i-1}$ is constructed from $\Gamma_{i-2}$ by
 a Case B extension.
 Let $\{v,w\}$ and $\{e,f\}$ be as in the description of Case B, so
 that $\Gamma_{i-1}=\Gamma_{i-2}\join v$ and $\Gamma_i$ is obtained from
 $\Gamma_{i-1}$ by coning-off the proper subgraph $\Gamma_{i-2}
 \subset \Gamma_{i-1}$ to $w$.
Since we never perform consecutive Case B extensions, the support of
$\Delta_{i-2}$ contains all vertices of $\Gamma_{i-2}$.
Since $\Delta_i$ contains $\Delta_{i-2}$ and $\{v,w\}$, its support is
all vertices of $\Gamma_i$. 
\end{proof}

\begin{corollary}
  An incomplete CFS graph with $n\geq 4$ vertices has at least $2n-4$ edges.
\end{corollary}
\begin{proof}
  The claim is true for a square. Proceed by induction on the
  construction by cone-offs, each of which adds
  one vertex and at least two edges. 
\end{proof}
\begin{corollary}
  A triangle-free CFS graph with $n\geq 4$ vertices has between $2n-4$ and
  $\lfloor\frac{n^2}{4}\rfloor$ edges, and both extremes are
  realized. 
\end{corollary}
\begin{proof}
  The minimum number of edges occurs for the suspension  $K_{2,n-2}$, and the maximum
  occurs for either $K_{\frac{n}{2},\frac{n}{2}}$ or
  $K_{\frac{n-1}{2},\frac{n+1}{2}}$, according to the parity of $n$.
\end{proof}
\begin{corollary}
  The property of being a triangle-free CFS graph is constructible by
  coning from a square.
\end{corollary}
\begin{proof}
  If $\Gamma$ is triangle-free then so are all subgraphs, so the proof
  of \fullref{inductiveconstruction} produces a triangle-free graph at
  each stage.
\end{proof}
\begin{example}
The property of being a cone-vertex-free CFS graph is not
  constructible by coning from a square:  The 1--skeleton of the
  octahedron is a CFS graph with no cone vertex, but removing any vertex leaves a cone on a
  square.  
\end{example}
\begin{example}
  The property of being strongly CFS is not constructible by coning
  from a square:
  The graph
  $\Gamma:=\texttt{circulant}(11,\{1,3\})=Cay(\mathbb{Z}_{11},\{1,3\})$
  shown in \fullref{fig:circulant}
  is the smallest example of a triangle-free strongly CFS graph such
  that for every vertex $v$ the graph $\Gamma\setminus\{v\}$ is not strongly CFS. Thus, $\Gamma$ cannot be built by coning from a square while
  remaining strongly CFS at each step.
  \begin{figure}[h]
    \centering
    \includegraphics{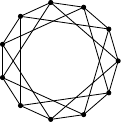}
    \caption{$\texttt{circulant}(11,\{1,3\})$}
    \label{fig:circulant}
  \end{figure}
\end{example}

The graph of \fullref{fig:circulant} turns out not to be RAAGedy.
In fact, we do not know an example of a triangle-free graph without
separating cliques that is RAAGedy but
cannot be constructed by coning from a square through RAAGedy graphs. 
\begin{question}\label{question:inductive_construction_raagedy}
  Is the property of being triangle-free and RAAGedy constructible
  by coning from a square?
\end{question}

\section{Establishing that a graph is RAAGedy}\label{sec:positive}
In \fullref{sec:neardouble} we consider the possibility that a graph
that is not a graph double might become a graph double after applying
a sequence of link doubling operations.
At the group level this corresponds to a sequence of passages to index-two subgroups, ending with one that is commensurable to a RAAG, so the
original group is also commensurable to a RAAG.
It turns out, \fullref{recognizeneardouble}, that if such a sequence
exists, then at most two link doubling steps are necessary, and the
existence of the sequence and identification of which vertices to
double over the links of are recognizable in the presentation graph. 

In \fullref{sec:clone} and \fullref{sec:unfolding} we introduce two
new operations, \emph{cloning} and \emph{unfolding}, that change a
graph $\Gamma$ without changing the quasiisometry type of $W_\Gamma$. 
This gives us three such operations, the third being link doubling.
Unlike link doubling, for cloning and unfolding we only know that they
produce quasiisometric groups, not whether the resulting group is
commensurable to the one we started with. 
In \fullref{sec:qiclasses} we explore the connections within our
enumeration of small CFS graphs given by these three graph
modification operations. 

\subsection{Near doubles}\label{sec:neardouble}
Recall the graph constructions of doubling, link doubling, and star
doubling of \fullref{prelim:graphs} and that
  $A_\Delta$ is commensurable to $W_{\double(\Delta)}$, which is \fullref{thm:davisjanuszkiewicz}.
  The proof of the following is elementary and is left to the reader.
We will not use the lemma directly, but its interpretation as a statement that
the property of being a graph double is stable under link doubling
inspired \fullref{def:neardouble}.
  \begin{lemma}\label{exchangedoubles}
  $\double^\circ_{(v,1)}(\double(\Gamma))\cong\double(\double^*_v(\Gamma))$, that is, they are isomorphic graphs.
\end{lemma}
\begin{definition}\label{def:neardouble}
  A \emph{near double} is a graph $\Gamma$ such that there
  exists a sequence of link doubles of $\Gamma$ such that the result
  is isomorphic to some $\double(\Delta)$. 
\end{definition}
\begin{proposition}\label{neardoublecommraag}
  If $\Gamma$ is a near double then $W_\Gamma$ is commensurable to a RAAG.
\end{proposition}
\begin{proof}
  Apply  \fullref{vertexdouble} and \fullref{thm:davisjanuszkiewicz}.
\end{proof}

We first give a characterization
of graph doubles, which then enables us to 
give an example of a near double that is not a
double.
After that we characterize near doubles in
\fullref{recognizeneardouble}.
It turns out that the length of the necessary sequence of link doubles
is bounded by 2, but \fullref{recognizeneardouble} is a more concrete
description than just testing all such bounded sequences.

Recall the definition of a twin module in \fullref{sec:twins}.

\begin{proposition}\label{characterizedouble}
  Let $\Gamma$ be a graph. The following are equivalent:
  \begin{enumerate}
  \item $\Gamma$ is a double, ie, there exists $\Delta$ such
    that $\Gamma\cong\double(\Delta)$.\label{item:isdouble}
      \item Every twin module in $\Gamma$ has even order. \label{item:evenorder}
    \item $\Gamma$ admits a fixed-point-free involution that fixes
      each twin module.\label{item:involution}
  \end{enumerate}
\end{proposition}
\begin{proof}
  \eqref{item:isdouble}$\implies$\eqref{item:evenorder}, since
  vertices of a double come in pairs with equal links. 

  \eqref{item:evenorder}$\implies$\eqref{item:involution} by choosing
  a pairing of the vertices within each twin module. 
  The map that exchanges the vertices of each such pair satisfies \eqref{item:involution}.

  \eqref{item:involution}$\implies$\eqref{item:isdouble} by defining a
  graph $\Delta$ by taking a
  vertex for each orbit of the involution $\phi$,  and connecting two vertices
  $v$ and $w$ by an edge if the twin module of $v$ is adjacent to the
  twin module of $w$ in the twin graph $\Gemini(\Gamma)$.
  By construction, there is a bijection between 
  vertices of $\double(\Delta)$ and vertices of $\Gamma$; it is
  $(v,0)\mapsto v$ and $(v,1)\mapsto \phi(v)$.
Every edge of $\double(\Delta)$ belongs to an induced square of the
form $\{(v,0),(v,1)\}*\{(w,0),(w,1)\}$, where $v\edge w$ is an edge of
$\Delta$, and corresponds to an induced square
$\{v,\phi(v)\}*\{w,\phi(w)\}$ of $\Gamma$.
Conversely, every edge of $\Gamma$ belongs to such an induced square,
since twin modules are anticliques and $v\edge w$ implies $v\edge
\phi(w)$, since $w$ and $\phi(w)$ are twins.
\end{proof}

 \begin{example}
    The graph $\Gamma$ on the left in \fullref{fig:first_link_doubling_example} is not a double; vertices 6 and 7 have no twin.
Two link doubles turn $\Gamma$ into a doubled octagon, as shown in \fullref{fig:first_link_doubling_example}.  We compactify notation by writing
  $(v,\alpha)$ as $v_\alpha$, 
  $(v,\alpha,\beta)$ as $v_{\alpha\beta}$, etc.
  \begin{figure}[h]
    \centering
    \begin{subfigure}{.18\textwidth}
      \centering
      \raisebox{-40pt}{\includegraphics{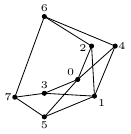}}
    \end{subfigure}
\hfill$\stackrel{\double^\circ_7}{\longrightarrow}$\hfill
      \begin{subfigure}{.25\textwidth}
      \centering
      \raisebox{-40pt}{\includegraphics{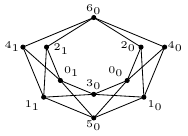}}
    \end{subfigure}
\hfill$\stackrel{\double^\circ_{(6,0)}}{\longrightarrow}$\hfill
      \begin{subfigure}{.27\textwidth}
      \centering
      \raisebox{-40pt}{\includegraphics{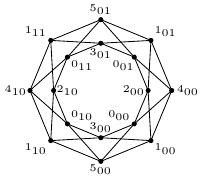}}
     \end{subfigure}
    \caption{First link doubling example.}
    \label{fig:first_link_doubling_example}
  \end{figure}
\end{example}

In Proposition \ref{recognizeneardouble} below, we describe twin modules as \emph{even/odd} according to their
orders.
We refer to the \emph{link} $\lk(M)$ of a twin module $M$, which
we take to mean $\lk(M):=\lk_{\Gemini(\Gamma)}(M)$, that is, 
the link of the vertex $M$ in the twin graph $\Gemini(\Gamma)$ (recall \fullref{def:twingraph}).
Let  $M_\Gamma(v)$ denote the twin module of $\Gamma$ containing $v$.

\begin{lemma}\label{doublingmodules}
  Let $\Gamma$ be an incomplete triangle-free graph without separating cliques.
  For $v\in \Gamma$,   
   the twin modules of $\double^\circ_v(\Gamma)$ are as follows:
    \begin{itemize}
    \item If $A=M_\Gamma(v)$ and $|A|>1$ then there is a single twin
      module corresponding to $A$  in $\double^\circ_v(\Gamma)$  of
      size $2(|A|-1)$, given by $(A\setminus\{v\})\times\{0,1\}$.
    \item If $A$ is a satellite of $M_\Gamma(v)$ in $\Gemini(\Gamma)$
      then there is a single twin module corresponding to $A$  in $\double^\circ_v(\Gamma)$  of
      size $2|A|$, given by $A\times\{0,1\}$.
   \item If $A\in\lk_{\Gemini(\Gamma)}M_\Gamma(v)$ then
      there is a single twin module corresponding to $A$ in  $\double^\circ_v(\Gamma)$ of
      size $|A|$, given by $A\times\{0\}=A\times\{1\}$.
     \item Otherwise, there are two distinct twin modules of size
       $|A|$ 
       corresponding to $A$  in $\double^\circ_v(\Gamma)$, given by
       $A\times\{0\}$ and $A\times\{1\}$.
     \end{itemize}
   
  \end{lemma}
  \begin{proof}
  The link of every vertex of $\Gamma$ is an anticlique of size at least 2.
    
    If $w \in M_{\Gamma}(v)\setminus \{v\}$, i.e., $w\neq v$ and $\lk_\Gamma(v)=\lk_\Gamma(w)$, then $w\notin
    \lk_\Gamma(v)$, so $w$ has distinct preimages $(w,0)$ and
    $(w,1)$ in
    $\double^\circ_v(\Gamma)$, but they have the same link, $\lk_{\double^\circ_v(\Gamma) }(w,0)=\lk_{\double^\circ_v(\Gamma) }(w,1)=\lk_\Gamma(v)\times\{0\}$, so
    the twin module $M_{\double_v^\circ(\Gamma)}(w,0)$ contains 2 copies of each vertex in
    $M_\Gamma(v)\setminus\{v\}$.

     If $\lk_\Gamma(w)\subsetneq\lk_\Gamma(v)$, then
     $w\notin\lk_\Gamma(v)$, but all of its neighbors are, so there 
     are two copies $(w,0)$ and $(w,1)$ of $w$ whose links
     are both $\lk_\Gamma(w)\times\{0\}$.
     Suppose $u$ is another vertex of $M_{\Gamma}(w)$, so $(u,0)$ has
     the same link as $(w,0)$ in $\double^\circ_v(\Gamma)$.
     Then $u\neq v$, since $\lk_\Gamma(w)\subsetneq\lk_\Gamma(v)$ is
     proper containment, and $u\notin\lk_\Gamma(v)$ by our hypotheses
     on $\Gamma$, since if it were then $u$ would be adjacent only to
     $v$, or to something also in $\lk_\Gamma(v)$, contradicting
     either that $\Gamma$ is connected without cut vertices, or that
     it is triangle-free.
     Therefore, $u\notin\lk_\Gamma(v)$, so $u$ contributes
     vertices $(u,0)$ and $(u,1)$ to $M_{\double^\circ_v(\Gamma)}(w,0)$. Thus, $M_{\double_v^\circ(\Gamma)}(w,0)$ contains 2 copies of each vertex in
    $M_\Gamma(w)$.

     If $w\in\lk_\Gamma(v)$ then there is only one vertex $(w,0)$ in
     the preimage of $w$.
     The fact that $\Gamma$ is triangle-free and connected without cut
     vertices implies that $\lk_\Gamma(w)$ contains only $v$ and a
     nonempty set of vertices from outside $\lk_\Gamma(v)$. 
     Choose $u\in\lk(w)\setminus\{v\}$, so $(w,0)$
     is adjacent to a distinct pair $(u,0)$ and $(u,1)$.
     Now, a vertex $(x,0)$ in $\double_v^\circ(\Gamma)$ is
     adjacent to $(u,0)$ if and only if $x$ and $u$ are adjacent in
     $\Gamma$, but $(x,0)$ and $(u,1)$ are adjacent if and only if $x$
     and $u$ are adjacent in $\Gamma$ and $x\in\lk_\Gamma(v)$.
     So, the vertices adjacent to both $(u,0)$ and $(u,1)$ are
     $(\lk_\Gamma(u)\cap\lk_\Gamma(v))\times\{0\}$.
     Thus, vertices with the same link as $(w,0)$ are of the
     form $(x,0)$ with $x\in\lk_\Gamma(v)$, where $\lk_\Gamma(w)=\lk_\Gamma(x)$.
     That is, $M_{\double^\circ_v(\Gamma)}(w,0)=M_\Gamma(w)\times\{0\}$.

     The remaining case is that $w\notin\lk_\Gamma(v)$ and $w$ is
     adjacent to at least one vertex $x \notin\lk_\Gamma(v)$.
     Then $(w,0)$ is adjacent to $(x,0)$ but not to $(x,1)$.
Suppose
$\lk_{\double^\circ_v(\Gamma)}(w,0)=\lk_{\double^\circ_v(\Gamma)}(u,\epsilon)$.
This implies $\lk_\Gamma(w)=\lk_\Gamma(u)\setminus\{v\}$.
However, it also implies $(u,\epsilon)$ is adjacent to $(x,0)$ but not
$(x,1)$, which does not happen if $u\in\lk_\Gamma(v)$, so, actually,
$\lk_\Gamma(w)=\lk_\Gamma(u)$.
Furthermore, $\epsilon=0$, and $(u,1)$ is adjacent to $(x,1)$ but not
$(x,0)$, so $(u,1) \notin M_{\double_v^\circ(\Gamma)}(w,0)$; we have
that $M_{\double_v^\circ(\Gamma)}(w,0)=M_\Gamma(w)\times\{0\}$ and
$M_{\double_v^\circ(\Gamma)}(w,1)=M_\Gamma(w)\times\{1\}$ are distinct
twin modules.
  \end{proof}

\begin{proposition}[Recognizing near doubles]\label{recognizeneardouble}
  A triangle-free graph $\Gamma$ without separating cliques is a near double if and only if one of the following is true:
  \begin{enumerate}
      \setcounter{enumi}{-1}
    \item There are no odd twin modules.\label{item:no_odd_modules} 
    \item There exists a twin module $A$ of $\Gamma$ such that every odd twin
      module of $\Gamma$ is either $A$ or a satellite of $A$ in $\Gemini(\Gamma)$. \label{item:all_odd_modules_satellites_of_one}
  \item There exist twin modules $A$ and $B$ of $\Gamma$ such that $A$
    and $B$ are adjacent in $\Gemini(\Gamma)$ and such that every odd
    twin module of $\Gamma$ is either $A$ or $B$ or a satellite of $A$ or
    $B$ in $\Gemini(\Gamma)$. \label{item:all_odd_modules_satellites_of_pair}
  \end{enumerate}
\end{proposition}
\begin{proof}
  Clearly
  \eqref{item:no_odd_modules}$\implies$\eqref{item:all_odd_modules_satellites_of_one}$\implies$\eqref{item:all_odd_modules_satellites_of_pair},
  but the conditions provide case distinctions.
  
  If $\Gamma$ is complete then it is either empty, a single vertex, or
  a single edge, and it is reduced to the empty graph by
  $|\Gamma|$--many link doublings. The empty graph is the double of
  itself, so the proposition is true in this case. Now assume $\Gamma$
  is incomplete so that we can apply \fullref{doublingmodules}. 

 In case~\eqref{item:all_odd_modules_satellites_of_pair}, pick $v\in
 A$ and consider $\double_v^\circ(\Gamma)$.
 From \fullref{doublingmodules} we see that the only possible odd
 modules of $\double_v^\circ(\Gamma)$ come from odd modules of
 $\Gamma$ that are not $A$ or one of its satellites in
 $\Gemini(\Gamma)$.
By \eqref{item:all_odd_modules_satellites_of_pair} the only possible
choices are $B$ or a satellite of $B$.  
Since $B$ is adjacent to $A$ in $\Gemini(\Gamma)$ there is a unique module $B\times\{0\}=B\times\{1\}$ in
$\double_v^\circ(\Gamma)$ corresponding to $B$.
If $C$ is a satellite of $B$ in $\Gemini(\Gamma)$ then for every $c\in
C$ and every $b\in B$ we have $\lk(c)\subset\lk(b)$, so for every
$\ell\in\lk(c)$ there are edges $c\edge\ell\edge b$ in $\Gamma$.
If $\ell=v$ then it does not appear in $\double_v^\circ(\Gamma)$.
Otherwise, $\double_v^\circ(\Gamma)$ contains edges:
\[(c,0)\edge (\ell,0)\edge(b,0)=(b,1)\edge(\ell,1)\edge(c,1)\]
Thus any module of  $\double_v^\circ(\Gamma)$ corresponding to a
satellite of $B$ in $\Gemini(\Gamma)$
is a satellite of $B\times\{0\}$ in $\Gemini(\double_v^\circ(\Gamma))$.
In particular,  every odd module of $\double_v^\circ(\Gamma)$ is
either $B\times\{0\}$ or a satellite of $B\times\{0\}$, so
$\double_v^\circ(\Gamma)$ belongs to case~\eqref{item:all_odd_modules_satellites_of_one}.

 In case~\eqref{item:all_odd_modules_satellites_of_one}, pick $v\in A$, and consider $\double_v^\circ(\Gamma)$.
 Again, by \fullref{doublingmodules} the possible odd
 modules of $\double_v^\circ(\Gamma)$ come from odd modules of
 $\Gamma$ that are not $A$ or one of its satellites in $\Gemini(\Gamma)$.
But \eqref{item:all_odd_modules_satellites_of_one} says there are none of these, so $\double_v^\circ(\Gamma)$
belongs to case~\eqref{item:no_odd_modules} 
  
In case~\eqref{item:no_odd_modules}  there are no odd modules, so $\Gamma$ is a double by
  \fullref{characterizedouble}.

\medskip

In the other direction we show that if there exists $v\in\Gamma$ such
that \eqref{item:all_odd_modules_satellites_of_pair} is true for
$\double_v^\circ(\Gamma)$ then
\eqref{item:all_odd_modules_satellites_of_pair}  was already true
in $\Gamma$.
Thus, by induction on the link doubling sequence, if \eqref{item:all_odd_modules_satellites_of_pair}  is false in
$\Gamma$ then it remains false in any iterated link double of
$\Gamma$.
Therefore, every iterated link double of $\Gamma$ contains odd twin
modules, so is not a graph double, so $\Gamma$ is not a near double.

Suppose there is $v\in\Gamma$ such that \eqref{item:all_odd_modules_satellites_of_pair} is true for
$\double_v^\circ(\Gamma)$.
Let $\sigma\from \double_v^\circ(\Gamma)\to \double_v^\circ(\Gamma)$
be the involution that fixes the first coordinate and exchanges 0 and
1 in the second coordinate. 
Let $\pi\from \double_v^\circ(\Gamma)\to\Gamma$ be projection to the first
coordinate.
By construction of $\double_v^\circ(\Gamma)$,  the map $\pi$ sends edges to
edges.
It follows that $\pi$ sends twin modules to twin modules, so it induces $\pi\from
\Gemini(\double_v^\circ(\Gamma))\to\Gemini(\Gamma)$ that sends
vertices to vertices and edges
to edges.

Furthermore, $\pi$ preserves the satellite relationship, as follows.
Suppose $B$ is a satellite of $A$ in
$\Gemini(\double_v^\circ(\Gamma))$.
Then $\lk(B)\subset\lk(A)$, so $\pi(\lk(B)))\subset\pi(\lk(A))\subset\lk(\pi(A))$.
However, it is possible that $\pi(\lk(B))\subsetneq
\lk(\pi(B))$, which happens when $\{v\}=M_\Gamma(v)\in\lk(\pi(B))$, so that
$\pi^{-1}(M_\Gamma(v))=\emptyset$.
To show $\pi(B)$ is a satellite
of $\pi(A)$ we must rule out the possibility that 
 $\{v\}=M_\Gamma(v)\in\lk(\pi(B))\setminus\lk(\pi(A))$.
 Suppose this is the case.
 Clearly $\pi(A)\notin\st(M_\Gamma(v))$. 
We cannot have $\pi(A)$ as a satellite of $M_\Gamma(v)$, because that
would mean that
$\lk(\pi(B))\setminus\{v\}\subset\lk(\pi(A))\subset\lk(M_\Gamma(v))$,
but then triangle-freeness implies $\lk(\pi(B))=\{v\}$, which would
make $v$ a cut vertex of $\Gamma$, which is a contradiction.
By \fullref{doublingmodules}, the remaining option is that $A$ and
$\sigma(A)$ are distinct.
The condition $\pi(B)\in\lk(M_\Gamma(v))$ implies $\sigma$ fixes the vertices in $B$.
For any $C\in\lk(B)\subset\lk(A)$ we have $\pi(C)\notin\st(M_\Gamma(v))$ by
triangle-freeness, so for all $a\in A$, $b\in B$, and $c\in C$ there
are segments $a\edge c\edge b$ and $\sigma(b)\edge\sigma(c)\edge\sigma(a)$,
with $b = \sigma(b) $, $\sigma(c)\neq c$, and $\sigma(a)\neq a$.
By construction of $\double_v^\circ(\Gamma)$, if $a\neq\sigma(a)$ and
$c\neq\sigma(c)$ and there exist edges
$a\edge c$ and $\sigma(a)\edge\sigma(c)$ then there do not exist edges
$a\edge\sigma(c)$ and $c\edge\sigma(a)$.
This contradicts $\sigma(c)\in\lk(b)\subset\lk(a)$.
Thus, $\pi(B)$ is a satellite of $\pi(A)$. 

\smallskip

Now we argue that \eqref{item:all_odd_modules_satellites_of_pair} for
$\double_v^\circ(\Gamma)$ implies \eqref{item:all_odd_modules_satellites_of_pair}  for $\Gamma$. 
Since \eqref{item:all_odd_modules_satellites_of_pair} is true for
$\double_v^\circ(\Gamma)$, there are modules $A$ and $B$ of $\double_v^\circ(\Gamma)$ that are
neighbors in $\Gemini(\double_v^\circ(\Gamma))$, and such that every
odd module of $\double_v^\circ(\Gamma)$  is either $A$ or $B$ or a
satellite of $A$ or $B$ in $\Gemini(\double_v^\circ(\Gamma))$.
Since $\pi\from
\Gemini(\double_v^\circ(\Gamma))\to\Gemini(\Gamma)$ preserves
satellites, the potential odd modules of $\Gamma$
are, according to \fullref{doublingmodules}, either $M_\Gamma(v)$ or
one of its satellites, or the projection of an odd module of
$\double_v^\circ(\Gamma)$, which by hypothesis must be 
$\pi(A)$ or $\pi(B)$ or a satellite of one of these two.  

Since $A$ and $B$ are adjacent, so are $\pi(A)$ and $\pi(B)$.
By triangle-freeness, they cannot both be in $\lk(M_\Gamma(v))$.
Suppose $\pi(B)$ is not in $\lk(M_\Gamma(v))$.
The alternatives are $\pi(B)=M_\Gamma(v)$, $\pi(B)$ is a satellite of
$M_\Gamma(v)$, or $\pi(B)$ is neither in $\st(M_\Gamma(v))$ nor a
satellite of $M_\Gamma(v)$.

If $\pi(B)=M_\Gamma(v)$ or $\pi(B)$ is a satellite of $M_\Gamma(v)$
then since $A$ and $B$ are adjacent, $\pi(A)\in\lk(M_\Gamma(v))$. 
If $C$ is a satellite of $B$ then $\lk(\pi(C))\subset\lk(\pi(B))\subset\lk(M_\Gamma(v))$, 
so $\pi(C)$ is a satellite of $M_\Gamma(v)$, so every odd module
of $\Gamma$ is either $M_\Gamma(v)$ or $\pi(A)$ or a satellite of one of
these. Since $\pi(A)\in\lk(M_\Gamma(v))$, vertices $\pi(A)$ and $M_\Gamma(v)$ are adjacent in $\Gemini(\Gamma)$ so \eqref{item:all_odd_modules_satellites_of_pair} is true for
$\Gamma$.

Suppose $\pi(B)$ is neither in $\st(M_\Gamma(v))$ nor a satellite of $M_\Gamma(v)$.
Then $\pi(B)\times\{0\}\neq\pi(B)\times\{1\}$ are distinct and
symmetric, so we may assume
$B=\pi(B)\times\{0\}$. We will show that $B$ and its satellites are even. 
Suppose $C$ is an odd satellite of $B$. 
Since $C$ is odd,  $\pi(C)$ is not $M_\Gamma(v)$ or a satellite of
$M_\Gamma(v)$, by~\fullref{doublingmodules}.
Since $C$ is a satellite of $B$ and $\pi$ preserves satellites,
$M_\Gamma(v)\notin\lk(\pi(B))\supset \lk(\pi(C))$, so  
$\pi(C)\times\{0\}\neq\pi(C)\times\{1\}$ are distinct odd modules.
Triangle-freeness implies that $\pi(B)\times\{0\}$, $\pi(B)\times\{1\}$, and
satellites of either of these two are distinct from $A$ and satellites
of $A$. 
So, by property \eqref{item:all_odd_modules_satellites_of_pair} for  $\double_v^\circ(\Gamma)$, both of $\pi(C)\times\{0\}$ and $\pi(C)\times\{1\}$ are satellites
of $B$. 
However, for $\pi(C)\times\{1\}$ to be a satellite of
$B=\pi(B)\times\{0\}$
implies $\lk(\pi(C)\times\{1\})\subset \pi^{-1}(\lk(M_\Gamma(v)))$, so $\pi(C)$ is a
satellite of $M_\Gamma(v)$ in $\Gemini(\Gamma)$, a contradiction, as we have already ruled this out. 
Similarly, if $B$ is odd then $\pi(B)\times\{1\}$ is a satellite of
$B=\pi(B)\times\{0\}$, so the same argument implies $\pi(B)$ is a
satellite of $M_\Gamma(v)$ in $\Gemini(\Gamma)$, which contradicts
that $B$ is odd.
Thus, $B$ and its satellites are even.

The same arguments apply after swapping the roles of $A$ and $B$.

In conclusion, if one of $\pi(A)$ or $\pi(B)$ is equal to or a satellite of  $M_\Gamma(v)$ then \eqref{item:all_odd_modules_satellites_of_pair} is true for
$\Gamma$ and we are done.
Otherwise, since $\pi(A)$ and $\pi(B)$ cannot both be in $\lk(M_\Gamma(v))$, we may assume that $\pi(B)$ is not, which implies that $B$ and its satellites are even. 
If $\pi(A)$ is also not in $\lk(M_\Gamma(v))$ then $A$ and its satellites are also even, leaving only $M_\Gamma(v)$ and its satellites as possible odd modules of $\Gamma$, so \eqref{item:all_odd_modules_satellites_of_one} is true for $\Gamma$.
On the other hand, if $\pi(A)\in\lk(M_\Gamma(v))$ then the potential odd modules of $\Gamma$ are $M_\Gamma(v)$, $\pi(A)$, or a satellite of one of these adjacent vertices, so \eqref{item:all_odd_modules_satellites_of_pair} is true for
$\Gamma$.
\end{proof}

\subsection{Cloning}\label{sec:clone}
\begin{definition}
  Let $v$ be a vertex of a graph $\Gamma$.
 To \emph{clone} $v$ means to add a new vertex $v'$ and
  connect it by an edge to each vertex of $\lk(v)$ to make a new graph
  $\Gamma':=\Gamma\join_{\lk(v)}v'$ in which $v$ and $v'$ are twins.
\end{definition}

\begin{definition}
  A vertex of a graph is a \emph{singleton} if it has no twins.
  A vertex is \emph{clonable} if it is a satellite of at least two
  other vertices, and \emph{unclonable} otherwise. 
\end{definition}

\begin{proposition}\label{prop:cloning}
Let $v$ be a clonable vertex of a connected,  triangle-free graph $\Gamma$.
  Cloning $v$ produces a quasiisometric RACG: the
  Davis complex
  $\Davis_\Gamma$ is quasiisometric to $\Davis_{\Gamma\join_{\lk(v)}v'}$ by a quasiisometry that restricts to a
  color preserving isomorphism on each translate of
  $\Davis_{\Gamma\setminus\{v\}}$. 
\end{proposition}

The proof of \fullref{prop:cloning} has two parts: we define base quasiisometries in
\fullref{lem:coloredtree1} that are designed to facilitate an
application of \fullref{tree_of_quasiisometries}.

  \begin{lemma}\label{lem:coloredtree1}
   For every $n\geq 2$  there is a quasiisometry $\phi$ between the $(n+1)$--valent tree $T_{n+1}$ with
    edges colored $a$, $b$, $c$, $x_1$,\dots,$x_{n-2}$ (with exactly one edge of each color at
    each vertex) and the $(n+2)$--valent tree $T_{n+2}$ with edges
    colored $a$, $b$, $c$, $d$, $x_1$,\dots,$x_{n-2}$ with the following properties:
    \begin{itemize}
    \item $\phi$ is bijective on vertices.
      \item $\phi$ induces a bijection between the set of 
        $S$--colored components of $T_{n+1}$ and
        the set of $S$--colored components of
        $T_{n+2}$, where $S:=\{a,b,x_1,\dots,x_{n-2}\}$ is the set of `static'
        colors. 
        \item $\phi$ restricts to a color-preserving isomorphism on each
           $S$--colored component.
    \end{itemize}
  \end{lemma}
  In the proofs of \fullref{lem:coloredtree1} and
  \fullref{lem:coloredtree2} we use without further explanation the
  following well-known construction.
   If $p$, $q$, $r$ are vertices of a tree such that $q\in\lk(r)$ is between
   $p$ and $r$ 
then the map that replaces the edge $q\edge r$ with an edge
$p\edge r$ is $(1+d(p,q))$--biLipschitz.
To see this, consider a geodesic $v_0,\dots,v_n$ in the original tree
structure such that $v_m=q$ and $v_{m+1}=r$, for some $0 \le m < n$.
In the new tree structure there is a path from $v_0$ to $v_n$ given by
$v_0,\dots,v_m$ followed by a geodesic from $v_m=q$ to $p$, then the
new edge $p\edge r=v_{m+1}$, then continuing with
$v_{m+2}$,\dots,$v_n$.  
Thus, the distance from $v_0$ to $v_n$ increased by at
most an additive factor of $d(p,q)$. Since vertex distances are
integral, distances increase by a multiplicative factor of at most
$1+d(p,q)$. The other inequality can be proved by making the same
argument for the inverse map.

In the literature this map is sometimes described as an `edge-slide'
move, where it is imagined that the target of the initial incidence
map for $q\edge r$
continuously `slides' along a specified edge path, which in this case
is the geodesic from $q$ to $p$. 
More generally, we can make infinitely many such edge-slides simultaneously
without compounding the biLipschitz constant, provided that no edge-to-be-slid is contained in the slide path of another. 

  \begin{proof}[Proof of \fullref{lem:coloredtree1}]
    Pick a base vertex $w_0$ in $T_{n+1}$. Given a vertex $v$ we speak of `incoming' and `outgoing' edges
    with respect to the chosen basepoint; ie, an outgoing edge at $v$
    leads farther from $w_0$.
    Since we are working in a tree, for every vertex $v$ and every
    word in letters $a$, $b$, $c$, $x_1$,\dots there is a unique edge path in the
    tree starting from $v$ whose edges are colored by successive
    letters of the given word.
    Words without repeated successive letters are geodesics, and if we
    identify $w_0$ with the empty word, then every vertex can be
    described uniquely by a geodesic word.

     Each  $S$--colored component has a unique vertex  $w_i$ closest to $w_0$. We assume that the vertices $w_i$ are ordered by increasing distance from $w_0$; that is, if $i<j$, then the distance from $w_j$ to $w_0$ is greater than or equal to the distance from $w_i$ to $w_0$. 
    For phase 1 of the construction, consider the following map that fixes all of the vertices and
    rearranges the $c$--colored edges. 
    All of the vertices $w_1,w_2,\dots$ have an incoming $c$--edge,
    and every other vertex has an outgoing $c$--edge.
    Working inductively with increasing distance from $w_0$, 
    for every vertex $v$ with an outgoing $c$--edge, remove the $c$--edge between $vca$ and
    $vcac$ and connect $v$ and $vcac$ by a $c$--edge.
    This map is  $3$--biLipschitz on the vertex set.
    In the new tree structure, $w_1, w_2,\dots$ still have one
    incoming $c$--edge, the vertices $w_1a, w_2a,\dots$ have either zero or two $c$--edges, and
    every other vertex has two outgoing $c$--edges.

    For phase 2 of the construction, 
    for each $i\geq 1$ consider the two geodesic rays 
        $w_i,w_ib,w_iba,w_ibab,\dots$ and $w_i,w_ia,w_iab,w_iaba,\dots$
     based at $w_i$.
    On the first ray, let every vertex after $w_i$ pass
    one of its $c$--edges to its predecessor.
    On the second ray, if $w_ia$ has two $c$--edges, do nothing, and if it has no $c$--edges, then let every vertex after $w_ia$ pass both of its $c$--edges its predecessor.
    Thus, $w_i$ receives one additional $c$--edge and passes none,
    $w_ia$ already has or receives two $c$--edges and passes none, and every other vertex
    receives the same number of $c$--edges from its successor as it
    passes to its predecessor.
    Hence, every vertex now has two incident $c$--edges.
    The phase 2 map is  $2$--biLipschitz on the vertex set.

The composition of the two phases is a $6$--biLipschitz map on the
vertex set taking $T_{n+1}$ to an $(n+2)$--valent tree such that every
vertex in the image has exactly two incident $c$--edges and one of
each color from $S$. 
        By construction, the map is a color preserving isometry on
        each $S$--colored component.
        The $\{c\}$--colored components of the image are biinfinite geodesics.
        On each such geodesic, recolor the edges so that they
        alternate $c$ and $d$.
        The result is isomorphic to $T_{n+2}$ as colored trees. 
  \end{proof}
  \begin{proof}[Proof of \fullref{prop:cloning}]
    Since $\Gamma$ is connected, $\emptyset\neq L:=\lk(v)$.
    Since $v$ is clonable, $C:=\{c\in\Gamma\setminus\{v\}\mid L\subset
    \lk(c)\}$ contains at least two vertices.
  Since $\Gamma$ is triangle-free and $L\neq\emptyset$, $C\cup\{v\}$ is an
  anticlique.
       Since $|C|\geq 2$, there is a quasiisometry  $\psi\from \Sigma_{C\cup\{v\}}\to \Sigma_{C\cup\{v,v'\}}$ given by \fullref{lem:coloredtree1} that
        restricts to color preserving isomorphisms on each copy of
        $\Sigma_C$. We may assume, up to pre-composing by translation in $W_{C\cup \{v\}}$, that $\psi$ fixes $\Sigma_C$.
        Crossing with the identity on $\Sigma_L$ gives a
        quasiisometry $\mathrm{Id}_{\Sigma_L}\times\psi$ from $\Sigma_L\times\Sigma_{C\cup\{v\}}=\Sigma_{L\cup
          C\cup \{v\}}$ to $\Sigma_L\times\Sigma_{C\cup\{v,v'\}}=\Sigma_{L\cup
          C\cup \{v,v'\}}$ that preserves edge colors for all edges
        colored by $L\cup C$. In particular, the map matches up
        translates of $\Sigma_{L\cup C}$ bijectively, and restricts to
        a color-preserving isomorphism on each such translate.

       If the vertex set of $\Gamma$ is $\{v\}\cup C \cup L$, we are done. Otherwise, we can write $W_\Gamma$ and $W_{\Gamma*_L v'}$ as amalgams:
        \begin{align*}
        W_\Gamma&=W_{\Gamma\setminus\{v\}}*_{W_L\times 
                  W_C}(W_L\times (W_C*W_{\{v\}}))\\
                  W_{\Gamma*_L {v'}}&=W_{\Gamma\setminus\{v\}}*_{W_L\times 
          W_C}(W_L\times(W_C*W_{\{v,v'\}}))\\
        \end{align*}
Apply \fullref{tree_of_quasiisometries} with base maps
$\mathrm{Id}_{W_{\Gamma\setminus\{v\}}}$ and
$\mathrm{Id}_{\Sigma_L}\times\psi$ on the vertex groups of the given
splittings.
We check that the hypotheses of \fullref{tree_of_quasiisometries} are
satisfied for $\mathrm{Id}_{\Sigma_L}\times\psi$.
The check for $\mathrm{Id}_{W_{\Gamma\setminus\{v\}}}$ is similar. 
The construction has been arranged so that Conditions~\eqref{item:base_consistency}  and~\eqref{item:coset_bijection} 
\fullref{tree_of_quasiisometries} are
satisfied.
For \eqref{item:consistency}, note that $W_{L\cup C}$
acts by color
preserving isomorphisms of $\Davis_{L\cup C}$, freely and transitively
on vertices.
Thus, if $(\mathrm{Id}_{\Sigma_L}\times\psi)((1,g)\Davis_{L\cup
  C})=(1,g')\Davis_{L\cup C}$ then let
$h'=(\mathrm{Id}_{\Sigma_L}\times\psi)((1,g)1)\in (1,g')W_{L\cup C}$
so that $\mathrm{Id}_{\Sigma_L}\times\psi$ and $h'
(\mathrm{Id}_{\Sigma_L}\times\psi) (1,g)^{-1}$ 
are color preserving isomorphisms of $(1,g)\Davis_{L\cup
  C}$ that agree on one vertex, hence on all.
\end{proof}

\begin{proposition}\label{modified_weighted_twin_graph}
  Let $\Gamma$ be a triangle-free graph without separating cliques. Then
  $W_\Gamma$ is quasiisometric to $W_{\Gamma'}$, where $\Gamma'$ is the
  blow-up graph $\Gemini(\Gamma)^{\omega'}$ for the
   weight function  $\omega'$ derived from the weight function
  $\omega(M):=|M|$ of $\Gemini(\Gamma)$ as follows:
  \begin{itemize}
    \item $\omega'(M)=1$ if $M$ is an unclonable singleton.
  \item $\omega'(M)=2$ if $M$:
    \begin{itemize}
    \item is a clonable singleton, or
    \item $\omega(M)=2$, or
      \item $\omega(M)>2$ and $M$ is a satellite in $\Gemini(\Gamma)$.
    \end{itemize}
  \item  $\omega'(M)=4$ otherwise.
  \end{itemize}
\end{proposition}
\begin{proof}
  We will change $\Gamma$ into $\Gamma'$ by cloning or retiring clones
  in such a way that $W_\Gamma$ and $W_\Gamma'$ are quasiisometric by \fullref{prop:cloning}.

  If $M$ is an unclonable singleton or if $|M| = 2$, then no change is necessary. 

  If $M=\{v\}$ is a clonable singleton then cloning $v$ creates a
  graph in which $v$ belongs to a twin module of size 2.

  If $M$ is a twin module of $\Gamma$ containing a vertex $v$ where $|M|>2$ and $M$ is a
  satellite of $N$ in $\Gemini(\Gamma)$ then for any vertex $w\in N$,
  $v$ is a satellite of $w$ and has at least two twins $v'$ and
  $v''$.
  Delete $v''$.
  In the resulting graph $v$ is clonable, since
  it is a satellite of $v'$ and $w$.
  Cloning $v$ is therefore a quasiisometry on the level of Coxeter
  groups, and results in a graph isomorphic to $\Gamma$, so deleting
  $v''$ induces the inverse quasiisometry on the level of Coxeter
  groups. Iterating this process, we can arrange $\omega'(M) = 2.$
  Similarly, if $M$ is not a satellite but $|M|>4$ then deleting a
  vertex from $M$ induces a quasiisometry of Coxeter groups, so we can
  arrange $\omega' (M)= 4$.

  Finally, if $|M|=3$ but $M$ is not a satellite in $\Gemini(\Gamma)$ then we
  can clone a vertex of $M$ without changing the quasiisometry type of
  $W_\Gamma$,  so we make $\omega'(M)=4$.
\end{proof}
\begin{corollary}\label{canoncai_weighted_representative}
  If $\Gamma_i$ and $(\Gemini(\Gamma_i),\omega'_i)$ for $i=1,2$ are as
  in  \fullref{modified_weighted_twin_graph} and
  $(\Gemini(\Gamma_1),\omega'_1)$ and $(\Gemini(\Gamma_2),\omega'_2)$
  are isomorphic as weighted graphs then $W_{\Gamma_1}$ and
  $W_{\Gamma_2}$ are quasiisometric. 
\end{corollary}
\begin{corollary}\label{einzelkind0}
  A triangle-free graph with no separating cliques and no unclonable singletons is RAAGedy.
\end{corollary}
\begin{proof}
  Let $\Gamma':=\Gemini(\Gamma)^{\omega'}$ as in
  \fullref{modified_weighted_twin_graph}, so that $W_\Gamma$ is
  quasiisometric to $W_{\Gamma'}$ and $\omega'$ does not take the
  value 1. 
  Since $\omega'$ takes only even
  values, $\Gamma'$ is isomorphic to the graph double of the
  blow-up graph $\Gemini(\Gamma)^{\omega'/2}$.
  Apply \fullref{thm:davisjanuszkiewicz}.
\end{proof}

As an application, we upgrade the conclusion of \fullref{mixed_multiples_are_strongly_CFS}:
\begin{corollary}\label{mixed_multiple_graphs_are_RAAGedy}
  Let $(\Delta,\omega)$ be a weighted graph such that $\Delta$ is
  connected, has more than one vertex, is triangle-free and twin-free, and such that
 $\omega$ takes the value 1 only on leaves of $\Delta$, and if
 $\Delta$ is a single edge then $\omega$ does not take the value 1.  
  Then the blow-up graph $\Delta^\omega$ is triangle-free without
  separating cliques and is RAAGedy. 
\end{corollary}
\begin{proof}
  Apply \fullref{einzelkind0}.
  This requires showing that
 $\Delta^\omega$ is triangle-free without separating cliques or
 unclonable singletons.
 By \fullref{mixed_multiples_are_strongly_CFS},
$\Delta^\omega$ is triangle-free and strongly CFS, 
 and the latter implies there are no separating cliques.

Singletons of $\Delta^\omega$ come from vertices of $\Delta$ with
weight 1.
By hypothesis, any such vertex is a leaf of $\Delta$. 
  Suppose $u$ is a leaf in $\Delta$ with weight 1 and $v$ is the vertex of
  $\Delta$ adjacent to $u$.
  By hypothesis, $u\edge v$ is not all of $\Delta$.
  Since $\Delta$ is connected, $v$ has some neighbor $w\neq u$.
  The vertex $w$ is not a leaf, since if it were then $u$ and $w$ would be twins, but
  $\Delta$ is twin-free.
  Hence $\omega(w)>1$.
  In $\Delta^\omega$, $(u,0)$ is a satellite of all the vertices
  $(w,0)$, \dots, $(w,\omega(w)-1)$, of which there are at least two, 
  so $(u,0)$ is clonable.
  Thus, $\Delta^\omega$ has no unclonable singletons. 
\end{proof}
If we assume the $\Delta$ is twin-free then a
blow-up graph $\Gamma=\Delta^\omega$ is a graph
double precisely when $\omega$ takes only even values.
In \cite{CasEdl} it is shown that when $\Gamma$ contains an induced
cycle of length greater than 6 with no 2--chord, then $\Gamma$ does
not admit a FIDL--$\Lambda$.
Thus, by taking $\Delta$ to be twin-free containing a long cycle
without 2--chords and
taking $\omega$ to be uneven we get many examples of blow-up graphs
$\Delta^\omega$ to which neither the Davis-Januszkiewicz nor
Dani-Levcovitz conditions apply, but that are RAAGedy by \fullref{mixed_multiple_graphs_are_RAAGedy}.

\medskip

Finally, we give an analogue of the `near-double' construction to
describe when passing to link doubles can eliminate unclonable
singletons:
\begin{theorem}\label{einzelkind}
  A triangle-free graph $\Gamma$ without
  separating cliques is RAAGedy if any of the following are true:
  \begin{itemize}
  \item There are no unclonable singletons. 
    \item There is a vertex $v$ such that the set of unclonable
      singletons is contained in the set consisting of $v$  and its satellites.
      \item There are adjacent vertices $v$ and $w$ such that the set
        of unclonable singletons is contained in the set consisting of
        $v$ and $w$ and their satellites.
  \end{itemize}
\end{theorem}
\begin{definition}
  We call a graph satisfying \fullref{einzelkind} a \emph{coarse near double}.
\end{definition}
\begin{proof}
  As in \fullref{modified_weighted_twin_graph}, without changing the
  quasiisometry type of $W_\Gamma$ we may replace $\Gamma$ by a graph
  $\Gamma'$ in which all twin modules are even except those coming
  from unclonable singletons of $\Gamma$.
The hypotheses control the relative arrangement of those unclonable
singletons and match the hypothesis of \fullref{recognizeneardouble},
so $\Gamma'$ is a near double.
Thus, $W_\Gamma$ is quasiisometric to $W_{\Gamma'}$, by
\fullref{prop:cloning}, and $W_{\Gamma'}$ is commensurable to a RAAG, by \fullref{neardoublecommraag}
\end{proof}

\begin{example}\label{ex:coarseneardoubles}
  \begin{figure}[h]
    \centering
    \includegraphics{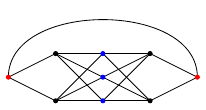}
\quad
    \includegraphics{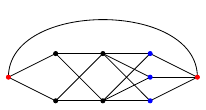}
\quad
\includegraphics{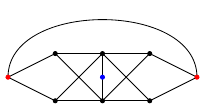}

\vspace{1cm}

\includegraphics{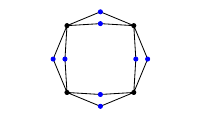}
\quad
    \includegraphics{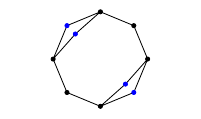}
\quad
\includegraphics{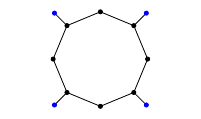}

    \caption{Some coarse near doubles $\Gamma$ (top row) and graphs
      $\Delta$ (bottom row)
      such that $W_\Gamma$ is quasiisometric to $A_\Delta$, respectively.}
    \label{fig:coarseneardouble}
  \end{figure}
The top row of \fullref{fig:coarseneardouble} shows the only three triangle-free CFS graphs with at most 9
  vertices that are not near doubles and for which Dani-Levcovitz does
  not produce a finite-index RAAG subgroup, but to which
  \fullref{einzelkind} applies.
  These are the smallest examples that we know for which $W_\Gamma$ is 
quasiisometric to a RAAG, but we do not know if $W_\Gamma$ is
commensurable to some RAAG. 

In each case there is an adjacent pair of unclonable singletons (the
extreme left and right vertices, in red) and one other odd module, which
consists of clonable vertices (blue).
Clone a vertex from this module to make it even.
The remaining odd modules are a pair of adjacent singletons.
Doubling over these two vertices produces a graph that is a double of
the corresponding graph in the bottom row.
\end{example}

\subsection{Unfolding}\label{sec:unfolding}
We introduce a new operation on graphs that induces a quasiisometry of
their respective RACGs. We first state a result, then give 
motivating examples, then prove the result. 

  \begin{proposition}\label{prop:unfold}
    Suppose $\Gamma$ is triangle-free with a
    separating join  $E\join F$.
    Let $G$ be a union of connected components of $\Gamma\setminus (E\join F)$, and let
    $H:=(\Gamma\setminus E\join F)\setminus G$.
    Let $\bar{G}$ be the union of $F$ and $G$ and the neighbors of $G$
    in $E$.
    Let $\bar{H}$ be the union of $F$ and $H$ and the neighbors of $H$
    in $E$.
    Partition $E$ as follows:
    \begin{align*}
      A&:=E\cap\bar{G}\cap\bar{H}^c=\{a_0,a_1,\dots,a_\ell\}\\
       B&:=E\cap\bar{G}^c=\{b_0,b_1,\dots,b_m\}\\
      C&:=E\cap\bar{G}\cap\bar{H}\\
    \end{align*}
    Suppose that $A$ and $B$ are nonempty and $C=\{c\}$ is a single
    vertex.
    
    Then $W_\Gamma$ is quasiisometric to $W_{\Gamma'}$, where
    $\Gamma'$ is a graph constructed as follows.
    The vertex set is the vertex set of
    $\Gamma$ plus one new vertex $d$, with $D:=\{d\}$.
    Add edges so that $(D\sqcup E)\join F\subset\Gamma'$.
    If $x\edge y$ is an edge in $\bar{G}$ then add an edge from $x$
    to $y$ in $\Gamma'$.
    If $x\edge y$ is an edge of $\bar{H}$ such that $x\neq c\neq y$
    then add an edge from $x$ to $y$ in $\Gamma'$.
    If $x\edge c$ is an edge of $\bar{H}$ then add an edge from $x$ to
    $d$ in $\Gamma'$.

The quasiisometry can be constructed so that each copy of $\Sigma_{\bar{G}}$ in
$\Sigma_\Gamma$ is sent to within uniformly bounded Hausdorff distance of a
copy of $\Sigma_{\bar{G}}$ in
$\Sigma_{\Gamma'}$, and each copy of $\Sigma_{\bar{G}}$ in
$\Sigma_{\Gamma'}$ is Hausdorff close to the image of a copy from
$\Sigma_\Gamma$. 

Furthermore, on each copy of $\Sigma_{\bar{G}}$ in
$\Sigma_\Gamma$ the quasiisometry restricts to a map that is uniformly
bounded distance from a color preserving
cubical isomorphism.

Similar statements are true for $\Sigma_{\bar{H}}$, except that
$c$--edges are sent to $d$--edges.
  \end{proposition}

  \begin{example}\label{ex:simple_unfolding}
    Consider the graphs $\Gamma$ and $\Gamma'$ of \fullref{fig:simple_unfolding}.

    \begin{figure}[h]
      \centering
      \begin{subfigure}[t]{.3\textwidth}
        \centering
        \includegraphics{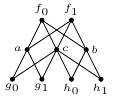}
        \caption{$\Gamma$}
        \label{fig:simple_unfolding_0}
      \end{subfigure}
      \hfill
       \begin{subfigure}[t]{.3\textwidth}
        \centering
        \includegraphics{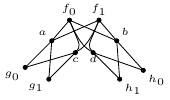}
        \caption{Intermediate}
        \label{fig:simple_unfolding_1}
      \end{subfigure}
      \hfill
      \begin{subfigure}[t]{.3\textwidth}
                \centering
        \includegraphics{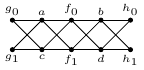}
        \caption{$\Gamma'$}
        \label{fig:simple_unfolding_2}
      \end{subfigure}
      \caption{Unfolding visualized as a continuous deformation.}
      \label{fig:simple_unfolding}
    \end{figure}
    
    $\Gamma'$ is obtained from $\Gamma$ by unfolding as in \fullref{prop:unfold}
   with $A:=\{a\}$, $B:=\{b\}$, $C:=\{c\}$,
    $E:=A\sqcup B\sqcup C$, $F:=\{f_0,f_1\}$, $G:=\{g_0,g_1\}$, and
    $H:=\{h_0,h_1\}$, so $W_\Gamma$ and $W_{\Gamma'}$ are quasiisometric.
  \end{example}
  In the previous example $\Gamma$ is a near double and admits a
  FIDL--$\Lambda$, so we already knew it was RAAGedy.
In the next example unfolding is the only way that we know to say that the graph is RAAGedy. 
\begin{example}\label{ex:iteratedcones}
  Consider the graphs $\Gamma$ and $\Gamma'$ of
  \fullref{fig:unfolding}.
  
  \begin{figure}[h]
    \centering
    \begin{subfigure}{.3\textwidth}
      \centering
      \includegraphics{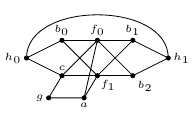}
      \caption{$\Gamma$}
      \label{fig:unfolding_Gamma}
    \end{subfigure}
      \begin{subfigure}{.3\textwidth}
      \centering
      \includegraphics{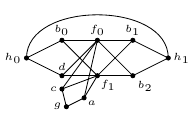}
      \caption{$\Gamma'$}
      \label{fig:unfolding_Gamma_prime}
    \end{subfigure}
      \begin{subfigure}{.3\textwidth}
      \centering
      \includegraphics{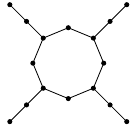}
      \caption{$\Delta$}
      \label{fig:unfolding_Delta}
    \end{subfigure}
    \caption{Unfolding example.}
    \label{fig:unfolding}
  \end{figure}

As in the previous example, the sets for \fullref{prop:unfold}  are
indicated by their lowercase vertex labels, and we conclude that $W_{\Gamma}$ is quasiisometric
to $W_{\Gamma'}$.

In $\Gamma$, vertices $h_0$, $h_1$, $b_0$, $c$, and $a$ are all unclonable singletons,
since $a$ and $b_0$ are only satellites of $c$, and vertices $h_0$, $h_1$, and
$c$ are not 
satellites at all.  The graph is not a near double.
It also contains an odd cycle, so it does not admit a FIDL--$\Lambda$.

In $\Gamma'$, $d$ and $b_0$ are twins and $c$ and $a$ are twins. 
The only unclonable singletons are $h_0$ and $h_1$, which are adjacent.
So, clone $g$ and then link double over $h_0$ and $h_1$. 
Conclude that $W_\Gamma$ and $W_{\Gamma'}$ are quasiisometric to
$A_\Delta$, for the $\Delta$ of \fullref{fig:unfolding_Delta}.
\end{example}

For the proof of \fullref{prop:unfold} we need a variation of \fullref{lem:coloredtree1}.
\begin{lemma}\label{lem:coloredtree2}
   For every $m,n\geq 0$  there is a quasiisometry $\phi$ between the $(m+n+3)$--valent tree $T_{m+n+3}$ with
    edges colored $a_0$,\dots, $a_m$, $b_0$,\dots,$b_n$, $c$ (with exactly one edge of each color at
    each vertex) and the $(m+n+4)$--valent tree $T_{m+n+4}$ with edges
    colored $a_0$,\dots, $a_m$, $b_0$,\dots,$b_n$, $c$, $d$  with the following properties:
    \begin{itemize}
      \item $\phi$ sends each 
        $\{a_0,\dots,a_m,c\}$--colored component of $T_{m+n+3}$
        within uniformly bounded Hausdorff distance of a unique
         $\{a_0,\dots,a_m,c\}$--colored component of $T_{m+n+4}$,
        and every such component of $T_{m+n+4}$ is the coarse image of
        a unique component of $T_{m+n+3}$.
        Furthermore, for each such component the quasiisometry
        restricts to a color-preserving isomorphism except at a single
        vertex.
          \item $\phi$ takes every $\{b_0,\dots,b_n,c\}$--colored component of $T_{m+n+3}$
        within uniformly bounded Hausdorff distance of a unique
         $\{b_0,\dots,b_n,d\}$--colored component of $T_{m+n+4}$,  and every such component of $T_{m+n+4}$ is the coarse image of
        a unique component of $T_{m+n+3}$.
Furthermore, for each $\{b_0,\dots,b_n,c\}$--colored component
of $T_{m+n+3}$ the quasiisometry restricts, except at a single vertex,
to an isomorphism that preserves  $b_j$--edges for each $j$  and takes $c$--edges to $d$--edges.
\end{itemize}
  \end{lemma}
  \begin{proof}
    As in \fullref{lem:coloredtree1}, choose a base vertex $w_0$ of $T_{m+n+3}$and
    describe vertices according to colored paths. 
   
    For phase~1 of the construction, consider every vertex $v$ with an
    outgoing $c$--edge.
    Add a new edge labelled $d$ at $v$, and call its opposite vertex $vd$.
    For each $j$, delete the edge between $vc$ and $vcb_j$ and instead connect $vcb_j$ to
    $vd$ by an edge labelled $b_j$. 
    The effect is to `unfold' all of the $c$--edges, leaving the
    $a_i$--edges in place at the end of the $c$--edge and moving the
    $b_j$--edges to the  end of the new $d$--edge.
See \fullref{fig:colortree2phase1}. 
    Call the result $\phi_1$.
    It is a coarse map, in the sense that some vertices are sent to a
    set of 2 vertices at uniformly bounded distance 2 from one another. 
   $\phi_1$  is a $(3,2)$--quasiisometry.
    It may be easier to visualize the inverse of the (coarse) map in
    \fullref{fig:colortree2phase1}: it
    `folds' co-incident edges $v\edge vc$ and $v\edge vd$ together to
    make a single edge $v\edge vc$.
    This is actually a map, but is not 1 to 1 on vertices. 

    The coarse map $\phi_1$ has the additional property that
    for any $\{a_0,\dots,a_m,c\}$--component $C$ there is a unique
    color preserving isomorphism $\phi_1^C$ such that for all vertices
    $v\in C$ we have $\phi_1^C(v)\in\phi_1(v)$.
    The same is true for $\{b_0,\dots,b_n,c\}$--components, except
    that edges colored $c$ are sent to edges colored $d$.
    This is to say that at the level of trees there is not a
    canonically nice way to make choices of image points such that
    $\phi_1$ is an honest map instead of a coarse map, but at the
    level of $\{a_0,\dots,a_m,c\}$--component and
    $\{b_0,\dots,b_n,c\}$--component there is, and it is unique, so
    we will not further belabor the point, and simply speak of
    \emph{the restriction of $\phi_1$ to $C$} as a well defined map.
    So in \fullref{fig:colortree2phase1}, we would say that $\phi_1$
    sends the unique $a_0c$ (red-green) colored geodesic through the basepoint
    on the left isomorphically to the unique $a_0c$  (red-green) colored geodesic through
    the basepoint on the right, and sends the unique $bc$ (blue-green) colored
    geodesic through the basepoint on the left isomorphically to the
    unique $bd$ (blue-olive) colored geodesic through the base point
    on the right.

    \begin{figure}[h]
      \centering
     \includegraphics{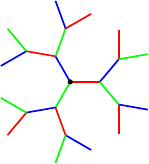}
\qquad\tikz[baseline=-40pt]\draw[thick,->] (0,0) -- ++ (1,0);\qquad
            \includegraphics{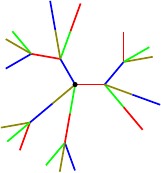}
      \caption{Phase~1 coarse map $\phi_1$, with  $a_0$ red, $b_0$ blue, $c$ green, $d$ olive}
      \label{fig:colortree2phase1}
    \end{figure}
    
    Vertices with an incoming $c$--edge have outgoing
    $\{a_0,\dots,a_m\}$--edges, but no incident
    $\{b_0,\dots,b_n,d\}$--edges.
    Vertices with an incoming $d$--edge have outgoing  $\{b_0,\dots,b_n\}$--edges, but no incident
    $\{a_0,\dots,a_m,c\}$--edges.
    All other vertices have one incident edge of each color.

    For phase~2 of the construction, define a map $\phi_2$ inductively with
    increasing distance to $w_0$, as follows. See also \fullref{fig:colortree2phase2}.
    Our map $\phi_2$ will be injective on vertices, but rearrange the
    placement of some edges.
    
    Suppose $v$ is a vertex of less than full valence (of valence
    less than $m+n+4$) and its incoming edge  is colored $c$.
    Then $v$ has no $\{b_0,\dots,b_n,d\}$--edges.
    The vertex $va_0db_0a_0$ had full
    valence at the end of phase~1, since it has incoming edge
    colored $a_0$.
    We claim it still has full valence now, so we can take all of its
    $\{b_0,\dots,b_n,d\}$--edges and donate them to $v$.
    We further claim that none of the other phase~2 moves affect an
    edge on the geodesic between $v$ and $va_0db_0a_0$.
    
    Now suppose $v$ has less than full valence with incoming
    $a_0$--edge.
    This could happen if it had full
    valence at the end of phase~1, but donated its  $\{b_0,\dots,b_n,d\}$--edges to a
    predecessor earlier in phase~2.
  The vertex $vca_0b_0a_0$  has
    full valence, because it did at the end of phase~1 and its last
    four edges are not $a_0db_0a_0$, so it was not called upon to
    donate to a vertex with incoming $c$--edge. 
Take its  $\{b_0,\dots,b_n,d\}$--edges and donate them to $v$.
    
    The construction for $v$ with incoming $d$-- or $b_0$--edges is
    similar, swapping the roles of $a_0$ and $b_0$ and those of $c$ and
    $d$.

    \begin{figure}[h]
      \centering
   \includegraphics{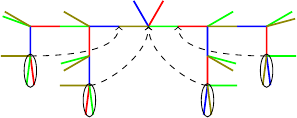}
      \caption{Choices of donors for phase~2 map.}
      \label{fig:colortree2phase2}
    \end{figure}

    Now we confirm the claim that these four types of donations do not
    interfere with one another.
    Given a vertex $u$, let $\delta(u)$
    denote its donor, if any.
    Suppose $v$ is a predecessor of $w$ and let $w'$ be the immediate
    predecessor of $w$.
    For the geodesic $[v,\delta(v)]$ to end on
    $[w,\delta(w)]$, and thus potentially have the donation to $v$
    interfere with the donation to $w$,  would require a terminal segment of
    $[v,\delta(v)]$ to coincide with an initial segment of $[w',\delta(w)]$.
    The four possible labels of
    $[v,\delta(v)]$  are: \[a_0db_0a_0\quad
      ca_0b_0a_0\quad b_0ca_0b_0 \quad db_0a_0b_0\]
    The four
    possible labels of $[w',\delta(w)]$ are:
    \[ca_0db_0a_0\quad
      a_0ca_0b_0a_0\quad db_0ca_0b_0\quad b_0db_0a_0b_0\]
    The maximal overlap between a suffix of the former and a prefix of
    the latter is a single letter, which means that
    $[v,\delta(v)]\cap[w,\delta(w)]$ is empty unless $w=\delta(v)$.
Even in this case we have arranged that the edges donated to $v$ by
$w$ are not one of the first edges on $[w,\delta(w)]$, since when $w$
donates $\{b_0,\dots,b_n,d\}$--edges, $[w,\delta(w)]$ starts with $c$,
and when $w$
donates $\{a_0,\dots,a_n,d\}$--edges, $[w,\delta(w)]$ starts with $d$.
  
    The new edges added
    connect vertices that were at distance at most 8 after phase~1, so
    the map $\phi_2$ of phase~2 is $8$--biLipschitz on vertices.

    Consider a $\{b_0,\dots,b_n,d\}$--component in the tree before phase~2,
    and suppose $w$ is its closest vertex to $w_0$.
   Suppose that during phase~2 a predecessor $v$ of $w$ took the
   $\{b_0,\dots,b_n,d\}$--edges from $w$.
   By construction, all predecessors of $v$ are already full valence,
   so no vertex of $\{b_0,\dots,b_n,d\}$--component at $w$ other than
   $w$ will be called upon to donate $\{b_0,\dots,b_n,d\}$--edges. 
   Thus, for any nontrivial geodesic word $u$ in letters
   $\{b_0,\dots,b_n,d\}$, $\phi_2(wu)=vu$.
   However, $\phi_2(w)=w$.
   So $\phi_2$ restricted to the $\{b_0,\dots,b_n,d\}$--component based at
   $w$ is a color preserving isomorphism except at the vertex $w$,
   since an isomorphism of this component would have sent $w$ to $v$. 

    Similarly, each $\{a_0,\dots,a_m,c\}$--component is sent to
    within bounded Hausdorff distance of an
    $\{a_0,\dots,a_m,c\}$--component, and the map restricts to a
    color preserving isomorphism except possibly at the unique vertex closest to $w_0$. 
  \end{proof}

  \begin{proof}[Proof of \fullref{prop:unfold}]    
    Since $\Gamma$ is triangle-free, 
    $E=A\cup B\cup C$ is an
    anticlique, so $\Sigma_E$ is a tree, as is $\Sigma_{E'}$ for $E':=A\cup B\cup C\cup D$.
    By \fullref{lem:coloredtree2} there is a quasiisometry
    $\psi\from\Sigma_{A\cup B\cup C}\to
    \Sigma_{A\cup B\cup C\cup D}$ that coarsely
    takes $\Sigma_{A\cup C}$--components to
    $\Sigma_{A\cup C}$--components, $\Sigma_{B\cup C}$--components to
    $\Sigma_{B\cup D}$--components, and on each
    such component restricts to a color preserving isomorphism except
    at a single vertex, and except for the fact that in the second case $c$--edges
    are sent to $d$--edges.
    Define:
    \begin{align*}
    \psi':=\mathrm{Id}_{\Sigma_{F}}\times\psi\from &\Sigma_{F\join E}=
      \Sigma_{F}\times\Sigma_{A\cup B\cup C}\\
      &\to \Sigma_{F}\times\Sigma_{A\cup B\cup C\cup
        D}=\Sigma_{F\join E'}  
    \end{align*}

    Because this map is just a product, the good behavior of $\psi$ on
    colored components carries over.
    Specifically, if the restriction of $\psi$ to the $(A\cup C)$--component of
    $\Sigma_E$ based at vertex $w$ sends it to the $(A\cup C)$--component of
    $\Sigma_{E'}$ based at $v$ then the restriction of $\psi$ is a
    color preserving isomorphism on that component, except at $w$ if
    $\psi(w)\neq v$, and
    the distance from $v$ to $\psi(w)$ is uniformly bounded. 
    Then $\psi'$ restricted to the $\Sigma_{F\join (A\cup
      C)}$--component of $\Sigma_{F\join E}$ based at $(1,w)$ sends it to the $\Sigma_{F\join (A\cup
      C)}$--component of $\Sigma_{F\join E'}$ based at $(1,v)$, and is
    a color preserving cubical isomorphism except along
    $\Sigma_F\times\{w\}$, which is sent to
    $\Sigma_F\times\{\psi(w)\}$ instead of $\Sigma_F\times\{v\}$.
   But these two sets are parallel at distance $d(v,\psi(w))$, which
   was uniformly bounded, so the restriction of $\psi'$ to an  $(A\cup
   C)$--component is uniformly bounded distance from a color
   preserving cubical isomorphism.

    Similar statements hold for copies of $\Sigma_{F\join (B\cup
      C)}$, except that $c$--edges change to $d$--edges. 

    According to our setup, $\Gamma=\bar G\cup\bar H$ and $\bar
    G\cap\bar H=F\cup C\subset F\join E$.
    This gives the following splittings of $W_\Gamma$ and $W_{\Gamma'}$ as graphs of
    groups, where each edge group is simply the intersection of its two
    vertex groups.
    \begin{align*}
      W_\Gamma&=W_{\bar{G}}\stackrel{W_{F\join (A\cup C)}}{\longdash}
                W_{F\join E}\stackrel{W_{F\join (B\cup C)}}{\longdash} 
                W_{\bar{H}}\\
           W_{\Gamma'}&=W_{\bar{G}}\stackrel{W_{F\join(A\cup
                        C)}}{\longdash}  W_{F\join
                        E'}\stackrel{W_{F\join (B\cup D)}}{\longdash} 
    W_{\bar{H}'}
    \end{align*}
 
    Now apply \fullref{tree_of_quasiisometries} to get a quasiisometry between $W_\Gamma$ and $W_{\Gamma'}$
    as a tree of quasiisometries with respect to the given splittings,
    with base maps:
    \begin{itemize}
    \item The identity on $W_{\bar{{G}}}$.
    \item $\psi'\from W_{F\join E}\to W_{F\join E'}$
      \item The
    isomorphism $W_{\bar{{H}}}\to W_{\bar{{H}}'}$
    fixing each $b_j$ and sending $c$ to $d$.\qedhere
  \end{itemize}
  The check that the hypotheses of \fullref{tree_of_quasiisometries}
  are satisfied is similar to the one in the proof of
  \fullref{prop:cloning}: Condition~\eqref{item:coset_bijection} has been
  established explicitly, and Conditions~\eqref{item:base_consistency}
  and \eqref{item:consistency} follow
  from the fact that $\psi'$ is uniformly bounded distance from a
  color preserving isomorphism on each coset. 
  \end{proof}
  
Here is an application of \fullref{prop:unfold}.
  \begin{proposition}\label{prop:removecutpaths}
  For every triangle-free CFS graph $\Gamma$ there is a triangle-free
  CFS graph $\Gamma'$ with no cut 2--paths, such that $W_\Gamma$ and
  $W_{\Gamma'}$ are quasiisometric. 
\end{proposition}
The proof will require the following lemma.
\begin{lemma}\label{components_contain_crossing_square}
  Suppose $\Gamma$ is incomplete, triangle-free, and  CFS.
Suppose $C$ is a cut, either an anticlique $\{a,b\}$ or a 2--path $a\edge
c\edge b$, such that $\Gamma\setminus C$ is not connected.
Then every component of $\Gamma\setminus C$ contains a vertex in $\lk(a)\cap\lk(b)$.
\end{lemma}
\begin{proof}
    Let $X$ be a component of $\Gamma\setminus C$, and let
    $\bar{X}:=X\cup C$.

  Suppose there is a vertex $\{x,y\}$ of $\diag(\Gamma)$ such that
  $x\in X$ and 
  $y\in\Gamma\setminus\bar{X}$.
  Then $x$ and $y$ are the
  diagonals of a square.
  The other diagonal of that square consists of two nonadjacent
  vertices in $\lk(x)\cap\lk(y)$, but every path from $x$ to $y$
  passes through $C$, and the only nonadjacent points of $C$ are $a$
  and $b$, so the square is $\{a,b\}\join\{x,y\}$.
In this case we are done: $x\in X\cap\lk(a)\cap\lk(b)$.

Now we show there must be such a vertex of $\diag(\Gamma)$ containing a point each from $X$ and $\Gamma\setminus\bar{X}$.
Choose any $x\in X$ and $y\in\Gamma\setminus\bar{X}$.
  Since $\Gamma$ is incomplete, triangle-free, and CFS, it has no
  cone vertices and $\diag(\Gamma)$
  contains a component whose support is all of $\Gamma$, so there is a nontrivial path in $\diag(\Gamma)$ from a vertex
with $x$ in its support to a vertex with $y$ in its support.
The path corresponds to a chain of squares $S_0,\dots,S_n$ in $\Gamma$
such that consecutive squares share a diagonal and $x\in S_0$ and
$y\in S_n$.
Let $m$ be the least index such that  $S_m$ contains a vertex of
$\Gamma\setminus\bar{X}$. Call that vertex $y'$.
If $S_m$ contains a vertex from $X$ we are done, so assume not.
This implies $m>0$
Consider $S_{m-1}$.
It is contained in $\bar{X}$, so it contains some $x'\in X$, since
$|S_{m-1}|=4$ and $|\bar{X}\setminus X|\leq 3$.
We have that $S_{m-1}$ does not contain $y'$ and $S_m$ does not
contain $x'$. 
The shared diagonal of $S_{m-1}$ and $S_m$ consists of two nonadjacent vertices in 
$\lk(x')\cap\lk(y')\subset C$, which must be $\{a,b\}$.
Conversely, $x'$ and $y'$ are then nonadjacent vertices in
$\lk(a)\cap\lk(b)$, so $\{a,b\}\join\{x',y'\}$ is a square with a
diagonal containing
vertices from $X$ and $\Gamma\setminus \bar{X}$.
\end{proof}

\begin{proof}[Proof of \fullref{prop:removecutpaths}]
  Suppose $\Gamma$ has a cut 2--path $x\edge c\edge y$.
  Take a component $\mathcal{C}$ of $\Gamma\setminus \{x,y,c\}$.
  Since $\Gamma$ is CFS it does not have separating cliques, and, by
  definition of cut 2--paths, $\{x,y\}$ is not a cut pair, so
  $\mathcal{C}$ contains vertices adjacent to $x$ and $y$ and $c$ and to no other
  vertices of $\Gamma\setminus\mathcal{C}$.
  Let $\bar{\mathcal{C}}:=\mathcal{C}\cup\{x,y,c\}$.
  By the same reasoning $\bar{\mathcal{C}}^c$ contains vertices adjacent to $x$ and
  $y$ and $c$.
  By \fullref{components_contain_crossing_square},
  $A:=\mathcal{C}\cap\lk(x)\cap\lk(y)\neq\emptyset$ and
  $B:=\bar{\mathcal{C}}^c\cap\lk(x)\cap\lk(y)\neq\emptyset$.
  Apply \fullref{prop:unfold} with $A$, $B$, $F = \{x,y\}$ and $C:=\{c\}$.
  In the resulting graph $\Gamma'$ the set $\{x,y\}$ becomes a cut
  pair.
  If $\Gamma'$ has other cut 2--paths repeat the argument until they
  have all been unfolded into cut pairs.  By construction, $\Gamma'$ is triangle-free, as a triangle in $\Gamma'$ would give a triangle in $\Gamma$ under the natural map collapsing $\Gamma'$ to $\Gamma$.  
  Finally, $W_{\Gamma'}$ has quadratic divergence, since it is quasiisometric to $W_{\Gamma}$, and $\Gamma$ is CFS, so $\Gamma'$ is CFS as well.  
\end{proof}
Illustrating \fullref{prop:removecutpaths}, 
in both \fullref{ex:simple_unfolding} and \fullref{ex:iteratedcones},
the graph $\Gamma$ contains a cut 2--path $f_0-c-f_1$ that becomes a
cut pair $\{f_0,f_1\}$ in $\Gamma'$ after unfolding.

\begin{example}\label{ex:bigunfold}
  One might wonder whether the restriction $|C|=1$ in
  \fullref{prop:unfold} is a limitation of the proof or of the
  concept.
  Here is an example with $|C|>1$ that shows the analogue of
    \fullref{prop:unfold} is not true.
    Consider the graphs $\Gamma$ and $\Gamma'$ in \fullref{fig:bigunfold}.

  \begin{figure}[h]
    \centering
    \begin{subfigure}{.4\textwidth}
      \centering
      \includegraphics{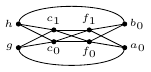}
      \subcaption{$\Gamma$}
    \end{subfigure}
    \quad
       \begin{subfigure}{.4\textwidth}
      \centering
      \includegraphics{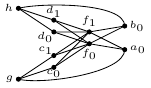}
      \subcaption{$\Gamma'$}
    \end{subfigure}
    \caption{Graphs showing that unfolding is not a quasiisometry when $|C|>1$.}
    \label{fig:bigunfold}
  \end{figure}
  
  As in \fullref{ex:iteratedcones}, the labelling of vertices in
  \fullref{fig:bigunfold}  suggests
the sets to which they belong in the statement of
\fullref{prop:unfold}. It appears that $\Gamma'$ is obtained from
$\Gamma$ by an unfolding-like operation on $F\join C$, but $W_\Gamma$ and $W_{\Gamma'}$ are
not quasiisometric, because their JSJ decompositions are incompatible,
according to \fullref{sec:jsjracg}: $W_\Gamma$ does not split over a
2--ended subgroup, since $\Gamma$ has no cut pairs or cut 2--paths,
while $W_{\Gamma'}$ splits over $W_{\{f_0,f_1\}}$.
  \end{example}

\subsection{Equivalence classes under the graph modification
  operations}\label{sec:qiclasses}
Consider the graph $\Xi$ whose vertices represent the triangle-free CFS
graphs with at most 12 vertices, with an edge between vertices if the
graph of one can be obtained from the other by either link
doubling, cloning, or unfolding.
If two graphs $\Gamma$ and $\Gamma'$ are in the same connected component
of $\Xi$, then $W_\Gamma$ and $W_{\Gamma'}$ are quasiisometric.

\begin{example}
  Consider the graphs $\Gamma$ and $\Gamma'$ shown in
  \fullref{fig:Xi_component}.
  They are contained in the same component of $\Xi$, and each is a
  local minimum in $\Xi$ with respect to number of
  vertices.

  \begin{figure}[h]
    \centering
    \begin{subfigure}{.45\textwidth}
      \centering
      \includegraphics{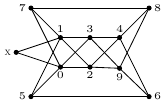}
      \caption{$\Gamma$}
      \label{fig:Xi_component_Gamma}
    \end{subfigure}
      \begin{subfigure}{.45\textwidth}
      \centering
      \includegraphics{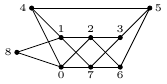}
      \caption{$\Gamma'$}
      \label{fig:Xi_component_Gamma_prime}
    \end{subfigure}
    \caption{Two graphs in the same $\Xi$ component.}
    \label{fig:Xi_component}
  \end{figure}
$\Gamma$ is not a coarse near double: 5, 6, 7, and 8 are all unclonable singletons. It does not admit a
FIDL--$\Lambda$ since it is not bipartite. 
  
$\Gamma'$ is a near double: it has a pair of adjacent singletons, 4
and 5, and one additional singleton, 8, that is a satellite of 4, so
$\double^\circ_{(5,0)}\circ\double^\circ_4(\Gamma')$ is a graph
double.

$\Gamma$ and $\Gamma'$ are connected in $\Xi$ by a cloning edge and a
link doubling edge; specifically, the graph obtained by cloning vertex
$\ten$ of $\Gamma$ is isomorphic to
$\double^\circ_{7}(\Gamma')$. 
\end{example}

\subsection{Inductive construction revisited}\label{inductive_construction}
Recall that \fullref{question:inductive_construction_raagedy} asks if
every RAAGedy, triangle-free graph without separating cliques is
constructible by coning from a square.
It seems implausible that this should be true in light of the fact
that the unfolding and cloning operations add vertices, but in fact we
can computationally verify an affirmative answer for triangle-free CFS
graph with at most 12 vertices that are known to be RAAGedy by the
results of this paper so far.
For graphs of any finite cardinality that are RAAGedy by virtue of satisfying Dani and
Levcovitz's or Davis and Januszkiewicz's conditions, the question  
has a positive answer, even if we strengthen the condition `RAAGedy'
to `commensurable to a RAAG':
\begin{proposition}
  If $\Gamma$ is a triangle-free graph without separating cliques that is
  either a graph double or admits a FIDL--$\Lambda$, then $\Gamma$ is constructible by coning from a
  square through graphs that define RACGs commensurable to RAAGs. 
\end{proposition}
\begin{proof}
  For graphs admitting a FIDL--$\Lambda$, this is the main result of
  \cite{CasEdl}; the idea is to use the fact that $\Lambda$ is a
  forest and show that there is a $\Lambda$--leaf that can be removed
  first, and then induct  on the number of vertices. 
 Here we prove the graph double case, also by induction on the
  number of vertices. 
  
  Suppose $\Gamma$ is a graph double.
  If it is a square we are done, so suppose not.
  Then $\Gamma=\double(\Theta)$ where $\Theta$ is a connected,
  triangle-free, incomplete graph.
  Recall that the vertex set of $\Gamma$ is identified with
  $\Theta\times\{0,1\}$.
  
  Since $\Theta$ is finite it contains some vertex $v$ that is not a
  cut vertex. 
  Let $\Theta':=\Theta\setminus\{v\}$ and
  $\Gamma'=\double(\Theta')\subset\Gamma$.
  Let
  $L:=\lk_{\Gamma}((v,0))=\lk_{\Gamma}((v,1))=\lk_\Theta(v)\times\{0,1\}$.
  Let $\Gamma'':=\Gamma\setminus\{ (v,1)\}$, so that  
  $\Gamma=\Gamma''\join_L (v,1)$ and $\Gamma''=\Gamma'\join_L(v,0)$. 
  Since $\Theta'$ is connected and triangle-free, it follows that
  $\Gamma'$ and $\Gamma''$ are triangle-free, incomplete graphs
  without separating cliques.
  Since $(v,0)$ is the only vertex of $\Gamma''$ with no twin, apply \fullref{recognizeneardouble} to see $\Gamma''$ is a near
  double, so defines a RACG commensurable to a RAAG, by
  \fullref{neardoublecommraag}.
  
  Now $\Gamma$ is constructible from a square by coning through graphs
  defining RACGs commensurable to RAAGs, first by applying the induction hypothesis to
  $\Gamma'$, which is a graph double with fewer vertices than
  $\Gamma$,  then coning once to get $\Gamma''$ and  once more to get $\Gamma$.
\end{proof}

\section{Obstructions to being RAAGedy}\label{sec:negative}
We now work from the other direction to find reasons that a
right-angled Coxeter group cannot be quasiisometric to any right-angled
Artin group. 

\subsection{Minsquare versus CFS}\label{sec:minsquare}

Recall that for incomplete triangle-free graphs $\Gamma$ without separating cliques, an obstruction to $W_{\Gamma}$ being RAAGedy is provided by $\Gamma$ not being minsquare (see \fullref{raagedy_implies_minsquare}) or not being CFS (see Section~\ref{sec:firstexamples}). In fact, we show in \fullref{raagedyimpliestrongcfs}
that $\Gamma$ must be strongly CFS if $W_{\Gamma}$ is RAAGedy.  In light of these results, we now explore the relationships between these properties.

\begin{lemma}\label{minsquareintermsofdiag}
  Let $\Gamma$ be a graph.
  Minsquare subgraphs of $\Gamma$ are in bijection with collections
  $\mathcal{C}$ of connected components of $\diag(\Gamma)$ satisfying:
  \begin{enumerate}
  \item $\mathcal{C}$ is nonempty.\label{item:nonempty}
    \item  If $\{v,w\}\subset\bigcup_{C\in\mathcal{C}}\supp(C)$ and $\{v,w\}$ is a vertex
  of $\diag(\Gamma)$ contained in connected component $C_0$, then
  $C_0\in\mathcal{C}$.\label{item:square_complete}
  \item $\mathcal{C}$ is minimal with respect to inclusion among
    collections of connected components of $\diag(\Gamma)$ satisfying
    the previous two conditions. \label{item:minimality}
  \end{enumerate}
\end{lemma}
\begin{proof}
  If $\mathcal{C}$ is a collection of connected components of
  $\diag(\Gamma)$ and $\Delta$ is a subgraph of $\Gamma$, 
  let $\Phi(\mathcal{C})$ be the span of
  $\bigcup_{C\in\mathcal{C}}\supp(C)$ in $\Gamma$ and let $\Psi(\Delta)$ be the
  collection of connected components $C$ of $\diag(\Gamma)$ such that
  $\Delta$ contains an induced square with vertices in $\supp(C)$.
We claim that $\Phi$ and $\Psi$ give inverse bijections between the
set of minsquare subgraphs of $\Gamma$ and collections of connected
components of $\diag(\Gamma)$ satisfying the conditions in the statement of the lemma. 
  
Since $\diag(\Gamma)$ has no isolated vertices, $\mathcal{C}$ is
nonempty if and only if $\Phi(\mathcal{C})$ contains an induced square
of $\Gamma$, and $\Delta$ contains an induced square of $\Gamma$ if
and only if $\Psi(\Delta)$ is a nonempty collection.

Now we want to show that Condition~\eqref{item:square_complete}
describes square completeness. 
If $\mathcal{C}$ satisfies Condition~\eqref{item:square_complete} then whenever
$\{v,w\}\subset\Phi(\mathcal{C})$ is a vertex of $\diag(\Gamma)$
contained in a connected component $C_0$, we have $C_0\in\mathcal{C}$.
Since $C_0$ is a connected component of $\diag(\Gamma)$, every vertex 
adjacent to $\{v,w\}$ is also
contained in $C_0$, so every induced square of $\Gamma$ with
diagonal $\{v,w\}$ is contained in $\supp(C_0)\subset
\Phi(\mathcal{C})$.
Thus, $\Phi(\mathcal{C})$ is square complete.

Conversely, if $\Delta$ is a square complete subgraph of $\Gamma$ and
if $\{v,w\}\subset \Delta$ is a vertex of $\diag(\Gamma)$ contained in
connected component $C_0$, then every neighbor $\{x,y\}$ of
$\{v,w\}$, of which there is at least one, corresponds to an induced
square $\{v,w\}\join\{x,y\}$ of $\Gamma$.
By square completeness, $\{x,y\}\subset\Delta$, so
$C_0\in\Psi(\Delta)$.
Furthermore, by induction on distance to $\{v,w\}$ in $C_0$,
$\supp(C_0)\subset\Delta$.
Applying this reasoning to each component in $\Psi(\Delta)$, we see that
$\Phi(\Psi(\Delta))\subset\Delta$. 
This shows that $\Psi(\Delta)$ satisfies Condition~\eqref{item:square_complete}, as follows.  If 
$\{v,w\}\subset\bigcup_{C\in\Psi(\Delta)}\supp(C)$ is a vertex of
$\diag(\Gamma)$ contained in a component $C_0$, then  $\{v, w\} \subset \Delta$, since 
$\bigcup_{C\in\Psi(\Delta)}\supp(C)$ induces the graph $\Phi(\Psi(\Delta))\subset \Delta$, so by the first part of this paragraph, $C_0\in\Psi(\Delta)$. 

Now observe that every collection $\mathcal C$ satisfying $(2)$ fulfills $\Psi(\Phi(\mathcal{C}))=\mathcal{C}$. Indeed, $C_0\in \Psi(\Phi(\mathcal{C}))$ means that
$\Phi(\mathcal{C})$ contains an induced square with vertices in
$\supp(C_0)$, so Condition~\eqref{item:square_complete} gives
$C_0\in\mathcal{C}$, ie.\ $\Psi(\Phi(\mathcal{C})) \subseteq
\mathcal{C}$.
Conversely, if $C_0 \in \mathcal C$ then $C_0 \in \Phi(\supp(C_0))$, hence, $C_0\in \Psi(\supp(C_0))\subseteq \Psi(\Phi(\mathcal C))$.

 Finally, we show that Condition~\eqref{item:minimality} is the
 analogue of the minimality condition in the definition of minsquare. 
Suppose $\mathcal{C}$ satisfies all three conditions, and define
$\Delta:=\Phi(\mathcal{C})$.
By the previous two steps of the argument, $\Delta$ is square complete and contains a square. 
Suppose that $\Delta'$ is a square complete subgraph of $\Delta$
that contains a square.
Then $\Psi(\Delta')$ is a subcollection of $\Psi(\Delta)$
that satisfies the first two conditions, so by minimality
$\Psi(\Delta')=\Psi(\Delta)$. 
Now:
\[\Delta'\subset
  \Delta=\Phi(\mathcal{C})=\Phi(\Psi(\Delta))=\Phi(\Psi(\Delta'))\subset\Delta'\]
Thus, $\Delta$ is square complete and contains a square, and there is
no proper subgraph of $\Delta$ with both properties, so $\Delta$ is
minsquare.

Conversely, if $\Delta$ is minsquare then it is square complete and
contains a square, so $\Psi(\Delta)$ is a 
collection of connected components of $\diag(\Gamma)$ that satisfies
the first two conditions.
Suppose $\mathcal{C}$ is a subcollection of $\Psi(\Delta)$
satisfying the first two conditions.
Then $\Phi(\mathcal{C})$ is a square complete subgraph of $\Delta$ that
contains a square.
By minimality of $\Delta$, we have $\Phi(\mathcal{C})=\Delta$. 
But then $\mathcal{C}\subset\Psi(\Delta)=\Psi(\Phi(\mathcal{C}))=\mathcal{C}$.
Thus, $\Psi(\Delta)$ is minimal with respect to inclusion among
collections of components of $\diag(\Gamma)$ satisfying the first two
conditions. 
\end{proof}

Genevois \cite[Example~7.3]{Gen22} 
showed by example that the minsquare
and CFS properties are independent.
Using \fullref{minsquareintermsofdiag}, we give triangle-free examples.

\begin{example}[minsquare does not imply CFS]\label{minsquarenotcfs}
 Consider the strongly CFS graph $\Gamma_0:= \begin{tikzpicture}[baseline=5pt,scale=.5]
  \filldraw (0,0) circle (2pt) (1,0) circle (2pt) (2,0) circle (2pt)
  (3,0) circle (2pt) (0,1) circle (2pt) (1,.66) circle (2pt) (3,1)
  circle (2pt);
  \draw (0,0) node[below]{\tiny$0$}--(1,0)
  node[below]{\tiny$1$}--(2,0) node[below]{\tiny$2$}--(3,0)
  node[below]{\tiny$3$}--(3,1) node[above]{\tiny$5$}--(0,1) node[above]{\tiny$4$}--(0,0)
  (0,1)--(1,.66)node[left,yshift=-2pt]{\tiny$6$}--(1,0) (1,.66)--(3,0) (1,0)--(3,1);
  \end{tikzpicture}$, which contains an isometrically embedded path $P:=(0,1,2,3)$
that contains a pair $\{0,2\}$ of vertices that is
not the diagonal of any square and a pair $\{1,3\}$ that is. 
Construct $\Gamma$ by taking two copies of $\Gamma_0$ and
identifying the two copies of $P$ with opposite orientations:
 \[\Gamma:=\begin{tikzpicture}[baseline=-2pt,scale=.8]\tiny
\coordinate[label={[label distance=0pt] 180:$0$}] (0) at (-2,0);
\coordinate[label={[label distance=0pt] -90:$1$}] (1) at (-1,0);
\coordinate[label={[label distance=0pt] 90:$2$}] (2) at (0,0);
\coordinate[label={[label distance=0pt] 0:$3$}] (3) at (1,0);
\coordinate[label={[label distance=0pt] 90:$4$}] (4) at (-2,1);
\coordinate[label={[label distance=0pt] 225:$6$}] (5) at (-1,0.7);
\coordinate[label={[label distance=0pt] 90:$5$}] (6) at (1,1);
\coordinate[label={[label distance=0pt] -90:$7$}] (7) at (1,-1);
\coordinate[label={[label distance=0pt] 45:$9$}] (8) at (0,-0.7);
\coordinate[label={[label distance=0pt] -90:$8$}] (9) at (-2,-1);
\draw (0)--(1);
\draw (0)--(4);
\draw (0)--(8);
\draw (0)--(9);
\draw (1)--(2);
\draw (1)--(5);
\draw (1)--(6);
\draw (2)--(3);
\draw (2)--(8);
\draw (2)--(9);
\draw (3)--(6);
\draw (3)--(7);
\draw (3)--(5);
\draw (4)--(5);
\draw (4)--(6);
\draw (7)--(8);
\draw (7)--(9);
\filldraw (0) circle (1pt);
\filldraw (1) circle (1pt);
\filldraw (2) circle (1pt);
\filldraw (3) circle (1pt);
\filldraw (4) circle (1pt);
\filldraw (5) circle (1pt);
\filldraw (6) circle (1pt);
\filldraw (7) circle (1pt);
\filldraw (8) circle (1pt);
\filldraw (9) circle (1pt);
\end{tikzpicture}
\quad\quad\diag(\Gamma)=
\begin{tikzpicture}[scale=4,baseline=5pt]\tiny
\coordinate[label={[label distance=-1pt] 90:$({0, 2})$}] (0) at (.25,0);
\coordinate[label={[label distance=-1pt] -90:$({1, 9})$}] (1) at (0,-0.125);
\coordinate[label={[label distance=-1pt] 90:$({8, 9})$}] (2) at (.5,0);
\coordinate[label={[label distance=-1pt] 90:$({0, 7})$}] (3) at (0.75,0);
\coordinate[label={[label distance=-1pt] 90:$({0, 6})$}] (4) at (0,0.25);
\coordinate[label={[label distance=-1pt] 90:$({1, 4})$}] (5) at (.25,0.25);
\coordinate[label={[label distance=-1pt] 90:$({0, 5})$}] (6) at (.5,0.25);
\coordinate[label={[label distance=-1pt] 90:$({1, 8})$}] (7) at (0,0);
\coordinate[label={[label distance=-1pt] -90:$({1, 3})$}] (8) at (.5,0.125);
\coordinate[label={[label distance=-1pt] 90:$({2, 6})$}] (9) at (.75,0.25);
\coordinate[label={[label distance=-1pt] -90:$({2, 5})$}] (10) at (0.75,0.125);
\coordinate[label={[label distance=-1pt] -90:$({5, 6})$}] (11) at (0.25,0.125);
\coordinate[label={[label distance=-1pt] -90:$({2, 7})$}] (12) at (0.5,-0.125);
\coordinate[label={[label distance=-1pt] -90:$({3, 9})$}] (13) at (0.75,-0.125);
\coordinate[label={[label distance=-1pt] -90:$({3, 8})$}] (14) at (0.25,-0.125);
\coordinate[label={[label distance=-1pt] -90:$({3, 4})$}] (15) at (0,0.125);
\draw (0)--(1);
\draw (0)--(2);
\draw (0)--(7);
\draw (2)--(3);
\draw (2)--(12);
\draw (4)--(5);
\draw (5)--(6);
\draw (5)--(11);
\draw (8)--(9);
\draw (8)--(10);
\draw (8)--(11);
\draw (11)--(15);
\draw (12)--(13);
\draw (12)--(14);
\filldraw (0) circle (.25pt);
\filldraw (1) circle (.25pt);
\filldraw (2) circle (.25pt);
\filldraw (3) circle (.25pt);
\filldraw (4) circle (.25pt);
\filldraw (5) circle (.25pt);
\filldraw (6) circle (.25pt);
\filldraw (7) circle (.25pt);
\filldraw (8) circle (.25pt);
\filldraw (9) circle (.25pt);
\filldraw (10) circle (.25pt);
\filldraw (11) circle (.25pt);
\filldraw (12) circle (.25pt);
\filldraw (13) circle (.25pt);
\filldraw (14) circle (.25pt);
\filldraw (15) circle (.25pt);
\end{tikzpicture}
\]

No induced square of $\Gamma$ enters both copies of
$\Gamma_0\setminus P$, so $\diag(\Gamma)$ consists of two disjoint
copies of $\diag(\Gamma_0)$, as shown. 
Thus, $\Gamma$ is not CFS.
On the other hand, $\{0,2\}$ is in the support of the top component of
$\diag(\Gamma)$ and also appears as a vertex in the bottom
component.
Similarly, $\{1,3\}$ is in the support of the bottom component and
appears as a vertex in the top component.
Since neither component has square complete support, 
\fullref{minsquareintermsofdiag} says $\Gamma$ is minsquare.
\end{example}

\begin{example}[CFS does not imply minsquare]\label{cfsdoesnotimplyminsquare}
Consider the graphs:
\[\Gamma_0=\begin{tikzpicture}[scale=.5,baseline=15pt]\tiny
  \filldraw (0,0) circle (2pt)(1,0) circle (2pt) (2,0) circle (2pt)
  (0,1) circle (1 pt) (1,1) circle (2pt) (2,1) circle (2pt) (0,2)
  circle (2pt) (1,2) circle (2pt) (2,2) circle (2pt) (2,3) circle
  (2pt);
  \draw  (0,0)  node[left]{$7$} --(0,2)  node[left]{$1$} --(2,3)  node[right]{$0$} --(2,0)  node[right]{$9$} --(0,0)
  (0,2)  node[left]{$1$} --(2,2)  node[right]{$3$} --(0,1)  node[left]{$4$} --(2,1)  node[right]{$6$} --(0,0)
  (1,0)  node[below]{$8$} --(1,1)  node[below left]{$5$}
  --(1,2)  node[above,xshift=2pt]{$2$} ;
\end{tikzpicture}
\quad\quad\Gamma:=\begin{tikzpicture}[scale=.5,baseline=15pt]\tiny
  \filldraw (0,0) circle (2pt)(1,0) circle (2pt) (2,0) circle (2pt)
  (0,1) circle (1 pt) (1,1) circle (2pt) (2,1) circle (2pt) (0,2)
  circle (2pt) (1,2) circle (2pt) (2,2) circle (2pt) (2,3) circle
  (2pt);
  \draw  (0,0)  node[left]{$7$} --(0,2)  node[left]{$1$} --(2,3)  node[above]{$0$} --(2,0)  node[below]{$9$} --(0,0)
  (0,2)  node[left]{$1$} --(2,2)  node[right]{$3$} --(0,1)  node[left]{$4$} --(2,1)  node[right]{$6$} --(0,0)
  (1,0)  node[below]{$8$} --(1,1)  node[below left]{$5$}
  --(1,2)  node[above,xshift=2pt]{$2$} ;
 \draw (2,0) .. controls (3,0) and (3,3) .. (2,3);
\end{tikzpicture}\]

$\Gamma_0$ is (strongly) CFS and contains a path $\gamma:=(0,3,6,9)$ that is
isometrically embedded and does not contain a diagonal of any
square.
Adding the edge $0\edge 9$ to make $\Gamma$ does not disturb the CFS
property, since the new edge does not kill any square.
The square spanned by $\{0,3,6,9\}$ in $\Gamma$ does not share a diagonal with any other
square, so it is a proper minsquare subgraph of a
CFS graph.
\end{example}

\begin{example}[minsquare and CFS does not imply strongly
  CFS]\label{minsquare_and_CFS_does_not_imply_strongly_CFS}
  Let $\Gamma_0$ be as in \fullref{minsquarenotcfs}.
  Let $\Gamma_1$  be a $(6,2)$--spider with one pincer foot (recall \fullref{snakewormspider}).
  Identify the non-pincer feet of the spider with vertices
  $\{1,3,4,5,6\}$, and the two vertices of the pincer foot with
  vertices $0$ and $2$. See \fullref{fig:spider_attack}. 
  All edges of the resulting graph $\Gamma$ come from either
  $\Gamma_0$ or $\Gamma_1$.
  The spider $\Gamma_1$ is strongly CFS and contains all vertices of
  $\Gamma$, so $\Gamma$ is CFS.

  \begin{figure}[h]
    \centering
    \includegraphics{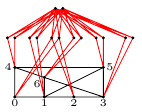}
    \caption{A spider attacking a small graph. }
    \label{fig:spider_attack}
  \end{figure}
  
We claim that $\diag(\Gamma)$ consists of one component isomorphic 
to $\diag(\Gamma_0)$, and one component consisting of $\diag(\Gamma_1)$ with
two additional leaf vertices attached, coming from the two possible
squares that use the segment $(0,1,2)$ of
$\Gamma_0$ and one of the common neighbors of $0$ and $2$ in
$\Gamma_1$.
This uses the fact that $\{0,2\}$ is not the diagonal of any square in
$\Gamma_0$, so we do not create any connection between the copies of
$\diag(\Gamma_0)$ and $\diag(\Gamma_1)$ in $\diag(\Gamma)$. 
Since $\diag(\Gamma)$ has two components, $\Gamma$ is not strongly CFS.

$\Gamma$ is minsquare because neither component of
$\diag(\Gamma)$ satisfies \fullref{minsquareintermsofdiag}.
\end{example}

\begin{corollary}\label{strongly_CFS_and_no_cone_implies_minsquare}
  Strongly CFS with no cone vertices implies minsquare.
\end{corollary}
\begin{proof}
  If $\Gamma$ is strongly CFS with no cone vertices then
  $\diag(\Gamma)$ is connected, so the lone component
  $C=\diag(\Gamma)$ has $\supp(C)=\Gamma$ and the collection $\{C\}$ satisfies the conditions
  of \fullref{minsquareintermsofdiag}, so $\Gamma=\Phi(\{C\})$ is
  minsquare. 
\end{proof}

\subsection{Stable subspaces and Morse boundaries}\label{sec:morse}
By a \emph{stable cycle} we mean a simple cycle in $\Gamma$ of length
at least 5 that is square complete. The corresponding special subgroup
is stable and 1--ended; it is virtually a hyperbolic surface group.
If $\Gamma$ is RAAGedy it cannot have stable cycles, by \fullref{Morsehyperbolic}.
Indeed, this is exactly the construction used by Behrstock
\cite{MR3966609} to give the first example of a non-RAAGedy CFS
graph (recall \fullref{fig:Behrstock}).
Since passing to iterated link double of $\Gamma$ induces a quasiisometry
on the level of RACGs, if $\Gamma$ is RAAGedy then no link double of
$\Gamma$ can contain a stable cycle.
In this subsection we address the possibility of finding stable cycles
in iterated link doubles. 

A conjecture proposed in \cite{Tra19} suggested that if a graph $\Gamma$ contains no stable cycles, then none would appear in any of its link doubles. This was disproven by a counterexample provided by Graeber et al.\ \cite{GraKarLaz21}. Their example was initially constructed by link-doubling a graph with triangles to produce a triangle-free graph without stable cycles  \cite{Karrer2021_1000129500}[Sec. 5.5]. A second link-doubling of this graph, however, resulted in the emergence of a stable cycle. Notably, none of the three graphs in this construction are CFS, prompting the question of whether a similar phenomenon could occur in CFS graphs. Furthermore, the example raised the guess that one doubling might always suffice to produce stable cycles in the triangle-free case.
As the following example shows, both assumptions are false—even within the class of triangle-free, strongly CFS graphs.

\begin{example}[Deeply buried stable
  cycle]\label{ex:buried_stable_cycle}
  \fullref{fig:gamma_buried_stable_cycle} gives a triangle-free, strongly CFS graph $\Gamma$ that has a deeply buried
  stable cycle; specifically:
  \begin{enumerate}
  \item $\Gamma$
    contains no stable cycle.
    \item $\forall v\in \Gamma$, $\double^\circ_v(\Gamma)$ contains no
      stable cycle.
      \item $\exists v,w\in \Gamma$,
        $\double^\circ_{(w,0)}(\double^\circ_v(\Gamma))$ contains a
        stable cycle. 
      \end{enumerate}
The first two claims are verified by enumerating and checking the
possibilities.
The third is achieved by doubling first over vertex 0 and then over vertex $(2,0)$.
      Consider the cycle $((9,0),0)$, $((8,0),0)$, $((4,0),1)$, $((\ten,0),1)$,
$((4,1),1)$, $((3,1),0)$, $((5,0),0)$, shown in red in
\fullref{fig:unburied_stable_cycle}.
It is induced and has a 2--chord $((4,0),1)\edge((3,0),0)\edge
((5,0),0)$, but there is no other 2--path between vertices of the
cycle that is not  a subsegment of the cycle, so it is square complete.
      \begin{figure}[h]
        \centering
        \begin{subfigure}{.3\textwidth}
          \centering
\includegraphics{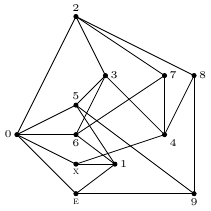}
          \subcaption{$\Gamma$}
        \label{fig:gamma_buried_stable_cycle}
      \end{subfigure}
      \quad
       \begin{subfigure}{.6\textwidth}
          \centering
\includegraphics{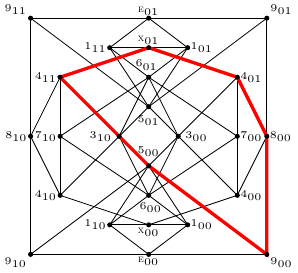}
          \subcaption{$\double^\circ_{(2,0)}(\double^\circ_0(\Gamma))$}
        \label{fig:unburied_stable_cycle}
        \end{subfigure}
        \caption{A graph with a deeply buried stable cycle for \fullref{ex:buried_stable_cycle}}
        \label{fig:buriedstablecycle}
      \end{figure}
\end{example}

We also know an example where a stable cycle appears only after
performing three link doubles, so it seems unlikely that there should
be any universal bound on how deeply stable cycles can be buried.
These examples were found by computer in our enumeration of small graphs;
we do not know how to build them on demand. 

\begin{question}
  Can stable cycles be buried arbitrarily deeply?
  That is, given $n$, does there exist a (triangle-free, strongly) CFS graph $\Gamma$
  such that for all $m<n$, no $m$--fold iterated link double of $\Gamma$ that contains
  a stable cycle, but there is an $n$--fold iterated link double of
  $\Gamma$ that does?
\end{question}

There is a notion of \emph{Morse boundary} of a group \cite{Cor17},
which is a boundary at infinity made of equivalence classes of Morse
geodesic rays.
Its topological type is quasiisometry invariant.
\begin{theorem}[{\cite{ChaCorSis23}}]
  The Morse boundary of a RAAG is totally disconnected.
\end{theorem}
A stable cycle appearing in some iterated link double of $\Gamma$
implies that the Morse boundary of $W_\Gamma$ contains a circle, so is
an obstruction to $W_\Gamma$ being quasiisometric to a RAAG. 
However, as discussed in \cite{GraKarLaz21}, the  1--skeleton of a 3--cube does
not have stable cycles in any link double, but it does have circles in
its Morse boundary.
Thus, connectivity in the Morse boundary is a better obstruction to
being RAAGedy than the existence of buried stable cycles.

Fioravanti and Karrer \cite{FioKar22} show that a group has totally
disconnected Morse boundary if it splits as amalgam of groups with
totally disconnected boundaries (including, possibly, the empty set)  over a subgroup with empty
relative Morse boundary, generalizing the results in \cite{Karrer23}.
For our purposes it is enough to say that if $\Gamma$ is a thick join
then $W_\Gamma$ has empty Morse boundary, and if $\Delta<\Gamma$ is a
subgraph contained in a thick join subgraph of $\Gamma$ then the
relative Morse boundary of $W_\Delta$ in $W_\Gamma$ is empty. 
\begin{proposition}[Base sufficient criteria for t.d.\ Morse
  boundary, cf \cite{FioKar22,Karrer23}]\label{FioKar}\leavevmode
  \begin{enumerate}
  \item If $\Gamma$ is a clique or a thick join then $W_\Gamma$ has
    empty Morse boundary.
  \item If $\Gamma$ contains a subgraph $\Delta$ such that all of the
    following are true:
    \begin{enumerate}
      \item $\Gamma\setminus\Delta$ has more than one component. 
    \item $\Delta$
      is a clique or is contained in a thick join subgraph of
      $\Gamma$.
      \item For each component $C$ of $\Gamma\setminus\Delta$,
        $W_{C\cup\Delta}$ has totally disconnected Morse boundary.
      \end{enumerate}
      Then $W_\Gamma$ has totally disconnected Morse boundary. 
  \end{enumerate}
\end{proposition}

\begin{definition}\label{FioKar_decomposition_sequence}
   A \textit{decomposition-sequence} of $\Gamma$ is a rooted tree $T$ of graphs where:
   \begin{itemize}
   \item The graph associated to the root is $\Gamma$.
      \item The graph associated to a vertex equals the union of the graphs of its descendants, and the intersection of its descendants is contained in a thick join subgraph or a clique of $\Gamma$. 
    \item Every non-leaf vertex of $T$ has at least 2 descendants.
\end{itemize}
\end{definition}

\begin{corollary}[Inductive sufficient criteria for t.d.\ Morse
  boundary]\label{inductive_FioKar}
  If $\Gamma$ admits a decomposition-sequence such that the graphs
  associated to the leaves are either cliques or thick joins, then
  $W_\Gamma$ has totally disconnected Morse boundary.
\end{corollary}

If \fullref{inductive_FioKar} is satisfied then we save ourselves the
work of iterating link doubles of $\Gamma$ and searching for stable
cycles; none exist.

\begin{corollary}\label{no_stable_cycles}
  If $\Gamma$ satisfies \fullref{inductive_FioKar} then no iterated
  link double of $\Gamma$ contains a stable cycle.
\end{corollary}

Forthcoming work of Cordes, Karrer, and Ruane will show that the
conditions of \fullref{inductive_FioKar} are also necessary.
This will imply, in the 2--dimensional case, that a RAAGedy $\Gamma$
must admit a decomposition-sequence.

\subsection{Obstructions from JSJ Decompositions}\label{sec:JSJ}

Recall from \fullref{jsj} that we know some properties of the JSJ
graph of cylinders of a RAAG.
For instance, it has no hanging vertices (\fullref{RAAGnohanging}), and rigid vertices are
special subgroups that admit no finite or 2--ended splittings (\fullref{RAAGflat}).
From the quasiisometry invariance of JSJ trees of cylinders we can
conclude that $\Gamma$ is not RAAGedy if the JSJ graph of cylinders of $W_\Gamma$
has any of the following:
\begin{itemize}
  \item A hanging vertex group. 
\item A rigid vertex group that splits over a finite or 2--ended subgroup.
  \item A rigid vertex group that is not quasiisometric to any RAAG.
\end{itemize}

Here is another obstruction that requires a little more work to justify.
\begin{lemma}[{Cf \cite[Proposition~4.44]{Edl24thesis}}]\label{raagjsj}
  Let $\Delta$ be a connected, triangle-free graph that contains a
  cut vertex $v$, and let $A_\Delta$ be the RAAG defined by $\Delta$.
  Let $\tilde v$ be the cylinder vertex in the JSJ tree of cylinders $\mathcal{T}$
  of $A_\Delta$ whose stabilizer is $A_{\st(v)}$.
    Let $\tilde r$ be a rigid vertex adjacent to $\tilde v$ in $\mathcal{T}$, such that $\tilde r$ is stabilized by
    $A_{\Delta_r}$, where $\Delta_r$ is some maximal biconnected
    subgraph of $\Delta$ containing $v$.
Then:
\begin{itemize}
\item $\tilde v$ has infinite valence in $\mathcal{T}$.
  \item If $\Delta_r$ is a single edge $v\edge w$ then $\tilde r$ has valence 2
    in $\mathcal{T}$, and its stabilizer and that of its two incident
    edges are $A_{\{v,w\}}\cong\mathbb{Z}^2$.
    \item If $\Delta_r$ is not a single edge then $\tilde r$ has
      infinite valence in $\mathcal{T}$ and the stabilizers of $\tilde
      r$ and all of its incident edges are RAAGs of rank at least
      3.
\end{itemize}
\end{lemma}
\begin{proof}
  Partition $\lk(v)$ into subsets $P_1$, $P_2$,... according to which component of
  $\Delta\setminus\{v\}$ each vertex belongs.
  There are at least two such parts, since $v$ is a cut vertex.
  For each $P_i$ there is a maximal biconnected subgraph $\Delta_i$
  containing $P_i$ as well as $v$, and containing no vertex from any other $P_j$. For each $i$ such that $\Delta_i$ is either not an edge or is an edge $v\edge w$ where $w$ is a cut vertex of $\Gamma$, there
 is a rigid vertex $\tilde r_i$ in $\mathcal{T}$ adjacent to
  $\tilde v$ whose stabilizer is $A_{\Delta_i}$. The
  stabilizer of the edge $\tilde e_i$ between $\tilde r_i$ and $\tilde v$ is
  $A_{\Delta_i}\cap A_{\st(v)}=A_{\{v\}\cup P_i}$.  By assumption, there is at least one rigid vertex $\tilde r_i$. For this $i$, since 
  $\lk(v)\setminus P_i$ is nonempty, $A_{\{v\}\cup P_i}$ has
  infinite index in $A_{\st(v)}$, so there are infinitely many
  distinct translates of $\tilde e_i$ in $\mathcal{T}$ incident to
  $\tilde v$.  

  Now consider a rigid vertex $\tilde r_i$ adjacent to $\tilde v$ in $\mathcal T$.  If $P_i=\{w\}$ is a singleton, then $w$
  is necessarily also a cut vertex and $\Delta_i$ is the single edge $w\edge v$.
  
  Then $\tilde r_i$ has incident edges in $\mathcal{T}$, connecting it to  $\tilde{v}$ and $\tilde w$, and $\tilde r_i$ as well as both of these edges
  are stabilized by $A_{v,w}\cong \mathbb{Z}^2$.

  If $P_i$ is not a singleton then $\Delta_i\cap\st(v)$ contains at least three
  vertices, since it contains $P_i$ and $v$, so the stabilizer of
  $\tilde r_i$ and the edge $\tilde e_i$ connecting it to $\tilde v$
  are RAAGs of rank at least 3. 

  Finally, since $\Delta$ is triangle-free, $P_i$ is an anticlique.
  This implies that the biconnected subgraph  $\Delta_i$ contains an additional vertex that is
  not in $\st(v)$, so $A_{\st(v)\cap\Delta_i}$ has infinite index in
  $A_{\Delta_i}$, which implies there are infinitely many edges in the
  orbit of $\tilde e_i$ incident to $\tilde r_i$. 
\end{proof}

\begin{corollary}\label{ZZobstruction}
  If $\Gamma$ is an incomplete, triangle-free graph with no separating
  clique and the JSJ graph of cylinders for $W_\Gamma$ contains a
  rigid vertex group that is not virtually $\mathbb{Z}^2$, but has an
  incident edge that is virtually $\mathbb{Z}^2$, then $\Gamma$ is not
  RAAGedy. 
\end{corollary}
\begin{example}
  Consider the graph $\Gamma$ of \fullref{fig:zsquaredobstruction}.
The JSJ graph of cylinders is:
\[W_\Gamma=W_{\{0,1\}\join\{2,7,8\}}\stackrel{W_{\{0,1\}\join\{2,7\}}}{\longdash}
  W_{\{0,1,2,3,4,5,6,7\}}\]
 
The vertex on the left is a cylinder; the vertex on the right is
rigid. 
The graph $\Gamma$ is not RAAGedy, by \fullref{ZZobstruction}, since
the JSJ graph of cylinders of $W_\Gamma$ has a non-virtually--$\mathbb{Z}^2$ rigid vertex group with an
incident virtually $\mathbb{Z}^2$ edge.
 \begin{figure}[h]
    \centering
    \includegraphics{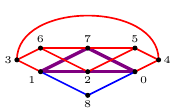}
    \caption{The blue/violet subgraph corresponds to a cylinder in the
      JSJ graph of cylinders of the RACG. The red/violet subgraph
      corresponds to a non-virtually--$\mathbb{Z}^2$ rigid
      vertex. Their intersection (violet) corresponds to a virtually--$\mathbb{Z}^2$ edge.}
    \label{fig:zsquaredobstruction}
  \end{figure}
\end{example}

We state another kind of obstruction to being RAAGedy that one
can derive from the JSJ decomposition.
We will not give a proof here, as \fullref{no_cycle_of_cuts} is a
special case of \fullref{nocycles}. 

\begin{theorem}[No cycles of cuts]\label{no_cycle_of_cuts}
  Let $\Gamma$ be a triangle-free graph without separating cliques.
  Suppose for some $n\geq 3$ there is an anticlique
  $\{a_0,\dots,a_{n-1}\}$ such that for all $i$  there is a cut $\{a_i
  - a_{(i+1)}\}$, for subscripts modulo $n$. 
Then $\Gamma$ is not RAAGedy. 
\end{theorem}

\section{The maximal product region graph}\label{mprg}
In this section we use the maximal product region graph to distinguish
some RACGs from RAAGs.
Recall \fullref{sec:oh}.
In \fullref{sec:mprg_connectivity} we show that RAAGedy graphs are
strongly CFS and establish connectivity properties of the MPRG.
In \fullref{sec:hhs} we show that in the strongly CFS case the MPRG is
equivariantly quasiisometric to the stability recognizing space of
Abbott, Behrstock, and Durham.
In \fullref{sec:ladders} we define ladders in the MPRG as an
obstruction to being RAAGedy, and give sufficient conditions for their
presence.
In \fullref{mprgexamples} we give examples and non-examples.

\subsection{Connectivity properties of the MPRG}\label{sec:mprg_connectivity}
In this section we establish connectivity properties of the maximal
product region graph. As a corollary:
\begin{theorem}\label{raagedyimpliestrongcfs}
  If $\Gamma$ is a triangle-free, RAAGedy graph with no
  separating clique then it is strongly CFS.
\end{theorem}

The proof of \fullref{raagedyimpliestrongcfs} is to apply 
quasiisometry invariance of the MPRG (\fullref{mpgqiinvariant})  and
compare the following result (\fullref{mprg_is_quasitree}, which
says the MPRG of a 1--ended, 2--dimensional RAAG is connected),  to the subsequent
\fullref{stronglycfsimpliesmprgconnected} about connectivity of the
MPRG for RACGs. 

\fullref{raagedyimpliestrongcfs} appeared as
\cite[Proposition~4.52]{Edl24thesis}.
The proof here is similar.
\cite{BehCicFal25} asserts that it is possible to
deduce the same result without the triangle-free hypothesis using
results from \cite{BehHagSis21}.
\begin{theorem}[{\cite[Corollary~4.9]{Oh22}}]\label{mprg_is_quasitree}
  The MPRG of a 1--ended,  irreducible, 2--dimensional RAAG 
  is an unbounded quasitree. In particular, it is connected.
\end{theorem}

 Recall \fullref{def:ric}: Denote the MPRG by  $\mprg_\Gamma$ and let $\ric_\Gamma=W_\Gamma\backslash\mprg_\Gamma$.
 Each vertex $v$ of $\ric_\Gamma$ corresponds to a maximal thick join
 $J_v$ of $\Gamma$, and $W_{J_v}$ is the stabilizer of $v$ for
 $W_\Gamma\act \mprg_\Gamma$, since it is the stabilizer of
 $\Davis_{J_v}$ for $W_\Gamma\act \Davis_\Gamma$. 

\begin{proposition}\label{stronglycfsimpliesmprgconnected}
  Let $\Gamma$ be a triangle-free CFS graph.
  The following are equivalent:
  \begin{itemize}
  \item $\mprg_\Gamma$ is connected.
  \item $\ric_\Gamma$ is connected.
  \item $\diag(\Gamma)$ is connected.
    \item $\Gamma$ is strongly CFS.
  \end{itemize}
\end{proposition}
\begin{proof}
Since $\Gamma$ is CFS, being strongly CFS is equivalent to
$\diag(\Gamma)$ being connected.
In addition, $\Gamma$ is triangle-free, so every vertex belongs to some
square, hence to some maximal thick join subgraph.
Thus, every generator $s$ of $W_\Gamma$ fixes at least one vertex of
$\ric_\Gamma$, so $\ric_\Gamma\cap s\Act\ric_\Gamma\neq\emptyset$.
By induction on word length, for every word $w$ whose $i$--th prefix
is $w_i$ there is a chain
$\ric_\Gamma=w_0\Act\ric_\Gamma,\dots,w_n\Act\ric_\Gamma=w\Act\ric_\Gamma$ such that $w_i\Act\ric_\Gamma\cap
w_{i+1}\Act\ric_\Gamma\neq \emptyset$.
If $\ric_\Gamma$ is connected this implies $\mprg_\Gamma$ is connected.
Conversely, \fullref{weakconvexity} implies that if $\mprg_\Gamma$ is
  connected then $\ric_\Gamma$ is too.

The proof is completed by establishing a bijection between connected
components of $\ric_\Gamma$ and connected components of $\diag(\Gamma)$.

By  \fullref{diagonaljoins}, we can define  maps $\phi$ from thick
joins of $\Gamma$ to joins of $\diag(\Gamma)$ by $\phi(A\join
B):=\binom{A}{2} \join \binom{B}{2}$ and $\psi$ from joins of
$\diag(\Gamma)$ to thick joins of $\Gamma$ by $\psi(A\join B):=
\supp(A)\join \supp(B)$.
Consider the two compositions: $\psi\circ\phi(A\join
B)=\psi(\binom{A}{2}\join\binom{B}{2})=A\join B$ and
$\phi\circ\psi(A\join
B)=\phi(\supp(A)\join\supp(B))=\binom{\supp(A)}{2}\join\binom{\supp(B)}{2}$,
which is a join subgraph of $\diag(\Gamma)$ containing $A\join B$. 

By \fullref{icuricracg}, a vertex $v\in\ric_\Gamma$ corresponds to a maximal
thick join  $J_v\subset\Gamma$, and $v\edge w$ is an edge of $\ric_\Gamma$
when $J_v$ and $J_w$ contain a common square. 
If $v\edge w$ is an edge of $\ric_\Gamma$ then $\phi(J_v)$ and $\phi(J_w)$
are joins in $\diag(\Gamma)$ with at least one edge in common, so
$\phi$ takes components of $\ric_\Gamma$ into components
of $\diag(\Gamma)$.

Conversely, if $\{a,b\}\edge\{c,d\}$ is an edge in $\diag(\Gamma)$
then $\psi(\{a,b\}\join\{c,d\})=\{a,b\}\join\{c,d\}$ is a square in
$\Gamma$, and the set of maximal thick joins containing
$\{a,b\}\join\{c,d\}$ is a nonempty clique in $\ric_\Gamma$.
 Moreover, if $\{a,b\}\edge\{e,f\}$ is another edge in $\diag(\Gamma)$
 then
 $\psi(\{a,b\}\join\{\{c,d\},\{e,f\}\})=\{a,b\}\join\binom{\{c,d,e,f\}}{2}$
 is a thick join containing  $\{a,b\}\join\{c,d\}$ and
 $\{a,b\}\join\{e,f\}$, so the set of maximal thick joins containing
 it is a nonempty clique in the intersection of the clique of those
 containing $\{a,b\}\join\{c,d\}$ and the clique of those containing
 $\{a,b\}\join\{e,f\}$.
 Thus, $\psi$ takes connected components of $\diag(\Gamma)$ into
 connected components of $\ric_\Gamma$.

 Finally, by the observation on compositions of $\phi$ and $\psi$, we
 have that $\phi$ and $\psi$ induce inverse bijections between
 connected components of $\ric_\Gamma$ and connected components of
 $\diag(\Gamma)$. 
\end{proof}

In the next results we will need to go between paths in the maximal
product region graph and paths in the square complex.
Let $\Upsilon$ be a triangle-free graph, let $G_\Upsilon$ be the RACG
or RAAG presented by $\Upsilon$, let $\Sigma_\Upsilon$ be its Davis
complex or universal cover of its Salvetti complex, respectively.
Let $\ric_\Upsilon=G_\Upsilon\backslash\mprg_\Upsilon$ be the
fundamental domain for the action on the MPRG. 
 Let $\gamma\from [0,L]\to\mprg_\Upsilon$ be a combinatorial path.
        We construct a path $\gamma'$ in $\Sigma_\Gamma$ \emph{shadowing} it
        as follows.
        For each $i\in [0,L)$, the maximal standard product regions $\gamma(i)$ and
        $\gamma(i+1)$ intersect in a standard product region, by
        definition of $\mprg_\Upsilon$.
        Further, $\ric_\Upsilon$ corresponds to maximal standard product
        regions of $\Sigma_\Upsilon$ containing the vertex 1.
        Choose vertices $a$ in the maximal standard product region
        $\gamma(0)$ and $b$ in the maximal standard product region
        $\gamma(L)$. 
        Let $\gamma'_0$ be a path in $\Sigma_\Upsilon$ starting at
        $a$, contained in the maximal standard product region $\gamma(0)$,
        and ending in $\gamma(0)\cap\gamma(1)$.
        For $i+1\in (0,L)$, let $\gamma'_{i+1}$ be a path in
        $\Sigma_\Upsilon$ that starts where $\gamma'_i$ ended, is
        contained in $\gamma(i+1)$, and ends in $\gamma(i+1)\cap\gamma(i+2)$.
        Let $\gamma'_L$ be a path in $\Sigma_\Upsilon$ contained in
        $\gamma(L)$ that starts at
        the end of $\gamma'_{L-1}$ and ends at $b$.
        Then the concatenation $\gamma'$ of the $\gamma'_i$ is a path
        from $a$ to $b$ in $\Sigma_\Upsilon$ composed of subsegments
        that are contained in product regions corresponding to
        successive vertices of $\gamma$.

 The next results \fullref{1bottleneck} and
 \fullref{cor:separatinggeodesics} are implicit in \cite{Oh22}, but we
 need more precise statements than what appear there explicitly. 
 \begin{lemma}[{cf.\ \cite[Section~4.1]{Oh22},\cite[Theorem~2.35]{Edl24thesis}}]\label{1bottleneck}
   Let $\Delta$ be a connected, triangle-free, incomplete graph without a cut
   vertex, and that is not a join.
   Let $\Salvetti_\Delta$ be the universal cover of the Salvetti
   complex of the RAAG $A_\Delta$, and let 
   $\ric_\Delta=A_\Delta\backslash\mprg_\Delta$ be the fundamental
   domain of its action on its maximal standard product region graph as in \fullref{def:ric}.
   Let $\wall_a$ be the wall in
   $\Salvetti_\Delta$ dual to the edge $1\edge a$ for $a\in\Delta$.
   Let $\sigma(\wall_a)\in\ric_\Delta\cap a\Act\ric_\Delta\subset\mprg_\Delta$ be the vertex corresponding to
   the maximal standard product region
   $\Sigma_{\lk(a)}\times\Sigma_{\{b\in\Delta\mid
     \lk(a)\subset\lk(b)\}}$.
   If $g$ and $h$ are separated by $\wall_a$ in $\Salvetti_\Delta$
   then $g\Act\ric_\Delta$ and $h\Act\ric_\Delta$ are separated by $\st(\sigma(\wall_a))$
   in $\mprg_\Delta$.
    \end{lemma}
   \begin{proof}
     There is a unique maximal join subgraph $\lk(a)\join\{b\in\Delta\mid
     \lk(a)\subset\lk(b)\}$  of
     $\Delta$ containing $\st(a)$, so $\sigma(\wall_a)\in\ric_\Delta$ is
     well-defined.
     Furthermore, since $a\in\st(a)$, this vertex is $a$--invariant,
     so $\sigma(\wall_a)\in\ric_\Delta\cap a\Act\ric_\Delta$.
     
     The set of edges dual to the wall $\wall_a$ is $A_{\lk(a)}
     (1\edge a)$.
     For $\ell\in A_{\lk(a)}$, any standard product region containing an edge $\ell(1\edge a)$
     contains a square $\ell(1\edge a \times 1\edge b)$ for some
     $b\in\lk(a)$, thus contains the flat $\ell\Sigma_{\{a,b\}}\subset
     \Sigma_{\st(a)}$.
     Thus, any maximal standard product region distinct from
     $\sigma(\wall_a)$ that contains an edge
     dual to $\wall_a$ is adjacent to
     $\sigma(\wall_a)$ in $\mprg_\Delta$.

A path $\gamma$ in $\mprg_\Delta$ from $g\Act\ric_\Delta$ to $h\Act\ric_\Delta$ can be shadowed by a
path $\gamma'$ in $\Salvetti_\Delta$ from $g$ to $h$, which must cross
$\wall_a$, by hypothesis, so $\gamma$ contains a vertex corresponding
to a maximal standard product region that contains an edge dual to $\wall_a$.
These are all in $\st(\sigma(\wall_a))$, so
$\st(\sigma(\wall_a))$ separates $g\Act\ric_\Delta$ from $h\Act\ric_\Delta$ in $\mprg_\Delta$.
   \end{proof}

As a corollary we make precise the way in which $\mprg_\Delta$ is a
quasitree. If it were actually a tree we would have the statement that
for every combinatorial geodesic $\gamma\from [0,L]\to\mprg_\Delta$
and every $t\in [0,L]$, the vertices $\gamma(0)$ and $\gamma(L)$ are
not in a common component of $\mprg_\Delta\setminus\{\gamma(t)\}$:
either $\gamma(t)$ coincides with one of the other two or it separates
them. The actual situation is weaker in two ways:
\begin{itemize}
\item We will need stars of vertices to separate $\mprg_\Delta$, not just single vertices.
\item $\ric_\Delta\subset\mprg_\Delta$ could be highly connected, so
  if $\gamma\cap\st(\gamma(t))\subset g\Act\ric_\Delta$ then it might
  be possible to detour around $\st(\gamma(t))$ in $g\Act\ric_\Delta$.
    It will turn out that the bottlenecks appear at transition points
    between different translates of $\ric_\Delta$, so instead of
   looking at $\gamma(t)$ we should shift to such a transition point
   on $\gamma$ near to $\gamma(t)$, where $\text{`near'}\approx \diam(\ric_\Delta)$.
\end{itemize}
 \begin{corollary}\label{cor:separatinggeodesics}
  With notation as in \fullref{1bottleneck}, 
  let $\gamma\from [0,L]\to \mprg_\Delta$ be a 
  combinatorial geodesic. 
For every $t\in [0,L]$ there is a vertex $v\in\mprg_\Delta$ such that
$\gamma(0)$ and $\gamma(L)$ are not contained in the same connected
component of $\mprg_\Delta\setminus\st(v)$, and such that $d_{\mprg_\Delta}(v,\gamma(t))\leq \diam(\ric_\Delta)+2$.
\end{corollary}
\begin{proof}
  We may assume that no translate of $\ric_\Delta$ contains both
  $\gamma(0)$ and $\gamma(L)$, since otherwise choosing $v:=\gamma(0)$
  satisfies the corollary vacuously. 

 Given $t$, by translating by the $A_\Delta$ action, if necessary, we may assume
$\gamma(t)\in \ric_\Delta$.
Take any $g,h\in A_\Delta$ with 
$\gamma(0)\in g\Act\ric_\Delta$ and $\gamma(L)\in h\Act\ric_\Delta$.
By assumption, $g\neq h$.
The 1--skeleton of a CAT(0) cube complex is a median graph, so there
exists a unique vertex $m$ of $\Salvetti_\Delta$ that is the median of
$\{1,g,h\}$.
This implies that, with respect to $m$, every
wall $\wall$ has a `majority side' containing $m$ and a complementary `minority
side'  containing at most one of  $1$, $g$, and $h$, counted with multiplicity.

We will use the fact that the map $\sigma$ can be extended $A_\Delta$--equivariantly to all walls of
$\Salvetti_\Delta$.
Furthermore, for distinct $b,c\in A_\Delta$ let $\wall_{b,c}$ be any
wall dual to the first edge of any geodesic from $b$
to $c$ in $\Salvetti_\Delta$.
Then  $\sigma(\wall_{b,c})\in
b\Act\ric_\Delta$, which follows from equivariance and the fact that
$\sigma(\wall_a)\in\ric_\Delta\cap a\Act\ric_\Delta$ for all $a\in \Delta$.

We estimate that
$d_{\mprg_\Delta}(\gamma(t),m\Act\ric_\Delta)\leq 2$.
If $m=1$ then the distance is 0, so the estimate holds.
Otherwise, let $\wall_{m,1}$ be a wall dual to the first edge on a
geodesic from $m$ to $1$ in $\Sigma_\Delta$.
We have $\sigma(\wall_{m,1})\in m\Act\ric_\Delta$.
Since $1$ is on the minority side of $\wall_{m,1}$, vertices $g$ and $h$ are on
the majority side, 
so $\wall_{m,1}$ separates $1$ from $\{g,h,m\}$ in
$\Salvetti_\Delta$.
By \fullref{1bottleneck}, $\st(\sigma(\wall_{m,1}))$ separates
$\ric_\Delta$ from both $g\Act\ric_\Delta$ and $h\Act\ric_\Delta$.
Thus, $t\geq t_0:=\min\{t'\mid \gamma(t')\in \st(\sigma(\wall_{m,1}))
\}$ and $t\leq t_1:=\max\{t'\mid \gamma(t')\in \st(\sigma(\wall_{m,1}))
\}$.
Since $\gamma$ is
geodesic, $t_1-t_0\leq 2$, so 
$d_{\mprg_\Delta}(\gamma(t),\{\gamma(t_0),\gamma(t_1)\})\leq 1$.
This gives $d_{\mprg_\Delta}(\gamma(t),m\Act\ric_\Delta)\leq d_{\mprg_\Delta}(\gamma(t),\sigma(\wall_{m,1}))\leq 2$.

Suppose $m\neq g$ and let $\wall_{m,g}$ be a wall in
$\Sigma_\Delta$ dual to the first edge of some geodesic from $m$ to
$g$.
Then it suffices to take $v:=\sigma(\wall_{m,g})\in m\Act\ric_\Delta$,
as follows.
The desired distance bound is satisfied, since:
 \[d_{\mprg_\Delta}(v,\gamma(t))\leq \diam(m\Act\ric_\Delta)+d(m\Act\ric_\Delta,\gamma(t))\leq\diam(\ric_\Delta)+2\]

By definition, $g$ is on the minority side of $\wall_{m,g}$, so $\wall_{m,g}$
separates $g$ from $h$ in
$\Salvetti_\Delta$, which, by \fullref{1bottleneck}, implies  $\st(\sigma(\wall_{m,g}))$ separates
$g\Act\ric_\Delta$ from $h\Act\ric_\Delta$.
Since $\gamma(0)\in g\Act\ric_\Delta$ and $\gamma(L)\in
h\Act\ric_\Delta$, vertices $\gamma(0)$ and $\gamma(L)$ are not
contained in a common component of $\mprg_\Delta\setminus\st(v)$.

If $m=g\neq h$ apply the same argument for
$v:=\sigma(\wall_{m,h})$.
\end{proof}
\begin{lemma}\label{ricraagdiameter}
With notation as in \fullref{1bottleneck},   $\diam(\ric_\Delta)\leq\diam(\Delta)+2$.
\end{lemma}
\begin{proof}
  Pick vertices $r$ and $s$ in $\ric_\Delta$ with $d(r,s)=\diam(\ric_\Delta)$.
  They correspond to maximal joins $J_r$ and $J_s$ of $\Delta$.
  Take a shortest geodesic $\gamma\from [0,L] \to\Delta$ from a vertex
  in $J_r$ to a vertex in $J_s$.
  For each integer $i\in [0,L]$ there is a maximal join $J_i$
  containing $\st(\gamma(i))$, and $J_i\cap J_{i+1}$ contains the edge
  $\gamma(i)\edge\gamma(i+1)$, so $J_i$ and $ J_{i+1}$ correspond to
  vertices of
  $\ric_\Delta$ whose distance is at most 1.
  Furthermore, $J_r$ contains $\gamma(0)$, so it has edges in common
  with $J_0$, so $r$ is at distance at most 1 in $\ric_\Delta$ from
  the vertex corresponding to $J_0$.
  Similarly, $s$ is distance at most 1 from the vertex of
  $\ric_\Delta$ corresponding to $J_L$.
  Thus, $\diam(\ric_\Delta)\leq L+2\leq \diam(\Delta)+2$.
\end{proof}
\begin{corollary}\label{no1bottleneck}
  If $\mprg$ is a graph containing sequences of vertices $(x_i)$ and $(y_i)$ such
  that $d(x_i,y_i)\stackrel{i\to\infty}{\longrightarrow}\infty$ and for all sufficiently large $i$ there
  does not exist a vertex $v$ at distance at least 3 from each of $x_i$
  and $y_i$ whose star separates $x_i$ and $y_i$, then $\mprg$ is not the
  MPRG of an irreducible, 1--ended, 2--dimensional RAAG. 
\end{corollary}
\begin{proof}
Suppose $\Delta$ is a finite, connected, triangle-free graph and
$x$ and $y$ are vertices in $\mprg_\Delta$ that are not separated by the star of any vertex
that is not in the 2--neighborhood of one of them.
By \fullref{cor:separatinggeodesics}, if $d(x,y)$ is large enough then there is an approximate midpoint $m$ of $x$
and $y$ whose star separates $x$ from $y$, and together with
\fullref{ricraagdiameter} it follows that 
 $d(m,\{x,y\})\geq
 d(x,y)/2-\diam(\Delta)-4$.
 But $d(m,\{x,y\})\leq 2$ by hypothesis, so $d(x,y)\leq 2\diam(\Delta)+12$.
 For any fixed $\Delta$, eventually
 $d(x_i,y_i)\stackrel{i\to\infty}{\longrightarrow}\infty$ exceeds this
 bound, so $\mprg\not\cong\mprg_\Delta$.
\end{proof}

Oh \cite[Lemma~4.13]{Oh22} characterizes cut vertices of the MPRG
    of a RAAG $A_\Delta$:
   a cut vertex of $\mprg_\Delta$ contained in
    $\ric_\Delta$ is either a cut vertex of $\ric_\Delta$ or is fixed
    by an element of $A_\Delta$ that does not fix any of its
    neighbors in $\ric_\Delta$. 
    \fullref{cut_vertex_in_mprg} is the analogous result for a RACG $W_\Gamma$.
    It is more complicated because $W_\Gamma\act\mprg_\Gamma$ is more
    complicated than $A_\Delta\act\mprg_\Delta$.
   Specifically, nonadjacent vertices in $\ric_\Gamma$ can have
    common elements in their stabilizers, and $\Gamma$ may contain edges that do not belong to any
    maximal thick join. Both of these phenomenon give rise to loops in $\mprg_\Gamma$ that are
    not visible in $\ric_\Gamma$. This observation will be key in \fullref{sec:ladders}.

      \begin{lemma}\label{cut_vertex_in_mprg}
        Let $\Gamma$ be a triangle-free strongly CFS graph that is not
        a join.
       One of the following conditions hold if and only if
       $v\in\ric_\Gamma$ is a cut vertex of $\mprg_\Gamma$. 
       \begin{enumerate}
         \item There are induced subgraphs $R_0$ and $R_1$ of
           $\ric_\Gamma$ properly containing $\{v\}$ with $R_0\cap R_1=\{v\}$ and
           $\mathrm{Edges}(\ric_\Gamma)=\mathrm{Edges}(R_0)\cup\mathrm{Edges}(R_1)$,
           and such that for $i\in\{0,1\}$ and $\Gamma_i$ defined to be the subgraph of $\Gamma$
           spanned by $\bigcup_{u\in R_i}J_u$, we have
           $\Gamma_0\cap\Gamma_1=J_v$ and 
           $\mathrm{Edges}(\Gamma)=\mathrm{Edges}(\Gamma_0)\cup\mathrm{Edges}(\Gamma_1)$.\label{item:ric_not_in_one_component}
        \item There is a vertex $s$ of $\Gamma$ such that $\st(s)\subset J_v$ and
          $s$  is not contained in any
          other maximal thick join.\label{item:ric_in_one_component}
        \end{enumerate}
      \end{lemma}
      \begin{proof}
               Since $\Gamma$ is a triangle-free strongly CFS graph
               that is not a join,  $\ric_\Gamma$ and $\mprg_\Gamma$ are connected and are not
        single vertices, by \fullref{stronglycfsimpliesmprgconnected}.

        Suppose $v\in\ric_\Gamma$ is a cut vertex of $\mprg_\Gamma$.
        If $v$ separates $\ric_\Gamma$ in $\mprg_\Gamma$
        then choose a complementary component and take  $R_0$ to be the subgraph of $\ric_\Gamma$ spanned by the
        union of $\{v\}$ and vertices of $\ric_\Gamma$ in that connected
        component of $\mprg_\Gamma\setminus\{v\}$.
        Let $R_1$ be the subgraph of $\ric_\Gamma$ spanned by $v$ and
        $\ric_\Gamma\setminus R_0$.
        By construction, $\{v\}\subsetneq R_0$, $\{v\}\subsetneq R_1$,
        $R_0\cap R_1=\{v\}$ and $\mathrm{Edges}(\ric_\Gamma)=\mathrm{Edges}(R_0)\cup\mathrm{Edges}(R_1)$.

        There are two ways for \eqref{item:ric_not_in_one_component}
        to fail, and we show that from either of them we can produce a contradictory
        path that connects $R_0$ to $R_1$ in $\mprg_\Gamma$ while
        avoiding $v$.
        We conclude that $v$ separating
$\ric_\Gamma$ in $\mprg_\Gamma$ implies
\eqref{item:ric_not_in_one_component}.

The first way for \eqref{item:ric_not_in_one_component} to fail is if
there exists
         $s\in \Gamma_0\cap\Gamma_1\setminus J_v$. Then $s\Act v\neq v$
        but $s$ fixes vertices in $R_0\setminus\{v\}$ and
        $R_1\setminus\{v\}$.
        Let $\gamma$ be a shortest path in $\ric_\Gamma$ from the
        fixed set of $s$ in $R_0$ to the fixed set of $s$ in $R_1$.
        Since $v$ separates $R_0$ from
        $R_1$ in $\mprg_\Gamma$, $\gamma$ goes through
        $v$.
        The path $s\Act\gamma$ does not go through $v$, since $s\Act
        v\neq v$, and it has the same endpoints as $\gamma$, so it connects $R_0$ to $R_1$ in
        $\mprg_\Gamma$ and avoids $v$.

        The other way for \eqref{item:ric_not_in_one_component} to
        fail is if there is an edge of $\Gamma$ from a vertex $s_0\in
        \Gamma_0\setminus \Gamma_1$ to a vertex
        $s_1\in\Gamma_1\setminus\Gamma_0$.
        By construction, this means there are vertices $u_i\in
        R_i\setminus \{v\}$ such that $s_i\in J_{u_i}\setminus J_v$.
        Furthermore, $s_0,s_1\notin J_v$ implies $W_{\{s_0,s_1\}}\cap
        W_{J_v}=\{1\}$. 
        Let $\gamma$ be a minimal length path in $\ric_\Gamma$
        from $u_0$ to $u_1$.
        The concatenation of  $s_0\Act\gamma$, $s_1s_0\Act\gamma=s_0s_1\Act\gamma$, and
        $s_1\Act\gamma$ connects $R_0$ and $R_1$ in $\mprg_\Gamma$ and
        avoids $v$.

        Suppose $v$ does not separate $\ric_\Gamma$ in
        $\mprg_\Gamma$.
       Then there is some
        translate $w\Act\ric_\Gamma$ such that $\ric_\Gamma\cap
        w\Act\ric_\Gamma=\{v\}$ and $v$ separates
        $\ric_\Gamma$ from $w\Act\ric_\Gamma$
        in $\mprg_\Gamma$.
        We induct on the word length of $w$ after considering what
        happens for generators. 
        Consider $s\in J_v$ not satisfying
        \eqref{item:ric_in_one_component}, so either $s$ fixes a
        vertex $u_s\in \ric_\Gamma\setminus\{v\}$, or $s$ only fixes
        $v$ but is adjacent in
        $\Gamma$ to a vertex $t$ that fixes $u_t\in
        \ric_\Gamma\setminus\{v\}$ but does not fix $v$.
        We claim in both cases that $v$ does not separate
        $\ric_\Gamma$ and $s\Act\ric_\Gamma$
        in $\mprg_\Gamma$.
        In the first case $v\neq u_s\in\ric_\Gamma\cap s\Act\ric_\Gamma$.
        In the second case take a shortest path $\gamma$ in $\ric_\Gamma$ from
        $v$ to $u_t$.
        Then $t\Act\gamma$ contains $u_t$ and $t\Act v$ but not $v$, and
        $st\Act\gamma$ contains $s\Act u_t\in s\Act\ric_\Gamma\setminus\{v\}$ and
        $st\Act v=ts\Act v=t\Act v$ but not $s\Act v=v$.
        Thus, there is a path in $\mprg_\Gamma$ from
        $s\Act\ric_\Gamma\setminus\{v\}$ to $\ric_\Gamma\setminus\{v\}$
        that avoids $\{v\}$.
        Now write $w$ as a minimal length word $s_1\cdots s_n$ for
        $s_i\in J_v$.
        If every $s_i\in J_v$ fails to satisfy
        \eqref{item:ric_in_one_component} then every pair
        $\ric_\Gamma$ and $s_i\Act\ric_\Gamma$
        are not separated by $v$ in $\mprg_\Gamma$.
        But then $s_1\cdots s_i\Act\ric_\Gamma$ is not separated from
        $s_1\cdots s_is_{i+1}\Act\ric_\Gamma$ by $v$ in $\mprg_\Gamma$, so all of $\ric_\Gamma\setminus\{v\}$,
        $s_1\Act\ric_\Gamma\setminus\{v\}$,\dots,$s_1\cdots
        s_n\Act\ric_\Gamma\setminus\{v\}=w\Act\ric_\Gamma\setminus\{v\}$ are
        in the same component of $\mprg_\Gamma\setminus\{v\}$,
        contradicting the choice of $w$.
        Thus, there is some $s\in J_v$ satisfying
        \eqref{item:ric_in_one_component}.

        In the other direction we suppose
        \eqref{item:ric_not_in_one_component} or
        \eqref{item:ric_in_one_component} and produce a cut
        vertex  $v$  of $\mprg_\Gamma$.

        First suppose  \eqref{item:ric_not_in_one_component}.
        Consider the splitting
        $W_\Gamma=W_{\Gamma_0}*_{W_{J_v}}W_{\Gamma_1}$.
        Let $T$ be its Bass-Serre tree.
        Edges of $T$ belong to a single orbit and they are in bijection
        with  cosets of $W_{J_v}$.
        There are two orbits of vertices in $T$ corresponding to
        cosets of $W_{\Gamma_0}$ and $W_{\Gamma_1}$.
        We define a
        $W_\Gamma$--equivariant map $\phi\from
       W_\Gamma\times\ric_\Gamma\to
       \mathrm{Vertices}(T)\cup\mathrm{Edges}(T)$  and then check it
       actually defines a map on $\mprg_\Gamma=W_\Gamma\Act\ric_\Gamma$. 
        \[\phi(w\Act u):=
          \begin{cases}
            \text{the vertex } wW_{\Gamma_0}&\text{if } u\in R_0\setminus\{v\}\\
            \text{the vertex } wW_{\Gamma_1}&\text{if } u\in
            R_1\setminus \{v\}\\
            \text{the edge } wW_{J_v}&\text{if
            }u=v
          \end{cases}
        \]
        Suppose $u\in R_0\setminus \{v\}$ and $w'\!\Act u=w\Act u$.
        Then $w'\in
        wW_{J_u}\subset wW_{\Gamma_0}$, so
        $\phi(w\Act u)=wW_{\Gamma_0}=w'W_{\Gamma_0}=\phi(w'\!\Act u)$.
        Similar arguments hold for $R_1\setminus \{v\}$ and $v$, so
        $\phi$ is well-defined on vertices of $\mprg_\Gamma$.

        Every edge in $\mprg_\Gamma$ is a translate of one in
        $\ric_\Gamma$ (recall \fullref{weakconvexity}), so it can be written $w\Act u_0\edge w\Act u_1$ where
        $u_0\edge u_1\subset \ric_\Gamma$. 
        If $u_0,u_1$ are both in $R_0\setminus\{v\}$ or both in
        $R_1\setminus\{v\}$ then they map to the same vertex of $T$. 
       If $u\in R_i$ is adjacent to $v$ then
       $\phi(v)=1W_{J_v}$ is an edge incident to the
       vertex $\phi(u)=1W_{\Gamma_i}$.
       By hypothesis, there are no edges between a vertex of
       $R_0\setminus\{v\}$ and a vertex of $R_1\setminus\{v\}$.
       Conclude that edge paths in $\mprg_\Gamma$ are sent by $\phi$ to
       connected subsets of $T$.
        
        Consider the edge of $T$ corresponding to the coset
        $1W_{\Gamma_0\cap\Gamma_1}$.
        Its two vertices are $1W_{\Gamma_0}$ and $1W_{\Gamma_1}$.
        Consider any path in $\mprg_\Gamma$ from a vertex of
        $R_0\setminus\{v\}$ to a
        vertex of
        $R_1\setminus\{v\}$.
        Its $\phi$--image is a connected set in $T$ that contains
        $1W_{\Gamma_0}$ and $1W_{\Gamma_1}$, so  it contains a vertex
        $w\Act u$ such that $\phi(w\Act u)$ is the edge $1W_{J_v}$.
        The definition of $\phi$ requires  $u=v$ and
        $wW_{J_v}=1W_{J_v}$, so
        $w\in W_{J_v}$, which  gives $w\Act v=v$.
        Thus, every path in $\mprg_\Gamma$ from
        $R_0\setminus\{v\}$ to $R_1\setminus\{v\}$ passes through $v$,
        so $v$ is a cut vertex of $\mprg_\Gamma$.

        Now suppose  \eqref{item:ric_in_one_component}.
        Let $\gamma\from [0,L]\to\mprg_\Gamma$ be a combinatorial path from 
        $\ric_\Gamma\setminus\{v\}$ to 
        $s\Act \ric_\Gamma\setminus\{v\}$, and let $\gamma'$ be a path
        from 1 to $s$ in $\Davis_\Gamma$ shadowing it.
        Since $\gamma'$ starts and ends on opposite sides of the wall $\wall_s$
        dual to the edge $1\edge s$, it contains some edge crossing
        $\wall_s$.
        The condition of \eqref{item:ric_in_one_component} saying $s$
        is not contained in any other maximal thick join implies that
        $\Davis_{J_v}$ is the only maximal standard product subcomplex
        containing the edge $1\edge s$. 
        The condition that $\lk(s)\subset J_v$ implies that every
        square containing the edge $1\edge s$ is contained in
        $\Davis_{J_v}$.
        It follows that every edge crossing $\wall_s$ is contained in
        $\Davis_{J_v}$ and not in any other maximal standard product
        subcomplex, so $\gamma'$ can only cross $\wall_s$ if $v\in\gamma$. Thus, $v$ is a cut vertex of $\mprg_\Gamma$.
      \end{proof}

\subsection{Relation to hierarchical hyperbolic structures}\label{sec:hhs}
This section is predominantly to relate the maximal product region
graph to other results in the literature, but, having done this, 
\fullref{cor:stability_recognizing} gives us an easy way to guarantee
that the orbit map gives a quasiisometric embedding of certain
subgroups of $W_\Gamma$ into $\mprg_\Gamma$, which will be useful in
the next subsection. 

The definition of hierarchically hyperbolic spaces and groups (HHS/HHG) is
complicated, and we will not repeat it; see \cite{BehHagSis17,
  BehHagSis17ad,MR3956144,BehHagSis21}.
There is a hyperbolic graph coming from the hierarchical
hyperbolic structure, and quasiisometries between HHSs induce
quasiisometries of these hyperbolic graphs.
We will show that if $\Gamma$ is strongly CFS then $\mprg_\Gamma$ is
$W_\Gamma$--equivariantly quasiisometric to the HHS graph for $W_\Gamma$.
For RAAGs the HHS graph is a quasitree and the MPRG is a quasitree
with bottleneck constant 1. 
In the next subsection we construct wide ladders in the MPRGs of
certain RACGs.
This allows the possibility that they are still quasitrees, but the bottleneck
constant is at least half the width of the ladder, so must be larger than 1.
To conclude that these groups are not quasiisometric to a RAAG we really
need the finer control that quasiisometries induce \emph{isomorphisms}
of MPRGs, not just that they induce quasiisometries between the HHS graphs.

For RAAGs the relevant hyperbolic graph from the standard HHS
structure is just the contact graph.
For RACGs some modifications must be made. 
Abbott, Behrstock, and Durham \cite[Theorems~A,B]{AbbBehDur21} show that
a HHG $G$ admits an action on a hyperbolic
space $ABD(G)$ with the following properties.
The structure of $ABD(W_\Gamma)$ will be
described in the course of the proof of \fullref{mprgcshhg}. 
\begin{itemize}
\item $G\act ABD(G)$ is a largest acylindrical action.
  \item In certain cases, including when $G$ is RACG, $G\act ABD(G)$
    is universal, in the sense that every
    generalized loxodromic element of $G$ acts loxodromically in this particular action.
    \item $ABD(G)$ is a stability recognizing space:  a finitely generated subgroup $H$ of $G$ is stable if and
only if any orbit map of $H$ into $ABD(G)$ is a quasiisometric
embedding. 
\end{itemize}

\begin{theorem}\label{mprgcshhg}
  Let $\Gamma$ be a triangle-free, strongly CFS graph.
  Then $\mprg_\Gamma$ is $W_\Gamma$--equivariantly
  quasiisometric to $ABD(W_\Gamma)$.
\end{theorem}
\begin{proof}
  Let $\X_0:=\Davis_\Gamma$, $\X_1:=ABD(W_\Gamma)$, $\X_2$ be the graph
  obtained from $\X_0$ by coning off each maximal standard product
  region, and $\X_3:=\mprg_\Gamma$.
  We will construct $W_\Gamma$--equivariant quasiisometries:
  \[\X_3\stackrel{\phi}{\to}\X_2\stackrel{\iota}{\to}\X_1\]
  
 The standard HHS structure on $W_\Gamma$ is $(\X_0,
 \mathfrak{S})$, where $\mathfrak{S}$ is the projection closure
 of hyperplane carriers, and for $S\in\mathfrak{S}$, the hyperbolic
 space $CS$ associated to $S$ is its contact graph.
 To construct $ABD(G)$, Abbott, Behrstock, and Durham modify this HHS
 structure \cite[Theorem~3.7]{AbbBehDur21} by replacing the top level
 hyperbolic space as follows.
 Start with $\X_0$.
 For $S\in \mathfrak{S}$ not the maximal element, if there exists $T,U\in\mathfrak{S}$
 with $S\subset T$ and $T\perp U$ with both $CT$ and $CU$ of infinite
 diameter, then cone off $S$ by adding a new vertex $c_S$ attached to
 each vertex of $S$.
 The resulting space is $\X_1:=ABD(W_\Gamma)$.

The inclusion map $\iota\from\X_2\to\X_1$ is 
$W_\Gamma$--equivariant and Lipschitz.
The idea for showing it is a quasiisometry is that while there may be more cone vertices in $\X_1$, the
extra ones are coning off subsets of $\X_0$  that were already coned off in $\X_2$,
so they are not making much difference. 
To see this precisely, we will define a coarse inverse $\bar\iota\from
\X_1\to\X_2$ to be the
identity on $\X_0$ and extend it to $\X_1\setminus\X_0$.

First we characterize $S$ such that there is a cone vertex $c_S\in\X_1\setminus\X_0$.
Such a cone vertex comes from a convex subcomplex $S$ of some hyperplane carrier, for a
hyperplane dual to edges labelled by some $s\in \Gamma$.
Up to the $W_\Gamma$--action, we may assume $1\in S\subset \Davis_{\st(s)}$
Any $T$ containing $S$ is a convex subcomplex of the full hyperplane carrier  $\Davis_{\lk(s)}\times\Davis_{s}=\Davis_{\st(s)}$ with  $T\subset \Davis_{\tau\cup\{s\}}\subset\Davis_{\st(s)}$, where
$\tau\cup\{s\}$ is the set of edge labels that occur in $T$.
Now, $\Davis_{\tau\cup\{s\}}$ has unbounded associated hyperbolic space
when $W_\tau$ is infinite and not a product, which, since $\Gamma$ is
triangle-free, is simply the case that $\tau$ has at least two
vertices. 
In the other direction, if $U\perp T$ then $U\perp
\Davis_{\tau\cup\{s\}}$, but, by triangle-freeness, the biggest subcomplex perpendicular to
$\Davis_{\tau\cup\{s\}}$ is $\Davis_\upsilon$, where $\upsilon$ is the
set of common neighbors of $\tau$. 
This has infinite associated hyperbolic space when $\upsilon$ has more
than one vertex.
Thus, we have a cone vertex $c_S\in\X_1$ when there exists a subset of
$\lk(s)$ that contains the non-$s$ edge labels in $S$, has at least 2
vertices, and has at least one common neighbor other than $s$.
This is equivalent to saying that $S$ is contained in a standard product region.

For $c_S\in\X_1\setminus\X_0$ define $\bar\iota(c_S)$ to be the set of cone vertices $c_{S'}$ of $\X_2$ such
that $S'$ is a maximal standard product region containing $S$.
Then $\bar\iota\circ\iota=Id_{\X_2}$, and we have
$\emptyset\neq \lk(c_S)\subset \bigcap_{c_{S'}\in \iota\circ\bar\iota(c_S)} \lk(c_{S'})$, which implies that  $\iota\circ\bar\iota$
is at distance at most 2 from $Id_{\X_1}$ and that distances between points
in $\X_0$ are the same in $\X_1$ as in $\X_2$.
Now apply \fullref{inversecoarselipimpliesqi} to see that $\iota$ and
$\bar\iota$ are inverse quasiisometries. 

Define:
\begin{align*}
  \phi&\from\X_3\to\X_2\from S\mapsto c_S\\
  \bar\phi&\from \X_2\to\X_3:
  \begin{cases}
    c_S\mapsto S&\text{ for }c_S\in \X_2-\X_0\\
    x\mapsto \{ S\mid x\in S\}&\text{ for } x\in\X_0
  \end{cases}
\end{align*}
By construction ,  $\phi$ is $W_\Gamma$--equivariant. 
If distinct maximal standard product regions $S_1$ and $S_2$ intersect in a
standard flat $S_0$ then in $\X_1$ there are cone vertices $c_{S_i}$
with $\emptyset\neq\lk(c_{S_0})\subset\lk(c_{S_1})\cap\lk(c_{S_2})$,
so $d_{\X_2}(\phi(S_1),\phi(S_2))=d_{\X_2}(c_{S_1},c_{S_2})=2$.
Since $\X_3$ is connected, by
\fullref{stronglycfsimpliesmprgconnected}, this implies $\phi$ is 2--Lipschitz.

The maximal standard product regions containing 1 correspond to
maximal nontrivial joins in $\Gamma$.
There are finitely many of these, and there do exist some, since the
CFS property implies $\Gamma$ contains a square.
Thus, $\bar\phi(1)$ is a non-empty set that is finite.
Since $\X_3$ is connected, finite sets have finite diameter.
By definition, $\bar\phi$ is $W_\Gamma$--equivariant, so $\bar\phi$ is
a well defined map from $\X_2$ to nonempty subsets of $\X_3$ of
uniformly bounded diameter.
When $x\in\X_0$ and $c_S\in\X_2\setminus\X_0$ is an adjacent cone
vertex then $x\in S$, so $\bar\phi(c_S)\in \bar\phi(x)$. 
Similarly, the CFS property implies that for adjacent vertices in
$\X_0$, say across an edge labelled $s$, there is a standard 2--flat
containing both, since the vertex $s\in\Gamma$ is in the support of a
square of $\Gamma$.
Thus, adjacent vertices in $\X_0\subset\X_2$ have intersecting
$\bar\phi$--images.
It follows that $\bar\phi$ is $(\diam\bar\phi(1),\diam\bar\phi(1))$--coarsely Lipschitz.
Clearly, $\bar\phi\circ\phi=Id_{\X_3}$.
The map $\phi\circ\bar\phi$ agrees with $Id_{\X_2}$ on $\X_2\setminus\X_0$ and
sends $x\in\X_0$ to the set of its cone vertex neighbors,
which are all adjacent to $x$.
Apply \fullref{inversecoarselipimpliesqi}.
\end{proof}

\begin{corollary}\label{cor:stability_recognizing}
   Let $\Gamma$ be triangle-free and  strongly CFS.
   Then $\mprg_\Gamma$ is a hyperbolic graph, and it is a stability
   recognizing space for $W_\Gamma$. 
 \end{corollary}

 \begin{corollary}\label{cor:stable_subgroup_implies_mprg_not_quasitree}
      Let $\Gamma$ be triangle-free and  strongly CFS.
If $W_\Gamma$ contains a 1--ended stable subgroup
   then $\mprg_\Gamma$ is not a quasitree. 
 \end{corollary}

\subsection{Ladders}\label{sec:ladders}
In this subsection we introduce an obstruction to a graph being the
MPRG of a RAAG by the presence of a `ladder', which will be a 2--ended
subgraph whose ends are not separated by stars of vertices, allowing
us to apply \fullref{no1bottleneck}.
While this formulation is conceptually clear, it is a large-scale
geometric condition in a locally infinite graph, so it is not
tangible.
We therefore develop explicit, computer verifiable conditions, described purely
in terms of the presentation graph $\Gamma$, that are sufficient to
imply the geometric condition.

The idea for building a ladder is to find a highly connected subset $\quadrat$ of the
fundamental domain
$\ric_\Gamma=W_\Gamma\backslash\mprg_\Gamma$ and a pair of generators $r$ and $s$ of $W_\Gamma$
such that $s\Act\quadrat\cap\quadrat$ and $r\Act\quadrat\cap\quadrat$
are large enough.
Then the `ladder' will be $\langle r,s\rangle\Act\quadrat$.
The precise conditions for $r$, $s$, and $\quadrat$ appear in \fullref{thm:ladders}.

\begin{lemma}\label{lem:key-lem}
  Suppose $\mprg$ is a graph that contains a sequence of connected
  subgraphs $\Theta_i\subset \mprg$ such that:
  \begin{itemize}
  \item $\diam_\mprg(\Theta_i)\geq i$
    \item For all $i$ and all $v\in\mprg$, $\Theta_i\setminus\st(v)$ has exactly one non-singleton component. 
  \end{itemize}  
Then
  $\mprg$ is not the MPRG of an irreducible, one-ended,
  2--dimensional  RAAG.
\end{lemma}
\begin{proof}
  Pick $x_i,y_i\in\Theta_i$ realizing the diameter.
  Since $\Theta_i$ is connected, $x_i$ has neighbors, and if $d(x_i,v)>2$ then no neighbor of $x_i$ is in $\st(v)$, so $x_i$
  is contained in the unique non-singleton component of
  $\Theta_i\setminus\st(v)$.
  The same is true for $y_i$, so $x_i$ and $y_i$ are not separated by
  $\st(v)$ except possibly if $d(v,\{x_i,y_i\})\leq 2$.
  Apply \fullref{no1bottleneck}.
\end{proof}

Our prototypical example for \fullref{lem:key-lem} is to construct an
infinite graph $\Theta$  by taking a square $\quadrat$ of side length at least
three contained in $\ric_\Gamma=W_\Gamma\backslash\mprg_\Gamma$, such that there are nonadjacent $r$ and $s$ in $\Gamma$
that act as reflections fixing opposite sides of $\quadrat$.
The $\Theta_i$ of \fullref{lem:key-lem} are an increasing nested
sequence of 
consecutive translates of $\quadrat$, with $\Theta$ as their infinite
union. See \fullref{fig:protoladder}.
\begin{figure}[h]
  \centering
  \begin{tikzpicture}[scale=.5]\tiny
    \draw  (-2.5,.5)--(3.5,.5) (-2.5,-.5)--(3.5,-.5);
    \begin{scope}\foreach \i in {-2,...,3}{
     \draw (\i,-.5)--(\i,.5);
   }\end{scope}
 \draw[thick, blue] (0,-.5)--(1,-.5)node [midway,above]
 {$\quadrat$}--(1,.5)--(0,.5)--cycle;
 \draw[<->] (-.2,.6) .. controls (0,.7) and (0,.7) .. (.2,.6) node [midway,above]
 {$r$};
  \draw[<->] (.8,.6) .. controls (1,.7) and (1,.7) .. (1.2,.6) node [midway,above]
 {$s$};
  \end{tikzpicture}
  \caption{Prototypical ladder.}
  \label{fig:protoladder}
\end{figure}
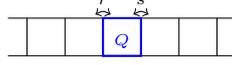

A variant  is shown in \fullref{fig:ladder}.
The variant highlights the reason that in \fullref{lem:key-lem} we say
that the graph minus a star has one non-singleton component, rather
than saying the graph minus a star is connected.
The ladder of  \fullref{fig:ladder} does obstruct $\mprg$ from being the
MPRG of a RAAG, because even though it is separated by the star of a
vertex $v$ (red), its ends are not.
\begin{figure}[h]
    \centering
    \includegraphics{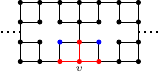}
    \caption{A non-prototypical example of a ladder}
    \label{fig:ladder}
\end{figure}

\begin{definition}[wide ladder]\label{def:wideladder}
    A \textit{wide ladder} $L$ in a graph $\mprg$ is a graph satisfying:
    \begin{enumerate}[(a)]
        \item There is a connected graph $\mathcal{C}$,
          each vertex of which is associated to a connected subgraph
          $\quadrat$ of $\mprg$, 
          and the subgraphs associated with adjacent vertices of $\mathcal C$ intersect. 
           $L$ is a graph of infinite diameter that is the union of these
           subgraphs.\label{item:chunks}    
        \item  The ladder \emph{has wide rungs}, in the sense that if
          $\quadrat$ and $\quadrat'$ are adjacent in $\mathcal{C}$,
          then $\diam_\mprg(\quadrat \cap \quadrat')\geq
          3$.
      \label{item:widerungs}
          \item For all $\quadrat\in\mathcal{C}$ and all vertices
            $v\in\mprg$, there is exactly one non-singleton component of
            $\quadrat\setminus\st(v)$.

            \label{item:chunkshardtoseparate}
           
  \item  There do not exist adjacent  $\quadrat$ and $\quadrat'$
    in $\mathcal{C}$ and $v\in \mprg$ such that every vertex of
    $(\quadrat\cap\quadrat')\setminus\st(v)$ is a singleton component of either
    $\quadrat\setminus\st(v)$ or $\quadrat'\setminus\st(v)$. See \fullref{fig:forbiddenconfiguration}.\label{item:forbiddenconfiguration}
      \end{enumerate}
    \end{definition}

    \begin{figure}[h]
      \centering      
      \makebox[.5\textwidth][c]{%
       \begin{tikzpicture}\tiny
        \coordinate[label={180:$v$}] (0) at (-.5,0);
        \coordinate[label={90:$$}] (1) at (0,1);
        \coordinate (2) at (-1,1);
        \coordinate (3) at (-1,0);
         \coordinate (4) at (-1,-1);
         \coordinate[label={-90:$$}] (5) at (0,-1);
        \coordinate (6) at (0,-.5);
        \coordinate (7) at (0,0);
        \coordinate (8) at (0,.5);
        \coordinate (9) at (1,1);
        \coordinate (10) at (1,0);
         \coordinate (11) at (1,-1);
        \coordinate (13) at (0,-1);
        \coordinate (14) at (0,0);
        \coordinate (15) at (0,1);
         \draw[blue,thick]
         (1)--(2)--(3)--(4)--(5) (3)..controls
         (-.5,-.5) and (-.5,-.5)..(7);
         \draw[red,thick]
         (1)--(9)--(10)--(11)--(5);
         \draw[violet,thick] (5)--(6)--(7)--(8)--(1);
         \draw (8)--(0)--(2) (6)--(0)--(4);
        \filldraw (0) circle (1pt) (1) circle (1pt) (2) circle (1pt)
        (3) circle (1pt) (4) circle (1pt) (5) circle (1pt) (6) circle
        (1pt) (7) circle (1pt) (8) circle (1pt) (9) circle (1pt) (10)
        circle (1pt) (11) circle (1pt) (13) circle
        (1pt) (14) circle (1pt) (15) circle (1pt);
      \end{tikzpicture}
      \hfill
      \begin{tikzpicture}\tiny
        \coordinate[label={45:$v$}] (0) at (0,0);
        \coordinate[label={90:$$}] (1) at (0,1);
        \coordinate (2) at (-.5,.5);
        \coordinate (3) at (-.5,0);
         \coordinate (4) at (-.5,-.5);
         \coordinate[label={-90:$$}] (5) at (0,-1);
         \coordinate (6) at (-.75,-.5);
        \coordinate (7) at (-.75,0);
        \coordinate (8) at (-.75,.5);
        \coordinate (9) at (.5,.5);
        \coordinate (10) at (.5,0);
         \coordinate (11) at (.5,-.5);
        \coordinate (13) at (.75,-.5);
        \coordinate (14) at (.75,0);
        \coordinate (15) at (.75,.5);
         \draw[blue,thick]
         (1)--(2)--(3)--(4)--(5)--(6)--(7)--(8)--(1);
         \draw[red,thick] (1)--(9)--(10)--(11)--(5)--(13)--(14)--(15)--(1);
         \draw (8)--(0)--(2) (11)--(0)--(13);
        \filldraw (0) circle (1pt) (1) circle (1pt) (2) circle (1pt)
        (3) circle (1pt) (4) circle (1pt) (5) circle (1pt) (6) circle
        (1pt) (7) circle (1pt) (8) circle (1pt) (9) circle (1pt) (10)
        circle (1pt) (11) circle (1pt) (13) circle
        (1pt) (14) circle (1pt) (15) circle (1pt);
      \end{tikzpicture}
      }  
      \caption{In each picture $\quadrat$ is blue, $\quadrat'$ is red, any edges in
their intersection are violet. These satisfy  \fullref{def:wideladder}~\ref{item:widerungs} and
        \ref{item:chunkshardtoseparate} but not
        \ref{item:forbiddenconfiguration}, and there is a vertex $v$
        whose star separates their union into multiple non-singleton components.}
      \label{fig:forbiddenconfiguration}
    \end{figure}
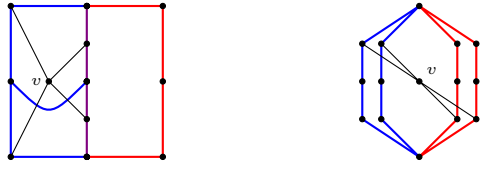

\begin{lemma} \label{lem:wide-ladders}
  A graph containing a wide ladder is not the MPRG of a RAAG. 
   \end{lemma}
   \begin{proof}
     Let $\mprg$ be a graph containing a wide ladder $L$ defined by
     $\mathcal C$ as in \fullref{def:wideladder}.
The goal is to show that $L$ satisfies the hypotheses of
\fullref{lem:key-lem}.
If there is a sequence of vertices $c_i$ of $\mathcal{C}$ such that the
corresponding graphs $\quadrat_{c_i}$ have diameters that grow without
bound then we can apply \fullref{lem:key-lem} with
$\Theta_i=\quadrat_{c_i}$.
Otherwise, we choose nested connected subgraphs $\mathcal{C}_i$ of
$\mathcal{C}$ such that the graphs $\Theta_i:=\bigcup_{c\in
  \mathcal{C}_i}\quadrat_c$ have diameters growing without bound.
This is possible since $L=\bigcup_{c\in
  \mathcal{C}}\quadrat_c$ has infinite diameter. 
We show by induction on the size of the subgraph of $\mathcal{C}$ that
such a $\Theta_i$ satisfies the condition that it has a unique
non-singleton complementary component for every star. 

Suppose $\quadrat$ and $\quadrat'$ are adjacent elements in $\mathcal{C}$
and $v$ is some vertex of $\mprg$.
Condition~\ref{item:forbiddenconfiguration} says for every $v\in
\mprg$ there is a $u\in(\quadrat\cap\quadrat')\setminus\st(v)$ such that $u$ is not a
singleton component of $\quadrat\setminus\st(v)$ and $u$ is not a
singleton component of $\quadrat'\setminus\st(v)$.
Condition~\ref{item:chunkshardtoseparate} says
$\quadrat\setminus\st(v)$ and $\quadrat'\setminus\st(v)$ have unique
non-singleton components $U$ and $U'$, respectively, so $u\in U\cap
U'$.
Thus, $U\cup U'$ is
connected.
Every vertex of $(\quadrat\cup\quadrat')\setminus (U\cup
U'\cup\st(v))$ is a singleton component of
$(\quadrat\cup\quadrat')\setminus\st(v)$, so $U\cup U'$  is the unique
non-singleton component of 
$(\quadrat\cup\quadrat')\setminus\st(v)$.

Now, suppose that every connected subset $\mathcal{A}\subset\mathcal{C}$ of size
at most $n$ has the property that for all $v\in\mprg$,
$\bigcup_{a\in\mathcal{A}}\quadrat_a\setminus\st(v)$ has a unique
non-singleton component containing the non-singleton component of
$\quadrat_a\setminus\st(v)$ for all $a\in\mathcal{A}$. 
Add one more graph  $\quadrat$ such that the corresponding vertex of
$\mathcal{C}$ is adjacent to $a_0\in\mathcal{A}$.
For any $v\in\mprg$, $\quadrat\cup\quadrat_{a_0}\setminus\st(v)$ and
$\bigcup_{a\in\mathcal{A}}\quadrat_a\setminus\st(v)$ have unique
non-singleton components, and they both contain the non-singleton
component of $\quadrat_{a_0}\setminus\st(v)$, so
$\quadrat\cup\bigcup_{a\in\mathcal{A}}\quadrat_a\setminus\st(v)$ has a
unique non-singleton component.
   \end{proof}

Next we would like to give sufficient conditions in terms of the graph
structure of $\Gamma$ for $\mprg_\Gamma$ to contain a wide ladder.
In the proof we will need the following lemma.
\begin{lemma}\label{starintersectionwithfundamentaldomain}
Let $\Gamma$ be an incomplete, triangle-free graph without a separating
clique.
  Suppose $Z=\{z_1,\dots,z_k\}\subset\st(v')\cap\ric_\Gamma$ for some
  $v'\in\mprg_\Gamma\setminus\ric_\Gamma$.
Let $v$ be the unique vertex in $W_\Gamma\Act v'\cap\ric_\Gamma$.  Then
$Z\subset\st(v)$ and $\big(\bigcap_{i=1}^k J_{z_i}\big)\setminus J_v\neq\emptyset$. 
\end{lemma}
\begin{proof}
By \fullref{weakconvexity},  for each $i$ there exists $g_i\in
  W_\Gamma$ such that  $g_i^{-1}\Act z_i=z_i$ and $g_i^{-1}\Act v'=v$.
  This shows $Z\subset\st(v)$.
  Furthermore, $g_iv=v'$ for each $i$, so $g_i\Davis_{J_v}$ is the maximal standard product region of $\Davis_\Gamma$ corresponding to $v'\in\mprg_\Gamma$. In particular, $g_i\Davis_{J_v}=g_1\Davis_{J_v}$ for all $i$.

Since $\ric_\Gamma$ corresponds to the maximal standard product
regions of $\Davis_\Gamma$ containing the identity vertex,
$\emptyset\neq \Davis_{J_v}\cap\bigcap_i\Davis_{J_{z_i}}$.
Thus, for all $i$,  $\emptyset\neq
g_i(\Davis_{J_{z_i}}\cap\Davis_{J_v})=\Davis_{J_{z_i}}\cap g_1\Davis_{J_{v}}$.
  The Helly Property for CAT(0) cube complexes says that since the
  convex subcomplexes
  $\{g_1\Davis_{J_{v}},\Davis_{J_{z_1}},\dots,\Davis_{J_{z_k}}\}$
  pairwise have nonempty intersection, their mutual intersection
  contains a vertex $w\in\Davis_\Gamma$.
  Now, $w\in\Davis_{J_{z_i}}$ implies that every minimal
  expression of $w$ as a word in $W_\Gamma$  uses only generators from $J_{z_i}$.
  On the other hand, $g_1\Davis_{J_v}$ and $\Davis_{J_v}$ are
  disjoint, so $w\notin \Davis_{J_v}$, which means that $w$ cannot be
  written using only generators from $J_v$.
  Thus, every minimal expression of $w$ uses only generators from
  $\bigcap_iJ_{z_i}$, and uses at least one that is not in $J_v$.
\end{proof}

   \begin{theorem} \label{thm:ladders}
     Let $\Gamma$ be 
     triangle-free and strongly CFS.
  Suppose there exist $r,s\in\Gamma$ and maximal thick joins
  $J_0,\dots,J_{n-1}$ of $\Gamma$ satisfying the following conditions:
  \begin{enumerate}
    \item The graph $\quadrat$ with a vertex $q_i$ for each $J_i$, and such
      that $q_i$ and $q_j$ span an edge when $J_i\cap J_j$ contains a
      square, is connected.\label{item:Jgraph}
\item Let $J$ be any maximal thick join of $\Gamma$.
  Let $I$ be a subset of $\{q_i\mid J\cap J_{q_i} \text{ contains a
    square}\}$ that is either the whole set or is a subset for which
  $\big(\bigcap_{q_i\in I}J_{q_i}\big)\setminus J\neq\emptyset$.
  Then  
    $\quadrat\setminus I$ has exactly one non-singleton
  component, $B$, and $\diam_\quadrat(\quadrat\setminus B)\leq
  2$.\label{item:wellconnected}
   \item The vertices $r$ and $s$ are not adjacent and not contained
      in a common  thick join of $\Gamma$.\label{item:dihedralaction}
    \item There  are indices
      $a_r,\,b_r,\,a_s,\,b_s\in\{0,\dots,n-1\}$, necessarily distinct, with:
      \begin{itemize}
      \item $r\in J_{a_r}\cap J_{b_r}$
        \item $J_{a_r}$ and $J_{b_r}$ do not share a square, and no
          maximal thick join of $\Gamma$ shares a square with both
          of them.
        \item $s\in J_{a_s}\cap J_{b_s}$
        \item $J_{a_s}$ and $J_{b_s}$ do not share a square, and no
          maximal thick join of $\Gamma$ shares a square with both
          of them. 
      \end{itemize}\label{item:widerungaction}
   
\end{enumerate}
Then $\mprg_\Gamma$ contains a wide ladder. Consequently, $\Gamma$ is not RAAGedy.
\end{theorem}
\begin{proof}
By the description in Item~\eqref{item:Jgraph}, we may regard
$\quadrat$ as an induced subgraph of $\ric_\Gamma$, which itself is induced
in $\mprg_\Gamma$, by \fullref{weakconvexity}.

Item~\eqref{item:dihedralaction} says $W_{\{r,s\}}$ is an infinite
dihedral group that does not act elliptically on $\mprg_\Gamma$. 
By \fullref{cor:stable} it is stable, so by
\fullref{cor:stability_recognizing} its orbit map
into $\mprg_\Gamma$ is a quasiisometric embedding.
We will take $\mathcal{C}$ to be the Cayley graph of $W_{\{r,s\}}$
with respect to $\{r,s\}$,
which is a line, and associate to vertex $g$ the subgraph
$g\Act\quadrat$ in $\mprg_\Gamma$.
The ladder $L:=\bigcup_{g\in W_{\{r,s\}}}g\Act\quadrat$ is unbounded,
since the orbit map of $W_{\{r,s\}}$ is a quasiisometric embedding. 
So far, we
have shown that condition \ref{item:chunks} of 
\fullref{def:wideladder} is satisfied.
Item~\eqref{item:widerungaction} gives 
condition \ref{item:widerungs}, since it implies $3\leq
d_{\ric_\Gamma}(q_{a_r},q_{b_r})$ and \fullref{weakconvexity} implies
$d_{\ric_\Gamma}(q_{a_r},q_{b_r})= d_{\mprg}(q_{a_r},q_{b_r})$, and analogously for $s$.

To show that condition~\ref{item:chunkshardtoseparate} is satisfied,
we claim that the intersection of the star of a vertex of $\mprg_\Gamma$ with
$\quadrat$ is one of the sets $I$ as described in
Item~\eqref{item:wellconnected}.
When the vertex is $v\in \ric_\Gamma$, then $\st(v)\cap\quadrat=\{q_i\mid J_v\cap J_i\text{
  contains a square}\}$.
When the vertex is  $v'\notin\ric_\Gamma$, let $\{v\}=W_\Gamma\Act v'\cap\ric_\Gamma$; then
\fullref{starintersectionwithfundamentaldomain} says $I:=\st(v')\cap\quadrat\subset\{q_i\mid J_v\cap J_i\text{
  contains a square}\}$
and that $\big(\bigcap_{i\in I} J_{q_i}\big)\setminus J_v\neq\emptyset$. 

Finally, we  show that condition~\ref{item:forbiddenconfiguration}
of  \fullref{def:wideladder} is satisfied.
Neighboring elements of $\mathcal{C}$ differ by the action of $r$ or $s$, so, up to a
symmetric argument and the group action, it suffices to consider that
the adjacent translates of condition~\ref{item:forbiddenconfiguration}
are 
$\quadrat$ and $r\Act\quadrat$.
We know that $\quadrat\cap r\Act\quadrat$ contains $q_{a_r}$ and
$q_{b_r}$, so the negation of
condition~\ref{item:forbiddenconfiguration} requires, up to reversing the roles of $q_{a_r}$ and $q_{b_r}$, that there is a
vertex $v\in\mprg_\Gamma$ such that either $q_{a_r}$ and $q_{b_r}$ are
both singleton components of $\quadrat\setminus\st(v)$, or  $q_{a_r}$ is a singleton component of $\quadrat\setminus\st(v)$
and $q_{b_r}$ is a singleton component of
$r\Act\quadrat\setminus\st(v)$.
Suppose the latter is
true.
Let $v'$ be the unique vertex
in  $W_\Gamma\Act v \cap
\ric_\Gamma$.
For $w\in\lk_\quadrat(q_{a_r})\subset\st(v)$, either
$w=v$ or there is
an element of $W_\Gamma$ taking the edge $w\edge v$ to $w\edge v'$ in $\ric_\Gamma$ by \fullref{weakconvexity}.
Similarly, for $w'\in\lk_{r\Act\quadrat}(q_{b_r})\subset \st(v)$,
either $w'=v$ and $r\Act w'=r\Act v=v'$ or there is an element of $W_\Gamma$ taking the edge $w'\edge v$ to
$r\Act w'\edge v'$ in $\ric_\Gamma$.
Thus, $\lk_\quadrat(q_{a_r})\cup\lk_\quadrat(q_{b_r})\subset\st(v')$,
so $q_{a_r}$ and $q_{b_r}$ are singleton components of
$\quadrat\setminus\st(v')$.
The  second claim of
Item~\eqref{item:wellconnected} forces a contradiction:\[  d_Q(q_{a_r}, q_{b_r})\leq
  2<3\leq d_\mprg(q_{a_r},q_{b_r})\leq d_Q(q_{a_r}, q_{b_r})\qedhere\]
\end{proof}

The original ladder example in \cite{Edl24thesis} was constructed by hand
directly for the graph $\Gamma$ of \fullref{ex:bricks}. 
\fullref{thm:ladders} applies to an iterated link double  of $\Gamma$.

\subsection{Examples}\label{mprgexamples}
\begin{example}\label{ex:bricks}
  \fullref{bricksgraph} depicts a graph $\Gamma$ with its maximal thick join
  subgraphs shadowed in different colors.
  The $\ric_\Gamma$ for $W_\Gamma\act\mprg_\Gamma$ is shown in \fullref{bricksmprg}.
  \begin{figure}[h]
    \centering
    \begin{subfigure}[b]{.4\textwidth}
          \centering
    \includegraphics{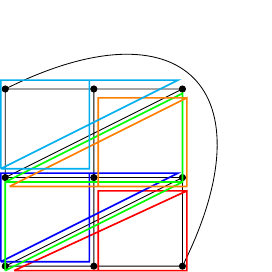}
  \caption{$\Gamma$}
  \label{bricksgraph}
    \end{subfigure}
    \begin{subfigure}[b]{.4\textwidth}
          \centering
    \includegraphics{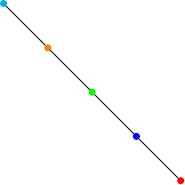}
  \caption{$\ric_\Gamma$}
    \label{bricksmprg}
  \end{subfigure}
  \caption{The graph $\Gamma$ of \fullref{ex:bricks} with its
    maximal thick joins colored and
    $\ric_\Gamma=W_\Gamma\backslash\mprg_\Gamma$ with vertices of
    matching colors.}
  \label{fig:bricks}
\end{figure}

There is one edge
of $\Gamma$
that does not belong to any thick join, and each of its endpoints
belongs to a unique maximal join, corresponding to the two ends of
$\ric_\Gamma$.
The orbit of $\ric_\Gamma$ by the action of the order 4 subgroup represented by that edge
makes a 16--cycle in $\mprg_\Gamma$.
We can see this more clearly by passing to the finite-index subgroup
obtained by link doubling over the two vertices of the extraordinary
edge.
The 16--cycle $\quadrat$ is the fundamental domain for the action of the subgroup
on $\mprg_\Gamma$, as seen in  \fullref{fig:brickssubgroupmprg}.
Notice that all vertices with first coordinate 2 belong to five
consecutive maximal join subgraphs; pick two on opposite sides, say $r:=2_{00}$ and $s:=2_{11}$.
Then $r$, $s$, and $\quadrat$ satisfy \fullref{thm:ladders}.
\begin{figure}[h]
  \centering
  \raisebox{-6ex}{
    \includegraphics{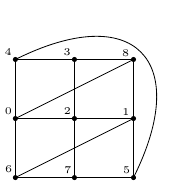}}
  $=$
  \raisebox{-7.5ex}{
\includegraphics{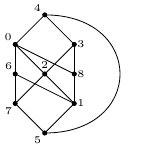}}
$\stackrel{\double^\circ_{(50)}\circ \double^\circ_4}{\longrightarrow}$
\raisebox{-11.5ex}{
\includegraphics{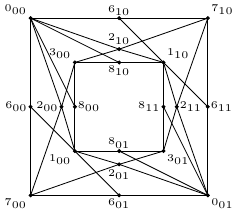}}
  \caption{Link doubling of $\Gamma$ in example \fullref{ex:bricks}.}
  \label{fig:bricksdoubled}
\end{figure}

\begin{figure}[h]
  \centering
  \includegraphics{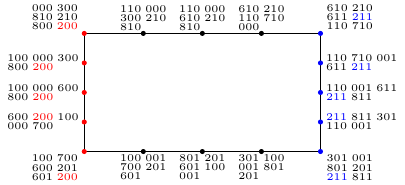}
  \caption{Fundamental domain for the action of the subgroup on the MPRG of
    \fullref{ex:bricks}.}
  \label{fig:brickssubgroupmprg}
\end{figure}

 Note also that $W_\Gamma$ is strongly CFS, has no  2--ended
  splittings, contains no compliant cycle as in \fullref{nocycles},
  and has totally disconnected Morse boundary, by \fullref{FioKar}, so
  contains no 1--ended stable subgroups. The ladder is the only way we
  know to say $\Gamma$ is not RAAGedy.  
\end{example}

\begin{example}\label{ex:Gamma73}
  \fullref{fig:secondnine} shows the other 9--vertex graph that is strongly CFS, has no  2--ended
  splittings, contains no compliant cycle as in \fullref{nocycles},
  and has totally disconnected Morse boundary, by \fullref{FioKar}, so
  contains no 1--ended stable subgroups. It has an edge not contained
  in a thick join. Link double over the two vertices of this edge, as
  in the previous example. 
  \begin{figure}[h]
    \centering
    $\Gamma:=$
    \raisebox{-10pt}{\includegraphics{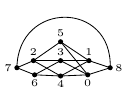}}
$\stackrel{\double^\circ_{(80)}\circ \double^\circ_7}{\longrightarrow}$
\raisebox{-62pt}{\includegraphics{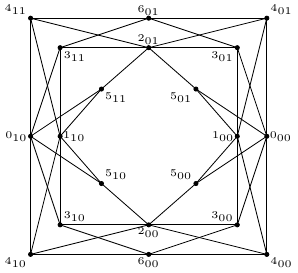}}
    \caption{The $\Gamma$ of \fullref{ex:Gamma73} and an
      iterated link double.}
    \label{fig:secondnine}
  \end{figure}

The fundamental domain for the action of the finite-index subgroup is
shown in \fullref{fig:secondninemprg}.
Take $\quadrat$ to be the entire fundamental domain.
There are choices  $r=2_{01}$ and $s=2_{00}$ that satisfy
\fullref{thm:ladders} for this $\quadrat$, so $\Gamma$ is not RAAGedy.

\begin{figure}[h]
  \centering
  \includegraphics{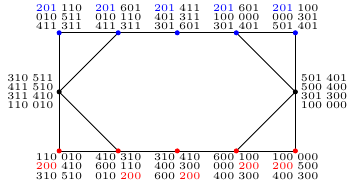}
  \caption{Fundamental domain of the MPRG in \fullref{ex:Gamma73}.}
  \label{fig:secondninemprg}
\end{figure}

In this example we cannot take $\quadrat$ to be just the outer
boundary of the fundamental domain as drawn in  \fullref{fig:secondninemprg} because that would not be an induced
subgraph, and we also cannot take just the interior octagon to be
$\quadrat$, since then neighboring translates of
$\quadrat$ would only intersect in  paths of length 2.
\end{example}

\section{Compliant cycles}\label{sec:compliant_cycles}

In this section we will show that certain cycles of subgraphs in $\Gamma$
give an obstruction to being RAAGedy.
Recall \fullref{raagtrivialprojection}: In RAAGs there is a dichotomy,
the closest point projection of one standard subcomplex to another has
diameter either zero or infinite, while in RACGs it is possible to have
finite, nonzero projection diameter.
We show that for some RACGs it is possible to build cycles 
$X_0$,\dots,$X_{m-1}$ of standard subcomplexes such that consecutive
pairs are close, the projection of any single $X_i$ to $X_0$ is small,
and the projection of $\cup_{i=1}^{m-1}X_i$ to $X_0$ is
large.
Then we would like to say that if such a RACG were quasiisometric to a
RAAG we could derive a contradiction, based on the dichotomy for
projection diameters in RAAGs.
This is accomplished in \fullref{nocycles}.
For the plan to make sense we need to know that these particular standard subcomplexes
$X_i$ in the RACG are sent by quasiisometry close to standard
subcomplexes of the RAAG.
This additional hypothesis is formalized in the first subsection,
where we bootstrap from the quasiisometry invariance of maximal
standard product subcomplexes to define a class of `compliant'
subcomplexes. 

\subsection{Compliant sets and subcomplexes}
The reader should refresh their memory of the background material on
the coarse geometry of standard subcomplexes of RAAGs and RACGs
established in \fullref{sec:catzero} and
\fullref{sec:coarse_geometry_standard}.
We will use it now.

The next statement defines compliant sets of the presentation graph,
and corresponding compliant subcomplexes of $\Sigma_\Upsilon$.
In terms of subcomplexes, the first two bullet points say the
collection is closed under adding or removing a finite direct factor.
Recalling \fullref{projection_of_standard_is_standard}, \ref{item:projection} says the collection is closed under projection, which includes intersection in the special
case that $T=\emptyset$.
Finally, \ref{item:pole} says the collection is closed under passing to the $\mathbb{R}$ factor of a
$\mathrm{tree}\times\mathbb{R}$ subcomplex.
\begin{definition}\label{def:constructible_compliant}
Let $G_\Upsilon$ be a 2--dimensional, 1--ended RACG or RAAG, and let
  $\Sigma_\Upsilon$ be its Davis complex or the universal cover of its
  Salvetti complex, respectively.
We recursively define strata $\compliant^n_\Upsilon$ of subsets of
 the vertex set of $\Upsilon$, inductively extend to higher strata, and finally take 
  $\compliant_\Upsilon:=\bigcup_{n=0}^\infty\compliant_\Upsilon^n$.
  
Seed $\compliant^0_\Upsilon$ by taking the collection of subsets of
vertices of $\Upsilon$ that consists of 
  the empty set and the vertex set of each maximal (thick, in the RACG case) join of $\Upsilon$.

Recursive~Phase: Having defined some subsets of
  $\compliant^n_\Upsilon$, join and `unjoin' spherical\footnote{A
    subset of vertices of the presentation graph of a Coxeter group is
  \emph{spherical} if the special subgroup they generate is
  finite. For RACGs this happens if and only if the vertex set is a
  clique. For RAAGs no nontrivial special subgroup is finite, so this
  phase does not apply.} factors; that is, if $S_0\join  S_1\subset\Upsilon$ and $S_0$ is spherical then:
  \begin{itemize}
  \item  If $S_1\in \compliant_\Upsilon^n$ set $\compliant_\Upsilon^n:=\compliant_\Upsilon^n\cup\{S_0\sqcup S_1\}$.
    \item  If $S_0\sqcup S_1\in \compliant_\Upsilon^n$ set $\compliant_\Upsilon^n:=\compliant_\Upsilon^n\cup\{S_1\}$.
  \end{itemize}
Repeat until $\compliant_\Upsilon^n$ stabilizes, which
happens in finitely many steps since $\Upsilon$ only has finitely many
subgraphs.

Inductive~Phase: Seed $\compliant_\Upsilon^{n+1}$ by taking 
  $\compliant_\Upsilon^{n}$ and all subsets of the following forms:
  \begin{enumerate}[(a)]
  \item If $S_0,S_1\in\compliant^n_\Upsilon$ and $T$ is a set of
    vertices of $\Upsilon$ such that either $T=\emptyset$ or
    $T\not\subset S_0$ and $T\not\subset S_1$,
    then $S_0\cap S_1\cap\bigcap_{t\in T}\lk(t)\in\compliant^{n+1}_\Upsilon$.\label{item:projection}
    \item If $S_0\sqcup S_1\in \compliant^n_\Upsilon$ with
      $S_0\join S_1\subset\Upsilon$ such that $\Sigma_{S_0}$ is a line
      and $\Sigma_{S_1}$ is a bushy tree then 
      $S_0\in\compliant^{n+1}_\Upsilon$.\label{item:pole}
    \end{enumerate}
    Perform the Recursive~Phase, joining and unjoining spherical
    factors to elements of $\compliant_\Upsilon^{n+1}$ until $\compliant_\Upsilon^{n+1}$ stabilizes.
        
        For $S\in\compliant_\Upsilon$, let its \emph{index} be $\mathrm{ind}(S):=\min\{n\mid
        S\in \compliant_\Upsilon^n\}$.
        
We call $\compliant_\Upsilon$ the \emph{compliant subsets} of
$\Upsilon$ and call the standard subcomplexes $g\Sigma_S$, for $g\in G_\Upsilon$
and $S\in\compliant_\Upsilon$,  the \emph{compliant subcomplexes} of
$\Sigma_\Upsilon$. 
\end{definition}

\begin{proposition}\label{qi_compliant}
  For all $L$, $A$, $N$, and $n$ there exists $C_n$ with the following
  property. 
  Let $G_\Upsilon$ be a 2--dimensional, 1--ended RACG or RAAG, and let
  $\Sigma_\Upsilon$ be its Davis complex or the universal cover of its
  Salvetti complex, respectively, and define $G_{\Upsilon'}$ and
  $\Sigma_{\Upsilon'}$ similarly.
  Suppose $|\Upsilon|\leq N$. 
  For every $g\in
  G_\Upsilon$, $S\in \compliant^n_\Upsilon$, and
  $(L,A)$--quasiisometry $\phi\from
  \Sigma_\Upsilon\to\Sigma_{\Upsilon'}$ 
  there exists $g'\in G_{\Upsilon'}$ and $S'\in\compliant^n_{\Upsilon'}$ such that
  $d_{Haus}(\phi(g\Sigma_S),g'\Sigma_{S'})\leq C_n$.

  Furthermore, $\compliant_\Upsilon=\compliant_\Upsilon^{2N}$, so we can
  take $C:=\max\{C_n\mid 0\leq n\leq 2N\}$ as a uniform Hausdorff
  distance bound for all of $\compliant_\Upsilon$.
\end{proposition}
\begin{proof}
  Let $C_0'$ be the constant of   \fullref{Oh}  with respect to $L$
  and $A$. 
  Let $C_0:=\max\{C_0',2L+A\}$.
  Consider $g\Sigma_S$ for $g\in G_\Upsilon$ and $S\in
  \compliant_\Upsilon^0$.
  If $S$ is the vertex set of a maximal (thick) join of $\Upsilon$
  then  \fullref{Oh}   and \fullref{Oh:standard} combine to say there
  exist $g'\in G'$ and $S'$ the vertex set of a maximal (thick) join
  of $\Upsilon'$ such that $d_{Haus}(\phi(g\Sigma_S),g'\Sigma_{S'})\leq C'_0\leq C_0$.
  By definition $S'\in\compliant_{\Upsilon'}^0$.
  If $S$ is such that $G_S$ is finite then the diameter of $\Sigma_S$
  is 0 if $G_\Upsilon$ is a RAAG or at most 2 if $G_\Upsilon$ is a
  2--dimensional RACG.
  Thus, for any $g'\in G_{\Upsilon'}$ such that
  $g'\Sigma_\emptyset\in\phi(g\Sigma_S)$ we have $\emptyset
  \in\compliant_{\Upsilon'}^0$ and 
  $d_{Haus}(\phi(g\Sigma_S),g'\Sigma_\emptyset)\leq 2L+A\leq C_0$.
  In the RAAG case this is all of $\compliant_\Upsilon^0$.
  In the RACG case it is also all of $\compliant_\Upsilon^0$, since
  the triangle-free hypothesis implies that we cannot add a cone
  vertex to a maximal thick join. 
  
  Now induct on index: supposing the proposition is true for all indices up to
  and including $n$, we extend it to $n+1$.  

Induction step \ref{item:projection} corresponds to projection between
subcomplexes.
Suppose $S:=S_0\cap S_1\cap\bigcap_{t\in
  T}\lk(t)\in\compliant_\Upsilon^{n+1}$ for some $S_0,S_1\in
\compliant_\Upsilon^n$ and $T$ either empty or not contained in either
of $S_0$ or $S_1$.
If $T=\emptyset$ define $h=1\in G_\Upsilon$.
Then $S=S_0\cap S_1$, so
$\Sigma_S=\Sigma_{S_0}\cap\Sigma_{S_1}=\pi_{\Sigma_{S_0}}(h\Sigma_{S_1})$. 
If $T\neq\emptyset$ then there exists $t_0\in T\setminus S_0$ and
$t_1\in T\setminus S_1$.
Let $h\in G_\Upsilon$ be any shortest word that begins with $t_0$ and
ends with $t_1$ and uses every letter in $T$ at least once.
So $|h|=|T|$ if $|T|=1$ or $t_0\neq t_1$, and $|h|=|T|+1$ if $|T|>1$
but $t_0=t_1$.
Again we have $\Sigma_S=\pi_{\Sigma_{S_0}}(h\Sigma_{S_1})$.
Define $h_0:=1\in G_\Upsilon$ and $h_1:=h$. 

By the induction hypothesis, for every $g\in G_\Upsilon$ and every
$(L,A)$--quasiisometry $\phi\from
\Sigma_\Upsilon\to\Sigma_{\Upsilon'}$ there are
$S_0',S_1'\in\compliant_{\Upsilon'}^n$ and $g_0',g_1'\in
G_{\Upsilon'}$ such that:
\[d_{Haus}(\phi(gh_i\Sigma_{S_i}),g_i'\Sigma_{S_i'})\leq C_n\]

Let $g'$ be a shortest element of $G_{S_0'}((g_0')^{-1}g_1')G_{S_1'}$, so
$|g'|=d_{\Sigma_{\Upsilon'}}(g_0'\Sigma_{S_0'},g_1'\Sigma_{S_1'})$
and if $g'$ is nontrivial then it begins with a letter not in $S_0'$ and ends with a letter not
in $S_1'$.
Let $T'$ be the letters appearing in $g'$, so that for $S':=S_0'\cap
S_1'\cap\bigcap_{t\in T'}\lk(t)\in \compliant_{\Upsilon'}^{n+1}$ we
have $g_0'\Sigma_{S'}=\pi_{g_0'\Sigma_{S_0'}}(g_1'\Sigma_{S_1'})$.

\fullref{coarse_intersection_in_cube_cpx} says
$g\Sigma_S=g\pi_{\Sigma_{S_0}}(h\Sigma_{S_1})\ceq(g\Sigma_{S_0}\cint
gh\Sigma_{S_1})$, and \fullref{qi_preserves_coarse_intersection} says $\phi$ sends this
coarse intersection 
to within bounded Hausdorff distance of $\phi(g\Sigma_{S_0})\cint
\phi(gh\Sigma_{S_1})$.
But using the induction hypothesis and
\fullref{coarse_intersection_in_cube_cpx} gives:
\begin{align*}
  \phi(g\Sigma_{S_0})\cint
  \phi(gh\Sigma_{S_1})&\ceq g_0'\Sigma_{S_0'}\cint g_1'\Sigma_{S_1'}\\
                     &\ceq \pi_{g_0'\Sigma_{S_0'}}(g_1'\Sigma_{S_1'})\\
  &=g_0'\Sigma_{S'}
\end{align*}
Moreover, the coarse equivalences provided by
\fullref{coarse_intersection_in_cube_cpx} and
\fullref{qi_preserves_coarse_intersection} depend on $L$ and $A$ and
the distances between the sets of the coarse intersections, so on 
$d(g\Sigma_{S_0},gh\Sigma_{S_1})=|h|\leq |T|+1\leq
|\Upsilon|+1$ and on:
\[d(g_0'\Sigma_{S_0'},g_1'\Sigma_{S_1'})\leq
  2C_n+d(\phi(g\Sigma_{S_0}),\phi(gh\Sigma_{S_1}))\leq
  2C_n+A+L(|\Upsilon|+1)\]
Thus,  $d_{Haus}(\phi(g\Sigma_S),g'\Sigma_{S'})$ is bounded by a
constant $C_{n+1}'$ depending on $L$, $A$, $C_n$, and
$|\Upsilon|$, hence on $L$, $A$, $N$, and $n$.

For induction step \ref{item:pole}, 
  suppose that $S_0\in\compliant_\Upsilon^{n+1}$ is obtained from
  $S\in\compliant_\Upsilon^n$ by virtue of $S$ decomposing as $S=S_0\sqcup S_1$, with $S_0\join S_1\subset\Upsilon$ where 
  $\Sigma_{S_0}$ is line and $\Sigma_{S_1}$ is a bushy tree.
  By the induction hypothesis, for every $g\in G_\Upsilon$ and every $(L,A)$--quasiisometry $\phi\from
  \Sigma_\Upsilon\to\Sigma_{\Upsilon'}$ there are $g'\in
  G_{\Upsilon'}$ and $S'\in\compliant^n_{\Upsilon'}$ such that $d_{Haus}(\phi(g\Sigma_{S}),g'\Sigma_{S'})\leq C_n$.
  Then
  $\pi_{g'\Sigma_{S'}}\circ\phi|_{g\Sigma_{S}}$
  is a quasiisometry from a $\text{(bounded valence bushy
    tree)}\times\mathbb{R}$ to $g'\Sigma_{S'}$ whose quasiisometry
  constants depend only on $L$, $A$, and $C_n$.
  Thus, $\Davis_{S'}$ is also a $\text{(bounded valence bushy
    tree)}\times\mathbb{R}$, which implies $S'=S_0'\sqcup S_1'$, where
  $S_0'\join S_1'\subset\Upsilon'$, with 
  $\Sigma_{S_0'}$ a line and $\Sigma_{S_1'}$ a bushy tree.
  Thus, $S_0'\in\compliant_{\Upsilon'}^{n+1}$.
  Furthermore, 
  since the $\mathbb{R}$--factor defines the only coarse equivalence
  class of separating quasiline in a $\text{(bounded valence bushy
    tree)}\times\mathbb{R}$, it follows from  \cite{Pap05} that up to
  post-composition by an element of $G_{S_1'}$, we have that $g\Sigma_{S_0}$ is sent $C''_{n+1}$--Hausdorff close to
  $g'\Sigma_{S_0'}$, with $C''_{n+1}$ depending on $L$, $A$, and $C_n$,
  hence on $L$, $A$, $N$, and $n$.

  In the RAAG case this is all of $\compliant_\Upsilon^{n+1}$.
  In the RACG case the Recursive~Phase allows for the possibility of
  taking joins with spherical subsets, which, since $n+1>0$ and
  $\Upsilon$ is triangle-free, simply means adding or removing a cone
  vertex. 
\fullref{cor:joinclique} says 
that if $S_0\join S_1\subset\Upsilon$ and $S_0'\join
  S_1\subset\Upsilon$ with $\Davis_{S_0}$ and $\Davis_{S_0'}$ finite
  then $\Davis_{S_1}$, $\Davis_{S_0\join S_1}$,  $\Davis_{S_0'\join
    S_1}$ are pairwise at Hausdorff distance at most 2.
  Suppose we know that one of these sets is in
  $\compliant_\Upsilon^{n+1}$ by one of the previous constructions, so its $\phi$--image is
$\max\{C_{n+1}',C_{n+1}''\}$ close to some $g'\Davis_{S'}$ for $g'\in
G_{\Upsilon'}$ and $S'\in\compliant_{\Upsilon'}^n$.
Then the
$\phi$--images of the other two are within
Hausdorff distance $2L+A+\max\{C_{n+1}',C_{n+1}''\}$ of the same
$g'\Davis_{S'}$. 

We have shown it suffices to take
$C_{n+1}:=2L+A+\max\{C_{n+1}',C_{n+1}''\}$.

For the further statement, in a RAAG there are no nontrivial spherical
subsets so there is no Recursive~Phase, only operations \ref{item:projection} and
\ref{item:pole} of the Inductive~Phase of \fullref{def:constructible_compliant}  apply, and these both
decrease the size of the sets involved, 
so the sequence $\compliant_\Upsilon^0\subset \compliant_\Upsilon^1\subset\cdots$
stabilizes after a number of steps bounded above by the size of 
largest set in $\compliant_\Upsilon^0$, which is the largest join in
$\Upsilon$, whose size is at most $|\Upsilon|\leq N$.
In a RACG we can add a cone vertex, increasing the size of a compliant
set without changing its index.
For operation \ref{item:pole}, it takes at least three generators to make a bushy tree, so
if $\Sigma_{S_0}$ is a line and $\Sigma_{S_1}$ is a bushy tree and
$\Sigma_{S_0*S_1}$ is compliant then $S_0$ is compliant and contains fewer elements
than $S_0\sqcup S_1$, and a cone on $S_0$ is also compliant and
contains fewer elements than $S_0\sqcup S_1$.

Now consider the case that $S_0,S_1\in\compliant^n_\Upsilon$ and
$S:=S_0\cap S_1\cap\bigcap_{t\in T}\lk(t)\in\compliant^{n+1}_\Upsilon\setminus\compliant^n_\Upsilon$
and $S'=S\join\{c\}$ is a cone on $S$.
Since $\Upsilon$ is triangle-free, $S$ is an anticlique. 
Suppose $|S'|\geq |S_0|$ and $|S'|\geq |S_1|$.
Neither of these can be strict, since $S\notin\compliant_\Upsilon^n$
implies $S\subsetneq S_0$ and $S\subsetneq S_1$. 
So $S_0=S\cup\{a_0\}$ and $S_1=S\cup\{a_1\}$.
Furthermore, $S\notin\compliant_\Upsilon^n$ implies $a_i$ is not a
cone vertex of $S_i$, which implies $S_i$ contains no cone vertex,
since $S$ is an anticlique. 
Thus, we entertain the notion that $|S'|=|S_0|=|S_1|$, but such an
inconvenience does not 
propagate any further, since this $S'$ having a
cone vertex means it cannot be the $S_0$ or
$S_1$ in an iterate of this paragraph.
Thus, we conclude that the size of a set in
$\compliant_\Upsilon^{n+2}$ is strictly less than the sets of
$\compliant_\Upsilon^n$ from which it is derived. Thus, 
$\compliant_\Upsilon=\compliant_\Upsilon^{2N}$.
\end{proof}

\begin{remark}
  One might guess a more general definition of `compliant
subcomplex' as one that is sent uniformly Hausdorff close to a
standard subcomplex by any quasiisometry. 
\fullref{def:constructible_compliant} represents the only examples we
know that satisfy this property, but the conclusion of
\fullref{qi_compliant} is stronger: images are not just
close to standard, they are close to the specific standard subcomplexes coming from
$\compliant_{\Upsilon'}$. 
\end{remark}

\begin{remark}
  There is a sense in which the class of products of trees is
  quasiisometrically rigid \cite{MR1934017}, but this involves
  projection to the factors.
  If $S_0\sqcup S_1\in \compliant_\Upsilon$ with $S_0\join
  S_1\subset\Upsilon$,  $\Sigma_{S_0}$ a bushy tree, and 
  $\Sigma_{S_1}$ a tree,  then we cannot
  automatically conclude $S_0\in\compliant_\Upsilon$.
  For example, the
  automorphism $x\mapsto xz$, $y\mapsto y$, $z\mapsto z$ of the RAAG $F_2\times\mathbb{Z}=\langle
  x,y\rangle\times\langle z\rangle$ does not send the standard
  subcomplex $\Salvetti_{\{x,y\}}$ close to a standard subcomplex. 
\end{remark}

\subsection{Coarse geometry of compliant cycles}
We would like the next result to say that if a configuration
of compliant subcomplexes exists in a RACG then it is not RAAGedy, but
the actual outcome is more subtle: it says if the configuration exists
and if there is a quasiisometry to a RAAG then there exists a
subcomplex $X'$ with a particular set of properties. The trailing
corollary says if there is no such subcomplex then the group was not
RAAGedy.
In \fullref{nocycles} we will translate these conditions to $\Gamma$,
with one set of conditions describing the configuration of compliant
sets, and a second set of conditions implying that the mystery subcomplex
$X'$ does not exist. If both sets of conditions are true then the group is not RAAGedy.
\begin{theorem}[Compliant cycles]\label{general_nocycles}
    Let $\Gamma$ be an incomplete triangle-free graph
    without separating cliques.
    Suppose there exists $B\geq 0$ such that for every sufficiently large $r$ there are compliant
    subcomplexes $X_0$, $X_1$, \dots, $X_{n-1}$ of the Davis complex $\Davis_\Gamma$, for some $n\geq 3$,
    such that for all $0\leq i<n$ there is a vertex $b_i\in X_i\cap\bar\nbhd_B(X_{(i+1)\mod n})$ and the following conditions are satisfied:
    \begin{enumerate}[(i)]
       \item $d(b_0,b_{n-1})\geq r$.\label{item:wide_base}
      \item $X_0$ has bounded coarse intersection with every other
        $X_i$.\label{item:bounded_intersection}
      \end{enumerate}
           If  $\Gamma$ is RAAGedy, then for all sufficiently large $r$ there is a quasiisometry $\psi\from\Davis_\Gamma\to \Davis_\Gamma$ taking $X_0$ to
          within bounded Hausdorff distance of a compliant subcomplex
          $X'$ that contains a quasigeodesic edge path $\bar\gamma'$ such that: 
          \begin{enumerate}
              \item $\bar\gamma'$ is contained in a bounded neighborhood of $X_0$.
              \item $\bar\gamma'$ comes close to $b_{i_0}$ for some $1\leq i_0\leq n-2$.
          \end{enumerate} 
    Furthermore, the quasiisometry constants of $\psi$, the
          quasigeodesic constants of $\bar\gamma'$,
          $d_{Haus}(X',\psi(X_0))$, and $d(\bar\gamma',b_{i_0})$ are
          independent of $r$.
        \end{theorem}
In applications we will arrange that for all $1\leq i\leq n-2$,
$d(b_i,X_0)\geq r$, since if some $b_i$ is close to $X_0$ we could
take $\psi=\mathrm{Id}_{\Davis_\Gamma}$, and the theorem would be vacuous. 
        \begin{corollary}\label{cor:nocycles}
          Suppose the hypotheses of \fullref{general_nocycles} are
          satisfied and there does not exist a quasiisometry
          $\psi\from\Davis_\Gamma\to \Davis_\Gamma$ taking $X_0$ to
          within bounded Hausdorff distance of a compliant subcomplex
          that has unbounded coarse intersection with $X_0$ and comes
          within the required distance of some $b_i$.
        Then $\Gamma$ is not RAAGedy.
      \end{corollary}
 \begin{proof}[Proof of \fullref{general_nocycles}]
    Suppose that $\phi\from W_\Gamma\to A_\Delta$ and $\bar\phi\from
    A_\Delta\to W_\Gamma$ are coarse inverse $(L,A)$--quasiisometries
    between $W_\Gamma$ and some RAAG $A_\Delta$.
    Let $C$ be the constant of \fullref{qi_compliant} for this $L$ and
    $A$ and $N=\max\{|\Gamma|,|\Delta|\}$. 

Assume $r$ is large compared to $A$, $B$, $C$, and $L$; specifically,
$r>L(6C+4A+3LB)$ is the estimate that will be needed later.
Choose $X_i$ and $b_i$ with respect to this $r$.

Since each $X_i$ is compliant, by \fullref{qi_compliant} there is a compliant subcomplex $Y_i$ of
$\Salvetti_\Delta$ at Hausdorff distance at most $C$ from
$\phi(X_i)$.
Let $\delta_i$ be a path in $Y_i$ from a point $\delta_i^-$ of $Y_i$ closest to
$\phi(b_{(i-1)\mod n})$ to a point $\delta_i^+$ of $Y_i$ closest to
$\phi(b_{i})$.
Let $\epsilon_i$ be a geodesic from $\delta_i^+$ to $\delta_{(i+1)\mod
  n}^-$.
Note that $|\epsilon_i|\leq 2C+LB+A$ for all $i$, and that $\delta_0$ and
the concatenation
$\epsilon_0+\delta_1+\epsilon_1+\delta_2+\epsilon_2+\cdots
+\delta_{n-1}+\epsilon_{n-1}$ are paths with the same endpoints
$\delta_0^+$ and $\delta_0^-$.
From the assumptions $d(b_0,b_{n-1})\geq r$ and $d(b_{n-1},X_0)\leq B$
we get an estimate:
\[d(\delta_0^+,\delta_0^-)\geq d(\phi(b_0),\phi(b_{n-1}))-LB-A-2C\geq r/L-LB-2A-2C\]

\begin{figure}[h]
  \centering
 \raisebox{-10ex}{\includegraphics{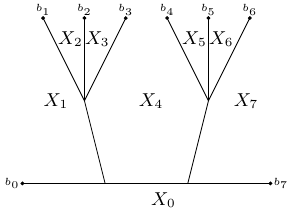}}
  $\stackrel{\phi}{\longrightarrow}$
 \raisebox{-10ex}{\includegraphics{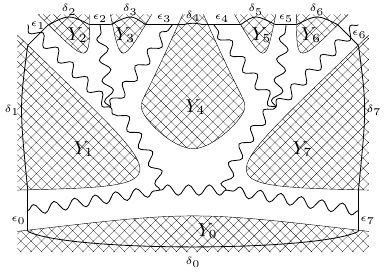}}

 \caption{Quasiisometry carrying compliant cycle into a RAAG. Each
   compliant set $X_i$ on the left, visualized as complementary
   regions of a tree in the plane,  is sent by $\phi$ close to a
   compliant subcomplex $Y_i$.
   Consecutive $Y_i$, $Y_{i+1}$ may no longer come $B$--close to each other, but they both
   come close to $\phi(b_i)$, so the path $\epsilon_i$ crossing
   between them is uniformly short.} 
  \label{fig:compliantqi}
\end{figure}

Combinatorial closest point projection to $Y_0$ in $\Salvetti_\Delta$
is a Lipschitz, hence combinatorial, map, so its sends the
concatenation of $\delta$ and $\epsilon$ paths to an edge path in $Y_0$
from $\delta_0^+$ to $\delta_0^-$.
Since $X_i\cint X_0$ is bounded, so is $Y_i\cint Y_0$, but for
standard subcomplexes of a RAAG this means that $\pi_{Y_0}(Y_i)$ is a
single vertex.
Thus, for all $i\neq 0$,  $\pi_{Y_0}(\delta_i)$ is a single vertex.

Since we assumed $r>L(6C+4A+3LB)$, we have:
\[d(\delta_0^-,\delta_0^+)>2(2C+LB+A)\geq
|\epsilon_0|+|\epsilon_{n-1}|\geq |\pi_{Y_0}(\epsilon_0)|+|\pi_{Y_0}(\epsilon_{n-1})|\]
This means that $\pi_{Y_0}(\epsilon_0)$ and $\pi_{Y_0}(\epsilon_{n-1})$ alone are not long
enough to reach from $\delta_0^+$ to $\delta_0^-$, so 
 $\pi_{Y_0}(\delta_1+\epsilon_1+\cdots+\epsilon_{n-2}+\delta_{n-1})$ is a
 nontrivial edge path in $Y_0$.
 All of the $\delta_i$ project to
single vertices, so there exists some $1\leq i_0\leq n-2$ such that $\epsilon_{i_0}$ has nontrivial
projection to $Y_0$.
Thus, there are parallel edges
$e\in Y_0$ and $e'\in \epsilon_{i_0}$. 

Recalling \fullref{commutator_labelling}, 
let $b\in A_\Delta$ be represented by a word labelling a geodesic from
$e$ to $e'$.
The fact that $e$ and $e'$ are parallel means that that they have the
same label $a$, and that $a$ and $b$
commute.
Up to translation, we may assume that $e$ is the edge between the vertices labelled $1$ and $a$ in $\Sigma_{\Delta}$, in which case, $e' = be$. 
Let $Y':=bY_0$.
Since it is standard, $Y'$ contains the entire standard geodesic $\gamma'$
containing $e'$, just as $Y_0$ contains the entire standard geodesic $\gamma$
containing $e$, and these geodesics are parallel, since every edge in
both geodesics is labelled $a$ and we have an element $b$ commuting
with $a$ realizing the parallel translation. 

Compliance of subcomplexes is preserved by group translation, by
definition, so applying \fullref{qi_compliant} to $Y'$ gives a compliant subcomplex
$X'$ of $\Davis_\Gamma$ at Hausdorff distance at most $C$ from
$\bar\phi(Y')$.
Define $\psi:=\bar\phi\circ (b\cdot)\circ\phi$, where $b\cdot$ is
left-multiplication by $b$, so that $\psi\from
\Davis_\Gamma\to\Davis_\Gamma$ is a quasiisometry taking $X_0$ within
bounded Hausdorff distance of $X'$.
The quasiisometry constants of $\psi$ only depend on $L$ and $A$, while
$d_{Haus}(\psi(X_0),X')$ depends only on $L$, $A$, and $C$.

Push the vertices of $\gamma'$ first to $\Davis_\Gamma$ via
$\bar\phi$, and then into $X'$ via $\pi_{X'}$.
Since $X'$ is convex and $d_{Haus}(X',\bar\phi(Y'))\leq C$, a standard
connect-the-dots argument says there is a quasigeodesic edge path
$\bar\gamma'$ contained in $X'$ whose quasigeodesic constants depend
only on $L$, $A$, and $C$, and that is bounded Hausdorff distance from
$\bar\phi(\gamma')$ and $\bar\phi(\gamma)$, hence contained in a bounded neighborhood of $X_0\ceq\bar\phi(Y_0)$.

The following estimate shows that  $d(\bar\gamma',b_{i_0})$ is bounded above, independent of $r$:
  {\allowdisplaybreaks\begin{align*}
                        d(\bar\gamma',b_{i_0})
                        &\leq d(\pi_{X'}(\bar\phi(\gamma')),b_{i_0})\\
                                     &\leq C+d(\bar\phi(\gamma'),b_{i_0})\\
                     &\leq C+A+d(\bar\phi(\gamma'),\bar\phi\phi(b_{i_0}))\\
                     &\leq C+2A+Ld(\gamma',\phi(b_{i_0}))\\
                     &\leq C+2A+LC+Ld(\gamma',\delta_{i_0}^+)\\
                     &\leq  C+2A+LC+L|\epsilon_{i_0}|\\
                                     &\leq  C+2A+LC+L(2C+LB+A)\qedhere
  \end{align*}}
\end{proof}

  \subsection{Graphical criteria}

  The next result gives practical, graphical criteria for applying
  \fullref{general_nocycles}.
  The point is that conditions
  \eqref{item3:paths}--\eqref{item3:szerolinks} imply the hypotheses of
  \fullref{general_nocycles}, but conditions 
  \ref{subcase:hyperbolic}--\ref{subcase:no_match} say that there is
  no possible target subcomplex for a quasiisometry as in
  \fullref{cor:nocycles}, so $\Gamma$ cannot be RAAGedy.

\begin{theorem}\label{nocycles}
  Let $\Gamma$ be an incomplete triangle-free graph without separating
  cliques, and let $\compliant_\Gamma$ be the compliant subsets of
  $\Gamma$, as in \fullref{def:constructible_compliant}.
Suppose there exist,  for $q\geq 1$ and all $0\leq i\leq q-1$,  sets $S_i\in\compliant_\Gamma$ and paths
  $P_i:=(a_{i,0},a_{i,1},\dots,a_{i,\ell(i)-1})$ in $\Gamma$, containing
  $\ell(i)\geq 1$ vertices, such that  all of
  the following hold (for $S_j$, $P_j$, and $a_{j,k}$
  we always implicitly take $j$ mod $q$ and $k$ mod $\ell(j)$):
   \begin{enumerate}
        \item The paths $P_i$ are disjoint, and $P:=\bigcup_jP_j$ is an
          induced subgraph of $\Gamma$.\label{item3:paths}
          \item $\forall i$, $S_{i}\cap
            P=\{a_{i-1,\ell(i-1)-1},a_{i,0}\}$, which are the last vertex of
            $P_{i-1}$ and the first
            vertex of $P_i$.\label{item3:spintersections}
                     \item If $q=1$ then $P_0\not\subset S_0$.\label{item3:free}
            \item $\forall i\neq 0$, $S_i\cap S_0$ is a
              clique. \label{item3:sintersections}
              \item $\forall (i,j)\notin \{(0,0),(q-1,\ell(q-1)-1)\}$,
                $\lk(a_{i,j})\cap S_0$ is a clique.\label{item3:szerolinks}
    \end{enumerate}
   Assume that no proper subset of $S_0$ belongs to $\compliant_\Gamma$ and
  contains the first and last vertices $a_{0,0}$ and
  $a_{q-1,\ell(q-1)-1}$ of $P$.
       Let $\Gamma_0$ be the subgraph of $\Gamma$ induced by
       $S_0$.
       If either of the following are true then $\Gamma$ is not RAAGedy:
  \begin{enumerate}[(a)]
    \item $\Gamma_0$ is square-free.\label{subcase:hyperbolic}
    \item No $S'\in\compliant_\Gamma$ spans a subgraph $\Gamma'$ 
       satisfying all of the following conditions:
       \begin{enumerate}[label=(\Roman*)]
           \item $\Gamma_0\cap\Gamma'$ is incomplete.\label{item:gamma_prime_cap_gamma_zero_incomplete}
              \item $\Gamma'\cap P$ is incomplete.\label{item:gamma_prime_cap_P_incomplete}
               \item $\{a_{0,0},a_{q-1,\ell(q-1)-1}\}\nsubset\Gamma'$.\label{item:end_pair}
               \item There exists a quasiisometry $\psi\from
             \Davis_\Gamma\to\Davis_\Gamma$  with $\psi(\Davis_{\Gamma_0})\ceq\Davis_{\Gamma'}$. \label{item:gamma_prime_qi_to_gamma_zero}
          \end{enumerate}\label{subcase:no_match} 
  \end{enumerate}
\end{theorem}
The minimality condition on $S_0$ is justified because a proper subset of
$S_0$ belonging to $\compliant_\Gamma$ and containing
$\{a_{0,0},a_{q-1,\ell(q-1)-1}\}$ satisfies hypotheses \eqref{item3:paths}--\eqref{item3:szerolinks}  for the
same choices of $S_i$ and $P_i$.

In \fullref{sec:not_same_orbit}  we give some ways to rule out condition
  \ref{item:gamma_prime_qi_to_gamma_zero}.
       
\begin{proof}
 Let $C_m$ be the cycle graph of length $m\geq 3$.
 Its commutator complex is a
  closed surface, since it is a connected square complex such that the
  link of every vertex is a circle.
  Its Euler characteristic is $2^m-m2^{m-1}+m2^{m-2}=2^{m-2}(4-m)$.
Think of  $P=\sqcup_iP_i$ as a subgraph of $C_m$ for  
$m:=\sum_{i=0}^{q-1}\ell(i)$, where the vertices $a_{i,j}$ are
ordered lexicographically, so that for each $i$ there is an edge 
$a_{i,\ell(i)-1}\edge a_{i+1,0}$ of $C_m$ that is not an edge of $P$. 
Thus, the commutator complex of $P$ is homotopy equivalent to the
commutator complex of $C_m$ after puncturing each square labelled by a
commutator $[a_{i,\ell(i)-1},a_{i+1,0}]$.
There are $q2^{m-2}$ such squares, so the Euler characteristic of the commutator complex of $P$ is
$2^{m-2}(4-m-q)$.
We claim this is a negative number, so $W_P$ is a virtually a nonAbelian
free group.
 To see this, consider that the alternative is that either $q=1$ and
 $\ell(0)\leq 3$ or $q=2$ and $\ell(0)=\ell(1)=1$.
 If $q=1$ then 
 Hypothesis~\eqref{item3:free} says $\ell(0)>2$, but it cannot be that
 $q=1$ and $\ell(0)=3$ because this would either give a triangle in
 $\Gamma$ or contradict Hypothesis~\eqref{item3:szerolinks}.
 We cannot have $q=2$ and $\ell(0)=\ell(1)=1$ because this would
 either violate Hypothesis~\eqref{item3:paths}, if $a_{0,0}$ and $a_{1,0}$ are
 adjacent, or Hypothesis~\eqref{item3:sintersections} if not.

$\Davis_P$ is the universal cover of a closed surface with some open
faces removed, so $\Davis_P$ admits a planar embedding in which all of its vertex links are copies of $P$ ordered as
in $C_m$ or its reverse, and whose boundary components are the bicolored geodesics
whose colors are  $a_{i,\ell(i)-1}a_{i+1,0}$, for each $i$; that is,
the boundary components are the lifts of the boundaries of the missing
faces.

If $\{a\}=P_i$ is an isolated vertex of $P$ then edges of $\Davis_P$
labelled $a$ belong to two different components of $\bdry\Davis_P$, one
bicolored with $a$ and the last vertex of $P_{i-1}$ and one bicolored
with $a$ and the first vertex of $P_{i+1}$.
If $a_{i,0}$ is the first vertex of a non-singleton component $P_i$
of $P$ then an edge $e$ of $\Davis_P$ labelled $a_{i,0}$ is contained in a
unique component of $\bdry\Davis_P$ and is a face of a unique square
whose sides are colored $a_{i,0}$ and $a_{i,1}$.
The opposite face of this square is the unique edge of $\Davis_P$
parallel to $e$.

Fix an identity vertex $1$ of $\Davis_P$, let $\gamma_0$ be the
$a_{0,0}a_{q-1,\ell(q-1)-1}$ bicolored geodesic through 1,
parameterized by arclength with $\gamma_0(0)=1$. 
For any $r>0$, consider the following set, where the overbar means closure in the 1--skeleton:
\[\hat\Davis_P:=\bar\nbhd_r(\gamma_0)\cap\overline{\pi_{\gamma_0}^{-1}(\gamma_0(0,r))}\]
This should be imagined as an $r$--tubular
neighborhood in $\Davis_P$ of the subsegment of $\gamma_0$ of length
$r$ starting at 1.
We specify a collection of components of $\bdry\Davis_P$ that 
contains all of the vertices of $\hat\Davis_P$. First, include every
boundary component of $\Davis_P$ that contains an edge in $\hat\Davis_P$. Then, for each vertex $x$ of $\hat\Davis_P$ that is not contained in one of these boundary components, choose any one of the components of $\bdry\Davis_P$ containing $x$. Clockwise with respect to the planar embedding, starting from $\gamma_0$,  consecutively number these boundary components $\gamma_0$, $\gamma_1$,\dots,$\gamma_{n-1}$. Orient each boundary component $\gamma_i$ accordingly, i.e.~choose a parametrization of $\gamma_i$ so that the vertex at which $\gamma_i$ enters $\hat \Sigma_P$ appears, in the clockwise ordering, before the vertex at which $\gamma_i$ leaves $\hat \Sigma_P$. Let $b_i$ be the last vertex of $\gamma_i$ in $\hat \Sigma_P$. Every vertex of $\Davis_P$ lies on some boundary component. Thus, the consecutive ordering of the boundary components $\gamma_i$ and the choice of the vertices $b_i$ yield $B:=2\geq d(b_i,\gamma_{i+1})$. See \fullref{fig:examplecompliant}
and \fullref{fig:treelike}.

Next we argue that all the $b_i$ with $i \notin \{0, n-1\}$ are at distance $r$ from $\gamma_0$. More specifically we show that the first vertex of $\gamma_1 \cap \hat \Sigma_P$ is $1$-close to $\gamma_0$, the last vertex  $b_{n-1}$ of $\gamma_{n-1} \cap \hat \Sigma_P$ is $1$-close to $\gamma_0$ and all the other endpoints of $\gamma_i \cap \hat \Sigma_P$ with $i \neq 0$ are at distance $r$ from $\gamma_0$. 

The region $\hat\Davis_P$ is bounded, so each geodesic $\gamma_i$
intersects it in a bounded subinterval (possibly a single vertex).
By construction, the extreme vertices of $\hat\Davis_P$ are those that
project to either $\gamma_0(0)$ or $\gamma_0(r)$ and those that are at
distance $r$ from $\gamma_0$. 
In particular, each endpoint of each $\gamma_i\cap\hat\Davis_P$ either is at
distance $r$ from $\gamma_0$ or projects to one of the endpoints of $\gamma_0\cap\hat\Davis_P$.
By construction, if $\gamma_0(0)\in\pi_{\gamma_0}(\gamma_i)$ then $\pi_{\gamma_0}(\gamma_i)$
contains the edge $\gamma_0(0,1)$.
This edge is either not a face of a square of $\Davis_P$, in which
case it is also contained in $\gamma_1$, or it is a face of a unique
square whose opposite face is a boundary edge, so is contained in
$\gamma_1$. Thus, $\gamma_1$ is the only $\gamma_i$ for $i\neq 0$
with $\gamma_0(0)\in\pi_{\gamma_0}(\gamma_i)$. 
Similarly, $\gamma_{n-1}$ is the only $\gamma_i$ for $i\neq 0$ such
that $\gamma_0(r)\in\pi_{\gamma_0}(\gamma_i)$.
Accordingly,  the first vertex of $\gamma_1 \cap \hat \Sigma_P$ and
$b_{n-1}$ are $1$-close to $\gamma_0$ and any other endpoint of any
$\gamma_i \cap \hat \Sigma_P$ with $i \neq 0$ leaves the bounded
region $\hat\Davis_P$ through a vertex at distance $r$ from
$\gamma_0$.

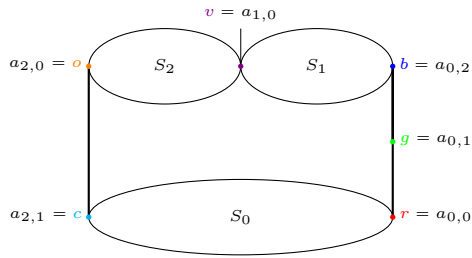
\begin{figure}[h]
  \centering
  \begin{tikzpicture}\tiny
    \draw (0,0) ellipse (2cm and .5cm) node {$S_0$};
    \draw (1,2) ellipse (1cm and .5cm) node {$S_1$};
    \draw (-1,2) ellipse (1cm and .5cm) node {$S_2$};
      \coordinate[label={[label distance=0pt] 90:${\color{violet}v}=a_{1,0}$}] (v) at (0,2.5);
      \draw[ultra thin] (0,2)--(v);
      \coordinate[label={[label distance=0pt]
        0:${\color{red}r}=a_{0,0}$}] (r) at (2,0);
      \coordinate[label={[label distance=0pt]
        0:${\color{green}g}=a_{0,1}$}] (g) at (2,1);
               \coordinate[label={[label distance=0pt] 0:${\color{blue}b}=a_{0,2}$}] (b) at (2,2);
      \coordinate[label={[label distance=0pt]
        180:$a_{2,0}={\color{orange}o}$}] (o) at (-2,2);
      \coordinate[label={[label distance=0pt]
        180:$a_{2,1}={\color{cyan}c}$}] (c) at (-2,0);
      \draw[thick] (r)--(b)--(g) (o)--(c);
               \fill[color=cyan] (c) circle (1pt);
    \fill[color=violet] (0,2) circle (1pt);
               \fill[color=orange] (o) circle (1pt);
                          \fill[color=blue] (b) circle (1pt);
               \fill[color=green] (g) circle (1pt);
               \fill[color=red] (r) circle (1pt);
  \end{tikzpicture}
  \caption{Example cycle $S_0$, $P_0$, $S_1$, $P_1$, $S_2$, $P_2$.}
  \label{fig:examplecompliant}
\end{figure}
\begin{figure}[h]
  \centering
  \includegraphics{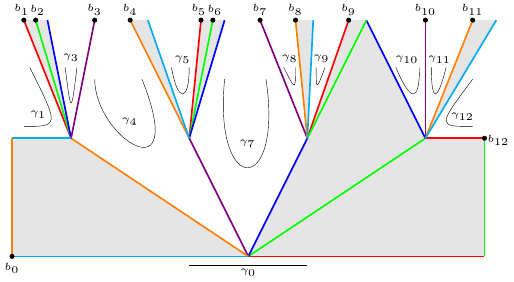}
  \caption{$\hat\Davis_P$ for $r=2$ and  $P$ as in
    \fullref{fig:examplecompliant}.}
  \label{fig:treelike}
\end{figure}

The inclusion of $\Davis_P$ into $\Davis_\Gamma$ is an
isometric embedding, by Hypothesis~\eqref{item3:paths}.
Hypothesis~\eqref{item3:spintersections} implies that for each
$\gamma_i$ there is a unique translate $X_i$ of one of the
$\Davis_{S_j}$ that contains $\gamma_i$, since for each $j$ only
$S_j$ contains both $a_{j,0}$ and $a_{j-1,\ell(j-1)-1}$. 
This also implies that $X_i \cap \Davis_P = \gamma_i$.
Furthermore, $b_i\in X_i$ and $B\geq d(b_i,\gamma_{i+1})\geq d(b_i,X_{i+1})$ and 
$d(b_0,b_{n-1})$ is either $r$ or $r+1$.
We have also arranged the $b_i$ for $i\notin\{0,n-1\}$ to be far from
$X_0$, since $X_0\cap\Davis_P=\gamma_0$ and
$d_{\Davis_P}(b_i,\gamma_0)=r$ imply $d_{\Davis_\Gamma}(b_i,X_0)=r$. 
This follows because $X_0\cap \Davis_P\neq \emptyset$, so the
projection of the convex subcomplex $\Davis_P$ to the convex subcomplex $X_0$ is just their intersection, which is
$\gamma_0$, so the closest point of $X_0$ to $b_i$ is a point of
$\gamma_0$.

To see that the hypotheses of \fullref{general_nocycles} are
satisfied, it remains to show that  $X_0$ has bounded coarse
intersection with every other $X_i$.

By Hypothesis~\eqref{item3:spintersections}, if $\wall$ is a wall dual
to an edge colored by some $a_{j,k}$ that is not on $\gamma_i$ then $\wall$ is not dual to any edge in $X_i$. 

First, suppose $\gamma_0$ and $\gamma_i$ are disjoint, so there is a
nontrivial shortest $\Davis_P$--path connecting them.
Let $\wall$ be the wall dual to the first edge of that path, let $\wall^+$ be the halfspace of $\wall$ containing
$1$, let $\wall^-$ be the complementary halfspace.
Then $X_0\subset \wall^+$.
By \fullref{bridge}, $\wall$ separates $X_0$ and $X_i$, so $X_i\subset \wall^-$.
By Hypothesis~\eqref{item3:szerolinks}, the label of $\wall$ commutes
with at most one edge label of $X_0$.
It then follows from \fullref{coarse_intersection_in_cube_cpx} and  \fullref{commutator_labelling},
that $X_0\cint \wall^-$ is bounded, so $X_0\cint X_i$ is bounded.

If $\gamma_0$ and $\gamma_i$ are distinct but not disjoint, then $X_i$ is
a translate of $\Davis_{S_j}$ for some $j\neq 0$, in which case $X_0\cint X_i$
is bounded as a consequence of Hypothesis~\eqref{item3:sintersections}. 

We have shown the hypotheses of 
\fullref{general_nocycles} are satisfied, so 
if $\Gamma$ is RAAGedy then there is a quasiisometry $\psi\from
\Davis_\Gamma\to \Davis_\Gamma$ taking $X_0$ to within bounded
Hausdorff distance of a compliant subcomplex $X'$ containing a
quasigeodesic edge path $\bar\gamma'$ that is contained in a bounded neighborhood of
$X_0$ and also comes close to $b_{i_0}$ for some $1\leq i_0\leq n-2$. 
We will show that either of 
Hypotheses~\ref{subcase:hyperbolic} or
\ref{subcase:no_match} leads to a contradiction, so $\Gamma$ cannot
have been RAAGedy.

Suppose  $X'=w\Davis_{S'}$ for some
$w\in W_\Gamma$ and $S'\in\compliant_\Gamma$, and let $\Gamma'$ be the
subgraph of $\Gamma$ induced by $S'$.
Since $\psi\from \Davis_\Gamma\to\Davis_\Gamma$ takes $X_0$ to within
bounded Hausdorff distance of $X'$ and they are both convex, hence
undistorted, $X_0$ and $X'$ are quasiisometric.
Adjusting by the group action, we have a quasiisometry
$(w^{-1}\cdot)\circ\psi:\Davis_\Gamma\to\Davis_\Gamma$ taking
$\Davis_{\Gamma_0}$ Hausdorff close to $\Davis_{\Gamma'}$, so $\Gamma'$ satisfies \ref{item:gamma_prime_qi_to_gamma_zero}.

The fact that $X'$ and $X_0$ have unbounded coarse intersection,
since their coarse intersection contains $\bar\gamma'$, means
$\Gamma_0\cap\Gamma'$ is incomplete. 
So $\Gamma'$ satisfies \ref{item:gamma_prime_cap_gamma_zero_incomplete}.

\fullref{general_nocycles} gives us that $d(\bar\gamma',b_{i_0})$ is
bounded, independent of $r$, so by taking $r$ large we can make 
$r-d(\bar\gamma',b_{i_0})$ as large as we like.
Assume $r-d(\bar\gamma',b_{i_0})\geq 3$.
Consider a shortest path $\zeta$
in $\Davis_P$ from $\gamma_0$ to $b_{i_0}\in \gamma_{i_0}$.
Let $\wall_1$, $\wall_2$, and $\wall_3$ be the walls
dual to the first three edges of $\zeta$.
For $k\in\{1,2,3\}$, let $z_k$ be the
generator labelling $\wall_k$, and let $\wall_k^-$ be the halfspace of
$\wall_k$ containing $b_{i_0}$.

By minimality of $\zeta$, $z_1\notin S_0$, so $X_0\subset \wall_1^+$.
If $\wall_1$ and $\wall_2$ cross then $z_1$ and $z_2$ commute and the
first two edges of $\zeta$ travel along the boundary of a square with
opposite side pairs labelled by $z_1$ and $z_2$.
By replacing the first two edges of $\zeta$ by the other two edges of
this square we get a path $\zeta'$ with the same endpoints and length as $\zeta$ that crosses
$\wall_2$ first and then $\wall_1$.
If $\wall_1$ crosses both $\wall_2$ and $\wall_3$ then $z_1$ commutes
with $z_2$ and $z_3$ and since $\Davis_P$ is $2$-dimensional, $z_2$
cannot commute with $z_3$.
Thus, up to exchanging  $\wall_1$ and $\wall_2$ we may assume that
either $\wall_1$ does not cross $\wall_2$ or $\wall_1$ crosses $\wall_2$ but not $\wall_3$. 

In the first case $X_0\subset \wall_1^+\subset\wall_2^+$.
As argued previously, $X_0$ has bounded coarse intersection with
$\wall_1^-$.
Likewise $X_0\cint \wall_2^-$ is bounded, since
$\wall_2^-\subset\wall_1^-$.
In particular, both ends of $\bar \gamma'$ are contained in
$\wall_1^+$, as otherwise we would have an unbounded subset of
$\wall_1^-$ contained in a bounded neighborhood of $X_0$,
contradicting that their coarse intersection is bounded.
We conclude that  $\bar\gamma'$ enters $\wall_1^+$. However, as $\bar \gamma'$ comes close to $b_{i_0}$, it enters $\wall_2^-$ as well. Accordingly, $\bar \gamma'$  crosses walls $\wall_1$ and $\wall_2$. 
Since $\bar\gamma'\subset X'$, the set $S'$ contains $z_1$ and $z_2$, with $z_1\in P\setminus S_0$
and $z_2\in P$ not adjacent to $z_1$.

In the second case we reach the same conclusions for $z_1$ and $z_3$. In either case, $S'$ contains $z_1$ and a non-adjacent vertex ($z_2$ or $z_3$), so 
$\Gamma'\cap P$ is incomplete. Thus
$\Gamma'$ satisfies \ref{item:gamma_prime_cap_P_incomplete}.

If Hypothesis~\ref{subcase:no_match}  is true then since $\Gamma'$
satisfies \ref{item:gamma_prime_cap_gamma_zero_incomplete},
\ref{item:gamma_prime_cap_P_incomplete}, and
\ref{item:gamma_prime_qi_to_gamma_zero}, it does not satisfy
\ref{item:end_pair}, so
$\{a_{0,0},a_{q-1,\ell(q-1)-1}\}\subset\Gamma'$. Since $S'\in\compliant_\Gamma$, the minimality condition on $S_0$ demands $S_0\cap S'=S_0$. So $S_0 \subseteq S'$ is contained in a join of two
anticliques $\Theta_0\join\Theta_1$.  
Suppose $z_1\in\Theta_1$.
Since $z_1\in (S'\cap P)\setminus S_0$,
Hypothesis~\eqref{item3:szerolinks} implies
$S_0\cap\lk(z_1)$ is at most one vertex.
But $S_0\cap\lk(z_1)=S_0\cap\Theta_0$, so the anticlique $\Theta_1$
contains all but at most one vertex of $S_0$.
This shows Hypothesis~\ref{subcase:hyperbolic} is true.

If Hypothesis~\ref{subcase:hyperbolic} is true then $X_0$ and
$X'$ are hyperbolic.

There are only finitely many isometry types of standard
   subcomplex in $\Davis_\Gamma$, so there is a uniform bound on
   hyperbolicity constants that occur, independent of $r$.
The quasigeodesic constants of $\bar\gamma'$ are also independent of
$r$. 
Thus, there is a stability constant, independent of $r$, 
   bounding the distance between $\bar\gamma'$ and a
   geodesic $\gamma''\subset X'$ asymptotic to it.
   Taking $r-d(\bar\gamma',b_{i_0})$ larger than this stability constant forces
the geodesic $\gamma''$ to enter $\wall_1^-$, but it still has both
ends in $\wall_1^+$, since $\bar\gamma'$ does. 
This is a contradiction: walls are convex, so $\gamma''$ cannot
cross from $\wall_1^+$ to $\wall_1^-$ and back to $\wall_1^+$.
\end{proof}

Observe that \fullref{no_cycle_of_cuts} is a consequence of
\fullref{nocycles} and \fullref{qi_compliant}, with the poles of
the cuts being the compliant sets and their intersections the
(singleton) connecting paths.

\subsection{Subcomplexes not in the same quasiisometry
  orbit}\label{sec:not_same_orbit}
For $W_\Gamma$ a RACG and $\Gamma'$ an induced subgraph
of $\Gamma$, let $\llbracket\Gamma'
\rrbracket$ denote the quasiisometry type of $\Davis_{\Gamma'}$.
Similarly, if $S$ is a set of vertices of $\Gamma$, let $\llbracket S\rrbracket$
denote the quasiisometry type of the special subgroup defined by $S$. 

We first state a result that we will use to rule out~\ref{item:gamma_prime_qi_to_gamma_zero} of Hypothesis~\ref{subcase:no_match} in certain applications of~\fullref{nocycles}. 
Recall that this condition says there exists a quasiisometry $\psi\from \Davis_\Gamma\to \Davis_{\Gamma}$ with $\psi(\Davis_{\Gamma_0})\ceq\Davis_{\Gamma'}$.
In other words, we wish to show that $\Davis_{\Gamma_0}$ and $\Davis_{\Gamma'}$ are not in the same orbit of the quasiisometry group of $\Davis_\Gamma$ for its action on coarse equivalence classes of subsets of $\Davis_\Gamma$.

\fullref{cor:ric_qi_obstruction} will be a
corollary of the subsequent, more general, \fullref{qi_implications}.
  \begin{corollary}\label{cor:ric_qi_obstruction}
    Let $\Gamma$ be an incomplete triangle-free graph without
    separating cliques, and let $\Davis_\Gamma$ be the Davis complex
    of $W_\Gamma$.
    Suppose $\phi\from \Davis_\Gamma\to\Davis_\Gamma$ is a
    quasiisometry.
Suppose $S_0$ and $T_0$ are 
    vertex sets of maximal thick joins of $\Gamma$ such that
    $\phi(\Davis_{S_0})\ceq \Davis_{T_0}$. 
Then $\llbracket
    S_0\rrbracket=\llbracket T_0\rrbracket$ and for every neighbor $S_2$ of
    $S_0$ in $\ric_\Gamma$ there is a neighbor $T_2$ of $T_0$ in
    $\ric_\Gamma$ such that $\llbracket S_2\rrbracket=\llbracket
    T_2\rrbracket$ and $\llbracket S_0\cap S_2\rrbracket=\llbracket
    T_0\cap T_2\rrbracket$.
  \end{corollary}

  \begin{lemma}\label{qi_implications}
    Let $\Gamma$ be an incomplete triangle-free graph without
    separating cliques, and let $\Davis_\Gamma$ be the Davis complex
    of $W_\Gamma$.
    Suppose $\phi\from \Davis_\Gamma\to\Davis_\Gamma$ is a quasiisometry.
    Suppose $S_0,S_1,S_2\in\compliant_\Gamma$ such that
$S_0\subset S_1$, and each of $S_0$ and $S_1\cap S_2$ contains vertex sets of squares of $\Gamma$.
Suppose $\phi(\Davis_{S_0})\ceq
\Davis_{T_0}$.
Then $T_0\in\compliant_\Gamma$ and there exist
$T_1,T_2\in\compliant_\Gamma$ such that for $i\in\{0,1,2\}$:
\begin{itemize}
\item Some translate of $\Davis_{T_i}$
is coarsely equivalent to $\phi(\Davis_{S_i})$.
\item $\llbracket {S_i}\rrbracket=\llbracket
  {T_i}\rrbracket$
\item $\mathrm{ind}(S_i)=\mathrm{ind}(T_i)$
  \item $T_0\cap T_1$ is either all of $T_0$ or all but a cone vertex
    of $T_0$.
      \item $\llbracket S_1\cap S_2\rrbracket=\llbracket T_1\cap T_2\rrbracket$
    \end{itemize}
  \end{lemma}
  \begin{proof}
For any sets $A$ and $B$ of vertices of $\Gamma$, if  $\phi(\Davis_{A})\ceq w\Davis_{B}$ then
$\pi_{\Davis_{B}}\circ (w^{-1}\cdot)\circ\phi|_{\Davis_{A}}\from \Davis_{A}\to
\Davis_{B}$ is a quasiisometry, so $\llbracket
A\rrbracket=\llbracket B\rrbracket$.
  By \fullref{qi_compliant}, for all $i$ there
  exist $w_i\in W_\Gamma$ (and we take $w_0=1$) and $T_i\in\compliant_\Gamma$ with
  $\mathrm{ind}(S_i)=\mathrm{ind}(T_i)$ and
  $\phi(\Davis_{S_i})\ceq w_i\Davis_{T_i}$, which implies $\llbracket
 S_i\rrbracket=\llbracket T_i\rrbracket$.

We may assume that all of the $S_i$ and $T_i$ have no cone vertex,
since removing a cone vertex does not change the coarse equivalence
class of the special subgroup, the index in $\compliant_\Gamma$, or the fact that
$S_0$ and $S_1\cap S_2$ contain vertex sets of  squares. 

As in the proof of \fullref{qi_compliant}, for $V_{ij}$ defined as
the set of generators occurring in minimal length elements of
$W_{S_i}(w_i^{-1}w_j)W_{S_j}$ we have:
\[\phi(\Davis_{S_i\cap S_j})\ceq
\pi_{w_i\Davis_{T_i}}(w_j\Davis_{T_j})=w_i\Davis_{T_i\cap T_j\cap
  \bigcap_{v\in V_{ij}}\lk(v)}\]
Thus, $\llbracket S_i\cap S_j\rrbracket=\llbracket T_i\cap T_j\cap
  \bigcap_{v\in V_{ij}}\lk(v)\rrbracket$. 
Since $\Gamma$ is triangle-free, links of vertices are anticliques,
so if $V_{ij}\neq\emptyset$ then $T_i\cap
  T_j\cap\bigcap_{v\in V_{ij}}\lk(v)$ is an anticlique and $\llbracket T_i\cap T_j\cap
  \bigcap_{v\in V_{ij}}\lk(v)\rrbracket$ is one of `point', `line', or
  `bushy tree'. 
But  $\Davis_{S_0}=\Davis_{S_0\cap S_1}$ and $\Davis_{S_1\cap S_2}$ contain 2--flats,
so $\llbracket S_i\cap S_j\rrbracket$ is not one of
  these, so
$V_{01}$ and $V_{12}$ are empty, which gives
$\llbracket S_0\rrbracket=\llbracket S_0\cap
S_1\rrbracket=\llbracket T_0\cap T_1\rrbracket=\llbracket T_0\rrbracket$ and $\llbracket
S_1\cap S_2\rrbracket=\llbracket T_1\cap T_2\rrbracket$.

Finally, $S_0\subset S_1$ implies
$\Davis_{T_0}\ceq\pi_{\Davis_{T_0}}(w_1\Davis_{T_1})=\Davis_{T_0\cap
  T_1\cap \bigcap_{v\in V_{01}}\lk(v)}=\Davis_{T_0\cap T_1}$.
By \fullref{joinclique} and the assumption that $T_0$ has no cone
vertex, $T_0=T_0\cap T_1$.
  \end{proof}

  \begin{proof}[Proof of \fullref{cor:ric_qi_obstruction}]
    Being a maximal thick join is the same as being a non-clique of $\compliant_\Gamma^0$.
    Suppose there is a $T_0$ that is in the quasiisometry orbit
    of $S_0$. 
    Take $S_1:=S_0$ and any
    neighbor $S_2$ of $S_0$ in $\ric_\Gamma$, which shares a square
    with $S_0$ by definition of $\ric_\Gamma$.
    Apply \fullref{qi_implications}  to get $T_1$ and $T_2$, which
    are maximal thick joins since $S_1$ and $S_2$ were.
    Since $T_0$ has no cone vertex, $T_0\subset T_1$, but $T_0$ is
    maximal, so $T_0=T_1$ and:
    \[\llbracket S_0\cap S_2\rrbracket=\llbracket S_1\cap S_2\rrbracket=\llbracket T_1\cap T_2\rrbracket=\llbracket T_0\cap T_2\rrbracket\qedhere\]
  \end{proof}

There are some other ways to rule out standard subcomplexes
being in the same quasiisometry orbit.
Instead of only considering single neighbor intersection quasiisometry types as in
\fullref{cor:ric_qi_obstruction}, one can consider the pattern of
intersection of $\Davis_{S_0}$ with all of its neighbors in
$\mprg_\Gamma$.
Something like this is done in \cite{raagedyii}.
In another direction, we can use automorphism orbits of $\mprg_\Gamma$
to distinguish maximal product regions. 
For instance, if there are two  maximal standard product regions and
one gives a cut vertex of $\mprg_\Gamma$ and the other does not then a
quasiisometry of $\Davis_\Gamma$ cannot take one
to the other.
Recall that we characterized cut vertices in \fullref{cut_vertex_in_mprg}.
    
\subsection{Examples}\label{sec:compliant_examples}
We give some example applications of \fullref{nocycles}.
\fullref{ex:cycleofcuts} gives some of the smallest triangle-free
strongly CFS graphs with compliant cycles. It turns out that all three
are already known to be non RAAGedy for other reasons.
After that we will give examples highlighting different aspects
of the theorem. 

\begin{example}\label{ex:cycleofcuts}
  \fullref{fig:compliantcycles} shows the smallest triangle-free CFS graphs with
  compliant cycles, which are therefore not RAAGedy, by
  \fullref{nocycles}.
  In each of these examples, $S_0$, $S_1$,  and $S_2$ are the three possible pairs of red
  vertices and the 
  $P_i$ are single red vertices that are the intersections of consecutive
  $S_i$, so $P:=\sqcup_iP_i$ is an anticlique. 
  In the first two the $S_i$ are cut pairs. In the third they are poles of maximal suspensions. 
 \begin{figure}[h]
   \centering
   \begin{subfigure}{.3\textwidth}
        \centering
    \includegraphics{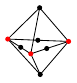}
\subcaption{Planar with cycle of cuts.}
\label{fig:planarcycleofcuts}
\end{subfigure}
\quad
\begin{subfigure}{.3\textwidth}
     \centering
  \includegraphics{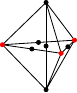}
\subcaption{Nonplanar with cycle of cuts.}
\label{fig:nonplanarcycleofcuts}
\end{subfigure}
\quad
\begin{subfigure}{.3\textwidth}
       \centering
 \includegraphics{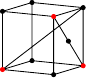}
\subcaption{Nonplanar,  with cycle of suspension poles.}
\label{fig:nonplanarcycleofpoles}
   \end{subfigure}
   \caption{Some CFS graphs with compliant cycles where the $S_i$ are
     pairs and $P$ is an anticlique.}
   \label{fig:compliantcycles}
 \end{figure}
 
The graph in \fullref{fig:planarcycleofcuts} is planar, so the fact
that it is not RAAGedy could have also been deduced 
by applying a theorem of Nguyen and Tran \cite{NguTra19}.
The graph in \fullref{fig:nonplanarcycleofcuts} is nonplanar, so
Nguyen-Tran does not apply, but it can be shown to be non-RAAGedy by
\fullref{no_cycle_of_cuts}. 
The graph in \fullref{fig:nonplanarcycleofpoles} has a JSJ graph of
cylinders with a single cylinder and single non-virtually
$\mathbb{Z}^2$ rigid subgroup connected by a virtually $\mathbb{Z}^2$
edge group, so is not RAAGedy by \fullref{raagjsj}.
\end{example}

\begin{example}
  Consider the graph $\Gamma$ of \fullref{fig:Gamma67}.
  The hypotheses of \fullref{nocycles} are satisfied for the path
  ${\color{green}P}=P_0:=(0,4,5,1)$ and the compliant set:
  \[{\color{red}S_0}:=\{2\}\join\{0,1\}=\{2,6\}\join\{0,1,7\}\cap\{2,8\}\join\{0,1,3\}\in\compliant_\Gamma^1\qedhere\]
  \begin{figure}[h]
    \centering
    \includegraphics{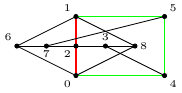}
\caption{An example $\Gamma$ such that \fullref{nocycles} can be
  satisfied by a single compliant subset $S_0$ and a single path $P=P_0$.}
\label{fig:Gamma67}
\end{figure}
\end{example}
\begin{example}\label{ex:Gamma2511}
  Consider the graph $\Gamma$ of \fullref{fig:Gamma2511}.
  By computer search, there is no single compliant set $S_0$ and connected
  $P$ satisfying the hypotheses of \fullref{nocycles}.
  For example, consider the maximal thick join
  ${\color{red}S_0}:=\{0,3\}\join\{1,7,8,9\}$.
  We would need a path $P$ whose endpoints are nonadjacent vertices of
  $S_0$ and whose interior vertices have link intersecting $S_0$ in a
  clique. The possible interior vertices are 4, 5, and
  $\ten$, which would mean the endpoints of $P$ would be adjacent
  vertices 0 and 9.

Instead, consider the maximal thick join
${\color{blue}S_1}:=\{0,2,6\}\join\{1,4,5\}$, which intersects $S_0$
in a clique, and paths 
  ${\color{green}P_0}:=(9,\ten,4)$ and
  $P_1:=(1)$.
 
  \begin{figure}[h]
    \centering
    \includegraphics{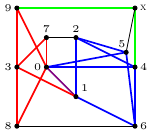}
    \caption{A graph with no compliant cycle with connected $P$, but
      having a compliant cycle consisting of two
      compliant sets and two connecting paths.}
    \label{fig:Gamma2511}
  \end{figure}
  In this example, no other method discussed in this paper works in order to show that it is not RAAGedy.
\end{example}

In the next two examples we use the considerations of \fullref{sec:not_same_orbit}  rule out condition \ref{item:gamma_prime_qi_to_gamma_zero} of \fullref{nocycles}.

\begin{example}\label{ex:Gamma375}
  Consider the graph $\Gamma$ in \fullref{fig:Gamma375}.
  Take ${\color{red}S_0}:=\{0,5\}\join\{2,6,7\}$, which is a maximal
  thick join that is a  non-square suspension.
  Consider $T:=\{2,6\}\join\{0,4,5\}$ and
  $T':=\{0,4\}\join\{1,2,3,6\}$, which are the two maximal thick joins that intersect $S_0$ in a
  non-clique.
  Both are non-square suspensions.
  The set of distinct pairs $(\llbracket
         \Upsilon\rrbracket,\llbracket
         \Upsilon\cap S_0\rrbracket)$, where $\Upsilon$ ranges
         over the neighbors of $S_0$ in $\ric_\Gamma$ consists of a
         single pair $(\text{tree}\times\text{line},\mathbb{E}^2)$.
         For $T$ we get the same result, but for $T'$ the answer is
         different: $T'$ has a neighbor in $\ric_\Gamma$ such that $(\llbracket
         \Upsilon\rrbracket,\llbracket
         \Upsilon\cap T'\rrbracket)=(\text{tree}\times\text{tree},\text{tree}\times\text{line})$,
         coming from $\Upsilon=\{1,2,3\}\join\{0,4,9\}$.
         Thus, we can see by \fullref{cor:ric_qi_obstruction} that $\Davis_{S_0}$ and $\Davis_{T'}$ are
         not  in the same orbit under the quasiisometry group of
         $\Davis_\Gamma$.
         Thus, we look for a compliant cycle based on $S_0$ such that
         $P$ avoids the vertex $\{4\}= T\setminus S_0$, but we are not forced
         to avoid $\{1,3\}=T'\setminus S_0$.
         ${\color{green}P}=P_0=(2,9,1,8,7)$ works.

        In this example, no other method discussed in this paper works in order to show that it is not RAAGedy.
  \begin{figure}[h]
    \centering
    \includegraphics{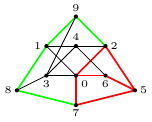}
    \caption{A graph with a compliant cycle where we need to recognize
    different quasiisometry orbits to verify the hypotheses of
    \fullref{nocycles} are satisfied.}
    \label{fig:Gamma375}
  \end{figure}
\end{example}

\begin{example}
  Consider the graph $\Gamma$ of \fullref{fig:cutvertex}.
Let $S_0:=\{2,3\}\join \{0,1,4,5\}$ and  
$S':=\{6,7\}\join\{4,5,8,9,14,15\}$, which are maximal thick joins, and
let $P:=(0,14,12,\elf,13,15,1)$ be a path.
For this $S_0$ and $P$, conditions
\eqref{item3:paths}--\eqref{item3:szerolinks} of \fullref{nocycles}
are satisfied.
Condition \ref{subcase:hyperbolic} is not true; $S_0$ contains
squares.
To conclude that $\Gamma$ is not RAAGedy we show that condition
\ref{subcase:no_match} is satisfied.
The only potential problem is $S'$, which fulfills
\ref{item:gamma_prime_cap_gamma_zero_incomplete}-\ref{item:end_pair}
of \ref{subcase:no_match},
so we need to show that \ref{item:gamma_prime_qi_to_gamma_zero} is
false by showing that $\Davis_{S_0}$ and $\Davis_{S'}$ are not in the
same orbit of the quasiisometry group of
$\Davis_\Gamma$. 

In this example \fullref{qi_implications} does not help, but \fullref{cut_vertex_in_mprg} does.

    \begin{figure}[h]
    \centering
    \raisebox{-9.5ex}{
    \includegraphics{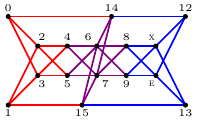}}
\quad\quad
\raisebox{-9.5ex}{
  \includegraphics{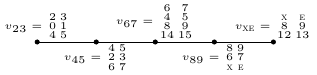}
   }
\caption{A graph $\Gamma$ and $\ric_\Gamma$ such that $\ric_\Gamma$
  contains a cut vertex of $\mprg_\Gamma$. The corresponding splitting
  of $\Gamma$ is illustrated by the red/violet and blue/violet subgraphs.
The cut vertex property is used to distinguish quasiisometry orbits of
maximal product regions in the
  construction of a compliant cycle.}
\label{fig:cutvertex}
\end{figure}
  Take $R_0$ and $R_1$ to be subgraphs of $\ric_\Gamma$ induced by
  vertex sets $\{v_{23}, v_{45},v_{67}\}$ and
  $\{v_{67},v_{89},v_{\ten\elf}\}$, respectively.
  As in \fullref{cut_vertex_in_mprg} for $i\in\{0,1\}$ let $\Gamma_i$
  be the subgraph of $\Gamma$ induced by $\bigcup_{v\in R_i}J_v$, so
  that $\Gamma_0$ is the red/violet subgraph and $\Gamma_1$ is the
  violet/blue subgraph, where the violet subgraph $J_{v_{67}}$ is their
  intersection.
  According to \fullref{cut_vertex_in_mprg}, the only cut vertices of
  $\mprg_\Gamma$ are those in the $W_\Gamma$--orbit of $v_{67}$.
  Quasiisometries of $\Davis_\Gamma$ induce automorphisms of
  $\mprg_\Gamma$, which preserve cut vertices, so  the
  $W_\Gamma$--orbit of $v_{67}$ coincides with its orbit under the
  action of the quasiisometry group of $\Davis_\Gamma$. 
Since $S_0$ corresponds to $v_{23}$ and $S'$ corresponds to $v_{67}$
we conclude that $\Davis_{S_0}$ and $\Davis_{S'}$ are not in the same orbit
of the quasiisometry group of $\Davis_\Gamma$. Thus,  \ref{item:gamma_prime_qi_to_gamma_zero} fails.
\end{example}

Sometimes passing to a link double is necessary to satisfy
\fullref{nocycles}:
\begin{example}\label{ex:Gamma8959}
  For the graph $\Gamma$ on the left of \fullref{fig:Gamma8959}, a
  computer search finds no configurations satisfying \fullref{nocycles}.
  After passing to  $\double^\circ_7(\Gamma)$ there is a good choice: ${\color{red}S_0}:=\{3_0,4_0\}\join\{0_0,2_0,5_0,8_0\}$ and
${\color{green}P}=P_0:=(0_0,\elf_0,\ten_1,6_1,9_1,2_0)$.
  \begin{figure}[h]
    \centering
    \raisebox{-7.5ex}{
    \includegraphics{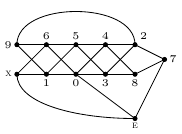}}
$\stackrel{\double^\circ_7}{\longrightarrow}$
\raisebox{-7.5ex}{
  \includegraphics{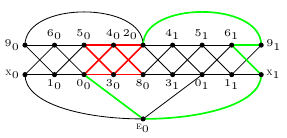}
   }
\caption{An example where taking a link double is necessary to satisfy
  the conditions of
  \fullref{nocycles}.}
\label{fig:Gamma8959}
\end{figure}
\end{example}

\begin{example}[Compliant cycle vs Morse boundary.]
 The graph of  \fullref{fig:planarcycleofcuts} has a compliant
 cycle, but its Morse boundary is totally disconnected, by
 \fullref{inductive_FioKar}.

 The graph of \fullref{fig:stable_no_compliant} displays a stable
 cycle (red), so its Morse boundary contains circles, but it has no
 compliant cycle (even after passing to a link double).
 \begin{figure}[h]
   \centering
   \includegraphics{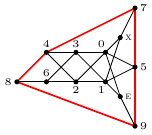}
   \caption{Graph with a stable cycle but no compliant cycle.}
   \label{fig:stable_no_compliant}
 \end{figure}
\end{example}

\section{Further questions}\label{sec:furtherexamples}

\begin{question}
 Quasiisometry versus commensurability:
  \begin{itemize}
  \item We have only shown that cloning and unfolding produce groups
    quasiisometric to the one we started with. Are they actually
    commensurable?
\item Does there exist a RACG that is not commensurable
  to any RAAG, but is quasiisometric to some RAAG?
  \end{itemize}
\end{question}
\begin{question}
  There are many other finite-index subgroups of RACGs than the
   ones we have considered. Would using these give any new results
   about commenurability between RACGs or between RACGs and RAAGs?
 \end{question}
 \begin{question}
 Generalize the dimension restrictions:
  \begin{itemize}
  \item Is Oh's theorem true when there are no 3--quasiflats (which is
    the case that $\Gamma$ is icosahedron-free)?
\item If yes, do all of our arguments generalize to this case?
\item Oh's theorem is not true in full generality in higher
  dimensions. What additional hypotheses on $\Gamma$ or $W_\Gamma$
  would we need to make to make it be true?
  Compare Huang assuming
  $\mathrm{Out}(A_\Delta)$ is finite. 
\end{itemize}
\end{question}
  \begin{question}
So far, we can answer RAAGedy/non-RAAGedy for all of the 533 triangle-free CFS
graphs with at most 10 vertices, and all but the following 8 of the
3405 
with 11 vertices. Which of them are RAAGedy?

$\Gamma_{1435}=\begin{tikzpicture}[scale=2,baseline=-10pt]\tiny
\coordinate[label={[label distance=0pt] -90:$0$}] (0) at (0,0);
\coordinate[label={[label distance=0pt] -90:$1$}] (1) at (0,-0.2);
\coordinate[label={[label distance=0pt] -45:$2$}] (2) at (-0.2,-0.2);
\coordinate[label={[label distance=0pt] 45:$3$}] (3) at (-0.2,0);
\coordinate[label={[label distance=0pt] 135:$4$}] (4) at (0.2,0);
\coordinate[label={[label distance=0pt] -135:$5$}] (5) at (0.2,-0.2);
\coordinate[label={[label distance=0pt] 180:$6$}] (6) at (-0.4,.2);
\coordinate[label={[label distance=0pt] 0:$7$}] (7) at (0.4,.2);
\coordinate[label={[label distance=0pt] 0:$8$}] (8) at (0.4,-0.4);
\coordinate[label={[label distance=0pt] 180:$9$}] (9) at (-0.4,-0.4);
\coordinate[label={[label distance=0pt] 90:$\ten$}] (10) at (0,0.2);
\draw (0)--(2);
\draw (0)--(3);
\draw (0)--(4);
\draw (0)--(5);
\draw (0)--(10);
\draw (1)--(2);
\draw (1)--(3);
\draw (1)--(4);
\draw (1)--(5);
\draw (2)--(6);
\draw (2)--(9);
\draw (3)--(6);
\draw (3)--(9);
\draw (4)--(7);
\draw (4)--(8);
\draw (5)--(7);
\draw (5)--(8);
\draw (6)--(10);
\draw (7)--(10);
\draw (8)--(9);
\filldraw (0) circle (.5pt);
\filldraw (1) circle (.5pt);
\filldraw (2) circle (.5pt);
\filldraw (3) circle (.5pt);
\filldraw (4) circle (.5pt);
\filldraw (5) circle (.5pt);
\filldraw (6) circle (.5pt);
\filldraw (7) circle (.5pt);
\filldraw (8) circle (.5pt);
\filldraw (9) circle (.5pt);
\filldraw (10) circle (.5pt);
\end{tikzpicture}$
\quad
$\Gamma_{1441}=\begin{tikzpicture}[scale=2,baseline=-5pt]\tiny
\coordinate[label={[label distance=0pt] -90:$0$}] (0) at (-0.1,-.1);
\coordinate[label={[label distance=0pt] 90:$1$}] (1) at (-0.1,0.1);
\coordinate[label={[label distance=0pt] -90:$2$}] (2) at (0.1,-.1);
\coordinate[label={[label distance=0pt] 135:$3$}] (3) at (0.1,0.1);
\coordinate[label={[label distance=0pt] 90:$4$}] (4) at (0.3,0.3);
\coordinate[label={[label distance=0pt] -90:$5$}] (5) at (0.075,-0.3);
\coordinate[label={[label distance=0pt] 90:$6$}] (6) at (0.3,0);
\coordinate[label={[label distance=0pt] 45:$7$}] (7) at (-0.3,0.1);
\coordinate[label={[label distance=0pt] 0:$8$}] (8) at (0.5,0);
\coordinate[label={[label distance=0pt] 90:$9$}] (9) at (-0.5,0.3);
\coordinate[label={[label distance=0pt] -90:$\ten$}] (10) at (-0.30,-.1);
\draw (0)--(2);
\draw (0)--(3);
\draw (0)--(5);
\draw (0)--(7);
\draw (0)--(10);
\draw (1)--(2);
\draw (1)--(3);
\draw (1)--(5);
\draw (1)--(7);
\draw (1)--(10);
\draw (2)--(4);
\draw (2)--(6);
\draw (3)--(4);
\draw (3)--(6);
\draw (4)--(8);
\draw (4)--(9);
\draw (5)--(8);
\draw (6)--(8);
\draw (7)--(9);
\draw (9)--(10);
\filldraw (0) circle (.5pt);
\filldraw (1) circle (.5pt);
\filldraw (2) circle (.5pt);
\filldraw (3) circle (.5pt);
\filldraw (4) circle (.5pt);
\filldraw (5) circle (.5pt);
\filldraw (6) circle (.5pt);
\filldraw (7) circle (.5pt);
\filldraw (8) circle (.5pt);
\filldraw (9) circle (.5pt);
\filldraw (10) circle (.5pt);
\end{tikzpicture}$
\quad
$\Gamma_{1454}=\begin{tikzpicture}[scale=2,baseline=-3pt]\tiny
\coordinate[label={[label distance=0pt] -90:$0$}] (0) at (0,-0.1);
\coordinate[label={[label distance=0pt] 90:$1$}] (1) at (0,.1);
\coordinate[label={[label distance=0pt,yshift=1pt] -135:$2$}] (2) at (0.2,-0.1);
\coordinate[label={[label distance=0pt,yshift=-1pt] 135:$3$}] (3) at (0.2,.1);
\coordinate[label={[label distance=0pt] -180:$4$}] (4) at (-0.2,0);
\coordinate[label={[label distance=0pt] 0:$5$}] (5) at (0.4,-0.1);
\coordinate[label={[label distance=0pt] 180:$6$}] (6) at (-0.3,0.4);
\coordinate[label={[label distance=0pt] 180:$7$}] (7) at (-0.3,-0.4);
\coordinate[label={[label distance=-2pt] -45:$8$}] (8) at (-0.2,-0.2);
\coordinate[label={[label distance=-2pt] 45:$9$}] (9) at (-0.2,0.2);
\coordinate[label={[label distance=0pt] 0:$\ten$}] (10) at (0.4,.1);
\draw (0)--(2);
\draw (0)--(3);
\draw (0)--(4);
\draw (0)--(8);
\draw (0)--(9);
\draw (1)--(2);
\draw (1)--(3);
\draw (1)--(4);
\draw (1)--(8);
\draw (1)--(9);
\draw (2)--(5);
\draw (2)--(10);
\draw (3)--(5);
\draw (3)--(10);
\draw (4)--(6);
\draw (4)--(7);
\draw (5)--(7);
\draw (6)--(9);
\draw (6)--(10);
\draw (7)--(8);
\filldraw (0) circle (.5pt);
\filldraw (1) circle (.5pt);
\filldraw (2) circle (.5pt);
\filldraw (3) circle (.5pt);
\filldraw (4) circle (.5pt);
\filldraw (5) circle (.5pt);
\filldraw (6) circle (.5pt);
\filldraw (7) circle (.5pt);
\filldraw (8) circle (.5pt);
\filldraw (9) circle (.5pt);
\filldraw (10) circle (.5pt);
\end{tikzpicture}$

$\Gamma_{2708}=\begin{tikzpicture}[scale=2,baseline=-3pt]\tiny
\coordinate[label={[label distance=0pt] 90:$0$}] (0) at (-0.2,0.1);
\coordinate[label={[label distance=0pt] -90:$1$}] (1) at (-0.2,-0.1);
\coordinate[label={[label distance=0pt] 0:$2$}] (2) at (0.2,-0.1);
\coordinate[label={[label distance=0pt] 90:$3$}] (3) at (0,0.1);
\coordinate[label={[label distance=0pt] -90:$4$}] (4) at (0.05,-0.4);
\coordinate[label={[label distance=0pt] 0:$5$}] (5) at (0.2,0.1);
\coordinate[label={[label distance=0pt,xshift=1pt] -90:$6$}] (6) at (0,-0.10);
\coordinate[label={[label distance=0pt] 180:$7$}] (7) at (-0.4,-0.10);
\coordinate[label={[label distance=-3pt] -135:$8$}] (8) at (-0.05,-0.2);
\coordinate[label={[label distance=0pt] 90:$9$}] (9) at (0.01,0.4);
\coordinate[label={[label distance=0pt] 180:$\ten$}] (10) at (-0.4,0.1);
\draw (0)--(3);
\draw (0)--(6);
\draw (0)--(7);
\draw (0)--(8);
\draw (0)--(10);
\draw (1)--(3);
\draw (1)--(6);
\draw (1)--(7);
\draw (1)--(8);
\draw (1)--(10);
\draw (2)--(3);
\draw (2)--(4);
\draw (2)--(6);
\draw (2)--(9);
\draw (3)--(5);
\draw (4)--(5);
\draw (4)--(7);
\draw (4)--(8);
\draw (5)--(6);
\draw (5)--(9);
\draw (9)--(10);
\filldraw (0) circle (.5pt);
\filldraw (1) circle (.5pt);
\filldraw (2) circle (.5pt);
\filldraw (3) circle (.5pt);
\filldraw (4) circle (.5pt);
\filldraw (5) circle (.5pt);
\filldraw (6) circle (.5pt);
\filldraw (7) circle (.5pt);
\filldraw (8) circle (.5pt);
\filldraw (9) circle (.5pt);
\filldraw (10) circle (.5pt);
\end{tikzpicture}$
\quad
$\Gamma_{2722}=\begin{tikzpicture}[scale=2,baseline=-3pt]\tiny
\coordinate[label={[label distance=0pt] -90:$0$}] (0) at (0.04,-0.1);
\coordinate[label={[label distance=0pt] 90:$1$}] (1) at (0.04,0.1);
\coordinate[label={[label distance=0pt] 90:$2$}] (2) at (-0.20,-.02);
\coordinate[label={[label distance=0pt] -163.42:$3$}] (3) at (-0.49,-0.3);
\coordinate[label={[label distance=0pt] -90:$4$}] (4) at (0.16,-0.1);
\coordinate[label={[label distance=0pt] 90:$5$}] (5) at (0.15,0.1);
\coordinate[label={[label distance=-1pt] 45:$6$}] (6) at (-0.20,0.15);
\coordinate[label={[label distance=-1pt] -45:$7$}] (7) at (-0.20,-0.15);
\coordinate[label={[label distance=0pt] -21.17:$8$}] (8) at (0.46,-0.3);
\coordinate[label={[label distance=0pt] 19.21:$9$}] (9) at (0.45,0.3);
\coordinate[label={[label distance=0pt] 151.92:$\ten$}] (10) at (-0.49,0.3);
\draw (0)--(2);
\draw (0)--(4);
\draw (0)--(5);
\draw (0)--(6);
\draw (0)--(7);
\draw (1)--(2);
\draw (1)--(4);
\draw (1)--(5);
\draw (1)--(6);
\draw (1)--(7);
\draw (2)--(3);
\draw (2)--(10);
\draw (3)--(6);
\draw (3)--(7);
\draw (3)--(8);
\draw (4)--(8);
\draw (4)--(9);
\draw (5)--(8);
\draw (5)--(9);
\draw (6)--(10);
\draw (9)--(10);
\filldraw (0) circle (.5pt);
\filldraw (1) circle (.5pt);
\filldraw (2) circle (.5pt);
\filldraw (3) circle (.5pt);
\filldraw (4) circle (.5pt);
\filldraw (5) circle (.5pt);
\filldraw (6) circle (.5pt);
\filldraw (7) circle (.5pt);
\filldraw (8) circle (.5pt);
\filldraw (9) circle (.5pt);
\filldraw (10) circle (.5pt);
\end{tikzpicture}$
\quad
$\Gamma_{3002}=\begin{tikzpicture}[scale=2,baseline=-3pt]\tiny
\coordinate[label={[label distance=0pt] 180:$0$}] (0) at (-0.30,0.125);
\coordinate[label={[label distance=0pt] 180:$1$}] (1) at (-0.30,-0.125);
\coordinate[label={[label distance=0pt] -90:$2$}] (2) at (0.02,-0.3);
\coordinate[label={[label distance=2pt] 0:$3$}] (3) at (0.02,0);
\coordinate[label={[label distance=0pt] 90:$4$}] (4) at (0.01,0.3);
\coordinate[label={[label distance=0pt] 90:$5$}] (5) at (0.03,0.125);
\coordinate[label={[label distance=0pt] -90:$6$}] (6) at (0.02,-0.125);
\coordinate[label={[label distance=0pt] 0:$7$}] (7) at (0.26,0.125);
\coordinate[label={[label distance=0pt] 0:$8$}] (8) at (0.41,0.25);
\coordinate[label={[label distance=0pt] 0:$9$}] (9) at (0.25,-0.125);
\coordinate[label={[label distance=0pt] 0:$\ten$}] (10) at (0.42,-0.25);
\draw (0)--(2);
\draw (0)--(3);
\draw (0)--(4);
\draw (0)--(5);
\draw (0)--(6);
\draw (1)--(2);
\draw (1)--(3);
\draw (1)--(4);
\draw (1)--(5);
\draw (1)--(6);
\draw (2)--(9);
\draw (2)--(10);
\draw (3)--(7);
\draw (3)--(9);
\draw (4)--(7);
\draw (4)--(8);
\draw (5)--(7);
\draw (5)--(8);
\draw (6)--(9);
\draw (6)--(10);
\draw (8)--(10);
\filldraw (0) circle (.5pt);
\filldraw (1) circle (.5pt);
\filldraw (2) circle (.5pt);
\filldraw (3) circle (.5pt);
\filldraw (4) circle (.5pt);
\filldraw (5) circle (.5pt);
\filldraw (6) circle (.5pt);
\filldraw (7) circle (.5pt);
\filldraw (8) circle (.5pt);
\filldraw (9) circle (.5pt);
\filldraw (10) circle (.5pt);
\end{tikzpicture}$

$\Gamma_{3632}=\begin{tikzpicture}[scale=2,baseline=4pt]\tiny
\coordinate[label={[label distance=0pt] 180:$0$}] (0) at (-0.2,0);
\coordinate[label={[label distance=0pt,xshift=-2pt] -90:$1$}] (1) at (0.2,0);
\coordinate[label={[label distance=0pt] 90:$2$}] (2) at (0.2,0.333);
\coordinate[label={[label distance=0pt] 90:$3$}] (3) at (-0.2,0.25);
\coordinate[label={[label distance=0pt] 90:$4$}] (4) at (-0.2,0.5);
\coordinate[label={[label distance=0pt] 90:$5$}] (5) at (0.2,0.5);
\coordinate[label={[label distance=0pt] 0:$6$}] (6) at (0.4,0.125);
\coordinate[label={[label distance=0pt] 90:$7$}] (7) at (0.2,.167);
\coordinate[label={[label distance=0pt] 0:$8$}] (8) at (0.4,0.375);
\coordinate[label={[label distance=0pt] -112.00:$9$}] (9) at (0.1,-0.25);
\coordinate[label={[label distance=0pt] -180:$\ten$}] (10) at (-0.45,0);
\draw (0)--(1);
\draw (0)--(2);
\draw (0)--(5);
\draw (0)--(7);
\draw (0)--(9);
\draw (1)--(3);
\draw (1)--(4);
\draw (1)--(6);
\draw (1)--(8);
\draw (2)--(3);
\draw (2)--(4);
\draw (2)--(6);
\draw (2)--(8);
\draw (3)--(5);
\draw (3)--(7);
\draw (3)--(10);
\draw (4)--(5);
\draw (4)--(7);
\draw (4)--(10);
\draw (5)--(8);
\draw (6)--(7);
\draw (6)--(9);
\draw (9)--(10);
\filldraw (0) circle (.5pt);
\filldraw (1) circle (.5pt);
\filldraw (2) circle (.5pt);
\filldraw (3) circle (.5pt);
\filldraw (4) circle (.5pt);
\filldraw (5) circle (.5pt);
\filldraw (6) circle (.5pt);
\filldraw (7) circle (.5pt);
\filldraw (8) circle (.5pt);
\filldraw (9) circle (.5pt);
\filldraw (10) circle (.5pt);
\end{tikzpicture}$
\quad
$\Gamma_{3794}=
\begin{tikzpicture}[scale=2,baseline=-3pt]\tiny
\coordinate[label={[label distance=0pt] -179.99:$0$}] (0) at (-.25,-0.00);
\coordinate[label={[label distance=0pt] -90:$1$}] (1) at (.5,0.25);
\coordinate[label={[label distance=0pt] -90:$2$}] (2) at (-0.00,-0.2);
\coordinate[label={[label distance=0pt] 0:$3$}] (3) at (.5,0.4);
\coordinate[label={[label distance=0pt] -90:$4$}] (4) at (-0.0,-0.00);
\coordinate[label={[label distance=0pt] 89.49:$5$}] (5) at (0.00,0.2);
\coordinate[label={[label distance=0pt] 0:$6$}] (6) at (1,-0.2);
\coordinate[label={[label distance=-2pt] 45:$7$}] (7) at (.5,-0);
\coordinate[label={[label distance=0pt] 0:$8$}] (8) at (.5,-0.2);
\coordinate[label={[label distance=0pt] 180:$9$}] (9) at (-0.00,-0.40);
\coordinate[label={[label distance=0pt] 180:$\textsc{x}$}] (10) at (-0.0,0.4);
\draw (0)--(2);
\draw (0)--(4);
\draw (0)--(5);
\draw (0)--(9);
\draw (0)--(10);
\draw (1)--(2);
\draw (1)--(4);
\draw (1)--(5);
\draw (1)--(6);
\draw (1)--(10);
\draw (2)--(3);
\draw (2)--(7);
\draw (2)--(8);
\draw (3)--(4);
\draw (3)--(5);
\draw (3)--(6);
\draw (3)--(10);
\draw (4)--(7);
\draw (4)--(8);
\draw (5)--(7);
\draw (5)--(8);
\draw (6)--(7);
\draw (6)--(9);
\draw (8)--(9);
\filldraw (0) circle (.5pt);
\filldraw (1) circle (.5pt);
\filldraw (2) circle (.5pt);
\filldraw (3) circle (.5pt);
\filldraw (4) circle (.5pt);
\filldraw (5) circle (.5pt);
\filldraw (6) circle (.5pt);
\filldraw (7) circle (.5pt);
\filldraw (8) circle (.5pt);
\filldraw (9) circle (.5pt);
\filldraw (10) circle (.5pt);
\end{tikzpicture}$
\end{question}


\bibliographystyle{hypersshort}
\bibliography{RAAGedyRACGs}
\end{document}